\documentclass[11pt]{book}
\usepackage{setspace}
\usepackage{amsmath,amssymb,mathrsfs}
\usepackage{amsthm}
\usepackage{algorithm, algpseudocode}
\usepackage{url}
\usepackage{fullpage}
\usepackage{makeidx}
\usepackage[vmargin=2.5cm]{geometry}
\usepackage{graphicx,float,psfrag,epsfig}
\usepackage{color}
\usepackage[font=small,justification=centering]{caption}
\usepackage[font=footnotesize,justification=centering]{subcaption}
\RequirePackage[colorlinks,citecolor=blue,urlcolor=blue,linkcolor=blue,linktocpage=true]{hyperref}
\usepackage{appendix}

\makeatletter
\newenvironment{chapquote}[2][2em]
  {\setlength{\@tempdima}{#1}%
  \def\chapquote@author{#2}%
   \parshape 1 \@tempdima \dimexpr\textwidth-1\@tempdima\relax%
   \itshape}
  {\par\normalfont \hspace{8cm}--- \chapquote@author\hspace*{\@tempdima}\par\bigskip}
\makeatother

\numberwithin{equation}{section}

\newtheoremstyle{myexample} 
    {\topsep}                    
    {\topsep}                    
    {\rm }                   
    {}                           
    {\bf }                   
    {.}                          
    {.5em}                       
    {}  

\newtheoremstyle{myremark} 
    {\topsep}                    
    {\topsep}                    
    {\rm}                        
    {}                           
    {\bf}                        
    {.}                          
    {.5em}                       
    {}  

\newtheorem{claim}{Claim}[section]
\newtheorem{lemma}[claim]{Lemma}

\newtheorem{conjecture}[claim]{Conjecture}
\newtheorem{theorem}{Theorem}
\newtheorem{proposition}[claim]{Proposition}
\newtheorem{corollary}[claim]{Corollary}
\newtheorem{definition}[claim]{Definition}

\theoremstyle{myremark}
\newtheorem{remark}{Remark}[section]

\theoremstyle{myremark}

\theoremstyle{myexample}

\newtheorem{exercise}{Exercise}[chapter]

\def\<{\langle}
\def\>{\rangle}

\def\atanh{{\rm atanh}}
\def\argmin{{\rm argmin}}

\def\det{{\rm det}}

\def\oh{\overline{h}}
\def\Entro{{\rm Ent}}
\def\sGauss{\mbox{\tiny\rm Gauss}}
\def\sBayes{\mbox{\tiny\rm Bayes}}
\def\sML{\mbox{\tiny\rm ML}}
\def\snew{\mbox{\tiny\rm new}}
\def\sRS{\mbox{\scriptsize\rm RS}}
\newcommand\sRSB[1]{\mbox{\scriptsize\rm {$#1$}RSB}}

\def\sF{{\sf F}}

\def\Info{{\rm I}}

\def\eps{{\varepsilon}}

\def\id{{\boldsymbol I}}
\def\sS{{\sf S}}
\def\sT{{\sf T}}
\def\proj{{\boldsymbol P}}

\def\bLambda{{\boldsymbol \Lambda}}
\def\bM{{\boldsymbol M}}
\def\bQ{{\boldsymbol Q}}
\def\bm{{\boldsymbol m}}
\def\bq{{\boldsymbol q}}
\def\tbQ{\tilde{\boldsymbol Q}}
\def\det{{\rm det}}
\def\hQ{\widehat{Q}}
\def\hbQ{\widehat{\boldsymbol Q}}
\def\hbM{\widehat{\boldsymbol M}}

\def\tu{\tilde{u}}
\def\cQ{\mathcal{Q}}
\def\Zrep{\mathcal{Z}}
\def\tZrep{\widetilde{\mathcal{Z}}}
\def\ob{\overline{b}}
\def\oq{\overline{q}}
\def\onu{\overline{\nu}}

\def\Symm{{\mathfrak S}}
\def\normal{{\sf N}}
\def\GOE{{\sf GOE}}

\def\Tr{{\rm {Tr}}}

\def\de{{\rm d}}

\def\cR{\mathcal{R}}

\def\cG{\mathcal{G}}
\def\d{{\mathrm{d}}}

\def\cF{{\mathcal F}}
\def\cZ{{\mathcal Z}}

\def\bfone{{\bf 1}}
\def\bfzero{{\bf 0}}

\def\hd{{\sf h}}

\def\cX{{\cal X}}
\def\cF{{\cal F}}

\def\bG{{\bf G}}

\def\ed{\stackrel{{\rm d}}{=}}

\def\sd{{\sf d}}

\def\cuP{\mathscr{P}}
\def\reals{\mathbb{R}}
\def\naturals{\mathbb{N}}
\def\integers{\mathbb{Z}}

\def\cE{{\mathcal{E}}}

\def\bA{{\boldsymbol A}}
\def\bD{{\boldsymbol D}}
\def\bG{{\boldsymbol G}}

\def\bI{{\boldsymbol I}}
\def\bW{{\boldsymbol W}}
\def\bJ{{\boldsymbol J}}

\def\bP{{\boldsymbol P}}
\def\bS{{\boldsymbol S}}
\def\bT{{\boldsymbol T}}
\def\bX{{\boldsymbol X}}
\def\bY{{\boldsymbol Y}}

\def\bu{{\boldsymbol u}}

\def\hbx{\hat{\boldsymbol x}}
\def\bx{{\boldsymbol x}}
\def\by{{\boldsymbol y}}
\def\bz{{\boldsymbol z}}
\def\blambda{{\boldsymbol \lambda}}

\def\bg{{\boldsymbol g}}
\def\bv{{\boldsymbol v}}
\def\bvz{{\boldsymbol v_0}}
\def\hbv{\widehat{\boldsymbol v}}
\def\bw{{\boldsymbol w}}

\def\bXz{{\boldsymbol X_0}}
\def\hbX{\widehat{\boldsymbol X}}

\def\bsigma{{\boldsymbol \sigma}}

\def\bh{{\boldsymbol h}}

\def\Poisson{{\rm Poisson}}

\def\prob{{\mathbb P}}
\def\E{{\mathbb E}}

\def\rP{{\rm P}}
\def\root{o}

\def\Ball{{\sf B}}

\def\Treg[#1]{T^{{\rm reg},#1}}
\def\GW[#1]{{\rm GW}(#1)}
\def\MGW[#1]{{\rm MGW}(#1)}

\def\sP{{\rm P}}
\def\sSG{{\rm SG}}
\def\sR{{\rm R}}

\def\P{{\mathbb P}}

\def\MSE{{\sf MSE}}
\def\mse{{\sf mse}}
\def\supp{{\rm supp}}

\def\ER{Erd\H{o}s-R\'{e}nyi\,} 

\title{A Friendly Tutorial on Mean-Field Spin Glass\\
Techniques for Non-Physicists}
 
\author{Andrea~Montanari\thanks{Department of Electrical 
Engineering and Department of Statistics, Stanford University} and Subhabrata Sen\thanks{Department of Statistics, Harvard University}} 

\makeindex

\begin{document}

\frontmatter

\maketitle
\tableofcontents

\mainmatter

\chapter*{Introduction}
\addcontentsline{toc}{chapter}{\protect\numberline{}Introduction}
\label{ch:Intro}
\begin{chapquote}[2.5cm]{\textup{Archimedes}, On the method}
I thought fit to write out for you and explain in detail in the same book the 
peculiarity of a certain method, by which it will be possible for you to get a start 
to enable you to investigate some of the problems in mathematics by means of mechanics.
[\dots] it is of course easier, when we have previously acquired, by the method, some 
knowledge of the questions, to supply the proof than it is to find it without any previous knowledge.
\end{chapquote}

\phantom{A}

This tutorial is based on lecture notes written for a
class taught in the Statistics Department at Stanford in the Winter Quarter of 2017
and then again in Fall 2021. The class was called `Methods From Statistical Physics (Stats 369)'
and was addressed to students from Statistics, Mathematics, and 
Engineering Departments with a solid background in probability theory, but no previous 
knowledge of physics.
The objective  was to provide a working knowledge of some of the techniques developed over the last
40 years by theoretical physicists and mathematicians to study mean field spin glasses
and their applications to high-dimenensional statistics and statistical learning.

\section*{Why mean field spin glasses?}

The history of the subject in truly remarkable\footnote{The reader interested in early historical overviews 
might wish to  consult the sequence of seven expository articles written by Phil Anderson, between 1988 \cite{anderson1988spin} 
and 1990 \cite{anderson1990reference}, or of course the unsurpassed collection of articles 
(and accompanying introductions) in \cite{SpinGlass}. For a personal
account, see \cite{VirasoroInterview}.}. 
Spin glass models were introduced by physicists in the 1970s 
to model the statistical properties of certain magnetic materials. Over the last 
half century, these models have motivated a blossoming line of mathematical work
with applications to multiple fields, at first sight distant from physics. 

From a mathematical point of view, spin glasses are high-dimensional probability
distributions, i.e.\ 
probability distributions over $\reals^n$, $n\gg 1$, which 
are usually written in the Gibbs-Boltzmann form 
\begin{align}
\mu(\de\bsigma) \propto e^{\beta H(\bsigma)} \, \nu_0(\de\bsigma)\, .\label{eq:FirstBoltzmann}
\end{align}
Here $\nu_0$ is a `simple' reference measure (for instance the uniform distribution over 
$\{+1,-1\}^n$ or the uniform distribution over the sphere $\sS^{n-1}$),
the exponential weight $H:\reals^n\to \reals$  is known as the Hamiltonian 
(the standard convention in physics is to refer to $-H$ as the Hamiltonian),
and $\beta>0$ is known as the inverse temperature.

Of course, any probability measure in $\reals^n$ can be written in the form 
\eqref{eq:FirstBoltzmann}. 
However in spin glass models, $H(\bsigma)$ is typically a sum of  polynomially many (in $n$)
 `simple' terms (e.g. low degree monomials in the coordinates $\sigma_i$),
and hence the form \eqref{eq:FirstBoltzmann} is meaningful.
It is worth mentioning that, while we focus for simplicity on probability measures over $\reals^n$, 
spin glass models have been studied in other product spaces $\cX^n$ as well.

In spin glass models, the Hamiltonian $H(\,\cdot\,)$ itself is random or, to be precise,
 $\{H(\bsigma)\}_{\bsigma\in\reals^n}$ is a stochastic process indexed by $\bsigma\in \reals^n$
(or $\bsigma\in \Sigma$, where $\Sigma$ denotes the support of $\nu_0$).
 Therefore the measure $\mu(\de\bsigma)$ is a random probability measure.
 A specific spin glass model is defined by specifying the distribution of the process 
 $H$ (alongside the reference measure $\nu_0$). 
 
 At first sight, studying random probability measures might seem a somewhat exotic
 endeavour. However, a little thought reveals that random probability measures
 are both ubiquitous and useful:
 \begin{enumerate}
 \item Consider a statistical estimation problem: we want to estimate an unknown 
 vector $\bx\in\reals^n$ from some data $\bY$. We will see several examples of this problem in 
 this tutorial. A Bayesian approach postulates a prior distribution over $\bx$,
 and then forms a posterior, namely the conditional distribution of $\bx$ given $\bY$,
 under that prior. The posterior is a random probability distribution over $\reals^n$ 
 (because $\bY$ is random).
 \item Consider an optimization problem of the form $\max_{\bsigma\in\Sigma} H(\bsigma)$.
 In many circumstances, we have a probabilistic model for $H$. For instance, this is the case in 
 empirical risk minimization, which is the standard approach to statistical learning.
 Of course, in order to understand the properties of this optimization problem,
 it is important to understand the geometry of the (random) superlevel sets 
 $\Sigma_E:=\{\bsigma\in\Sigma:\; H(\bsigma)\ge E\}$. It turns out that---in many cases---the distribution \eqref{eq:FirstBoltzmann} is closely related to the uniform distribution 
 over $\Sigma_{E(\beta)}$ for a certain $E(\beta)$.
 Therefore, the Gibbs measure \eqref{eq:FirstBoltzmann} is a powerful tool to explore 
 the random geometry of superlevel sets.
  \item In physics, the Hamiltonian $H$ is random because, for instance, the spins
 $\sigma_i$ are magnetic moments associated to impurities at random locations
 in an otherwise non-magnetic material. More generally, $H$ can be random because it is the 
 Hamiltonian of a disordered system, whose `disorder' degrees of freedom are 
 not thermalized.
 \end{enumerate}

 \emph{Mean field} spin glass models are a special subfamily of spin glass models.
 Roughly speaking, they are characterized by the fact that the coordinates of $\bsigma$ 
 (the `spins' in physics language) are `indistinguishable' from the point of view 
 of the process\footnote{Formally, for any fixed number of vectors
 $\bsigma_1,\dots,\bsigma_k$, the joint distribution of 
 $(H(\bsigma_1),\dots,H(\bsigma_k))$ depends on
 $\bsigma_1,\dots,\bsigma_k$ only through the joint empirical distribution of their  
 coordinates $n^{-1}\sum_{i=1}^n\delta_{\sigma_{1,i},\dots,\sigma_{k,i}}$.}
 $H(\,\cdot\,)$. 
 The first model of this type, which we will 
 study in Chapter \ref{ch:SK}, was first introduced by David Sherrington and Scott Kirkpatrick in 1975 
  \cite{sherrington1975solvable} to provide an idealized model
  that could be amenable to mathematical analysis. 
  
  In the near half-century since then, 
  and starting with the invention of `replica symmetry breaking' by Giorgio Parisi
  \cite{parisi1979infinite}, physicists have developed a number of sophisticated 
  non-rigorous techniques to analyze mean field spin glasses and characterize their high-dimensional 
  behavior.
 While several of the physicists' results and technique are outstanding challenges for mathematicians, 
 since the early 2000's, there has been  increasing success in rigorizing some 
 of these ideas. These developments have given birth to a rich and rapidly evolving 
 area of probability theory.
 
 In parallel with these developments, it has become increasingly clear that
  understanding the behavior of random high-dimensional
 probability distributions $\mu(\de\bsigma)$ and of random high-dimensional objective 
 functions $H(\bsigma)$ is crucial in a number of mathematical disciplines beyond 
 theoretical physics. 
We mentioned above high-dimensional statistics and statistical learning: tools 
and intuitions from physics have found countless applications in these areas in the last 10-15 years. 
In the opposite direction, high-dimensional statistics and statistical learning have brought new 
questions and stimulated new developments in spin glass theory.

This tutorial is mainly aimed at researchers in statistics, mathematics, computer science, 
who want to learn some of the important tools and ideas in this area.
 
\section*{The style of this tutorial}

This tutorial is deliberately written in a somewhat non-standard style, 
from several viewpoints:
\begin{description}
\item[Concrete problems.] Rather than developing the theory in the most general setting, 
we focus on two concrete problems that are motivated by questions in statistical estimation.
Each of the next two chapters is dedicated to one such problem.

We use each of these examples as a pretext for presenting a number of mathematical techniques.
We believe it is best to learn these technique on concrete applications.
\item[Non-exhaustive.] Our treatment is far from exhaustive, even for each of the specific 
problems that we treat. On the other hand, while we use these examples as motivation,
we do not hesitate in pursuing detours that are interesting, but indirectly related to the 
original questions posed by those examples.
\item[Rigorous vs.\ non-rigorous techniques.] We present a mixture of non-rigorous
and rigorous techniques. Whenever something is proven (or a proof in the literature is indicated,
or sketched) we emphasize this by using the labels `theorem', `lemma', and so on.
 On the other hand,  we explain non-rigorous techniques on examples for which rigorous alternatives (yielding 
the same conclusions) are available.

There are countless reasons for learning non-rigorous techniques
in parallel with rigorous ones. Among others: $(i)$~They have driven this research area;
$(ii)$~Properly used, they provide the correct answer with significantly less work; 
$(iii)$~They apply more broadly; $(iv)$~They can provide invaluable insights/conjectures
 for rigorous research.
 
 As shown by the quote above, reason $(iv)$ was already acknowledged by
  Archimedes more than two millennia ago.
\item[Exercises.] As explained above, this tutorial is based on a class taught at Stanford.
We include the exercises developed for that class (often generalizations of the models treated
in the main text). The importance of hands-on practice in mathematics cannot be overstated.
Even more so when learning a completely different point of view (e.g.\ learning non-rigorous techniques
for a mathematically minded person, or vice versa).
\end{description}

\section*{Acknowledgements}

Teaching a class is always a dialogue between the teacher(s) and the student(s). Even more
so for a small specialized class like this. Therefore, we are primarily indebted to 
the students who took Stats 369 at Stanford:
Jie Jun Ang, Mona
Azadkia,
Erik
Bates,
Zhou
Fan,
Henry
Friedlander,
Yanjun
Han,
Sifan
Liu,
Song
Mei,
Aran
Nayebi,
Phan Minh
Nguyen,
Allan
Raventos Knohr,
David
Ritzwoller,
Christian
Serio,
Pragya
Sur,
Matteo
Sesia,
Javan
Tahir,
Andy
Tsao,
Atsushi
Yamamura,
Guanyang
Wang,
Jun
Yan,
Zhenyuan
Zhang,
Kangjie
Zhou,
Zhengqing
Zhou.

We also acknowledge funding from the NSF grant CCF-2006489, the ONR grant N00014-18-1-2729
(AM) and a Harvard Dean's Competitive Fund  Award(SS).

\chapter{The $p$-spin model and tensor PCA}
\label{ch:Pspin}

\section{Introduction}

In many modern data analysis problems, data take the form of a multi-dimensional array
\cite{morup2011applications}. Examples include collaborative filtering
\cite{barak2016noisy},
 applications of the moment method \cite{hsu2012spectral},
 image inpainting \cite{liu2013tensor}, hyperspectral imaging 
   \cite{li2010tensor,signoretto2011tensor}, and geophysical imaging \cite{kreimer2013tensor}.

 Information is extracted from such data by positing a latent low-rank structure.
Tensor principal component analysis (PCA) provides a simple formalization of this problem,
which has attracted significant interest over the last few years. 

It also provides an excellent sandbox for learning spin glass techniques, because the 
corresponding spin glass model (known as the $p$-spin spherical model) has been 
studied in physics for over 20 years, as discussed in the bibliography notes at the end of this 
chapter.

\subsection{Statistical model and Gibbs measure}
\label{sec:Pspin}

Consider the following signal-plus-noise model for a tensor $\bY\in(\reals^{n})^{\otimes k}$:
\begin{align}
\bY = \lambda\, \bvz^{\otimes k} + \bW. \, \label{eq:SpikedTensor}
\end{align}
Here  $\bvz\in \reals^n$, $\|\bvz\|_2=1$ is an unknown vector,
$\lambda\in\reals_{\ge 0}$ is the signal-to-noise ratio, and $\bW\in (\reals^{n})^{\otimes k}$ is a noise tensor.
We want to estimate $\bvz$ given an observation of the tensor
$\bY$. We will assume $\bW$ to be a Gaussian symmetric tensor with
independent entries up to symmetries. More precisely, let $\bG = (G_{i_1,\dots,i_k})_{i_1,\dots,i_k\in [n]}$ be a tensor with i.i.d.
entries $G_{i_1,\dots,i_k}\sim\normal(0,1)$. For a permutation over $k$ elements, $\pi\in\Symm_k$, we
let $\bG^{\pi}$ be the tensor obtained by permuting the indices of $\bG$, namely $G^\pi_{i_1,\dots,i_k} = G_{i_{\pi(1)},\dots,i_{\pi(k)}}$.
We then let
\begin{align}
\bW = \frac{1}{\sqrt{k! \, n}}\sum_{\pi \in \Symm_k} \bG^{\pi}\, .\label{eq:Wconstruction}
\end{align}
Note in particular that $\bW= \bW^{\pi}$ and, for $i_1<i_2<\cdots<i_k$
we have $W_{i_1,\dots,i_k}\sim_{i.i.d.}\normal(0,1/n)$.
Our analysis is fairly insensitive to the distribution of entries with coinciding indices, but the present choice is particularly
convenient because the resulting distribution of $\bW$ is invariant under 
rotations\footnote{Formally this means that, for any $n\times n$ orthogonal 
matrix $\bQ$, defining the tensor $\bW'$ via $W'_{i_1\dots i_k}:=
\sum_{j_1\dots j_k\le n}W_{j_1\dots j_k}Q_{j_1i_1}\cdots Q_{j_ki_k}$,
we have $\bW'\ed \bW$ ($\bW'$ is distributed as $\bW$).} in $\reals^n$.

The maximum likelihood estimator for $\bvz$ minimizes the negative log-likelihood
\begin{align}
-\log p(\bY|\bv_0=\bsigma) = {\rm const.}+\frac{n}{2k!}\|\bY-\lambda\bsigma^{\otimes k}\|_F^2
\end{align}
over unit norm vectors $\bsigma$
(where ${\rm const.}$ is a constant independent of $\bsigma$).
Here $\|\bT\|_F$ is the Frobenius norm of a tensor $\bT$,  
namely $\|\bT\|_F^2:= \sum_{i_1,\dots,i_k}T_{i_1\dots i_k}^2$.

We will denote by $\<\,\cdot\, ,\,\cdot\,\>$ the standard scalar product between tensors.
Namely, given two tensors $\bS,\bT\in (\reals^n)^{\otimes k}$, we let
\begin{align}
\<\bS,\bT\> = \sum_{i_1,\dots,i_k=1}^nS_{i_1,\dots,i_k}T_{i_1,\dots,i_k}\, .
\end{align}
In particular, $\|\bT\|_F^2=\<\bT,\bT\>$.
 With this notation, minimizing the negative likelihood
 $-\log p(\bY|\bv_0=\bsigma)$ is equivalent to solving the following maximization problem
\begin{align}
\mbox{\rm maximize}&\;\;\;\;\;\<\bY,\bsigma^{\otimes k}\>\, ,\\
\mbox{\rm subject to}&\;\;\;\;\;\bsigma\in \sS^{n-1}\, ,
\end{align}
where $\sS^{n-1}\equiv \{ \bx\in\reals^n: \; \|\bx\|_2=1\}$ is the unit sphere in $n$ dimensions.
Explicitly, the objective that is being maximized is
$\<\bY,\bsigma^{\otimes k}\> = \sum_{i_1,\dots,i_k=1}^nY_{i_1,\dots,i_k}\sigma_{i_1}\cdots\sigma_{i_k}$.

An alternative estimation procedure is obtained by introducing a prior distribution on $\bvz$.
 Note that the problem is invariant under 
orthogonal rotations of $\bvz$. Hence it is natural to use a prior that is  also invariant under rotations.
The only such prior is the uniform 
probability distribution over the unit sphere $\sS^{n-1}$, to be denoted by $\nu_0(\,\cdot\,)$
(We alert he reader that this is `nu-zero', not `v-zero' which we are using to denote the signal.).
 Applying Bayes rule, we obtain the following posterior
\begin{align}
\mu_{\sBayes}(\de\bsigma)&= \frac{1}{\widetilde{Z}_{\sBayes}(\lambda)} \;
\exp\Big\{-\frac{n}{2(k!)}\big\|\bY-\lambda\bsigma^{\otimes k}\big\|_F^2\Big\}\, \nu_0(\de\bsigma)\\
& = \frac{1}{Z_{\sBayes}(\lambda)} \; \exp\Big\{\frac{n\lambda}{k!}\<\bY,\bsigma^{\otimes k}\>\Big\} \, \nu_0(\de\bsigma)\, . \label{eq:gibbs_measure}
\end{align}
The constant $Z_{\sBayes}$ is completely determined by the normalization condition 
$\int \mu_{\sBayes}(\de\bsigma) =1$.
It depends on $\lambda$ and $\bY$ but we will often omit these dependencies whenever 
clear from the context. 

It is worth opening a parenthesis on why it makes sense to study the uniform prior:
 the choice of a specific prior can lead to estimation methods that might
perform poorly in general.
In the present case (and considering a rotation invariant loss),
 the uniform prior is `least favorable' in the sense of 
statistical minimax theory. This is a consequence of the invariance of the problem and 
is formalized by the Hunt-Stein theorem \cite{lehmann2006theory}. 
The least favorable property means two things:
\begin{enumerate}
\item The Bayes error \footnote{The Bayes error represents the lowest expected loss that can be attained by an estimator. In turn, any estimator attaining the Bayes error is called a Bayes estimator.}
for any other prior is not larger than the Bayes error for the 
uniform prior.
\item The Bayes estimator with respect to the uniform prior is minimax optimal,
in the sense of performing at least as well on any distribution of $\bv_0$ as
on the uniform prior.
\end{enumerate}
(In fact in the present case proving these properties on $\nu_0$ is a straightforward
exercise.)

Motivated by the above two approaches to estimation (maximum likelihood and Bayes), 
we consider the following family of
probability measures indexed by $\beta, h\in \reals$:
\begin{align}
\mu_{n,\lambda,\beta, h}(\de\bsigma)& = \frac{1}{Z_{n}(\beta,\lambda,h)} \; \exp\Big\{\frac{n\beta}{\sqrt{2(k!)}}\,
\<\bY,\bsigma^{\otimes k}\> +  n h \<\bvz,\bsigma\>\Big\} \, \nu_0(\de\bsigma)\, . \label{eq:PspinDef}
\end{align}
In the following we will omit some of the arguments  $n,\lambda,\beta, h$ unless necessary.
We will refer to $\mu_{n,\lambda,\beta, h}(\de\bsigma)$ as the `Gibbs (or Boltzmann)
measure'. We will also refer to the quantity $H(\bsigma)\equiv (n\beta/\sqrt{2(k!)})\,
\<\bY,\bsigma^{\otimes k}\> +  n h \<\bvz,\bsigma\>$ appearing in the exponent 
as  the `Hamiltonian' (the standard definition in statistical physics differs
from this by a constant factor).

A few remarks:
\begin{enumerate}
\item The measure $\mu_{n,\lambda,\beta, h}(\de\bsigma)$ is in fact a random probability measure because it depends on the random
tensor $\bY$.
\item The maximum likelihood estimator is recovered as
  $\beta\to\infty$, with $h=0$. Indeed, in this limit the measure $\mu_{n,\lambda,\beta, 0}$
  concentrates on the (almost surely unique up to symmetry $\bsigma\to -\bsigma$) 
  maximizer of $\<\bY,\bsigma^{\otimes k}\>$
  over the sphere. 
\item The Bayes posterior corresponds to $\beta = \lambda\sqrt{2/k!}$, $h=0$:
$\mu_{n,\lambda,\beta= \lambda\sqrt{2/k!}, 0}(\de\bsigma) = \mu_{\sBayes}(\de\bsigma)$.

The line $\beta = \lambda\sqrt{2/k!}$ in the $\beta,\lambda$ plane is known in physics as the 
`Nishimori line,' and some consequences of the tower property of conditional expectation 
are known as Nishimori identities. In fact, these identities are quite natural from a Bayes perspective.
\item For any  $\beta\in\reals_{\ge 0}$, $h=0$ the above measure can be used to construct
 interesting estimators.

The term $h\<\bv_0,\bsigma\>$  is instead not accessible to a statistician, and
is introduced as a device to analyze the properties of $\mu$. Eventually it needs to be set to 
zero.

\item To develop techniques to analyze the measure $\mu_{n,\lambda,\beta, h}(\cdot)$,
 we will sometimes restrict to the special case $\lambda=0$ throughout this tutorial.
  The $\lambda=0$ case corresponds to the pure noise model, and is not interesting 
  from a statistical perspective. However, from a mathematical standpoint, it presents the same 
  challenges as the general $\lambda \neq 0$ case. 
  In fact, extensions from $\lambda=0$ to $\lambda\neq 0$ are typically straightforward
   (e.g. see Section \ref{sec:rep_symmetry_lambda} for an analysis of the general $\lambda$
    case within the framework of replica theory). We will use the $\lambda=0$ case 
    in some of our discussions to simplify the notational overhead. 

\end{enumerate}

\noindent
\textbf{Organization:} Before diving into the details, we provide a roadmap for this chapter. 
We introduce the relevant statistical formalism in Section \ref{sec:TensorEstimator}. Section
 \ref{sec:free_energy} delineates the connections between the free energy, a key 
 quantity in statistical physics and the mutual information (and other central notions)
  from information theory. The rest of the chapter is dedicated to the analysis of the
   probability distribution $\mu(\,\cdot\,)$ using ideas from statistical physics. As a 
   first attempt in this direction, we introduce the replica symmetric approximation 
   in Sections \ref{sec:replica_symmetry} and \ref{sec:symmetric_phase}. Although
    insightful and sometimes asymptotically exact, the replica symmetric approximation can 
    predict the wrong asymptotics in certain parameter regimes. In Section \ref{sec:pSpin1RSB}, 
    we elaborate on this shortcoming and introduce the replica symmetry breaking
     approximation scheme. In a different direction, one can extract substantial 
     information regarding the distribution $\mu(\,\cdot\,)$ by analyzing the 
     distribution of critical points of the random Hamiltonian $H$. Section 
     \ref{sec:critical_points} derives some rigorous results regarding the 
     expected number of critical points using the Kac-Rice formula. In Section
      \ref{sec:replica_critical} we turn back to replica theory, and discuss 
      connections between the replica symmetry breaking predictions and the 
      distribution of critical points derived in Section \ref{sec:critical_points}. 
      Finally, Section \ref{sec:omitted_calc} collects some technical proofs.

\subsection{Estimators}
\label{sec:TensorEstimator}

As mentioned above, we are interested in estimating the unknown vector $\bv_0$.
An estimator is a map $\hbv:\bY\mapsto \hbv(\bY)\in\reals^n$. We will, in general, not
require $\hbv(\bY)$ to have unit norm. 
We can evaluate the quality of an estimator through a risk function, which takes the form
\begin{align}
R_n(\hbv) := \E\,\ell(\hbv(\bY),\bv_0)\, .
\end{align}
Note that for $k$ even, it is only possible to estimate $\bv_0$ up to an overall sign.
We will therefore consider losses: $\ell:\reals^n\times\reals^n\to\reals$
that are invariant to this sign change.
Three well known examples are
\begin{itemize}
\item \emph{The overlap:}
\begin{align}
Q_n(\hbv)\equiv\E\Big\{
\frac{|\<\hbv(\bY),\bvz\>|}{\|\hbv(\bY)\|_2\|\bvz\|_2}\Big\}\, .\label{eq:OverlapDef}
\end{align}
Note that good estimation corresponds to large overlap. Also notice that for this 
metric we can restrict without loss of generality to normalized estimators, e.g.
estimators satisfying $\|\hbv(\bY)\|_2=1$ almost surely.
\item \emph{The mean square error:}
\begin{align}
\mse_n(\hbv)\equiv\E\big\{\min_{s\in\{+1,-1\}}
\big\|\hbv(\bY)-s\bvz\|_2\Big\}\, .
\end{align}
The inner minimization over the global sign $s$ is introduced to 
make the loss invariant under $\bvz\to -\bvz$.
\item \emph{The matrix mean square error.} In this case we attempt to estimate
$\bvz\bvz^{\sT}$, and it is therefore natural to consider a general matrix
estimator $\hbM:\bY\mapsto\hbM(\bY)$:
\begin{align}
\label{eq:MatrixMSE}
\MSE_n(\hbM)\equiv\E\big\{
\big\|\hbM(\bY)-\bvz\bvz^{\sT}\|_F^2\Big\}\, .
\end{align}
Of course, given a vector estimator $\hbv$, we can construct a matrix estimator via
$\hbM(\bY) \equiv \hbv(\bY)\hbv(\bY)^{\sT}$.
\end{itemize}

We can use the measure $\mu_{n,\lambda,\beta,h}$ to construct estimators in a number of ways.
The simplest one is by taking expectation
\begin{align}\label{eq:VectorEstimate}
\hbv_{\beta,\lambda}(\bY) := \int_{\sS^{n-1}}\!\bsigma\; \mu_{n,\lambda,\beta,0}(\de\bsigma)\, .
\end{align}
It is an elementary fact that (for $\beta = \lambda\sqrt{2/k!}$), $\hbv_{\beta,\lambda}(\bY)$ maximizes
a version of the overlap $Q^0_n(\hbv)$ in which the absolute value is dropped.
Also it minimizes a version of the mean square error $\mse^0_n(\hbv)$ in which 
the minimization over $s$ is dropped. In formulas
:%
\begin{align}
Q^0_n(\hbv)\equiv\E\Big\{
\frac{\<\hbv(\bY),\bvz\>}{\|\hbv_{\beta}(\bY)\|_2\|\bvz\|_2}\Big\}\, ,\;\;\;
\mse^0_n(\hbv)\equiv\E\big\{ 
\big\|\hbv(\bY)-\bvz\|_2\Big\}\, .
\end{align}

A natural matrix estimator is 
\begin{align}
\label{ea:MatrixExpect}
\hbM_{\beta,\lambda}(\bY) := \int_{\sS^{n-1}}\bsigma\bsigma^{\sT}\, \mu_{n,\lambda,\beta,0}(\de\bsigma)\, .
\end{align}
This minimizes the matrix mean square error of Eq.~\eqref{eq:MatrixMSE}.

Note that the conditional expectation estimator $\hbv_{\beta,\lambda}(\bY)$
of Eq.~\eqref{eq:VectorEstimate}
becomes identically zero for $k$ even, because the Gibbs measure $\mu_{n,\lambda,\beta,0}$
is invariant under reflections ($\mu_{n,\lambda,\beta,0}(A) = \mu_{n,\lambda,\beta,0}(-A)$
for any $A \subseteq \sS^{n-1}$).
Indeed, no non-trivial estimation is possible under the risk metrics $Q^0_n(\hbv)$ and
$\mse^0_n(\hbv)$ in this case, because $\bY$ is left unchanged upon change of $\bvz$
into $-\bvz$.

Nevertheless the Gibbs measure $\mu_{n,\lambda,\beta,0}$ contains useful information 
despite the reflection symmetry. Namely, it is expected that $\mu_{n,\lambda,\beta,0}$
decomposes into a combination 
of components that are symmetric to each other:
\begin{align}
\mu_{n,\lambda,\beta,0} &= \frac{1}{2}\mu^+_{n,\lambda,\beta,0}+\frac{1}{2}
\mu^-_{n,\lambda,\beta,0}\, ,\label{eq:MuDecompose}\\
& \mu^-_{n,\lambda,\beta,0}(A) = \mu^+_{n,\lambda,\beta,0}(-A)\;\;\;\; \forall A\subseteq\sS^{n-1}\, .
\end{align}
Further each of the two components $\mu^{\pm}_{n,\lambda,\beta,0}$ should be 
significantly more concentrated than the mixture, and $\mu^{+}_{n,\lambda,\beta,0}$
is mostly supported on $\<\bsigma,\bvz\> >0$.
Setting $h>0$ and letting $h\to 0$ after $n\to\infty$ is expected to be equivalent to 
computing expectation with respect to $\mu^+_{n,\lambda,\beta,0}$.

This picture suggests several estimators that are non-trivial, even 
if the conditional expectation \eqref{eq:VectorEstimate} is $0$. A simple idea would be
to sample $\bsigma \sim \mu_{n,\lambda,\beta,0}$, and use this vector as an estimate.
This however introduces an additional sampling variance term in the estimation error.

An elegant alternative is to compute the matrix expectation $\hbM_{\beta,\lambda}(\bY)$, and evaluate its top eigenvector $\bv_1(\hbM_{\beta,\lambda}(\bY))$ and the
corresponding eigenvalue $\lambda_1(\hbM_{\beta,\lambda}(\bY))$. We
then estimate $\bvz$ using 
\begin{align}
\hbv^+_{\beta,\lambda}(\bY) \equiv \sqrt{\lambda_1(\hbM_{\beta,\lambda}(\bY))}\cdot \bv_1(\hbM_{\beta,\lambda}(\bY))\, .
\end{align}
The reason for this notation is that $\hbv^+(\bY)$ should be asymptotically equivalent 
to the expectation of $\mu^+_{n,\lambda,\beta,0}$.

In statistical mechanics analysis, the technical difficulties related to the fact that $\mu_{n,\lambda,\beta,0}$
is symmetric under $\bsigma \mapsto -\bsigma$ are addressed by introducing the 
symmetry breaking term $+nh\<\bv_0,\bsigma\>$ with $h >0$ 
in the Gibbs measure. The behavior of various estimators is analyzed by taking $n\to\infty$ first and $h\to0$
thereafter. Because of the decomposition \eqref{eq:MuDecompose} it is assumed that this is
asymptotically equivalent to computing expectations with respect to 
$\mu^+_{n,\lambda,\beta,0}$.

\subsection{Free energy, and its connections with estimation and information theory}
\label{sec:free_energy} 

In the statistical mechanics approach, the first crucial step to study the properties of the 
Gibbs measure $\mu_{n,\lambda,\beta,h}(\de\bsigma)$ is to
compute the asymptotic free  energy density:
\begin{align}
\label{eq:free_energy_density} 
\phi(\beta,\lambda,h) &= \lim_{n\to\infty}\frac{1}{n}\log Z_n(\beta,\lambda,h)\, ,\\
Z_n(\beta,\lambda,h) &=\int_{\sS^{n-1}}\exp\Big\{\frac{n\beta}{\sqrt{2(k!)}}\,
\<\bY,\bsigma^{\otimes k}\> +  n\,h \<\bvz,\bsigma\>\Big\} \, \nu_0(\de\bsigma)\, ,
\end{align}
where the limit is assumed to exist almost surely and be non-random. In fact
 it is not hard to show that $\log Z_n$ is a Lipschitz continuous function of the 
standard Gaussian vector $\bG$ (remember the definition of $\bW$ 
in Eq.~\eqref{eq:Wconstruction}) and use  Gaussian concentration 
(see  Theorem \ref{thm:GaussianConcentration} in Appendix \ref{app:Probability}) to prove that $\log Z_n$ concentrates exponentially 
around its expectation. This argument shows that the 
asymptotic free energy density is also equal to  
\begin{align}
\label{eq:free_energy_density_concentration} 
\phi(\beta,\lambda,h) = \lim_{n\to\infty}\Phi_n(\beta,\lambda,h)\, ,\;\;\;\;\;
\Phi_n(\beta,\lambda,h)\equiv \frac{1}{n}\E\log Z_n(\beta,\lambda,h)
\end{align}
assuming that the limit on the right-hand side exists. Proving the existence of the limit 
for the expectation is  a more difficult mathematical question, which 
can sometimes be addressed by showing that the sequence $\Phi_n(\beta,\lambda,h)$
is superadditive. We refer to the bibliographic notes at the end of this chapter. 

A clarification on the terminology. The quantity $\phi$ defined above is a `density' 
because of the normalization by $n$: it is the free energy `per particle'.
Also, the standard physics convention is to refer to the quantity $-\phi/\beta$
as the free energy. The quantity $\phi$ is sometimes referred to as `pressure' or
`free entropy'. Here we will neglect the factor $(-1/\beta)$ for simplicity.

The free energy density $\phi(\beta,\lambda,h)$ has interesting general properties and 
can be used in a number of ways to characterize the estimation problem. 
These properties are  most easily stated in terms of  the non-asymptotic quantity
\begin{align}
\Phi_n(\beta,\lambda,h)\equiv \frac{1}{n}\E\log Z_n(\beta,\lambda,h)\, .
\end{align}
We also define the noiseless tensor $\bXz$, and the expectations 
$\hbv_{\beta,\lambda,h}(\bY)$, $\hbX_{\beta,\lambda,h}(\bY)$:
\begin{align}
\bXz  & \equiv \bvz^{\otimes k}\, ,\\
\hbv_{\beta,\lambda,h}(\bY) & \equiv \int_{\sS^{n-1}} \bsigma\;  \mu_{n,\beta,\lambda,h}(\de\bsigma)\, ,\\
\hbX_{\beta,\lambda,h}(\bY)& \equiv \int_{\sS^{n-1}} \bsigma^{\otimes k}\;  \mu_{n,\beta,\lambda,h}(\de\bsigma)\, .
\end{align}
Note that $\hbv_{\beta,\lambda,h}(\bY)$, $\hbX_{\beta,\lambda,h}(\bY)$ generalize the 
definitions in the previous sections to $h\neq 0$. However, these are not statistical 
estimators unless $h=0$.

The following identities show that various expectations of interest can be estimated
by computing derivatives of $\Phi_n(\beta,\lambda,h)$ with respect to its arguments.
\begin{lemma}\label{lemma:DerivativesFE}
With the above definitions:
\begin{align}
\frac{\partial \Phi_n}{\partial h}(\beta,\lambda,h) & =\E\Big\{\int_{\sS^{n-1}}\<\bvz,\bsigma\> \, \mu_{n,\beta,\lambda,h}(\de\bsigma)\Big\}  
= \E\big\{\<\bvz,\hbv_{\beta,\lambda,h}(\bY)\>\big\}\, ,\label{eq:PartialPhiH}\\
\frac{\partial \Phi_n}{\partial \lambda}(\beta,\lambda,h) & =\frac{\beta}{\sqrt{2(k!)}}\E\Big\{\int_{\sS^{n-1}} \<\bvz,\bsigma\>^k \, \mu_{n,\beta,\lambda,h}(\de\bsigma)\Big\}  
= \frac{\beta}{\sqrt{2(k!)}}\E\big\{\<\bX_0,\hbX_{\beta,\lambda,h}(\bY)\>\big\}\, ,\label{eq:PartialPhiLambda}\\
\frac{\partial \Phi_n}{\partial \beta}(\beta,\lambda,h) & = \frac{1}{\sqrt{2(k!)}}\E\Big\{\int_{\sS^{n-1}} \<\bY,\bsigma^{\otimes k}\> \, \mu_{n,\beta,\lambda,h}(\de\bsigma)\Big\}  = \frac{1}{\sqrt{2(k!)}}\E\big\{\<\bY,\hbX_{\beta,\lambda,h}(\bY)\>\big\}\, .
\end{align}
\end{lemma}
\begin{proof}
Differentiating $\Phi_n$ in $h$, we obtain, 
\begin{align}
\frac{\partial \Phi_n}{\partial h}(\beta,\lambda,h) &= \E\Big[ \frac{1}{Z_n(\beta,\lambda,h)} \int_{\sS^{n-1}}\<\bvz,\bsigma\> \exp\Big\{\frac{n\beta}{\sqrt{2(k!)}}\,
\<\bY,\bsigma^{\otimes k}\> +  n\,h \<\bvz,\bsigma\>\Big\} \, \nu_0(\de\bsigma)\Big] \nonumber \\
&=\E\Big[\int_{\sS^{n-1}}\<\bvz,\bsigma\> \, \mu_{n,\beta,\lambda,h}(\de\bsigma)\Big] 
= \E\big[\<\bvz,\hbv_{\beta,\lambda,h}(\bY)\>\big]\,  \nonumber 
\end{align}
where the final equality follows by the linearity of expectation under $\mu_{n,\beta,\lambda,h}(\cdot)$.  
In the above computation, we have interchanged the derivative and expectation. This can be justified using the Dominated Convergence Theorem, and the observation that 
\begin{align}
&\Big| \frac{\partial}{\partial h} \Big(\frac{1}{n}\E\log Z_n(\beta,\lambda,h) \Big) \Big|  \nonumber \\
=& \Big| \frac{1}{Z_n(\beta,\lambda,h)} \int_{\sS^{n-1}}\<\bvz,\bsigma\> \exp\Big\{\frac{n\beta}{\sqrt{2(k!)}}\,
\<\bY,\bsigma^{\otimes k}\> +  n\,h \<\bvz,\bsigma\>\Big\} \, \nu_0(\de\bsigma) \Big|  \nonumber \\
=&\Big|  \int_{\sS^{n-1}}\<\bvz,\bsigma\> \, \mu_{n,\beta,\lambda,h}(\de\bsigma) \Big| \leq 1. \nonumber 
\end{align}

\noindent
The other computations are similar. 
\end{proof}

Our objective is to determine the asymptotics of the estimation error under any
of the metrics defined in the previous section, cf. Eqs.~\eqref{eq:OverlapDef} to
\eqref{eq:MatrixMSE}, for the estimator $\hbv^+_{\beta,\lambda}(\bY)$.
As mentioned in the previous section, the effect of a strictly positive $h>0$ 
should be to break the symmetry and select the $\mu^+_{\beta,\lambda}$ component
of the Gibbs measure, cf. Eq.~\eqref{eq:MuDecompose}. In particular, this suggests
\begin{align}
\lim_{n\to\infty}\E\big\{\<\hbv^+_{\beta,\lambda}(\bY),\bvz\>\big\} &= 
\lim_{h\to 0+}\lim_{n\to\infty}\E\big\{\<\hbv_{\beta,\lambda,h}(\bY),\bvz\>\big\}
=\lim_{h\to 0+} \frac{\partial \phi}{\partial h}(\beta,\lambda,h)\, ,\label{eq:LimCorr}\\
\lim_{n\to\infty}\E\big\{\|\hbv^+_{\beta,\lambda}(\bY)\|_2^2\big\} &= 
\lim_{h\to 0+}\lim_{n\to\infty}\E\big\{\|\hbv_{\beta,\lambda,h}(\bY)\|_2^2\big\}\,.\label{eq:LimNorm}
\end{align}
Note that in order to determine the estimation error, we need to compute (the asymptotics of)
$\|\hbv^+_{\beta,\lambda,h}(\bY)\|_2^2$. While Lemma \ref{lemma:DerivativesFE}
does not describe how to do this, it will result as a byproduct 
the calculations in the next pages.
The quantities in expectations on the right-hand side of Eqs.~\eqref{eq:LimCorr},
\eqref{eq:LimNorm} should concentrate around their mean, 
and therefore the above imply
\begin{align}
\lim_{n\to\infty}Q_n(\hbv^+_{\beta,\lambda}) &= \lim_{h\to 0}\lim_{n\to\infty}Q_n(\hbv_{\beta,\lambda,h})\, ,\\
\lim_{n\to\infty}\mse_n(\hbv^+_{\beta,\lambda}) &= \lim_{h\to 0}\lim_{n\to\infty}\mse_n(\hbv_{\beta,\lambda,h})\, .
\end{align}
In other words, by computing the free energy, and taking derivatives of the result,
we can access the estimation error.
In the following, we will often omit mentioning explicitly the limit $h\to 0+$.

The connection between free energy, information theory and estimation theory is particularly
 elegant in the case of Bayes estimation, which corresponds to taking 
$h=0$, $\beta =\lambda\sqrt{2/k!}$. We use the notation
\begin{align}
\Phi_{n,\sBayes}(\lambda) = \Phi_n\Big(\beta =\lambda\sqrt{\frac{2}{k!}},\lambda,h=0\Big) \, . \label{eq:phi_bayes} 
\end{align}

The next lemma connects free energy with mutual information and estimation error. 
Recall that the relative entropy (Kullback-Leibler divergence) between 
two probability distributions $\mu$, $\nu$ with $\mu$ absolutely continuous with respect to 
$\nu$ is defined as
\begin{align}
D(\mu\|\nu) &\equiv \E_{\mu}\Big\{\log \frac{\de \mu}{\de\nu}\Big\}\\
& =\E_{\nu}\Big\{\frac{\de \mu}{\de\nu}\log \frac{\de \mu}{\de\nu}\Big\}\, ,
\end{align}
where $\frac{\de \mu}{\de\nu}$ is the Radon-Nikodym derivative of $\mu$ with respect to $\nu$. 
Given two random variables $X$, $Y$ on the same probability space with laws $\mu_X$, $\mu_Y$ and joint law $\mu_{X,Y}$, their mutual information
is defined as 
\begin{align}
\Info(X;Y) &\equiv D(\mu_{X,Y}\|\mu_X\times \mu_Y) = \E\Big\{\log \frac{\de \mu_{X,Y}}{\de(\mu_X\times\mu_Y)}(X,Y)\Big\}\\
& = \E\Big\{\log \frac{\de \mu_{X|Y}}{\de\mu_X}(X|Y)\Big\}\, ,
\end{align}
where in the last expression $\mu_{X|Y}$ is the conditional distribution of $X$ given $Y$.
For the reader who might not be familiar with information theory notation,
we emphasize  that $\Info(X;Y)$ is not a function of the random variables $X$, $Y$ or
of the value they take, but of their joint law.

Since $\mu_{\sBayes}$ is the conditional distribution of $\bvz$ given $\bY$, the mutual information is also given by
\begin{align}
\Info(\bvz ;\bY) =\E\left\{\log \Big(\frac{\de\mu_{\sBayes}}{\de \nu_0}(\bvz)\Big)\right\}\, .
\end{align}
Note that in the above formula, the Radon-Nikodym derivative $\frac{\de\mu_{\sBayes}}{\de \nu_0}(\bvz)$ is evaluated at $\bvz$ and depends implicitly 
on $\bY$ because $\mu_{\sBayes}$ depends on $\bY$.
\begin{lemma}\label{lemma:I-MMSE}
With the above definitions, and letting  $\hbX_{\sBayes}(\bY)\equiv \int \bsigma^{\otimes k} \mu_{\sBayes}(\de\bsigma)$,
we have the identities
\begin{align}
\frac{1}{n}\, \Info(\bvz ;\bY)&= \frac{\lambda^2}{k!}-\Phi_{n,\sBayes}(\lambda) \, ,\\
\frac{2}{n\lambda} \frac{\partial\phantom{\lambda}}{\partial\lambda} \Info(\bvz ;\bY) & = \frac{1}{k!}\, 
\E\big\{\|\hbX_{\sBayes}(\bY)-\bXz\|_F^2\big\}\, .
\end{align}
\end{lemma}
A proof can be found in Section \ref{sec:appendix_proof_immse}, and is an example of a general relationship between information
and estimation.

\section{Replica symmetric asymptotics}
\label{sec:replica_symmetry} 

Recall the free energy density $\phi(\beta,\lambda)$ from Eq.~\eqref{eq:free_energy_density}. 
We next describe a first attempt at computing $\phi(\beta,\lambda)$ using the  replica method. 
As we will see, the result is correct at small enough $\beta$ but not in general\footnote{However, it is correct
on the important case of Bayes estimation, namely on the line $\beta = \lambda\sqrt{2/k!}$.}. 
In the next section we will see how to modify this prediction using
the idea of `replica symmetry breaking', and obtain the correct free energy density.

\subsection{The replica symmetric calculation}

Using Eq.~\eqref{eq:free_energy_density_concentration}, the computation of $\phi(\beta,\lambda)$ reduces to the computation of $\E\log Z_n$.
The replica method tries to deduce this  from
the moments of $Z_n$, using the identity
\begin{align}
\E\log Z_n= \lim_{r\to 0}\frac{1}{r}\log \E\{ Z_n^r\}\, .\label{eq:ReplicaTrick}
\end{align}
This identity is perfectly correct, e.g. it follows from dominated convergence
under very mild integrability conditions on $Z_n$ (which are valid in our case) at fixed $n$. 
However, a few strange things will happen next:
\begin{enumerate}
\item  We consider $r$ integer.
\item We obtain a formal expression for $\E\{Z_n^r\}$ that is only valid to leading exponential order,
i.e. we compute a formula for
\begin{align}
\lim_{n\to\infty} \frac{1}{nr}\log \E\{ Z_n^r\}. 
\end{align}
\item  We finally take the limit $r\to 0$.
\end{enumerate}
The success of this procedure entirely lies in choice of the `formula'  for 
$\lim_{n\to\infty} (1/nr)\log \E\{ Z_n^r\}$ 
and its dependence on $r$. A naive hope would be that the integer moments 
$\E\{ Z_n^r\}$ uniquely determine the distribution of $Z_n$. However:
$(i)$~In most cases of interest the integrability conditions for this to happen 
(e.g. as in Carleman's condition) are violated; 
$(ii)$~Most importantly, we do not plan to compute these moments exactly, but only 
their exponential growth rate
$\lim_{n\to\infty} (1/nr)\log \E\{ Z_n^r\}$. 

It is easy to show\footnote{For instance, let $Z_n= e^{nX_n}$ where the 
random variables $X_n$ satisfy a large deviation principle at rate $n$, with convex 
rate function $I(x)$. Evaluating the exponential growth rate of $\E\{Z_n^r\}$ 
only determines $I(x)$ on a set of points away from its minimum.}
 that evaluating this exponential growth rate at integer $r$ 
does not uniquely determine $\lim_{n\to\infty}\E\log Z_n$. Hence
the actual mathematical content of the replica method is not in the
elementary identity \eqref{eq:ReplicaTrick}, but rather in the specific construction 
of the formula for  $\lim_{n\to\infty} (1/nr)\log \E\{ Z_n^r\}$ which surprisingly captures  
 the probabilistic structure of the problem.

We now turn to the implementation of the replica symmetric approximation to this setting. Fix an integer $r \geq 1$. Let us start by computing  $\E\{ Z_n^r\}$. 
%
Recall from Eq.~(\ref{eq:free_energy_density}) that
\begin{align}
Z_n \equiv \int_{\sS^{n-1}}\; \exp\Big\{n h \<\bvz,\bsigma\>+\frac{n\beta\lambda}{\sqrt{2(k!)}}\,
\<\bvz,\bsigma\>^k+\frac{n\beta}{\sqrt{2(k!)}}\,
\<\bW,\bsigma^{\otimes k}\> \Big\} \, \nu_0(\de\bsigma) . 
\end{align}
Using this formula $r$ times, and writing the expectation over $\bvz$ explicitly as an integral over $\bsigma^0$, we get
\begin{align}
\E\{Z_n^r\} = \int_{(\sS^{n-1})^{r+1}}\! e^{n h \sum_{a=1}^r\<\bsigma^0,\bsigma^a\>+ \frac{n\beta\lambda}{\sqrt{2(k!)}}\,
\sum_{a=1}^r\<\bsigma^0,\bsigma^a\>^k}\, 
\E \exp\Big\{\frac{n\beta}{\sqrt{2(k!)}}\,\sum_{a=1}^r
\<\bW,(\bsigma^a)^{\otimes k}\>\Big\} \, \nu_{0,r+1}(\de\bsigma)\, ,\label{eq:EZ-pspin-first}
\end{align}
where the integral is over $\bsigma = (\bsigma^0,\bsigma^1,\dots,\bsigma^r)\in (\sS^{n-1})^{r+1}$, 
and $\nu_{0,r+1}\equiv \nu_0\times\cdots\times\nu_0$. We can now take the expectation over $\bW$.
Recalling the definition \ref{eq:Wconstruction}, we get
\begin{align}
\E \exp\Big\{\frac{n\beta}{\sqrt{2(k!)}}\,\sum_{a=1}^r \<\bW,(\bsigma^a)^{\otimes k}\>\Big\} &=
\E \exp\Big\{  \sqrt{\frac{n\beta^2}{2}}\,\sum_{a=1}^r \<\bG,(\bsigma^a)^{\otimes k}\>\Big\} \\
&= \exp\Big\{ \frac{n\beta^2}{4}\,\sum_{a,b=1}^r\<\bsigma^a,\bsigma^b\>^k\Big\}
\end{align}
We can substitute this in Eq.~(\ref{eq:EZ-pspin-first}) to get
\begin{align}
\E\{Z_n^r\} = \int_{(\sS^{n-1})^{r+1}}  \!\exp\Big\{n h \sum_{a=1}^r\<\bsigma^0,\bsigma^a\>+ \frac{n\beta\lambda}{\sqrt{2(k!)}}\,
\sum_{a=1}^r\<\bsigma^0,\bsigma^a\>^k+ \frac{n\beta^2}{4}\,\sum_{a,b=1}^r\<\bsigma^a,\bsigma^b\>^k\Big\}\nu_{0,r+1}(\de\bsigma)
\label{eq:replica_matrix_exp1} 
\end{align}
Consider the $(r+1)\times (r+1)$ matrix $\hbQ = (\hQ_{ab})_{0\le a,b\le r}$ defined by 
\begin{align}
\hQ_{ab} & = \<\bsigma^a,\bsigma^b\>\, \;\;\;\;\;\mbox{for $a,b\in\{0,1,\dots,r\}$.}
\end{align}
Note that the RHS of Eq.~(\ref{eq:replica_matrix_exp1}) is solely a function of $\hbQ$. To continue our computation, we will integrate the RHS against the marginal distribution of $\hbQ$.
Denote by $f_{n,r}(\bQ)$ the joint density of the random variables $(\hQ)_{0\le a<b\le r}$ when $\bsigma^0,\bsigma^1,\dots,\bsigma^r$
are distributed according to $\nu_{0,r+1}=\nu_0\times\cdots\times\nu_{0}$. We then have
\begin{align}
\E\{Z_n^r\} & = \int \exp\Big\{n h\sum_{a=1}^r Q_{0,a}+\frac{n\beta\lambda}{\sqrt{2(k!)}}\,
\sum_{a=1}^r Q_{0a}^k+ \frac{n\beta^2}{4}\,\sum_{a,b=1}^rQ_{ab}^k\Big\}\, f_{n,r}(\bQ)\, \de\bQ\, .\label{eq:Integral-Pspin-1}
\end{align}
We need to estimate $f_{n,r}(\bQ)$ to leading exponential order. 
For $\bLambda = (\Lambda_{a,b})_{0\le a, b\le r}$ a symmetric 
positive-definite matrix, consider the following Gaussian measure on $(\reals^n)^{r+1}$:
\begin{align}
\gamma_{\bLambda} (\de\bsigma) = \left(\frac{n}{2\pi}\right)^{n(r+1)/2} \det(\bLambda)^{n/2}\; 
\exp\Big\{-\frac{n}{2}\sum_{a,b=0}^r\Lambda_{a,b}\<\bsigma^a,\bsigma^b\>\Big\}\, \de\bsigma
\end{align}
We claim  that (see Section \ref{sec:ProofQdensity})
\begin{align}
f_{n,r}(\bQ) &= e^{-n(r+1)/2+n\<\bLambda,\bQ\>/2}\det(\bLambda)^{-n/2}\frac{G_{\bLambda}(\bQ)}{g_{n}(1)^{r+1}}\, ,\label{eq:Q-density}
\end{align}
where $G_{\bLambda}$ is the density of $(\hQ_{ab})_{0\le a\le b\le r}$, $\hQ_{a,b}=\<\bsigma^a,\bsigma^b\>$, 
when $\bsigma$ has distribution
$\gamma_{\bLambda}$ and $g_n(x)$ is the density of $\|\bz\|_2^2$ at $x$ when $\bz\sim\normal(0,\id_n/n)$.
By the local central limit theorem, we have $g_n(1) = \Theta(n^{1/2})$, and $G_{\bLambda}(\bQ) = \Theta(n^{r(r+1)/2})$,
if $\bLambda = \bQ^{-1}$.
Hence 
\begin{align}
f_{n,r}(\bQ) &\doteq \inf_{\bLambda}e^{n\{\<\bLambda,\bQ\>-\Tr(\id_{r+1})\}/2}\det(\bLambda)^{-n/2}\\
& = \det(\bQ)^{n/2}\, .
\end{align}
where the last line follows by substituting $\bLambda = \bQ^{-1}$.
(Here and below $\doteq$ denotes equality to the leading exponential order. Namely,
we write $f(n)\doteq g(n)$ if $\lim_{n\to\infty}n^{-1}\log[f(n)/g(n)]=0$.)

Substituting in Eq.~(\ref{eq:Integral-Pspin-1}) and computing the integral over $\bQ$ by Laplace method, we get
\begin{align}
\E\{Z_n^r\} & \doteq \sup_{\bQ} \exp\Big\{n h\sum_{a=1}^r Q_{0,a}+\frac{n\beta\lambda}{\sqrt{2(k!)}}\,
\sum_{a=1}^r Q_{0a}^k+ \frac{n\beta^2}{4}\,\sum_{a,b=1}^rQ_{ab}^k+\frac{n}{2}\Tr\log(\bQ)\Big\}\\
&= \exp\big\{n \sup_{\bQ} S(\bQ)\big\}\, ,\label{eq:Action-p-spin}
\end{align}
where we introduced the function
\begin{align}
S(\bQ)\equiv h\sum_{a=1}^r Q_{0,a}+\frac{\beta\lambda}{\sqrt{2(k!)}}\,
\sum_{a=1}^r Q_{0a}^k+ \frac{\beta^2}{4}\,\sum_{a,b=1}^rQ_{ab}^k+\frac{1}{2}\Tr\log(\bQ)\, . \label{eq:RS_Action}
\end{align}

Recalling (\ref{eq:ReplicaTrick}) we get the following formal expression
\begin{align}
\phi(\beta,\lambda,h) = \lim_{r\to 0} \frac{1}{r} \sup_{\bQ} S(\bQ) \, . \label{eq:ReplicaTrickS}
\end{align}
Before continuing our investigations into the RHS of Eq.~(\ref{eq:ReplicaTrickS}), we take a brief detour to interpret the optimizer $\bQ$ in the RHS of Eq.~(\ref{eq:ReplicaTrickS}).
What is the interpretation for the matrix $\bQ$? 
Recalling that $\hbv_{\beta}(\bY) = \int_{\sS^{n-1}} \bsigma \mu_n(\de\bsigma)$, we have
\begin{align}
\E\big\{\|\hbv_{\beta}(\bY)\|^2_2\big\} &= \E\left \< \int_{\sS^{n-1}} \bsigma \mu_n(\de\bsigma) , \int_{\sS^{n-1}} \bsigma \mu_n(\de\bsigma) \right \> \, \nonumber \\
&= \E \left\{ \int_{(\sS^{n-1})^{2}} \hQ_{1,2} \mu_{n}(\de\bsigma^1) \mu_{n}(\de \bsigma^2)  \right\}\,   \nonumber\\
&= \frac{2}{r(r-1)}\sum_{1\le a<b\le r}\E\left\{\int_{(\sS^{n-1})^{r}}\hQ_{a,b}\mu_{n}(\de\bsigma^1)\cdots \mu_{n}(\de\bsigma^r)\right\}\,  , 
\end{align}
where the last equality follows by linearity of expectation. Under the double limit $n\to \infty$ (evaluated by Laplace method) followed by $r \to 0$, one conjectures that the behavior of the system should be governed by the supremum $\bQ^*$ in Eq.~(\ref{eq:ReplicaTrickS}). In turn, this implies
%
\begin{align}
\lim_{n\to\infty}\E\big\{\|\hbv_{\beta}(\bY)\|^2_2\big\} = \lim_{r\to 0}\frac{2}{r(r-1)} \sum_{1\le a<b\le r}Q^*_{a,b}\, . \label{eq:Qinterpretation}
\end{align}
Proceeding analogously for $\bvz$ that is represented as $\bsigma^0$, we get
\begin{align}
\lim_{n\to\infty}\E\big\{\<\hbv_{\beta}(\bY),\bvz\>\big\} = \lim_{r\to 0}\frac{1}{r} \sum_{a=1}^r Q^*_{0,b}\, . \label{eq:Q0interpretation}
\end{align}

Let us emphasize that the expression (\ref{eq:ReplicaTrickS}) was the result of a formal 
manipulation, and giving mathematical sense to such formulas is still an open problem.
Indeed the 
asymptotic formula (\ref{eq:Action-p-spin}) only holds for $r\ge 1$ a positive integer.

We return to Eq.~(\ref{eq:ReplicaTrickS}), and continue the derivation of the limiting free energy density.
The way one makes sense of this formula  in the replica method is to proceed as follows:
\begin{itemize}
\item Make an ansatz for the saddle point $\bQ$,  such that the resulting expression depends 
 analytically on $r$.
\item Take the limit $r\to 0$.
\end{itemize}
The ansatz will be chosen  to be a stationary point of $S(\,\cdot\,)$. However, this
is not sufficient to fix the choice of $\bQ$. We will start from a particularly simple choice 
and proceed gradually towards the full form of the `replica symmetry breaking' ansatz.
This will be presented in the context of the Sherrington-Kirkpatrick model 
 in the next chapter. 

\subsection{Replica-symmetric ansatz}
\label{sec:rep_symmetry_lambda} 

Note that the function $S(\,\cdot\,)$ is invariant under permutation
of the indices $\{1,\dots, r\}$ of the columns/rows of $\bQ$. Therefore there must exist stationary points
that are invariant under such permutations. These must have the form (the first constraint comes from the fact that 
$\|\bsigma^a\|_2=1$):
\begin{align}
Q_{aa} & = 1\, ,\;\;\; \mbox{ for $a\in\{0,1,\dots, r\}$,}\nonumber\\
Q_{0a} & = b\, ,\;\;\; \mbox{ for $a\in\{1,\dots, r\}$,}\label{eq:FirstRS}\\
Q_{a_1a_2}& = q\, ,\;\;\; \mbox{ for $a_1\neq a_2$, $a_1,a_2\in\{1,\dots, r\}$.}\nonumber
\end{align}
This is known as a \emph{replica-symmetric ansatz}, or \emph{replica-symmetric saddle point}, and we will use the notation
$\bQ^{\sRS}$ whenever useful to specify that we are using this ansatz.

Recall the function $S(\bQ)$ from Eq.~\eqref{eq:RS_Action}.
We next evaluate $S(\bQ)$ for this ansatz. The tricky part is to compute $\Tr\log(\bQ) = \log\det(\bQ)$. 
Let $\tbQ\in\reals^{r\times r}$ be the principal submatrix of $\bQ$  which corresponds to rows/columns with indices $\{1,\dots,r\}$.
By a simple application of the matrix determinant lemma
\begin{align}
\Tr\log(\bQ) = \log\big(1-b^2\<\bfone,\tbQ^{-1}\bfone\>\big)+\Tr\log(\tbQ)\, .
\end{align}
Next note that the eigenvalues of $\tbQ$ are $\lambda_1(\tbQ) = 1+(r-1)q$ (with eigenvector equal to $\bfone$)
and $\lambda_2(\tbQ) = \dots = \lambda_r(\tbQ) = 1-q$ (with eigenspace given by the orthogonal complement of $\bfone$). 
This also implies $\<\bfone,\tbQ^{-1}\bfone\> = r/(1+(r-1)q)$.
It thus follows
\begin{align}
\Tr\log(\bQ) = \log \Big(1-\frac{b^2 r}{1+(r-1)q}\Big) +\log\big(1+(r-1)q\big)+(r-1)\log(1-q)\, .
\end{align}
We can then evaluate the exponential rate $S(\bQ)$, cf. Eq.~(\ref{eq:RS_Action}) on this ansatz
\begin{align}
S(\bQ^{\sRS}) &= hrb + \frac{\beta \lambda}{\sqrt{2 (k!)}} r b^k +\frac{\beta^2r}{4} +\frac{\beta^2 r (r-1)}{4} \, q^k\nonumber\\
& +\frac{1}{2}\log\Big(1-\frac{r b^2}{1+(r-1)q}\Big) + \frac{1}{2}\log \Big(1+\frac{rq}{1-q}\Big)+\frac{1}{2} r\, \log (1-q)\, .
\end{align}

While we derived this formula for $r$ an integer, the resulting expression makes sense for arbitrary $r$, and in particular is of order $\Theta(r)$
as $r\to 0$. We therefore take this limit and define $\Psi_{\sRS}(b,q;\beta,\lambda,h) = \lim_{r\to 0} r^{-1} \, S(\bQ^{\sRS})$. We get
\begin{align}
\Psi_{\sRS}(b,q;\beta,\lambda,h)& = h b +\frac{\beta\lambda}{\sqrt{2k!}} \,b^k +\frac{\beta^2}{4}(1-q^k) -\frac{1}{2}\frac{b^2}{1-q}
+\frac{q}{2(1-q)} +\frac{1}{2}\log(1-q)\, . \label{eq:PsiRS-pspin}
\end{align}
We will omit the argument $\beta,\lambda,h$ whenever clear from the context.
According to (\ref{eq:ReplicaTrickS}), we should look for a critical point of $\Psi_{\sRS}$. The partial derivatives
are
\begin{align}
\frac{\partial\Psi_{\sRS}}{\partial b}(b,q) & =h+\beta\lambda\sqrt{\frac{k}{2(k-1)!}} \, b^{k-1} - \frac{b}{1-q}\, ,\label{eq:GradPspinRS1}\\
\frac{\partial\Psi_{\sRS}}{\partial q}(b,q) & =-\frac{k\beta^2}{4} q^{k-1} -\frac{b^2}{2(1-q)^2} +\frac{q}{2(1-q)^2}\, . \label{eq:GradPspinRS2}
\end{align}

We are then lead to the following conjecture.
\begin{conjecture}\label{conj:PspinRS}
There exists a region $\cR_{\sRS}\subseteq \reals^3$ such that, for $(\beta, \lambda, h)\in \cR_{\sRS}$, we have
\begin{align}
\phi(\beta,\lambda,h) = \Psi_{\sRS}(b_*,q_*;\beta,\lambda,h)\, ,
\end{align}
where $(b_*(\beta,\lambda,h),q_*(\beta,\lambda,h))$ is a critical point of $\Psi_{\sRS}$, i.e. satisfies $\nabla\Psi_{\sRS} (b,q) = \bfzero$.

Further, the asymptotic estimation accuracy of the $\beta$ estimator $\bv_{\beta}(\bY)$, and its norm are  given by
\begin{align}
\lim_{n\to\infty}\E\big\{|\<\bvz,\hbv_{\beta}(\bY)\>\big|\big\} & =b_*(\beta,\lambda,h=0)\, , \label{eq:overlap1}\\
\lim_{n\to\infty}\E\big\{\|\hbv_{\beta}(\bY)\|^2_2\big\} & =q_*(\beta,\lambda,h=0)\,  \label{eq:overlap2}.
\end{align}
\end{conjecture}
\noindent
The last two limits are motivated by Eqs.~(\ref{eq:Qinterpretation}) and (\ref{eq:Q0interpretation}).

\begin{remark}
Note that we did not write the RS ansatz as a supremum anymore. Indeed one puzzling
feature of the replica trick is that the supremum is sometimes changed into an infimum or a saddle point.
This can---sometimes---be understood better using the variational approach, which we will consider in the next chapter.
\end{remark}

\begin{remark}
The above conjecture is somewhat vague in what concerns the region of validity of the replica symmetric ansatz $\cR_{\sRS}$.
Indeed our calculation gives no clue as to where the latter does hold. We will see in Section \ref{sec:PureNoisePspin} that 
the conjecture cannot possibly be correct for all values of the parameters.

The replica-symmetry breaking calculation of Section \ref{sec:pSpin1RSB} allows to determine the precise region of validity $\cR_{\sRS}$.
\end{remark}

\begin{remark}
The asymptotic free energy density of this model was characterized rigorously by 
Michel Talagrand (for even $k$) \cite{talagrand2006free} and  Wei-Kuo Chen (for general $k$) 
\cite{chen2013aizenman}, in the case $\lambda=0$. In fact these papers cover 
all values of $\beta$, beyond the regime within which the replica symmetric asymptotics
is correct and prove a more general, replica symmetry breaking, expression.
For pointers to the mathematical literature in the case $\lambda>0$ we refer to the 
bibliographical notes at the end of this chapter.
\end{remark}

In the next sections we will consider some special choices of the parameters 
$(\beta,\lambda,k)$ that lead to special simplifications. This will help to develop some 
intuition about the replica symmetric asymptotics and to 
understand some of its implications. Throughout we will set $h=0$, and omit the argument $h$.

\subsection{Special cases: I. The sequence model, $k =1$}
\label{eq:SphericalSequenceModel}

In the cases $k\in \{1,2\}$ the replica symmetric ansatz is expected to be correct. 
There are several heuristic reasons for this, which should be clear once we discuss 
replica-symmetry breaking. One important remark is that in these cases the Hamiltonian
$H(\bsigma)$ has a unique local maximum on $\sS^{n-1}$ and hence the Gibbs distribution
$\mu_{\beta,\lambda,n}$ is unimodal. As we will see, replica-symmetry breaking is related to 
a (quite dramatic) breakdown of unimodality.

Further, the model (\ref{eq:SpikedTensor}) can be treated by elementary
means in these cases. It correponds to the standard sequence model for $k=1$ and to the
 symmetric matrix spiked model for $k=2$ (for which 
classical tools from random marix theory can be used). 

For $k=1$, the model (\ref{eq:SpikedTensor}) reduces to 
\begin{align}
\by = \lambda \,\bvz +\bw\, .
\end{align}
In words: we are observing a uniformly random unit vector $\bvz\in\reals^n$, $\|\bvz\|_2 =1$ corrupted by i.i.d. Gaussian noise

Equation (\ref{eq:PsiRS-pspin}) yields
\begin{align}
\Psi_{\sRS}(b,q;\beta,\lambda)& = \frac{\beta\lambda}{\sqrt{2}} \,b +\frac{\beta^2}{4}(1-q) -\frac{1}{2}\frac{b^2}{1-q}
+\frac{q}{2(1-q)} +\frac{1}{2}\log(1-q)\, . \label{eq:PsiRS_k1}
\end{align}
Solving the equations $\nabla\Psi_{\sRS}(b,q) = 0$
yields\footnote{There are in fact two solutions. We select the branch
  that has $q_*(\beta,\lambda) \to 0$ as $\beta\to 0$.}
\begin{align}
q_*(\beta,\lambda) & = 1+x_*(\beta,\lambda)-\sqrt{2x_*(\beta,\lambda)+x_*(\beta,\lambda)^2}\, ,\label{eq:QstarK1}\\
b_*(\beta,\lambda) & = \frac{\beta\lambda}{\sqrt{2}}\Big\{-x_*(\beta,\lambda)+\sqrt{2x_*(\beta,\lambda)+x_*(\beta,\lambda)^2}\Big\}\, \\
x_*(\beta,\lambda)& \equiv \frac{1}{\beta^2(1+\lambda^2)}\, .\label{eq:Xstar}
\end{align}
Substituting in Eq.~(\ref{eq:PsiRS_k1}) we obtain the RS prediction for $\phi(\beta,\lambda)$. 
It is interesting to further specialize this result.

Consider Bayes estimation corresponding to $\beta = \lambda\sqrt{2}$. Using Eqs.~(\ref{eq:QstarK1}) to (\ref{eq:Xstar}) it is immediate
that 
\begin{align}
b_{\sBayes}(\lambda) = q_{\sBayes}(\lambda) = \frac{\lambda^2}{1+\lambda^2}\, ,
\end{align}
where we used the notation $b_{\sBayes}(\lambda) = b_*(\beta=\lambda\sqrt{2/k!},\lambda)$, 
$q_{\sBayes}(\lambda) = q_*(\beta=\lambda\sqrt{2/k!},\lambda)$.
Hence, Conjecture \ref{conj:PspinRS} implies
\begin{align}
\lim_{n\to\infty}\E\big\{|\<\bvz,\hbv_{\sBayes}(\by)\>\big|\big\} = b_{\sBayes}(\lambda)  = \frac{\lambda^2}{1+\lambda^2}\,, \label{eq:SequenceMSE}\\
\lim_{n\to\infty}\E\big\{\|\hbv_{\sBayes}(\by)-\bvz\big\|^2_2\big\} = 1-2b_{\sBayes}(\lambda)+q_{\sBayes}(\lambda)  = \frac{1}{1+\lambda^2}\,. \label{eq:SequenceMSE2}
\end{align}

It is easy to rederive Eq.~(\ref{eq:SequenceMSE}) using classical tools. Here is a sketch. Recall that the signal $\bvz$ is uniformly distributed on the unit sphere $\sS^{n-1}$. For large $n$,
this distrinution is close to the isotropic Gaussian. We will therefore assume, without further justification, the simper model $\bvz\sim\normal(0,\id_n/n)$, and 
leave it as an exercise for the reader to deal with the uniform distribution over $\sS^{n-1}$. In this case, the Bayes posterior is also Gaussian. A simple calculation yields
\begin{align}
\mu_{\sBayes,\sGauss}(\de\bsigma) = \frac{1}{Z}\;\exp\Big\{-\frac{n}{2}(1+\lambda^2)\,\|\bsigma\|_2^2+n\lambda\<\by,\bsigma\>\Big\}\;\de\bsigma\, .
\end{align}
We therefore have
\begin{align}
\hbv_{\sBayes}(\by) = \frac{\lambda}{1+\lambda^2}\, \by\, ,
\end{align}
whence Eqs.~(\ref{eq:SequenceMSE}),  (\ref{eq:SequenceMSE2})  follow.

\subsection{Special cases: II. The spiked matrix model, $k =2$}
\label{sec:SpikedMatrix}

For $k=2$, the model (\ref{eq:SpikedTensor}) yields
\begin{align}
\bY = \lambda \,\bvz\,\bvz^{\sT} +\bW\, ,
\end{align}
where $\bW\in\reals^{n\times n}$, $\bW = \bW^{\sT}$ is a matrix sampled with the  $\GOE(n)$ (Gaussian Orthogonal Ensemble) distribution. Namely
$(W_{ij})_{i\le j}$ are independent random variables with $W_{ii} \sim \normal(0,2/n)$ for $1\le i\le n$, and $W_{ij}\sim\normal(0,2/n)$ for $1\le i<j\le n$.
\begin{align}
\Psi_{\sRS}(b,q;\beta,\lambda)& =\frac{\beta\lambda}{2} \,b^2 +\frac{\beta^2}{4}(1-q^2) -\frac{1}{2}\frac{b^2}{1-q}
+\frac{q}{2(1-q)} +\frac{1}{2}\log(1-q)\, .
\end{align}
The stationarity condition $\nabla\Psi_{\sRS} (b,q)= 0$ has three solutions with $q_*,b_*\ge 0$, corresponding to different `phases' of the model.
\begin{enumerate}
\item \emph{Uninformative/paramagnetic phase:}
\begin{align}
q_{{\rm P}}(\beta,\lambda) = b_{{\rm P}}(\beta,\lambda) = 0\, .
\end{align}
Remembering the interpretation of $q_*(\beta,\lambda)$, $b_*(\beta,\lambda)$ in Conjecture \ref{conj:PspinRS}, we obtain 
\begin{align}
\lim_{n\to\infty}\E\big\{\|\hbv_{\beta}(\bY)\|_2^2\big\} = 
\lim_{n\to\infty}\E\big\{\big|\<\bvz,\hbv_{\beta}(\bY)\>\big|\big\} = 0\, .
\end{align}
In other words, under $\mu_{n}(\de\bsigma)$, $\bsigma$ has zero mean. 
Obviously the estimator $\hbv_{\beta}(\bY)$ is not useful in this case.
Further analysis reveals that $\mu_{n}(\de\bsigma)$ is roughly uniform
in the sense that the distribution of $(\sqrt{n}\sigma_{i_1},\dots,\sqrt{n}\sigma_{i_{\ell}})$
converges to $\normal(0,\id_\ell)$ for any fixed $\ell$ as $n\to\infty$. On the other
hand correlations are strong enough for $\bsigma\bsigma^{\sT}$ to be positively correlated with $\bY$.
Namely, using Lemma \ref{lemma:DerivativesFE}, the replica symmetric formula implies
\begin{align}
\lim_{n\to\infty}\E\Big\{\int_{\sS^{n-1}} \<\bY,\bsigma\bsigma^{\sT}\> \, \mu_{n}(\de\bsigma)\Big\} 
=2\frac{\partial \Psi_{\sRS}}{\partial \beta}(0,0;\beta,\lambda)\, .
\end{align}

We expect the paramagnetic solution to  give the correct asymptotics 
for $\beta,\lambda$ small. Expanding around this stationary point, we get
\begin{align}
\Psi_{\sRS} (b,q) = \frac{\beta^2}{4}-\frac{1}{2}(1-\beta\lambda) b^2 +\frac{1}{4}(1-\beta^2)q^2+ o(b^2,q^2)\, .
\end{align}
Note that for $\beta<1$ and $\lambda<1/\beta$, $(b_{{\rm P}}, q_{{\rm P}})$ is a maximum with respect to  $b$ and a minimum with
respect to $q$. At these boundaries, new solutions appear.
\item \emph{Symmetric/spin glass phase.} For $\beta>1$ a distinct solution appears:
\begin{align}
q_{{\rm SG}}(\beta,\lambda) & = 1-\frac{1}{\beta}\, ,\\
b_{{\rm SG}}(\beta,\lambda) &= 0\, .
\end{align}
In this case the measure $\mu_n(\de\bsigma)$ has non-zero mean $\hbv_{\beta}(\bY)$, with
\begin{align}
\lim_{n\to\infty}\E\big\{\|\hbv_{\beta}(\bY)\|_2^2\big\} = 1-\frac{1}{\beta}\, .
\end{align}
However this mean is random and asymptotically uncorrelated with $\bvz$:
$\lim_{n\to\infty}\E\{\<\hbv_{\beta}(\bY),\bvz\>\} = 0$.
 Again, the estimator $\hbv_{\beta}(\bY)$ is not useful in this case:
 the Gibbs measure is concentrated around a random direction, which is however uncorrelated with the signal.
 Notice that this is particularly dangerous from an inference perspective, since 
 it might lead to overconfidence in a completely noisy estimate.
 
The value of the free energy is
\begin{align}
\Psi_{\sRS} (b_{{\rm SG}},q_{{\rm SG}}) = \beta-\frac{3}{4}-\frac{1}{2}\log\beta\, .
\end{align}
Expanding around this point, we get
\begin{align}
\Psi_{\sRS} (b,q) = &\Psi_{\sRS} (b_{{\rm SG}},q_{{\rm SG}}) +\frac{1}{2}\beta^2(\beta-1) \big(q-q_{{\rm SG}}\big)^2-\frac{1}{2}\beta(1-\lambda)\, \big(b-b_{{\rm SG}}\big)^2\\
&+o\big(\big(b-b_{{\rm SG}}\big)^2+\big(q-q_{{\rm SG}}\big)^2\big)\, .
\end{align}
Hence  for $\beta>1$ (which is required for this to be a solution)
and $\lambda<1$, $(b_{{\rm SG}}, q_{{\rm SG}})$ is a maximum with respect to  $b$ and a minimum with
respect to $q$. At the boundary $\lambda=1$, a new solution appears.
\item \emph{A recovery/ferromagnetic phase.} This solution exists for $\beta<1$ and $\lambda>1/\beta$ or $\beta\ge 1$ and $\lambda>1$:
\begin{align}
q_{{\rm R}}(\beta,\lambda) & = 1-\frac{1}{\beta\lambda}\, ,\label{eq:RecoveryK2a}\\
b_{{\rm R}}(\beta,\lambda) &= \sqrt{\left(1-\frac{1}{\lambda^2}\right)\left(1-\frac{1}{\beta\lambda}\right)}\, . \label{eq:RecoveryK2b}
\end{align}
\end{enumerate}

Which of the three solutions above should we pick? Of course for large $\lambda$ the recovery solution  
$(q_{{\rm R}},b_{{\rm R}})$ is the one that makes sense, while for small $\lambda$
we should switch to either $(q_{{\rm P}},b_{{\rm P}})$ or $(q_{{\rm SG}},b_{{\rm SG}})$. 
In particular for $\lambda$ small and $\beta$ small, we expect $\mu_n(\de\bsigma)$ to be 
approximately uniform over the sphere and hence $(q_{{\rm P}},b_{{\rm P}})$ should make sense.
At first sight, the derivation in the previous section, and in particular
Eq.~(\ref{eq:ReplicaTrickS}) seems to suggest that we have to maximize $\Psi_{\sRS}$ and
hence select the stationary point that corresponds to a largest value of $\Psi_{\sRS}(q,b)$. 
Somewhat surprisingly, this turns out to be incorrect. For instance, for $\lambda=0$ and 
$\beta$ small $(q_{{\rm P}},b_{{\rm P}})$ is the only stationary point, and it seems to make intuitive sense.
However, it is a maximum with respect to $b$ and a minimum with respect to $q$. 

Indeed, the following two heuristic rules are normally used:
\begin{itemize}
\item[(A)] The correct solution $(q_*,b_*)$ should be a maximum with respect to $b$ and a minimum with respect to $q$.
Namely, $\Psi_{\sRS}(q_*,b_*) = \max_b \Psi_{\sRS}(q_*,b)$ and $\Psi_{\sRS}(q_*,b_*) = \min_q \Psi_{\sRS}(q,b_*)$.
\item[(B)] At a phase transition line between solutions $(q_1,b_1)$ and $(q_2,b_2)$ we must have $\Psi_{\sRS}(q_1,b_1) = \Psi_{\sRS}(q_2,b_2)$.
\end{itemize}
Rule (B) follows from the observation that $\Phi_n(\beta,\lambda)$ is uniformly continuous in the parameters $(\beta,\lambda)$.
We will provide justification of rule (A) when we study the interpolation method in 
the next chapter. 

Some terminology is useful at this point.
The model parameters $(\beta,\lambda)$ take values in $\cR= \reals_{\ge 0}^2$.
In the present context, \emph{phase diagram} is a partition of the space of parameters $\cR$ into regions: $\cR = \cR_1\cup\cR_2\cup \dots$ such that
within region $\cR_i$, a specific solution $q_i(\beta,\lambda), b_i(\beta,\lambda)$, depending analytically on $\beta, \lambda$ yields the correct prediction
for the free energy. Each region typically corresponds to a qualitatively different 
behavior of the Gibbs measure $\mu_n$, as in the present example.

\begin{figure}[t]
\centering
\includegraphics[width=0.7\textwidth,keepaspectratio]{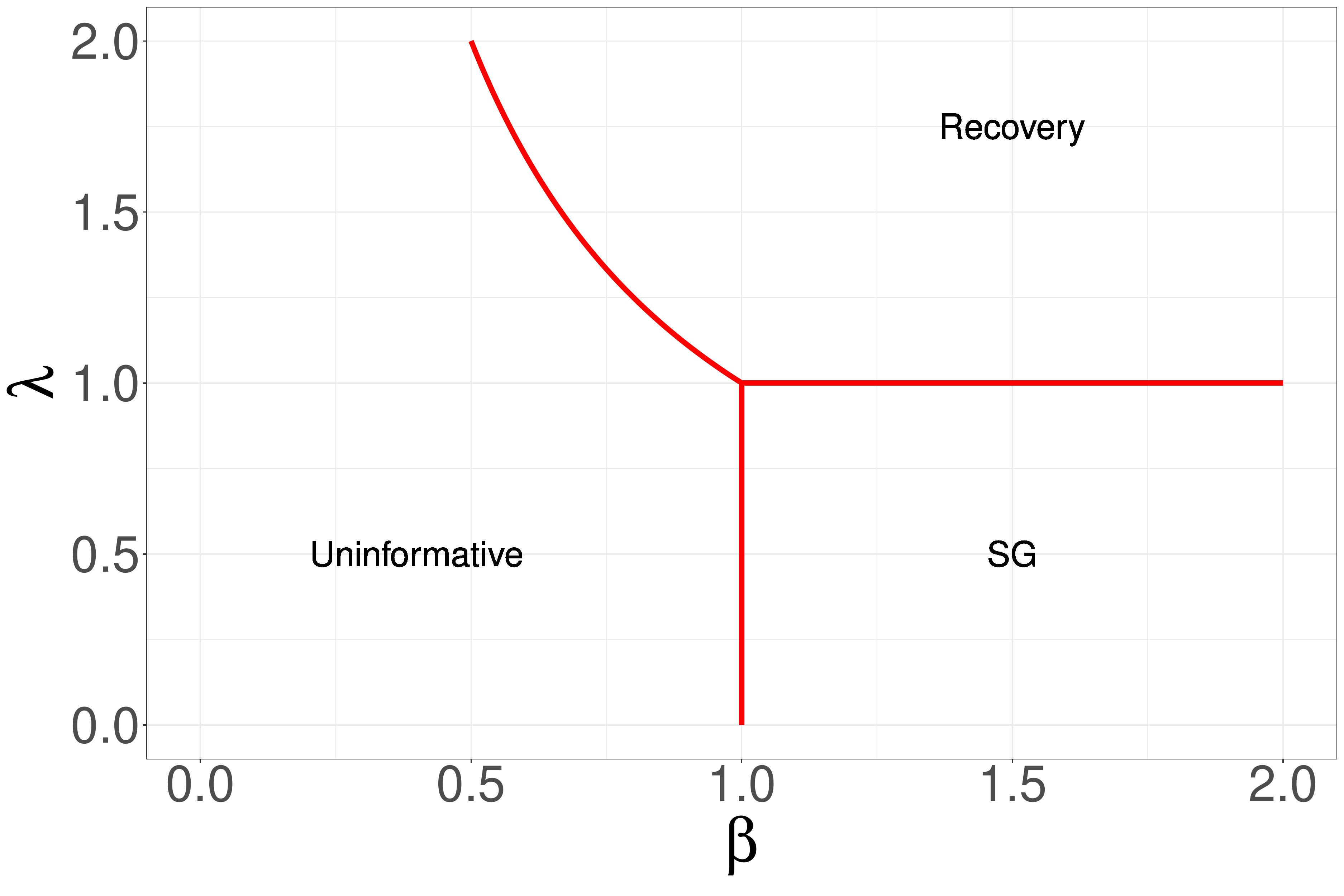}
\caption{The phase diagram of the spiked matrix model (case $k=2$ of the spiked tensor).}\label{fig:k2phase}
\end{figure}

Figure \ref{fig:k2phase} depicts the phase diagram for the $k=2$ model:
\begin{align}
\mbox{Uninformative}: & \;\;\;\;\; \cR_{{\rm P}} = \big\{(\beta,\lambda):\;\; \beta<1, \lambda<1/\beta\big\}\, ,\\
\mbox{Spin glass}: & \;\;\;\;\; \cR_{{\rm SG}} = \big\{(\beta,\lambda):\;\; \beta \ge 1, \lambda<1\big\}\, ,\\
\mbox{Recovery}: & \;\;\;\;\; \cR_{{\rm R}} = \big\{(\beta,\lambda):\;\; \lambda\ge \max(1,1/\beta) \big\}\, .
\end{align}
We can use Eqs.~(\ref{eq:RecoveryK2a}), (\ref{eq:RecoveryK2b}) to compute the mean square error in the recovery  phase:
\begin{align}
\MSE(\beta,\lambda) = \lim_{n\to\infty}\E\big\{\min\big(\|\hbv_{\beta}(\bY)-\bvz\|_2^2,\|\hbv_{\beta}(\bY)+\bvz\|_2^2\big)\big\}\, .
\end{align}
We get, for $(\beta,\lambda)\in \cR_{{\rm R}}$,
\begin{align}
\MSE(\beta,\lambda) = \frac{1}{\lambda^2} +\left[\sqrt{1-\frac{1}{\lambda^2}}-\sqrt{1-\frac{1}{\beta\lambda}}\right]^2\, .
\end{align}

It is interesting to consider two specific lines in the $(\beta,\lambda)$ plane, corresponding 
to Bayes optimal estimation (and therefore minimax optimal estimation, since the uniform prior
is least favorable), and maximum likelihood estimation.

On the Bayes line $\beta=\lambda$, the mean square error is minimized:
\begin{align}
\mse_{\sBayes}(\lambda) = \min\left(1,\frac{1}{\lambda^2}\right)\, .
\end{align}

The limit $\beta\to\infty$ corresponds to the maximum likelihood estimator $\hbv_{\infty}(\bY) = \bv_1(\bY)$ and 
\begin{align}
\mse_{\sML}(\lambda) = 2-2\sqrt{\Big(1-\frac{1}{\lambda^2}\Big)_+}\, .
\end{align}

As mentioned above, for $k=2$ the replica results can be proved easily by random matrix theory, cf. Exercise \ref{exer:SpikedMatrix}.

\subsection{Special cases: III. Bayes-optimal estimation}

Bayes optimal estimation is recovered for $\beta = \lambda\sqrt{2/k!}$,
and the fact that in this case the Gibbs measure $\mu_n$ coincides with the posterior
leads to additional simplifications. In this case (setting $h=0$), the replica symmetric 
stationarity conditions $\nabla_{(b,q)}\Psi_{\sRS}=0$, cf.  Eqs.~\eqref{eq:GradPspinRS1},
\eqref{eq:GradPspinRS2},  admit a solution with $q=b$. Indeed, this identity 
is expected to hold as a consequence of the interpretation of the 
`order parameters' $b,q$ in Eq.~(\ref{conj:PspinRS}):
\begin{align}
\E\big\{\<\bvz,\hbv_{\sBayes}(\bY)\>\big\}=\E\big\{\<\E\{\bvz| \bY\},\hbv_{\sBayes}(\bY)\>\big\} = \E\big\{\|\hbv_{\sBayes}(\bY)\|^2_2\big\}\, .
\end{align}

\begin{figure}[t]
\begin{subfigure}[t]{.5\textwidth}
\centering
\includegraphics[width=\textwidth]{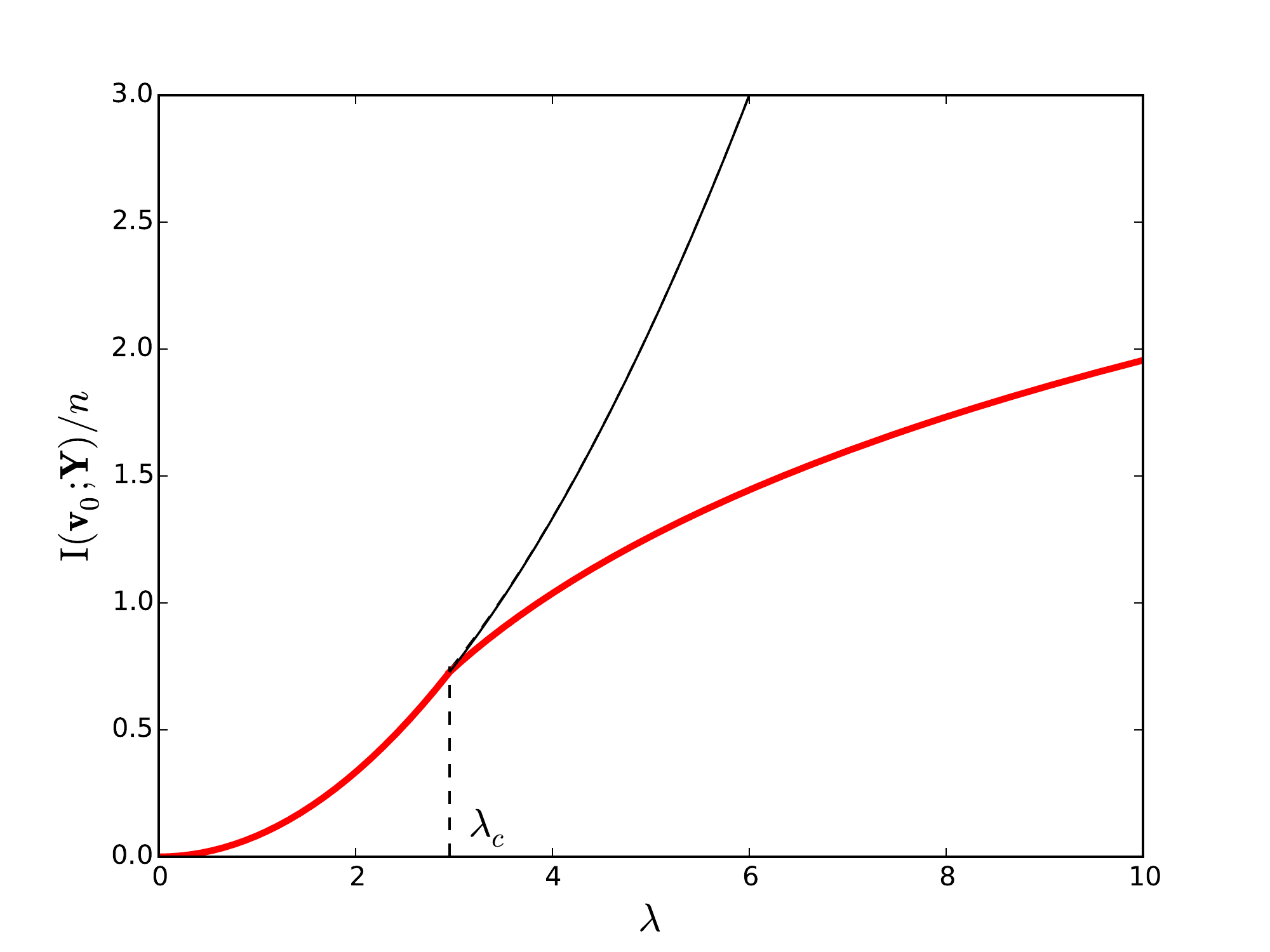}
\end{subfigure}
\begin{subfigure}[t]{.5\textwidth}
\centering
\includegraphics[width=\textwidth]{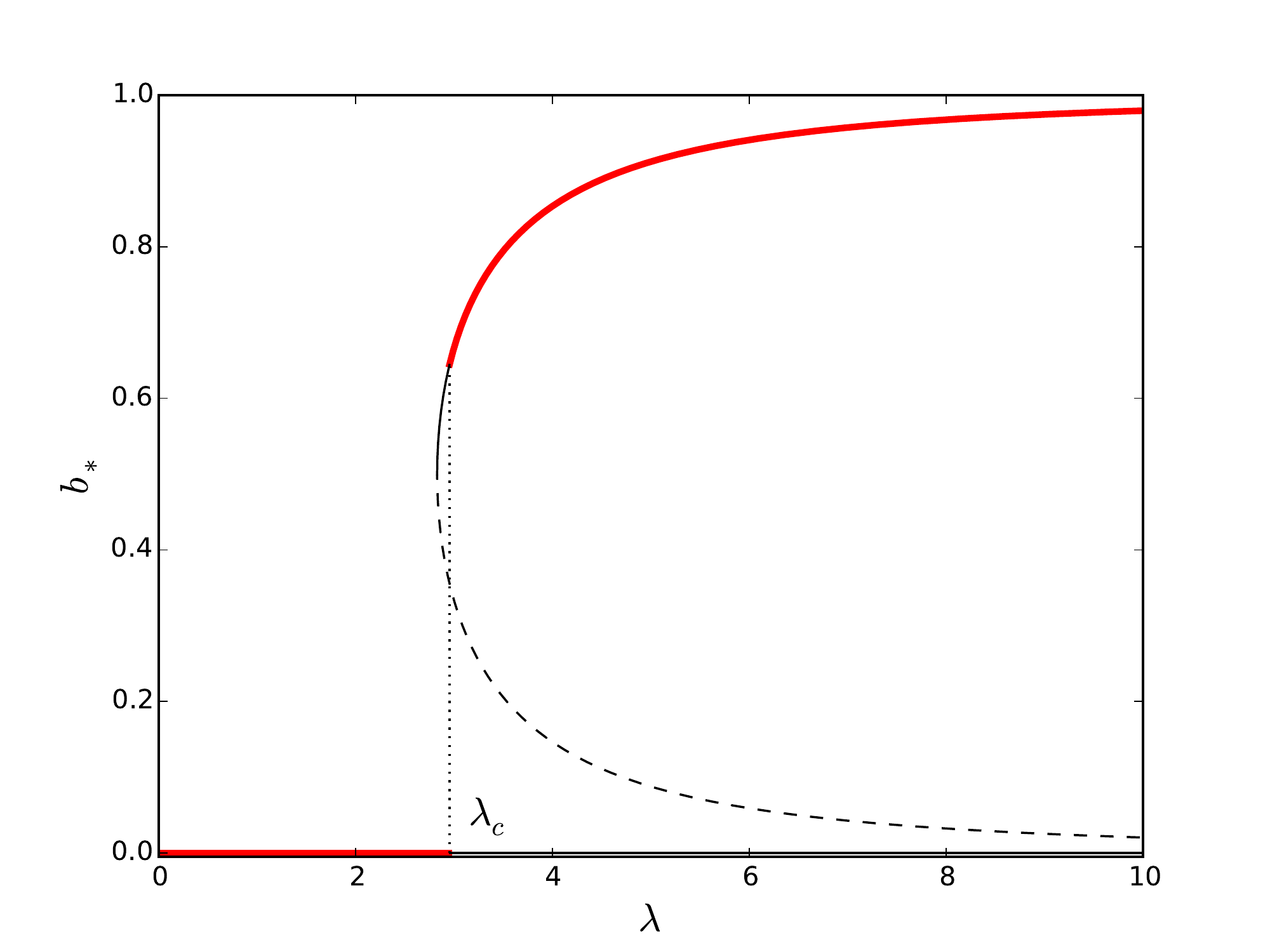}
\end{subfigure}
\caption{Spiked tensor model with $k=3$, on the Bayes line. Left:
  Mutual information. Right: Order parameter.}\label{fig:PspinBayes}
\end{figure}

Substituting this in the replica symmetric free energy (\ref{eq:PsiRS-pspin}), for $\beta = \lambda\sqrt{2/k!}$, we obtain 
$\Psi_{\sBayes}(b;\lambda)\equiv \Psi_{\sRS}(b,q=b;\beta=\lambda\sqrt{2/k!},\lambda,h=0)$, 
\begin{align}
\Psi_{\sBayes}(b;\lambda)=\frac{\lambda^2}{2k!}(1+b^k) +\frac{1}{2} b+\frac{1}{2} \log (1-b)\, .
\end{align}
Lemma \ref{lemma:I-MMSE} suggests that $\Psi_{\sBayes}(b;\lambda)$ should determine the asymptotic behavior of the 
mutual information. This fact can be proved rigorously (we refer to the bibliographic notes for further 
references.)
\begin{theorem}[\cite{lesieur2017statistical}]
With the above definitions
\begin{align}
\lim_{n\to\infty}\frac{1}{n} \Info(\bvz ;\bY) =\frac{\lambda^2}{k!}-\sup_{b\ge 0} \Psi_{\sBayes}(b;\lambda)\, .
\end{align}
\end{theorem}

\begin{figure}[t]
\begin{subfigure}[t]{.5\textwidth}
\centering
\includegraphics[width=\textwidth]{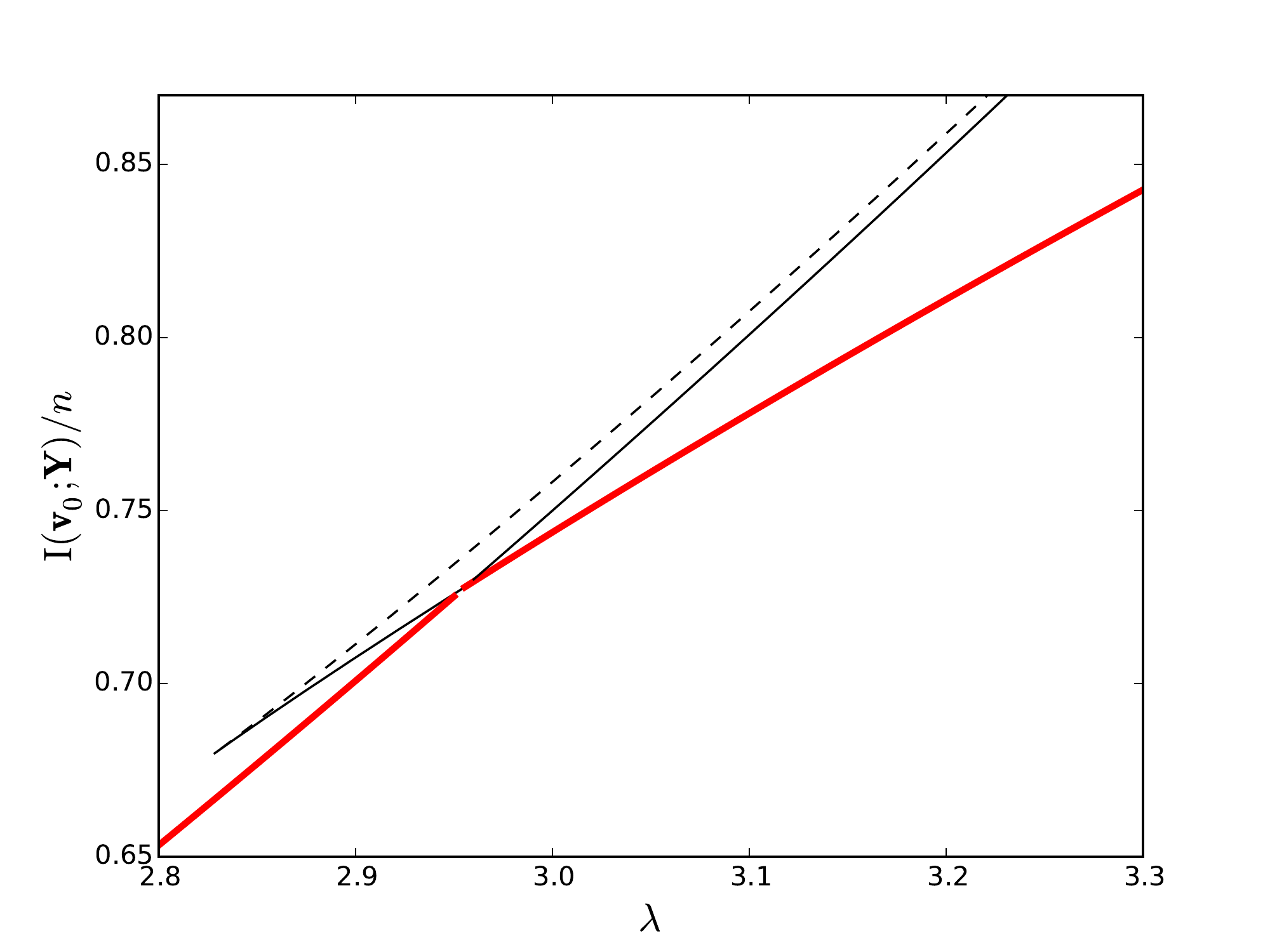}
\end{subfigure}
\begin{subfigure}[t]{.5\textwidth}
\centering
\includegraphics[width=\textwidth]{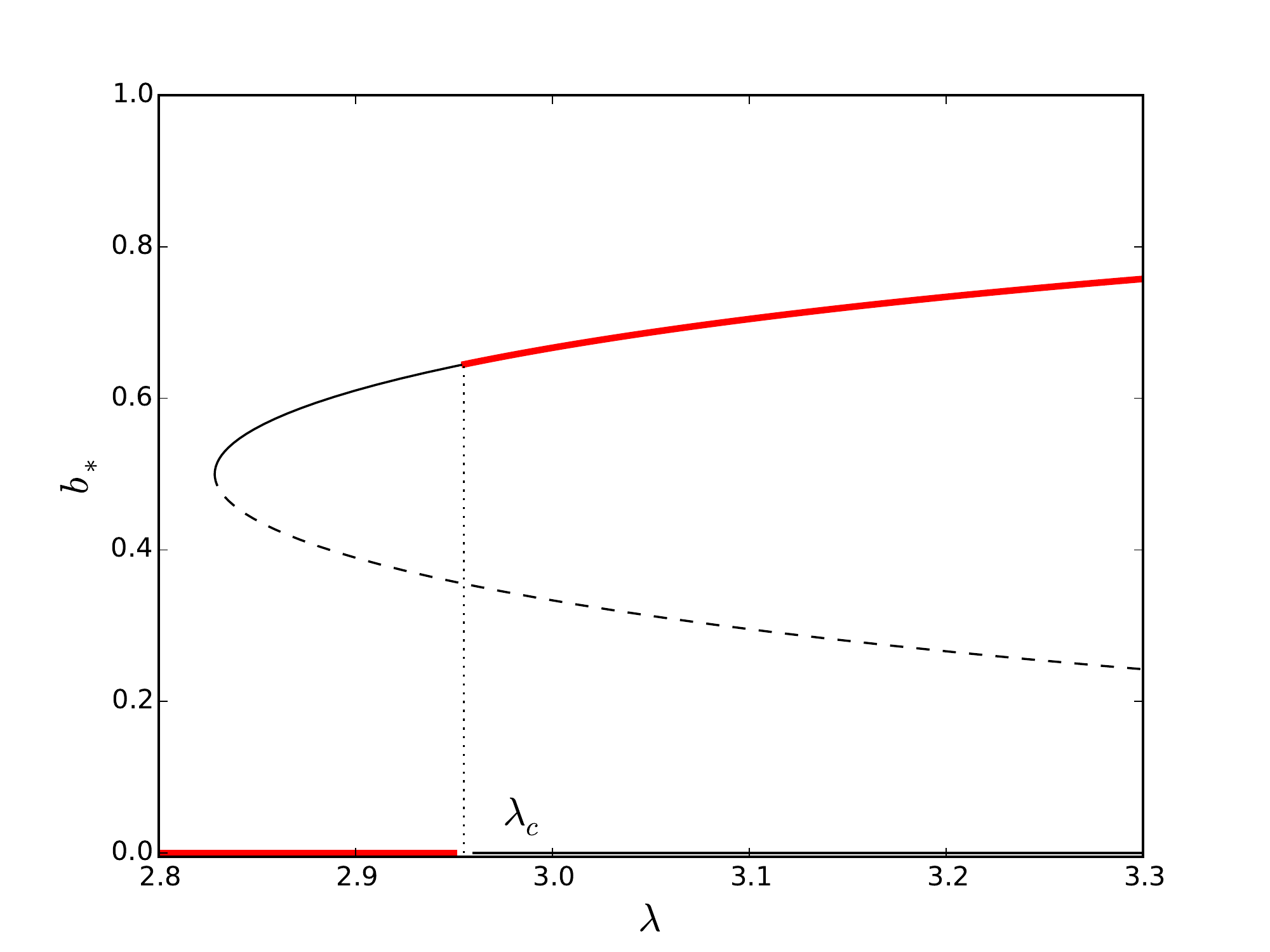}
\end{subfigure}
\caption{Same as in Figure \ref{fig:PspinBayes}: zoom around the critical point.}\label{fig:PspinBayesZoom}
\end{figure}
Setting to $0$ the derivative of $\Psi_{\sBayes}(b;\lambda)$ with respect to $b$
(or, equivalently, using Eqs.~(\ref{eq:GradPspinRS1}), (\ref{eq:GradPspinRS2})), we obtain that the stationary point $b_*$ must satisfy
\begin{align}
b &= \frac{ \xi\, b^{k-1} }{1+\xi\, b^{k-1}}\, ,\;\;\;\;\;\;\;\;\;\;
\xi \equiv \frac{\lambda^2}{(k-1)!}\, .\label{eq:PspinBayes-eqb}
\end{align}
The behavior of the solution is illustrated in Figures \ref{fig:PspinBayes}, \ref{fig:PspinBayesZoom}.
The uninformative point $b_{\rm P} (\lambda)=0$ is always a solution of Eq.~(\ref{eq:PspinBayes-eqb}), and a local maximum of the free energy
$b\mapsto \Psi_{\sBayes}(b;\lambda)$.
For $\lambda>\lambda_s(k)$, two new solutions appear: a second local maximum $b_{{\rm R}}(\lambda)>0$,  corresponding to non-trivial recovery,
and a local minimum $0<b_{{\rm unst}}<b_{{\rm R}}$. The point at which a pair of
stationary points of the free energy functional appear is referred to as
`spinodal point' in statistical physics. In this case, it is explicitly given by
\begin{align}
\frac{\lambda_s(k)^2}{(k-1)!} = \frac{(k-1)^{k-1}}{(k-2)^{k-2}}\, .
\end{align}
The boundary between the two phases, $\lambda_c(k)>\lambda_s(k)$ is determined by the condition
\begin{align}
\Psi_{\sBayes}(b_{{\rm P}}(\lambda);\lambda) =
\Psi_{\sBayes}(b_{{\rm R}}(\lambda);\lambda) \, .
\end{align}
Solving this equation numerically,  we get $\lambda_c(k=3) \approx 2.955$.
%
%
\subsection{Special cases: IV. Pure noise, $\lambda=0$}
\label{sec:PureNoisePspin}

For $\lambda=0$, the condition $\frac{\partial\Psi_{\sRS}}{\partial  b}=0$, cf. Eq.~(\ref{eq:GradPspinRS1}), implies $b=0$.
Substituting in Eq.~(\ref{eq:PsiRS-pspin}), and dropping the arguments $b=0$, $\lambda=0$, we get 
\begin{align}
\Psi_{\sRS}(q;\beta)& = \frac{\beta^2}{4}(1-q^k) +\frac{q}{2(1-q)} +\frac{1}{2}\log(1-q)\, . \label{eq:PsiRS-pspin-NoSignal}
\end{align}
The stationarity condition (\ref{eq:GradPspinRS1}) is satisfied provided
\begin{align}
\frac{q}{(1-q)^2} = \frac{k\beta^2}{2} \, q^{k-1}\, .
\end{align}
For every $\beta$, this admits a paramagnetic solution $q_{{\rm P}}=0$, which is a local minimum of $\Psi_{\sRS}$.
For $\beta>\beta_{s,\sRS}(k)$, two new solution appear $0<q_{\rm unst}<q_{{\rm SG}}$ which are the solutions of
\begin{align}
q^{k-2}(1-q)^2 = \frac{2}{k\beta^2}\, .\label{eq:QpspinRS}
\end{align}
The spinodal point is given explicitly by
\begin{align}
\beta_{s,\sRS}(k)^2 = \frac{(k-2)^{k-2}}{2k^{k-1}}\, .
\end{align}
The spin glass solution $q_{{\rm SG}}$ is a local minimum of $\Psi_{\sRS}$ and becomes a global minimum for $\beta>\beta_{c,\sRS}(k)$,
where the critical point $\beta_{c,\sRS}(k)>\beta_{s,\sRS}(k)$, is given by the condition $\Psi_{\sRS}(q_{{\rm P}};\beta)=\Psi_{\sRS}(q_{{\rm SG}};\beta)$.

We append the subscript RS to emphasize that --in this case-- the replica symmetric solution is only
a good approximation but ultimately incorrect (i.e. not asymptotically exact as $n\to\infty$). 
Within the replica method, this can be seen after we study replica-symmetry breaking.
%
%
%
%
%
%
%
%
%
%
%
\section{The replica symmetric phase diagram, $h=0$}
\label{sec:symmetric_phase} 

For $h=0$, and general $\lambda, \beta$, the stationarity condition $\nabla\Psi_{\sRS}(b,q) =0$ reduces to (cf. Eq.~(\ref{eq:GradPspinRS1}), (\ref{eq:GradPspinRS2})): 
\begin{align}
\beta\lambda\sqrt{\frac{k}{2(k-1)!}} \, b^{k-1} &= \frac{b}{1-q}\, ,\\
\frac{k\beta^2}{4} q^{k-1} +\frac{b^2}{2(1-q)^2} &=\frac{q}{2(1-q)^2}\, .
\end{align}

We note that $q_{\rm P} = b_{\rm P} =0$ is always a stationary point for this problem. This corresponds to the \textit{paramagnetic/ uninformative} fixed point for this problem. To obtain the \textit{spin-glass} stationary point, we set $b_{\rm SG} =0$ and solve for $q_{\rm SG}$ to obtain the following fixed point equation 
\begin{align}
q_{\rm SG}^{k-2} (1- q_{\rm SG})^2 = \frac{2}{k\beta^2}. 
\end{align}

Finally, the \textit{recovery} fixed point can be obtained by solving for $b_{\sRS}, q_{\sRS}$. As usual, we expect the recovery fixed point to be the correct one for $\lambda$ large, while the paramagnetic or the spin glass fixed point should be the correct one for small $\beta, \lambda$. 

%
%
%
\section{Replica-symmetry breaking}
\label{sec:pSpin1RSB}

\subsection{The need for replica-symmetry breaking}
\label{sec:NeedRSB}

Throughout the previous section, we have implicitly assumed that the RS
 formula \ref{eq:PsiRS-pspin} gives the correct formula for the limiting free energy. 
 However, it is known that for $\beta$ large enough and $\lambda$ sufficiently small, 
 $\Psi_{\sRS}$ no longer gives the correct free energy.
  Historically, this was determined first in the context of the Sherrington-Kirkpatrick (SK) model,
  which we will study in the next chapter.  Here we will give a short preview about this
  model, to motivate the need to go beyond replica symmetry.

The SK model is a Gibbs probability measure over $\bsigma \in \{ \pm 1 \}^n$.
Let $\bW \sim \GOE(n)$ be a random matrix from the Gaussian Orthogonal Ensemble
(this coincides with the $k=2$ case of the random tensor defined in the previous sections). 
At any inverse temperature $\beta >0$, we consider the Gibbs measure and free energy 
\begin{align}
\mu_{\beta}(\bsigma) &\equiv\frac{1}{Z_n(\beta)}  e^{H(\bsigma)} \,,\;\;\;\;
H(\bsigma) = \frac{\beta}{2} \sum_{i,j} W_{ij} \sigma_i \sigma_j\, ,\\
\Phi_n (\beta) &= \frac{1}{n} \E \log \Big\{\sum_{\bsigma \in \{ \pm 1 \}^n} e^{H(\bsigma)}  \Big\}\,. 
\end{align}

A direct computation reveals that 
\begin{align}
\Phi_n(\beta) - \beta \frac{\partial \Phi_n}{\partial \beta } = \frac{1}{n}\mathbb{E}[\Entro(\mu_\beta)] \geq 0.\nonumber 
\end{align}
Here $\Entro(\nu)$ is the Shannon entropy of the discrete probability distribution
$\nu$: $\Entro(\nu):=-\sum_x\nu(x)\log\nu(x)$.
If the RS ansatz is correct, then  $\lim_{n\to\infty}\Phi_n(\beta) = \Psi_{\sRS}(\beta)$
and therefore 
\begin{align}
0\le \lim_{n\to\infty}\frac{1}{n} \mathbb{E}[\Entro(\mu_\beta)]  =
\Psi_{\sRS}(\beta) - \beta \frac{\partial \Psi_{\sRS}}{\partial \beta}(\beta) \, . \label{eq:EntroIneq}
\end{align}
(Note that the limit $n\to\infty$ and the derivative with respect to $\beta$ can be exchanged
because $\beta\mapsto \Phi_n(\beta)$ is convex.)

In their  1975 paper,  Sherrington and Kirkpatrick 
\cite{sherrington1975solvable} computed the RS prediction  $\Psi_{\sRS}(\beta)$:
we will repeat the same calculation in the next chapter. 
Among other things they computed the prediction for the asymptotic entropy 
density \eqref{eq:EntroIneq}, and showed that the non-negativity constraint 
is violated for $\beta>\beta_{0}$, with $\beta_0$ a sufficiently large constant. 

This implies that the 
replica symmetric free energy fails to provide the correct answer at sufficiently low temperature. 
Surprisingly, the physicists' reaction was not that the replica method
(i.e. guessing the limit as the number of replicas $r\to 0$ from the 
computations for integer $r$) was incorrect, but that the the assumption of a replica-symmetric
stationary point $\bQ$, cf. Eq.~\eqref{eq:FirstRS} has to be replaced by a different one.
Even if the functional $S(\bQ)$ is invariant under permutations of the $r$ replicas, the
dominating stationary points had to be non-symmetric. In physics language, the 
group of symmetries $\Symm_r$ is `spontaneously broken.'

\subsection{One-step replica symmetry breaking (1RSB) free energy}

The correct form the matrix $\bQ$ (the correct way to break the replica symmetry)
was introduced by Parisi in 1979 \cite{parisi1979infinite}. 
The construction of the replica symmetry breaking (RSB) matrix $\bQ$  is hierarchical and 
proceeds in rounds. Here we will describe the first step of this construction
(one-step replica symmetry breaking, or 1RSB), which is sufficient for the model
studied here. 
The full construction will be described in the next chapter\footnote{An interesting 
question is the following: `Is there a simple criterion to decide how many steps of the RSB construction
need to be taken to get the exact asymptotics of the free energy?' While physicists have some intuition
about this, no good mathematical theory exists to provide guidance, and this question is addressed on a case-by-case basis.}.

In 1RSB, we divide the indices $[r]=\{1, \cdots, r\}$ into $r/m$ blocks, each block having 
$m$ elements. We note that the functional $S(\bQ)$ in \eqref{eq:RS_Action} is 
invariant to the permutation of the indices in each block, as well 
to the permutation of the distinct blocks. These form a subgroup of the group of permutations $\Symm_r$. 
 We look for a stationary point $\bQ$ that is invariant under this subgroup. 
 Namely, we fix a partition $[r]= \cup_{\ell =1}^{r/m} \mathcal{I}_{\ell}$,
 $|\mathcal{I}_{\ell}|=m$ and set
\begin{align}
Q_{aa} &= 1 \quad a \in \{0, 1, \cdots, r/m\}, \nonumber \\
Q_{0a} &= b \quad a \in \{ 1, \cdots , r/m\}, \nonumber \\
Q_{ab} &= q_1 \quad a, b \in \mathcal{I}_{\ell}, \ell \in \{1, \cdots, r/m\}, \nonumber  \\ 
Q_{ab} &= q_0 \quad \mbox{otherwise.} \nonumber 
\end{align}
It is not difficult to check that this is indeed the most general matrix that is invariant under the 
subgroup described above (under the additional condition $Q_{aa} =1$ for all $a$
 that follows from the interpretation of $\bQ$).

This is usually referred to as the $\rm{1RSB}$ ansatz and we will refer to the stationary point 
as $\bQ^{\sRSB{1}}$ when we use this ansatz. We next compute $S(\bQ^{\sRSB{1}})$. 
The tricky part again is to compute $\Tr\log (\bQ^{\sRSB{1}})$. To this end, 
we proceed similarly to the replica-symmetric analysis. We partition the matrix and denote by
 $\tilde{\bQ}$ as the $r \times r$ sub-matrix corresponding to the indices 
$\{1, \cdots, r\}$. Using the Schur-complement formula, we have,
\begin{align}
\Tr\log (\bQ^{\sRSB{1}}) = \log ( 1- b^2 \langle \mathbf{1}, \tilde{\bQ}^{-1} \mathbf{1} \rangle ) +
\Tr \log  (\tilde{\bQ}). \label{eq:schur}
\end{align}

To evaluate the expression above, we need a detailed understanding of the 
eigenvalues and eigenvectors of the matrix $\tilde{\bQ}$. It is not hard to see that 
$\lambda_1 = 1 + (m-1)q_1 + (r-m) q_0$ is an eigenvalue with eigenvector $\mathbf{1}$ 
with multiplicity one. The other distinct eigenvalues are $\lambda_2= 1- q_1 + m(q_1 - q_0)$
 with multiplicity $(r/m) -1$
 (the corresponding eigenvectors being vectors that are constant within blocks, and orthogonal 
 to the all-ones vector), and $\lambda_3= 1- q_1$ with multiplicity $(r/m)(m-1)$
 (for the orthogonal complement of the previous subspaces). 
  Now, we can compute
\begin{align}
\langle \mathbf{1}, \tilde{\bQ}^{-1} \mathbf{1} \rangle = \frac{r}{1 + (m-1)q_1 + (r-m) q_0}.
\end{align}
Therefore \eqref{eq:schur} immediately implies that 
\begin{align}
\rm{Tr} \log (\bQ^{\sRSB{1}}) &= \log \Big( 1 - \frac{b^2 r}{1+ (m-1) q_1 + (r-m)q_0} \Big) + \log ( 1 + (m-1) q_1 + (r-m) q_0 )
\nonumber  \\ 
&+ \Big(\frac{r}{m} - 1 \Big) \log ( 1 - q_1 + m (q_1 - q_0)) + \frac{r}{m} (m-1) \log(1- q_1). 
\end{align}
Recall the functional $S(\bQ)$ from \eqref{eq:RS_Action}:
\begin{align}
S(\bQ) = \frac{\beta \lambda}{\sqrt{2 (k!)}} \sum_{a=1}^{k} Q_{0a}^k + \frac{\beta^2}{4}\sum_{a,b =1}^r Q_{ab}^k + \frac{1}{2} \rm{Tr} \log (\bQ). 
\end{align}
To evaluate this functional under the $1$RSB ansatz, we note that:
\begin{align}
\sum_{a=1}^{k} Q_{0a}^k = r b^k , \quad \quad
\sum_{a,b =1}^k Q_{ab}^k = r + \frac{r}{m} m (m-1) q_1^k + (r^2 - \frac{r}{m} m^2 ) q_0^k.
\end{align}
We define the $1$RSB free energy functional
\begin{align}
\Psi_{\sRSB{1}}(b,q_0,q_1,m) = \lim_{r \to 0} \frac{1}{r} S(\bQ^{\sRSB{1}}). 
\end{align} 
Plugging in the expressions obtained above and taking the limit as $r \to 0$, we obtain 
\begin{align}
\Psi_{\sRSB{1}}(b, q_0, q_1, m) &= \frac{\beta \lambda}{\sqrt{2 (k!)}} b^k + \frac{\beta^2}{4} [1 - (1-m) q_1^k - m q_0^k] 
+ \frac{1}{2} \frac{q_0 - b^2}{1- (1-m) q_1 - m q_0}\nonumber \\
&+ \frac{1}{2m} \log(1- (1-m) q_1 - m q_0) - \frac{1-m}{2m} \log (1 - q_1). \label{eq:Pshi1RSB}
\end{align}
Notice that $\Psi_{\sRSB{1}}$ also depends on $\beta,\lambda$, but we will omit this dependency unless needed.

Before analyzing the $1$RSB free energy, we carry out  some sanity checks to ensure the
 consistency of this solution with the RS solution studied earlier.
 Notice that the space of replica symmetric matrices $\bQ$ is a subset of
 the set of 1RSB matrices, and can be recovered in at least three different ways:
\begin{enumerate}
\item $q_1 = q_0 = q$. In this case, it is easy to check that 
the free energy is independent of $m$ and 
\begin{align*}
\Psi_{\sRSB{1}} (b, q, q,m) = \Psi_{\sRS}(b,q)\, . 
\end{align*}

\item $m\to 0$. It is somewhat less obvious that we should recover the RS ansatz
in this case, but we recall that $r=0$, and therefore this limit coincides with $m\to r$.
 Indeed, we obtain
\begin{align*}
\lim_{m\to 0}\Psi_{\sRSB{1}}(b,q_0, q_1, m) = \Psi_{\sRS}(b,q_1)\, . 
\end{align*}

\item $m = 1$. In this case the 1RSB ansatz $\bQ^{\sRSB{1}}$ reduces
to the RS one with $q=q_0$.
It is easy to see that in this case $\Psi_{\sRSB{1}}(b,q_0, q_1, m) \to \Psi_{\sRS}(b,q_0)$.
\end{enumerate}

Now we analyze the stationary points of the 1RSB free energy functional $\Psi_{\sRSB{1}}$. 
To simplify the analysis, throughout this discussion, we set 
\begin{align*}
\lambda=0\, .
\end{align*}
We will leave it as an exercise for the reader to generalize the discussion below to $\lambda>0$.

Recall, from equation \eqref{eq:overlap1} that $b$ is the asymptotic 
 scalar product  between the signal $\bvz$ and the estimate $\hbv_{\beta}(\bY)$. 
 As a result, we expect $b=0$ in for $\lambda=h=0$, and indeed, this choice satisfies the
 stationarity condition $\partial_b \Psi_{\sRSB{1}}=0$.

 In order to determine $q_0,q_1$, we need to solve the stationarity conditions: 
\begin{align}
\frac{\partial \Psi_{\sRSB{1}}}{\partial q_0} = \frac{\partial \Psi_{\sRSB{1}}}{\partial q_1 } =0. \nonumber 
\end{align}
It is easy to check that there are always stationary points of the form $(q_0 =0, q_1)$. 
We will restrict ourselves to stationary points of this form. This assumption is in fact not
required, and we could study the larger space of solutions with $q_0\ge 0$, 
but eventually the correct solutions (maximizing $\Psi_{\sRSB{1}}$) are obtained by setting $q_0=0$.
Although the choice $q_0=0$ appears ad hoc 
at this point, its motivation will become clear when we discuss the pure states decomposition
of the state space under the $1$RSB heuristic in the rest of this section.

Setting $q_0=b=0$ (and dropping for simplicity the dependence upon these arguments), 
the free energy 
functional reduces to the following simple form:
\begin{align}
\Psi_{1\rm{RSB}}(q_1, m) &= \frac{\beta^2}{4} [1 - (1-m)q_1^k ] + \frac{1}{2m} \log (1 - (1-m) q_1) - \frac{1-m}{2m} \log (1- q_1). \label{eq:SimplePsi1RSB} \\
\frac{\partial \Psi_{1\rm{RSB}}}{\partial q_1} &= - \frac{k\beta^2}{4} (1-m) q_1^{k-1} + \frac{1-m}{2} \frac{q_1}{(1- (1-m)q_1)(1-q_1)}. \nonumber 
\end{align} 
For $k\geq 3$, $q_1=0$ is always a stationary point. This corresponds to the paramagnetic
replica symmetric stationary point studied in the previous sections. 
Thus we look for conditions under which other non-vanishing stationary points appear.
 If we set $q_1 \neq 0$, any stationary point must satisfy the equation 
\begin{align}
f_1(q_1) : = q_1^{k-2}(1-q_1) (1- (1-m)q_1) = \frac{2}{k\beta^2} = \frac{2}{k} T^2. \label{eq:q_fixedpt}
\end{align}
Here $T=1/\beta$ is the temperature.

\begin{figure}
\centering
\begin{subfigure}[b]{0.45\linewidth}
\centering
\includegraphics[width=0.82\linewidth,keepaspectratio]{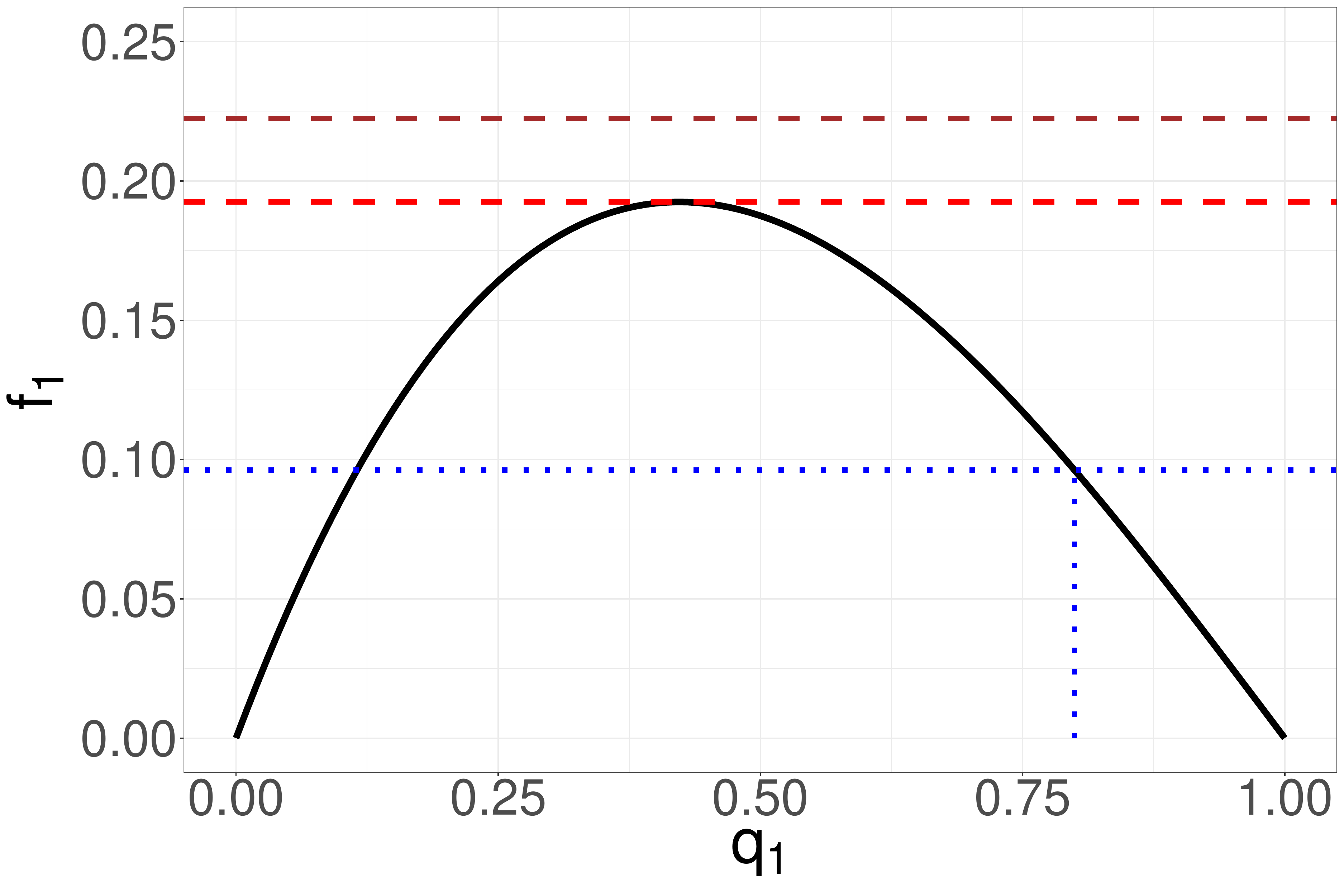}
\label{fig:pspin_rsb_stationary}
\caption{}
\end{subfigure}
\hspace{5pt}
\begin{subfigure}[b]{0.45\linewidth}
\centering
\includegraphics[width=0.82\linewidth,keepaspectratio]{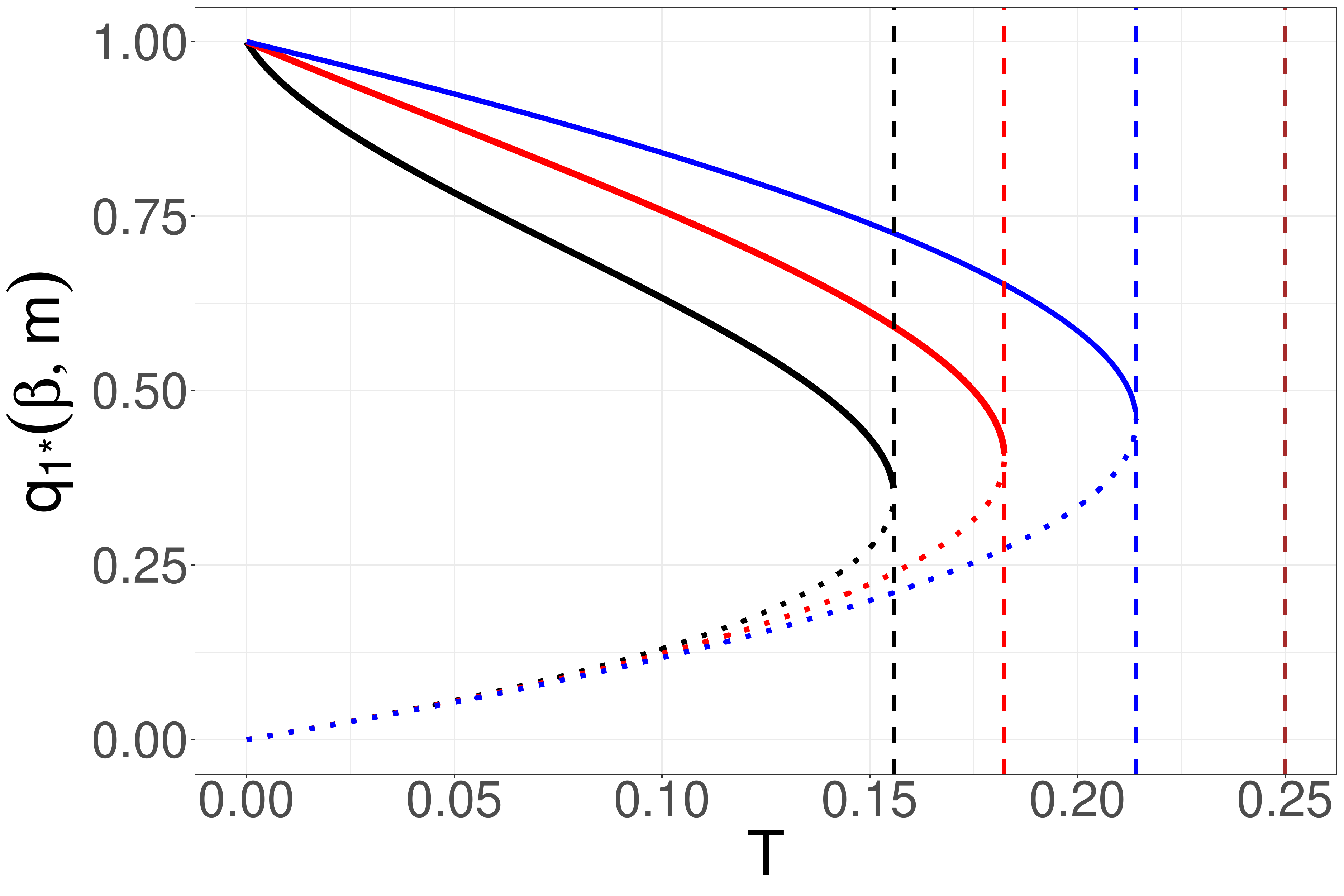}
\caption{}
\end{subfigure}
\caption{On the left, we plot $f_1$ as a function of $q_1$, $k=3$, $m=0.5$. In the high-temperature regime (shown in \small{\textsf{BROWN}}) \eqref{eq:q_fixedpt} does not have any solution. New solutions appear at $T= T_d(m)$ (shown in \textsf{RED}). For $T<T_d(m)$ (shown in \textsf{BLUE}), there are two roots---the larger root is physically relevant.  On the right, we plot the dependence of $q_{1*}(\beta,m)$ on $T$. The distinct curves represent different values of $m$.}
\label{figure_trial} 
\end{figure}

The behavior of the fixed points is shown in Fig \ref{figure_trial}.
 We note that for $T > T_d(m)$, the fixed point equation has no roots and therefore the paramagnetic fixed point is the only solution. Therefore, we expect that for $T$ sufficiently large, the replica symmetric free energy gives the right solution. Two new solutions appear for $T < T_d(m)$. We will choose the larger root $q_{1, *}(\beta, m)$ as the physically meaningful stationary point. To motivate this choice, recall the RS heuristic, where \eqref{eq:overlap2} motivates that $q_* \to 1$ as $\beta \to \infty$. Although we are working with the $1$RSB heuristic, this choice turns out to be the correct one. Further, analysis of the free energy functional reveals that for $T > T_d(m)$, $q_{1,*}(\beta)$ is a local minimizer of the free energy functional, and $\Psi_{1\rm{RSB}}(q_{1,*}(\beta, m) , m) < \Psi_{1\rm{RSB}}(0,m)$. The explicit dependence of $q_{1, *}(\beta,m)$ on $T$ and effect of varying $m$ are exhibited in Figure \ref{figure_trial}. 
 

 This clarifies the choice of $q_1$ for a fixed $m,\beta$. We next address the choice of $m$. 
 
 \subsection{Interpretation of the 1RSB ansatz and choice of $m$}
 
 Following the same heuristics as that for $q_0,q_1$, we will seek to minimize the functional with respect to $m$.
  However, we first need to specify the domain of the parameter $m$. 
  To this end, we postpone this analysis for the moment, and go back to the replica method 
  to develop an interpretation of the parameter $m$. 
  
  Recall the Gibbs measure $\mu_n$ of Eq.~\eqref{eq:PspinDef}, which we will write in compact form as 
  \begin{align}
  \mu_n(\de\bsigma) = \frac{1}{Z_n}\, e^{H_n(\bsigma)}\, \de\bsigma\, ,
  \end{align}
  with $\de\bsigma$ the uniform probability measure on the sphere $\sS^{n-1}$.
  Let $\bsigma_1, \bsigma_2\sim^{i.i.d.}\mu_n$
  (equivalently $(\bsigma_1,\bsigma_2)\sim\mu_n\otimes\mu_n$). 
  \begin{definition}[Replicas] 
  We will refer to i.i.d. samples from the Gibbs measure $\mu_n(\cdot)$ as \emph{replicas} following standard statistical physics terminology. 
  \end{definition} 
  We will compute the moments of 
  $\langle \bsigma_1, \bsigma_2 \rangle$, averaged with respect to the law of the Gaussian tensor $\bW$. 
  
  For any integer $\ell$, we have
 \begin{align}
 \E\Big[ \int_{(\sS^{n-1})^2} \langle \bsigma_1, \bsigma_2 \rangle^\ell \mu_n(\de \bsigma_1) \mu_n(\de \bsigma_2) \Big]
 &= \E\Big[ Z_n^{-2} \int_{(\sS^{n-1})^2} \langle \bsigma_1, \bsigma_2 \rangle^{\ell} e^{H_n(\bsigma_1) + H_n(\bsigma_2)}
  \de\bsigma_1 \de\bsigma_2  \Big] . \nonumber 
 \end{align}
 Using the replica trick, we have,
 \begin{align}
 \E\Big[ \int_{(\sS^{n-1})^2} \langle \bsigma_1, \bsigma_2 \rangle^{\ell} \mu_n(\de \bsigma_1) \mu_n(\de \bsigma_2) \Big] &=\
  \lim_{r \to 0} \frac{1}{\E[Z_n^r]}\E\Big[ Z_n^{r-2} \int_{(\sS^{n-1})^2} \langle \bsigma_1, \bsigma_2 \rangle^{\ell} e^{H_n(\bsigma_1) + H_n(\bsigma_2) } \de\bsigma_1 \de\bsigma_2  \Big].  \nonumber
 \end{align}
 We evaluate the right-hand side for $r$ a positive integer. Using the definition of the partition function $Z_n$, we obtain, 
 \begin{align}
 &\E\Big[ Z_n^{r-2} \int_{(\sS^{n-1})^2} \langle \bsigma_1, \bsigma_2 \rangle^{\ell} e^{H_n(\bsigma_1) + H_n(\bsigma_2)} \de\bsigma_1 \de\bsigma_2  \Big] \nonumber \\
 &= \E\Big[ \int_{(\sS^{n-1})^r} \langle \bsigma_1, \bsigma_2 \rangle^{\ell} e^{H_n(\bsigma_1) + \cdots + H_n(\bsigma_r)} \de\bsigma_1 \cdots \de \bsigma_r \Big] \nonumber \\
 &= \frac{2}{r(r-1)} \sum_{1\leq a < b \leq r} \E\Big[ \int \langle \bsigma_a , \bsigma_b \rangle^{\ell} e^{\sum_{a=1}^{r} H(\bsigma_a)}  \de\bsigma_1 \cdots \de \bsigma_r\Big] \nonumber \\
 &= \frac{2}{r(r-1)} \sum_{1\leq a < b \leq r} \int Q_{ab}^{\ell} \exp{(n g(\bQ))} f_{n,r}(\bQ) \de \bQ. 
 \end{align}
 Proceeding analogously for the factor  $\E[Z_n^r]$, and inverting as usual the limits $r\to 0$ and $n\to\infty$
 (and recalling the definition of $S(\bQ)$),
  we obtain
 \begin{align}
 \lim_{n\to\infty}\E\Big[ \int_{(\sS^{n-1})^2} \langle \bsigma_1, \bsigma_2 \rangle^\ell \mu_n(\de \bsigma_1) \mu_n(\de \bsigma_2) \Big]
 & =
  \lim_{r\to\infty}\frac{2}{r(r-1)} \sum_{1 \leq a < b  \leq r} \lim_{n\to \infty}\frac{\int Q_{ab}^\ell \exp{(n S(\bQ))} \de \bQ}
  {\int \exp{(n S(\bQ))} \de \bQ} . \label{eq:overlap3}
 \end{align}
 As $n \to \infty$, the integral is dominated by the overlap matrix $\bQ^{*}$ which maximizes the functional $S(\bQ)$. 
 We will evaluate the functional at the $1$RSB stationary point $\bQ^{\sRSB{1}}$:
 \begin{align}
  \lim_{n\to\infty}\E\Big[ \int_{(\sS^{n-1})^2} \langle \bsigma_1, \bsigma_2 \rangle^{\ell} \mu_n(\de \bsigma_1) \mu_n(\de \bsigma_2) \Big] 
  &= \lim_{r \to 0}  \frac{2}{r(r-1)} \sum_{1\leq a < b \leq r} (Q_{ab}^{\sRSB{1}})^{\ell}. \nonumber 
 \end{align}
 
 Next, we compute the right-hand side for the $1$RSB stationary point $\bQ^{\sRSB{1}}$ to obtain
 \begin{align}
 \sum_{1 \leq a < b \leq r} (Q_{ab}^{\sRSB{1}})^{\ell} = - \frac{r}{2} (1-m) q_1^{\ell} - \frac{r}{2} (1- (1-m)) q_0^\ell + O(r^2). \nonumber 
 \end{align}
 This implies 
 \begin{align}
   \lim_{n\to\infty}\E\Big[ \int \langle \bsigma_1, \bsigma_2 \rangle^\ell \mu_n(\de \bsigma_1) \mu_n(\de \bsigma_2) \Big]  = m q_0^\ell + (1-m) q_1^\ell. \nonumber 
 \end{align}
 If we set $\rho_n$ to be the law of the overlap $\langle \bsigma_1, \bsigma_2 \rangle$, 
 then the heuristic computation 
 above suggests that, as $n\to\infty$,
 \begin{align}
 \rho_n \stackrel{w}{\Rightarrow} \rho:= m \delta_{q_0} + (1-m) \delta_{q_1}. \label{eq:OverlapDistr1RSB}
 \end{align}
 (Here $\stackrel{w}{\Rightarrow}$ denotes weak convergence.)
 
Equation \eqref{eq:OverlapDistr1RSB} suggests  that we should minimize the free energy functional over $m \in [0,1]$. 
Further, it implies the following picture for the behavior of the random Gibbs distribution 
$\mu_{n}$ (We will focus for simplicity on the case $\beta=\lambda=h=0$).
 For $\beta$ sufficiently small, we expect the replica symmetric ansatz to be correct, 
and  $q_0 = q_1=0 $ or $m=1$. In this case, two iid samples $\bsigma_1, \bsigma_2$ are typically 
approximately orthogonal, i.e., $\langle \bsigma_1 , \bsigma_2 \rangle \approx 0$. 
This behavior is observed for iid samples from the uniform distribution on $\sS^{n-1}$.  
In this respect, for $\beta$ sufficiently small, the measure $\mu_{\beta}$ is qualitatively similar
 to the uniform distribution: it does not have strong correlation among its variables.
 
  When $\beta$ crosses a critical point,  a genuine the $1$RSB solution appears 
  (i.e. a solution with $q_0<q_1$ and $m\in (0,1)$).
  Hence, the behavior of the overlap distribution changes dramatically. In this case, for 
  $\bsigma_1, \bsigma_2\sim^{i.i.d.}\mu_{\beta}$, the overlap $\<\bsigma_1,\bsigma_2\>$ 
  concentrates around one of two values, either $q_0 =0$ or $q_1 > 0$. The interpretation
  is as follows.
  The measure $\mu_{n}$ concentrates around regions of sphere which correspond to high values 
  of the Hamiltonian $H_n(\,\cdot\,)$. For $\beta$ sufficiently large, there are 
  $O(1)$ such regions on the sphere, that are well-separated, i.e. separated a distance of order one. 
  These regions are often referred to as ``pure states" or ``clusters" in the physics parlance.
 The fact that $q_0=0$ indicates that these regions are 
   approximately orthogonal. For two iid samples drawn from the measure 
$\mu_{n}$, either they both belong to the same cluster, in which case their overlap concentrates around $q_1$, 
or they belong to different clusters, in which case their overlap concentrates 
around $0$. Intuitively, we expect $q_1 \to 1$ as $\beta\to \infty$, as the pure state concentrates 
very close to the configurations with the highest values of the Hamiltonian $H_n(\, \cdot\, )$. 

We can decompose the Gibbs measure into pure states as
\begin{align}
\mu(\de\bsigma) &= \sum_{\alpha=1}^Nw_{\alpha}\mu_{n,\alpha}(\de\bsigma) \, ,\\
\mu_{n,\alpha}(\de\bsigma)& = \frac{1}{\mu_n(\Omega_{\alpha})}\mu_n(\de\bsigma) \bfone_{\Omega_{\alpha}}(\bsigma)\,,\;\;\;\;
w_{\alpha} = \mu_n(\Omega_{\alpha})\, .
\end{align}
Under this description of the measure $\mu_{\beta}$, the parameter $m$ assumes added significance. 
Note that $w_{\alpha}$ is the probability that $\bsigma$ belongs to state $\alpha$, and therefore the
probability  that two replicas belong to the same state is given by $\sum_{\alpha} w_{\alpha}^2$.
Therefore we expect
\begin{align}
1-m = \lim_{n\to\infty}\E\Big[ \sum_{\alpha=1}^N w_a^2 \Big]. 
\end{align}
As we will see below, the weights $(w_{\alpha})$ are random and not equal, even in the limit $n\to\infty$.
If we nevertheless assumed there was a non random number of states $\overline{N}$, all
of equal size, the above would imply $\overline{N} = 1/(1-m)$.
While this last statement is ---as we said--- incorrect, and in fact we will correct it below,
the intuition is still valid. The number $m$ quantifies the number of 
pure states, with $m$ close to $0$ corresponding to one or a few pure states, while $m$ close to one
corresponding to a large number of pure states  (still order one in $n$). 

\subsection{Back to the free energy}

Next, we return to the study of the $1$RSB free energy functional $\Psi_{\sRSB{1}}(q_1, m)$. We wish to minimize
 this functional jointly with respect to $q_1, m$. We have already studied the minimization problem with respect to
  $q_1$, for any fixed value of $m \in [0,1]$ yielding  the minimizer $q_{1,*}(\beta, m)$. 
  
We now study the minimization of $\Psi_{1\rm{RSB}} (q_{1,*}(\beta,m), m)$ as a function of $m$. 
Recall that
\begin{align}
\Psi_{1\rm{RSB}}(0,m) = \frac{\beta^2}{4} = \Psi_{\sRS}(q_0 =0), \nonumber \\
\Psi_{1\rm{RSB}}(q_1, 1) = \frac{\beta^2}{4} = \Psi_{\sRS}(q_0 =0). \nonumber 
\end{align}
Thus if we set $q_1=0$, the value is independent of $m$. We further recall that $q_1=0$ is always a stationary point. 
Thus at sufficiently high temperature, $q_1=0$ is the only stationary point, and we recover the replica symmetric solution. 

At lower temperatures, we set 
\begin{align}
g(m) \equiv \Psi_{1\rm{RSB}}(q_{1,*}(\beta,m),m) - \frac{\beta^2}{4}, \nonumber 
\end{align}
where for any $m > 0$, we recall that $q_{1,*}(\beta,m)$ is the largest root of the equation \eqref{eq:q_fixedpt},
 whenever such a root exists. 

\begin{figure}
\centering
\begin{subfigure}[b]{0.45\textwidth}
\centering
\includegraphics[width=0.9\linewidth,keepaspectratio]{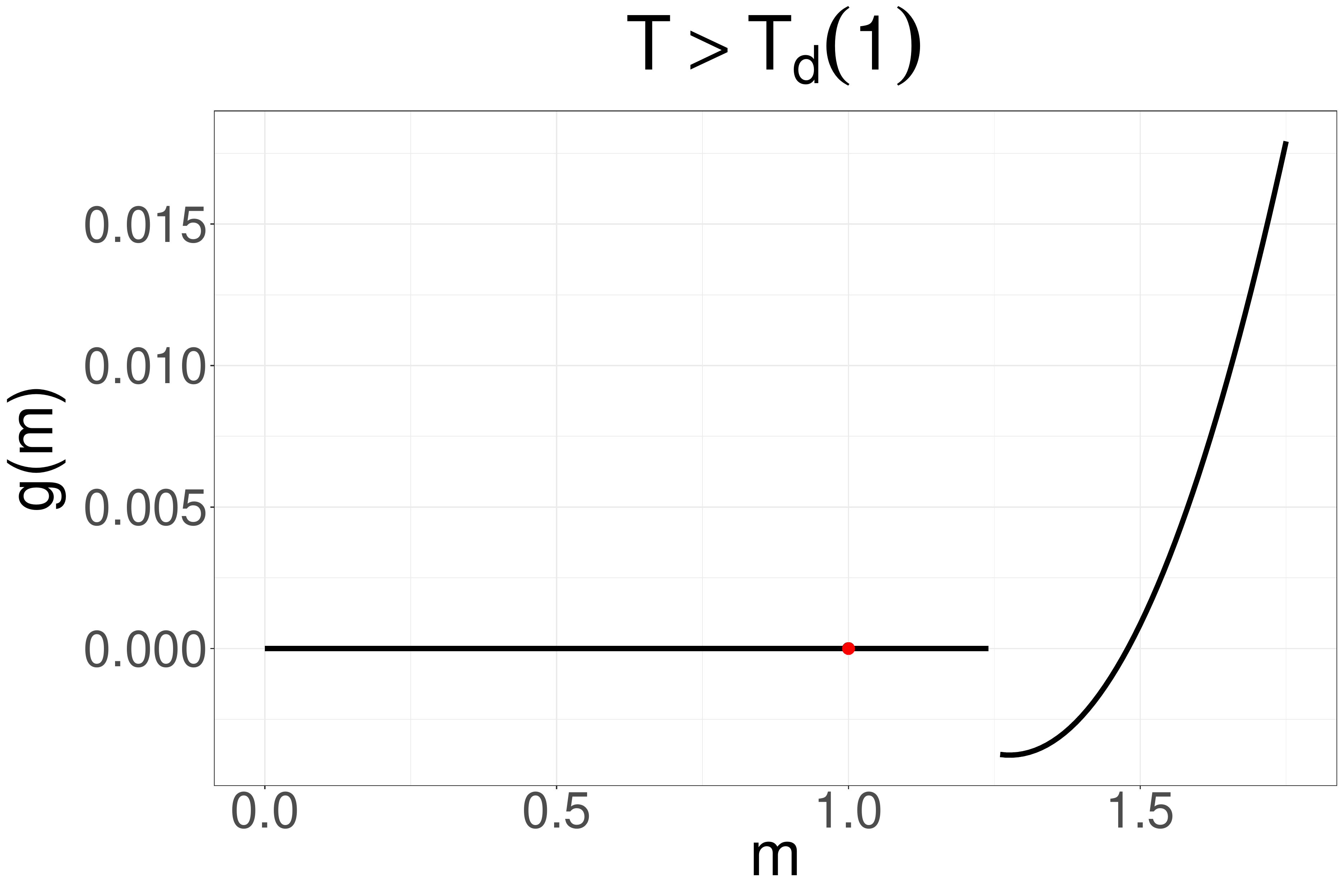}
\label{fig:TbigTd}
\caption{}
\end{subfigure}
\hspace{1pt}
\begin{subfigure}[b]{0.45\textwidth}
\centering
\includegraphics[width=0.9\linewidth,keepaspectratio]{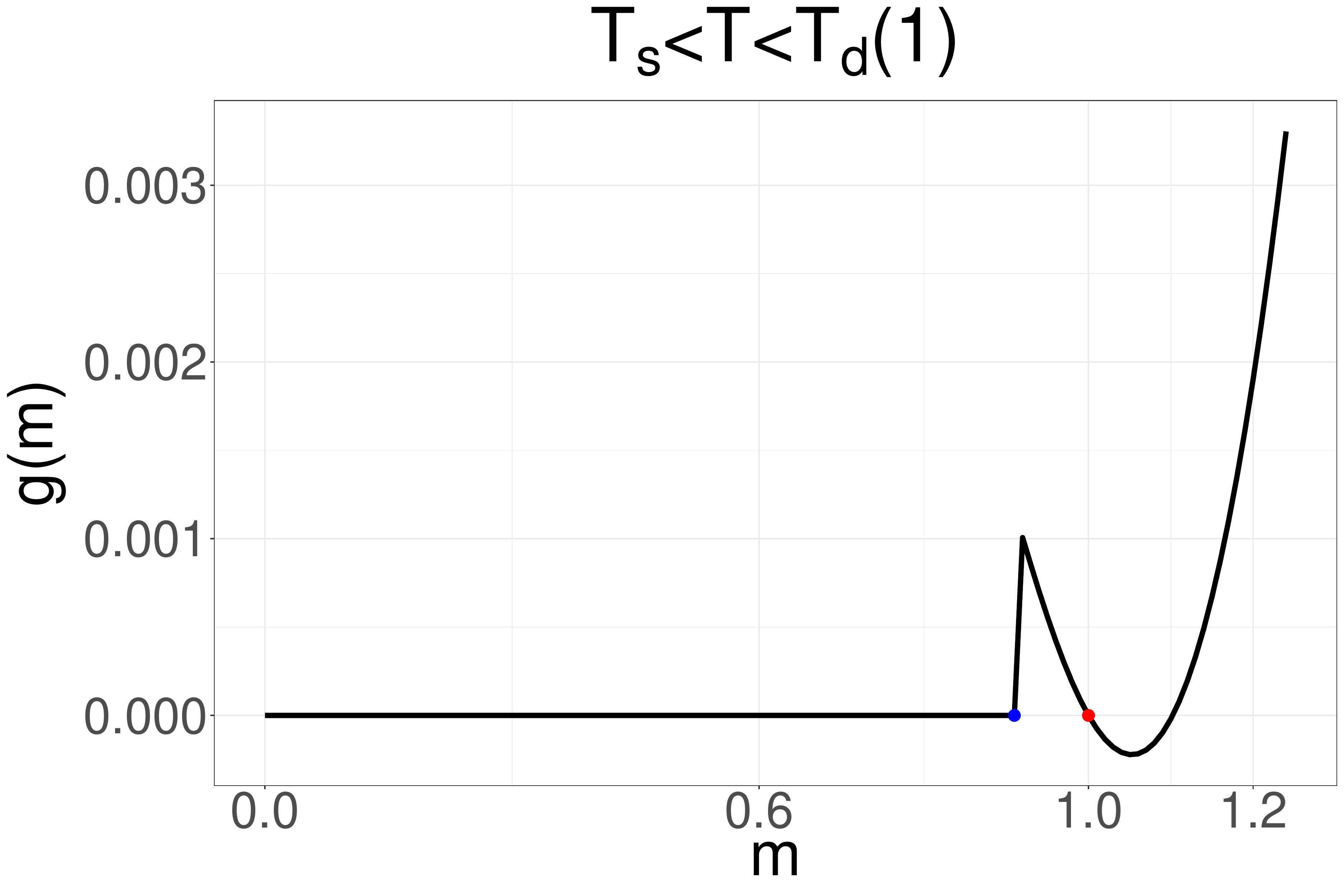}
\caption{}
\end{subfigure}
\newline
\begin{subfigure}[b]{0.45\textwidth}
\centering
\includegraphics[width=0.9\linewidth,keepaspectratio]{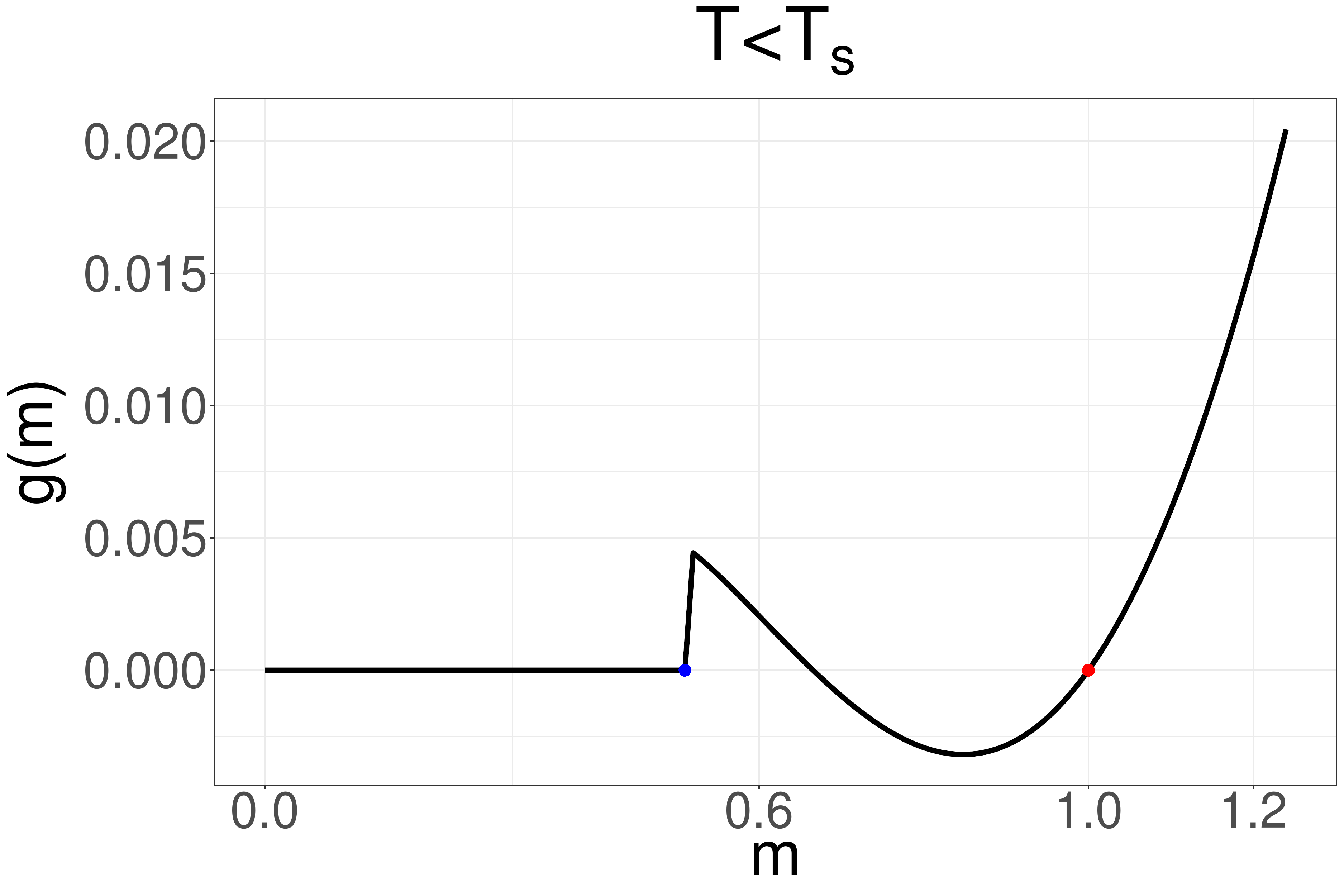}
\caption{}
\end{subfigure}
\caption{We plot $g(m)$ vs. $m$ for $k=3$. In subplot (a), we consider $T > T_d(1)$; the only physically relevant solution is $q_1=0$ and we are back to the paramagnetic solution. In (b), we consider $T_s < T < T_d(1)$; $m(T)$ is represented in \textsf{BLUE}. The function $g$ restricted to the interval $[m(T),1]$ is non-negative, and thus the replica symmetric free energy is correct in this regime. Finally (c) represents the setting $T<T_s$. $g(m)$ on the interval $[m(T),1]$ (represented by the \textsf{BLUE} and \textsf{RED} points attains negative values). Thus the 1RSB free energy strictly improves on the replica symmetric approximation.  }
\label{fig:psivsm} 
\end{figure}


We encounter the following cases according to the temperature $T$:
\begin{enumerate}
\item \emph{$T > T_d(1)$.} In this case, a non-trivial solution $q_{1,*}(\beta,m)$ exists only 
for $m \geq  m(T) >1$. The only physically relevant fixed point has $q_1=0$ and we are back to the paramagnetic solution. 
\item \emph{$T_d(1) > T > T_s$.} In this case, there exists $m(T)$ such that a non-trivial fixed point $q_*(\beta, m)$ exists for $m \in [m(T), 1]$. However, in this regime, $\Psi_{1\rm{RSB}}(q_{1,*}(\beta,m),m) > \Psi_{\sRS}(q_0=0)$ for all $m \in [m(T),1]$. In this case, as before, the relevant fixed point is $m=1$, a situation where $q_1$ is undetermined. Thus the replica symmetric free energy is still correct in this temperature regime.

\item \emph{$T< T_s$,} In this regime, there exists $m(T)$ such that a non-trivial fixed point $q_*(\beta, m)$ exists for $m \in [m(T), 1]$. Moreover, there exists $m_{*}(T) \in [m(T),1]$ such that 
\begin{align}
\Psi_{1\rm{RSB}}(q_{1,*}(\beta,m_*(T)), m_*(T))  < \Psi_{\sRS}(q_0=0). \nonumber 
\end{align}
In this case, we choose the fixed point $(m,q_1) = (m_*(T), q_{1,*}(\beta, m_*(T)))$. 

\end{enumerate}
Finally, we exhibit the dependence of $m_*(T)$ on $T$ in Figure \ref{fig:monT}. 
\begin{figure}
\centering
\hspace{2pt}
\includegraphics[width=.52\linewidth,height=0.36\linewidth]{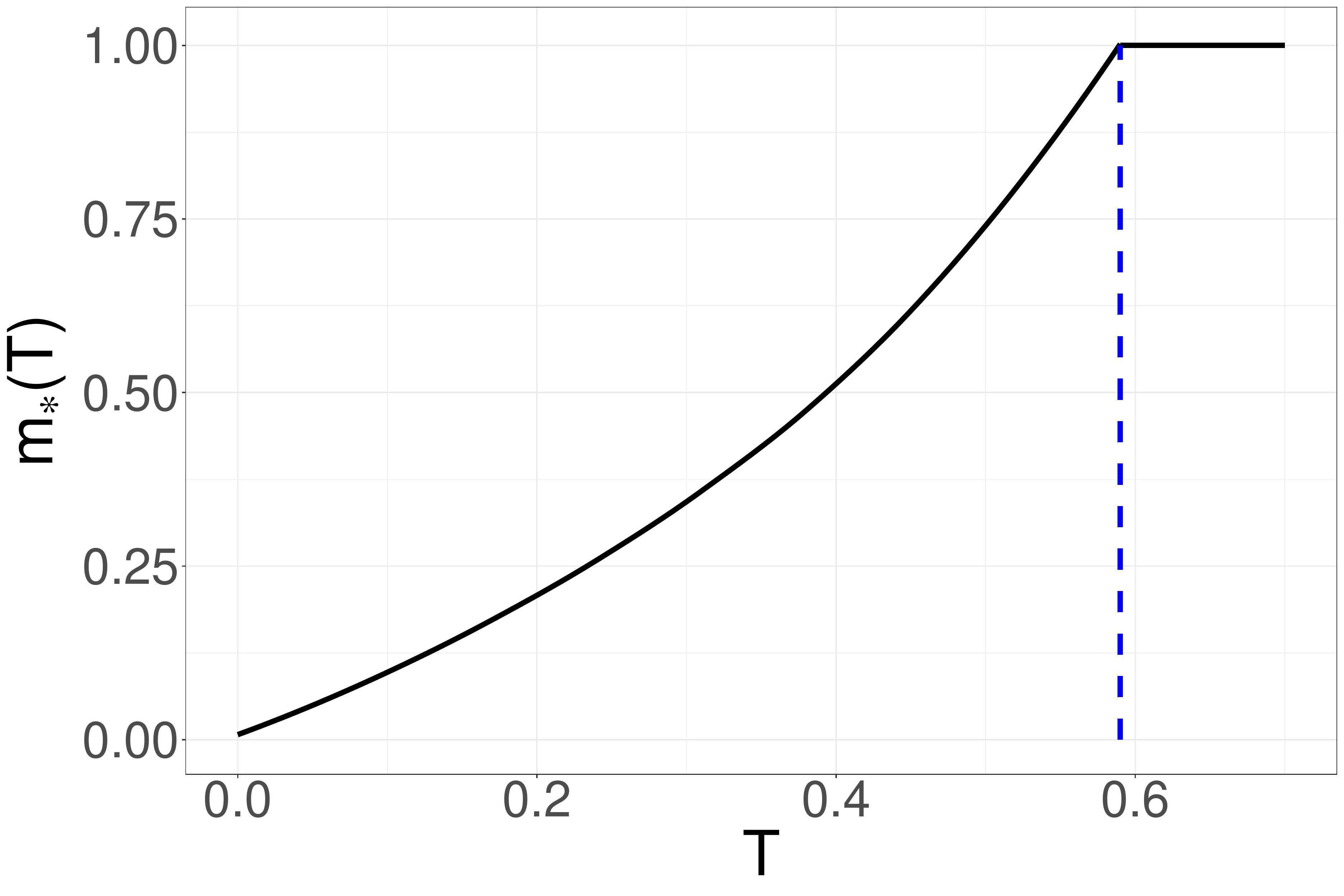}
\caption{The dependence of $m_*(T)$ on $T$. Note the critical static temperature $T_s$ (in blue)}
\label{fig:monT}
\end{figure} 

One might wonder whether the transition temperature $T_d(1)$  has any physical significance,
since the free energy density does not exhibit any non-analyticity at this temperature.
It turns out that his temperature corresponds to an important change in the 
structure of the Gibbs measure, which we will further explore below, but is useful to summarize here.

 Imagine that we lower the temperature $T$ of the system continuously. 
 At sufficiently high temperature, the 
 measure $\mu_{\beta}$ behaves roughly like the uniform measure. 
 At $T_d(1)$, the system spontaneously 
 breaks up into exponentially many pure states, each of which has negligible weight. 
Since the number of states that contribute to the Gibbs measure is exponentially large,
two replicas have negligible probability to belong to the same state. Hence, the overlap
concentrates around a single value, and the replica-symmetric free energy gives the correct asymptotics.

 While this transition does not affect the free energy,  it is 
 important in determining the  behavior of various sampling algorithms. 
This splitting of the state space into many pure states should hinder the 
 exploration of the state space by 
 Glauber or Langevin  dynamics. This is sometimes referred to as  the dynamical
 phase transition. 
 
 As the temperature is further lowered,
   the pure states become smaller and a smaller number of them dominates the Gibbs measure. 
   At the point $T_s$, an $O(1)$ number of pure 
   states dominate, and this dramatically changes the nature of the measure $\mu_{n}$,
   and the overlap distribution.
    This  is usually referred to as the static phase transition (or `random first order'
    phase transition). 
    Below this temperature, the replica symmetric free energy fails to be correct, and must 
    instead be replaced by the $1$RSB free energy.

%
%
\section{The number of critical points}
\label{sec:critical_points} 

In the previous sections, we have studied the Gibbs measure 
$\mu_{n}(\de\bsigma)= \exp\{H(\bsigma)\}\de\bsigma/Z_n$ using the replica method.
 In this section, we turn to the direct study of properties of the landscape $H(\bsigma)$.
 
 We will focus again on the case $\lambda=0$. We rescale the Hamiltonian by a factor 
 $n$ to define $\hd(\bsigma)\equiv H(\bsigma)/n$, that is 
 \begin{align}
 \hd(\bsigma) = \frac{1}{\sqrt{2(k!)}} \langle \mathbf{W}, \bsigma^{\otimes k} \rangle\, ,
 \end{align}
 where the symmetric Gaussian tensor $\bW$ is defined as per Eq.~\eqref{eq:Wconstruction}.
 (This is not to be confused with the perturbation $h\<\bvz,\bsigma\>$ that we previously added to the Hamiltonian and now set to zero.)
 
Of course, the free energy density can be thought of as a way to explore the structure of 
the energy $H(\bsigma)$:
\begin{align}
\phi(\beta) = \lim_{n \to \infty} \frac{1}{n} \log \int \exp{\{ n \beta \hd (\bsigma) \}} \nu_0(\d\bsigma). \nonumber 
\end{align}
For instance, it is not hard to prove that
\begin{align}
\lim_{\beta \to \infty} \frac{\phi(\beta)}{\beta} = \lim_{n\to \infty} 
\E\Big[ \max_{\bsigma \in \sS^{n-1}} \hd(\bsigma) \Big]. \nonumber 
\end{align}
The quantity on the right-hand side is referred to as the `ground state energy'  in physics
(the usual convention in physics is to call $-\hd$ the Hamiltonian, and minimize the latter).

Here, we will take a different approach, and study the structure of the critical points
of the function $\hd(\,\cdot\,)$ directly. 
 For $B\subset \mathbb{R}$, we denote by $C_n(B)$ the number of critical points $\bsigma$
 of $\hd$ with $\hd(\bsigma)\in B$:
\begin{align}
C_n(B) = \big| \{ \bsigma : {\rm grad}\, \hd (\bsigma)  =0, \;\; \hd (\bsigma) \in B \} \big|\, . 
\end{align}
Here ${\rm grad}\hd $ denotes the Riemannian gradient on the sphere $\sS^{n-1}$
and $|A|$ is the cardinality of the set $A$. In suitable coordinates, we have
${\rm grad} \hd (\bsigma) = \proj_{\bsigma}^{\perp}\nabla \hd(\bsigma)$,
where $\proj_{\bsigma}^{\perp} = \id - \bsigma \bsigma^{\sT}$ is the projector onto
 the linear space orthogonal to $\bsigma$. 
 
 In the rest of this section we will estimate the expectation of $C_n(B)$.
 
 \subsection{The Kac-Rice formula}
 
The Kac-Rice formula allows us to compute the expected number of critical 
points for Gaussian processes, subject to certain regularity conditions. 
\begin{lemma}[Kac-Rice formula \cite{adler2007random}] 
\label{lemma:kac-rice}
We have, for any Borel set $B \subset \mathbb{R}$, 
\begin{align}
\E[C_n(B)] &= \int_{\sS^{n-1}} \E[ D(\bsigma) \mathbf{1}_{\{\hd(\bsigma) \in B\} } | \proj_{\bsigma}^{\perp} \nabla \hd(\bsigma) =0] \phi_{\bsigma}(0) \de \bsigma\, , \\
D(\bsigma) &\equiv\big|\det_{\bsigma^{\perp} } [ \proj_{\bsigma}^{\perp} \nabla^2 \hd(\bsigma) \proj_{\bsigma}^{\perp} - \langle \bsigma, \nabla \hd(\bsigma) \rangle \id ] |,  \nonumber 
\end{align}
where $\phi_{\bsigma}(\,\cdot\,)$ is the density of $\proj_{\bsigma}^{\perp}\nabla \hd(\bsigma)$ 
with respect to the Lebesgue measure on $V^{\perp}_{\bsigma}\cong\mathbb{R}^{n-1}$
(where $V^{\perp}_{\bsigma}$ is the tangent space to $\sS^{n-1}$ at $\bsigma$,
which we identify with the linear space orthogonal to $\bsigma$). 
Finally, the determinant is computed on $V^{\perp}_{\bsigma}$ (namely, $\det_{\bsigma^{\perp}}(\bM)$
is the determinant of the matrix obtained by representing the operator $\bM$ 
with respect to an arbitrary basis on $V^{\perp}_{\bsigma}$).
\end{lemma}

A detailed proof of the above lemma may be found in the paper by Auffinger, Ben-Arous, \u{C}ern\'{y} 
 \cite{auffinger2013random}, derived using the differential geometric framework of 
  \cite{adler2007random}. To keep the discussion self-contained, we will sketch the
   proof of this result.
\begin{proof}
We will consider the special case $B = \mathbb{R}$,
i.e. computing the expectation of the total number of critical points  $C_n=C_n(\reals)$.
The generalization to other sets $B$ uses the same ideas.

We note that $\bsigma$ is a critical point if and only if $\proj_{\bsigma}^{\perp}\nabla H(\bsigma) = 0$. For $\varepsilon >0$, define $\eta_{\varepsilon}: \mathbb{R}^{n} \times \sS^{n-1} \to \mathbb{R}$, 
\begin{align}
\eta_{\varepsilon}(\bv, \bsigma) = 
\begin{cases}
\frac{1}{\overline{\omega}_{n-1} \varepsilon^{n-1} }\quad {\rm{if}} \quad  \| \proj_{\bsigma}^{{\perp}} \bv \|_2 \leq \varepsilon, \nonumber  \\
0 \quad \rm{o.w.} , 
\end{cases} 
\end{align}
where $\overline{\omega}_n$ is the volume of the unit ball in $\mathbb{R}^n$. We define 
\begin{align}
C_{n,\varepsilon} = \{ \bsigma \in \sS^{n-1} : \| \proj_{\bsigma}^{\perp} \nabla H(\bsigma) \|_2 \leq \varepsilon \}. \nonumber 
\end{align}
We will need the following lemma. 
\begin{lemma}
\label{lemma:smooth}
If all the critical points are non-degenerate, then $\displaystyle\lim_{\varepsilon \to 0} C_{n, \varepsilon} = C_n$. 
\end{lemma}
We delay the proof of Lemma \ref{lemma:smooth} and complete the sketch. We have, 
\begin{align}
\E[C_n] &= \E\Big[ \lim_{\varepsilon \to 0} \int_{\sS^{n-1}} D(\bsigma) \eta_{\varepsilon}(\nabla H(\bsigma); \bsigma) \de \bsigma \Big] \nonumber\\
&\stackrel{(a)}{=} \int_{\sS^{n-1}} \lim_{\varepsilon \to 0} \E[D(\bsigma) \eta_{\varepsilon} (\nabla H(\bsigma) ; \bsigma)] \de \bsigma \nonumber\\
&= \int_{\sS^{n-1}} \lim_{\varepsilon \to 0} \E[D(\bsigma) | \| \proj_{\bsigma}^{\perp} \nabla H(\bsigma) \|_2 \leq \varepsilon] \frac{1}{\overline{\omega}_{n-1} \varepsilon^{n-1}} \mathbb{P} [ \| \proj_{\bsigma}^{\perp}\nabla H(\bsigma) \|_2 \leq \varepsilon] \de \bsigma \nonumber \\
&= \int_{\sS^{n-1}} \E[D(\bsigma) | \proj_{\bsigma}^{\perp} \nabla H(\bsigma) =0 ] \phi_{\bsigma}(0) \de \bsigma. \nonumber 
\end{align}
Note that the inversion of limit and expectations in step $(a)$ requires further 
justification, for which we defer to  \cite{adler2007random,auffinger2013random}.
\end{proof}

We next sketch the proof of Lemma \ref{lemma:smooth}.
\begin{proof}[Proof sketch of Lemma \ref{lemma:smooth}]
Note that $\| \proj_{\bsigma}^{\perp}\nabla \hd(\bsigma) \|_2 = 0$ if and only if $\bsigma$ 
is a critical point. Further, in this case, it is possible to check that there are only 
finitely many critical points almost surely. We will refer to these critical points 
as $\bsigma_1, \cdots, \bsigma_m$. 
We define 
\begin{align}
\Omega_{\varepsilon} = \{ \bsigma \in \sS^{n-1} : \| \proj_{\bsigma}^{\perp} \nabla \hd(\bsigma) \|_2 \leq \varepsilon \}. \nonumber 
\end{align}
For all $\varepsilon>0$ small enough, $\Omega_{\varepsilon} = \cup_{i=1}^{m} \Omega_{\varepsilon}^{(i)}$
 where $\bsigma_i \in \Omega_{\varepsilon}^{(i)}$ and $\Omega_{\varepsilon}^{(i)} \subseteq \Ball(\bsigma_i , \delta(\varepsilon))$ where 
$\Ball(\bx_0,r)$ is the ball of radius $r$ around $\bx_0$ and $\delta(\varepsilon) \downarrow 0$ 
as $\varepsilon \to 0$. In this case, we have, 
\begin{align}
\Omega_{\varepsilon} &= \sum_{i=1}^{n} \int_{\Omega_{\varepsilon}^{(i)}} 
D(\bsigma) \frac{1}{\overline{\omega}_{n-1} \varepsilon^{n-1}} \de \bsigma \nonumber \\
&= \sum_{i=1}^{n} \frac{D(\bsigma_i)}{\overline{\omega}_{n-1} \varepsilon^{n-1}} {\rm{Vol}}_{\sS^{n-1}} (\Omega_{\varepsilon}^{(i)}) ( 1 + o_{\varepsilon}(1) ). \nonumber 
\end{align}
Thus it remains to compute ${\rm{Vol}}_{\sS^{n-1}}(\Omega_{\varepsilon}^{(i)})$. 
For $\varepsilon>0$ sufficiently small, we consider the approximate decomposition 
$\bsigma = \sqrt{ 1 - \| \bx \|_2^2 } \,\,\bsigma_i + \bx$, where 
$\bx \in V_{\bsigma}^{\perp}$. We set 
\begin{align}
\overline{\Omega}_{\varepsilon}^{(i)} = \{ \bx \in V_{\bsigma_i}^{\perp}: \sqrt{1 - \| \bx\|_2^2} \,\,\bsigma_i + \bx \in \Omega_{\varepsilon}^{(i)} \}. \nonumber 
\end{align}
For $\varepsilon$ small, we have, 
\begin{align}
{\rm{Vol}}_{\sS^{n-1}} (\Omega_{\varepsilon}^{(i)}) = {\rm{Vol}}_{V_{\bsigma}^{\perp}}(\overline{\Omega}_{\varepsilon}^{(i)}) (1 + o_{\varepsilon}(1)). \nonumber 
\end{align}
To compute ${\rm{Vol}}_{V_{\bsigma}^{\perp}}(\overline{\Omega}_{\varepsilon}^{(i)})$, we note that for any vector $\bsigma(x) = \bsigma_i + \bx$, we have, 
\begin{align}
\proj_{\bsigma(\bx)}^{\perp} = I - \bsigma(\bx) \bsigma(\bx)^{T} = \proj_{\bsigma_i}^{\perp} - \bx \bsigma_i^{T} - \bsigma_i  \bx^{T} + O(\| \bx \|_2^2). \nonumber 
\end{align}
Thus we have,
\begin{align}
\nabla \hd(\bsigma( \bx)) &= \nabla \hd(\bsigma_i) + \nabla^2 \hd(\bsigma_i) \bx + O(\| \bx \|_2^2). \nonumber \\
\proj_{\sigma}^{\perp}\nabla \hd(\bsigma( \bx)) &= [ \proj_{\bsigma_i}^{\perp} \nabla^2 \hd(\bsigma_i) \proj_{\bsigma_i}^{\perp} - \langle \bsigma_i , \nabla \hd(\bsigma_i) \rangle\id] \bx + O(\| \bx \|_2^2) := M(\bsigma_i) \bx + O(\| \bx\|_2^2). \nonumber 
\end{align}
We define $\mathcal{E}_{\varepsilon}^{(i)} = \{ \bx \in V_{\bsigma_i}^{\perp} : \| M(\bsigma_i) \bx \|_2 \leq \varepsilon \}$. The above approximate computation suggests that 
\begin{align}
{\rm{Vol}}(\overline{\Omega}_{\varepsilon}^{(i)}) = {\rm{Vol}}(\mathcal{E}_{\varepsilon}^{(i)}) ( 1 + o_{\varepsilon}(1)). \nonumber 
\end{align}
We note that $\mathcal{E}_{\varepsilon}^{(i)} = \{ \bx \in V_{\bsigma_i}^{\perp} : \| M(\bsigma_i) \bx \|_2 \leq \varepsilon \}$ is an ellipse. Therefore, the volume follows upon changing coordinates and computing the volume of a sphere. We conclude that 
\begin{align}
{\rm{Vol}}(\overline{\Omega}_{\varepsilon}^{(i)}) = {\rm{Vol}}(\mathcal{E}_{\varepsilon}^{(i)}) (1 + o_{\varepsilon}(1)) = D(\bsigma_i) \overline{\omega}_{n-1} \varepsilon^{n-1}. \nonumber 
\end{align}
This concludes the proof sketch.  
\end{proof}

\subsection{Applying the Kac-Rice formula}

We next apply the Kac-Rice Formula \ref{lemma:kac-rice} to the spherical $p$-spin model,
obtaining the following exact expression for the expected number of critical points.
\begin{lemma}
\label{lemma:crit_points}
For any Borel set $B \subset \mathbb{R}$, we have, 
\begin{align}
\E[C_n(B)] = \Big[\frac{(k-1)(n-1)}{2} \Big]^{\frac{n-1}{2}} \frac{2 \sqrt{\pi}}{\Gamma(\frac{n}{2})} \E\big[\big|{\rm{det}}(\bW_{n-1} - t Z \bI_{n-1})\big| \mathbf{1}_{\{ Z \in \sqrt{2n} B \} }\big], \label{eq:critical-points}
\end{align}
where $\bW_{n-1}\sim \GOE(n-1)$ (in particular has independent entries above the diagonal 
$W_{ij}\sim \normal(0,1/(n-1))$) is independent of $Z \sim \normal(0,1)$, and $t\equiv \sqrt{\frac{k}{(k-1)(n-1)}}$. 
\end{lemma}

\begin{proof}
For our specific Gaussian process $\hd(\bsigma)$, recalling the definition of the symmetric Gaussian tensor, we have, 
\begin{align}
\hd(\bsigma) &= \frac{1}{\sqrt{2n}} \frac{1}{k!} \sum_{\pi, i_1 , \cdots, i_k} G^{\pi}_{i_1, \cdots, i_k} \sigma_{i_1} \sigma_{i_2} \cdots \sigma_{i_k}. \nonumber \\
(\nabla \hd(\bsigma))_i &= \frac{k}{\sqrt{2n}} \frac{1}{k!}  \sum_{\pi, l_1, \cdots, l_{k-1}} G^{\pi}_{i, l_1, \cdots, l_{k-1}} \sigma_{l_1} \cdots\sigma_{ l_{k-1}}. \nonumber\\ 
(\nabla^2 \hd(\bsigma))_{ij} &= \frac{k(k-1)}{\sqrt{2n}} \frac{1}{k!} \sum_{\pi, l_1, \cdots, l_{k-2}} G^{\pi}_{i,j, l_1, \cdots, l_{k-2}} \sigma_{l_1} \cdots \sigma_{l_{k-2}}. \nonumber 
\end{align}
Thus $(\hd(\bsigma), \nabla \hd(\bsigma) , \nabla^2 \hd(\bsigma))$ is jointly Gaussian. We note that the covariances depend only on the inner products and are thus invariant under rotations. Therefore, to describe the marginal distribution of $(\hd(\bsigma), \nabla \hd(\bsigma) , \nabla^2 \hd(\bsigma))$ for any fixed $\bsigma \in \sS^{n-1}$, we will assume, without loss of generality that $\bsigma = \mathbf{e}_1 = (1, 0, \cdots,0)$. In this case, it is easy to see that 
\begin{align}
\hd(\bsigma) &= \frac{1}{\sqrt{2n}} G_{1, \cdots,1} \nonumber \\
(\nabla \hd(\bsigma) )_i &= \frac{1}{\sqrt{2n}} [ G_{i,1, \cdots,1} + G_{1,i, 1,\cdots,1} + \cdots + G_{1, \cdots, 1, i}]. \nonumber \\
(\nabla^2 \hd(\bsigma))_{ij} &= \frac{1}{\sqrt{2n}} [G_{i,j,1,\cdots, 1} + \cdots + G_{1, \cdots,1, i,j}]. \nonumber 
\end{align}
Thus in this case, $(\hd(\bsigma), \proj_{\bsigma}^{\perp} \nabla \hd(\bsigma) , \proj_{\bsigma}^{\perp} \nabla^2 \hd(\bsigma) \proj_{\bsigma}^{\perp})$ are mutually independent and the marginal distributions may be described as follows. 
\begin{align}
\hd(\bsigma) &\ed \frac{1}{\sqrt{2n}} Z  \nonumber \\
\proj_{\bsigma}^{\perp}\nabla \hd(\bsigma) &\ed \sqrt{\frac{k}{2n}} g_{n-1}, \nonumber \\
\proj_{\bsigma}^{\perp} \nabla^2 \hd(\bsigma) \proj_{\bsigma}^{\perp} &\ed \sqrt{\frac{k(k-1) (n-1)}{2n}} \bW_{n-1}, \nonumber
\end{align}
where $Z\sim \normal(0,1)$, $g_{n-1} \sim \normal(0, \id_{n-1})$ and 
$\bW_{n-1}\sim GOE(n-1)$. As stated earlier, $Z, g, \bW$ are independent. 

We can now use Lemma \ref{lemma:kac-rice} in our setup. We have, 
$\phi_{\bsigma}(0) = \Big( \frac{n}{\pi k} \Big)^{\frac{n-1}{2}}$. 
Also, we note that the integrand in Lemma \ref{lemma:kac-rice} is independent of $\bsigma$ in this case. Further, we observe that 
$\langle \bsigma, \hd(\bsigma) \rangle = k \hd(\bsigma)$ and therefore, using independence, we are left with the unconditional expectation. This implies that 
\begin{align}
\E[C_n(B)] =  \Big( \frac{n}{\pi k} \Big)^{\frac{n-1}{2}} \frac{2 \pi^{n/2}}{\Gamma(\frac{n}{2})} \E\Big[ \Big| {\rm{det}}\Big(\sqrt{\frac{k(k-1) (n-1)}{2n}} \bW_{n-1} - \frac{k}{\sqrt{2n}} Z  \bI_{n-1}  \Big) \Big| \mathbf{1}_{ \{ \frac{1}{\sqrt{2n}}Z \in B \} } \Big], \nonumber 
\end{align}
where we have plugged in $A_n = \frac{2 \pi^{n/2}}{\Gamma(\frac{n}{2})}$, the surface measure of 
$\sS^{n-1}$. A final simplification of the coefficients completes the proof. 
\end{proof}

\subsection{Asymptotics of the number of critical points}

Armed with Lemma \ref{lemma:crit_points}, we can study the asymptotic behavior of the 
expected number of critical points. 
\begin{lemma}
We have, as $n\to \infty$, 
\begin{align}
\E[C_n(B)] \doteq \exp\big\{n \sup_{x \in B}  S(x)\big\}\, , \nonumber
\end{align}
where 
\begin{align}
S(x) & \equiv \Omega \Big(x \sqrt{\frac{2k}{k-1}} \Big) - x^2 + \frac{1}{2} \log (k-1) + \frac{1}{2},
\label{eq:FirstCplx}\\
 \Omega(x) &= \begin{cases}
\frac{x^2}{4} - \frac{1}{2} & \mbox{ for $|x| \le  2$}, \\
\frac{x^2}{4} - \frac{1}{2} - \frac{|x| \sqrt{x^2 -4}}{4} + \log \Big(\frac{|x|}{2} + \sqrt{\frac{x^2}{4} -1} \Big) & \mbox{ for $|x|\ge 2$.}
\end{cases}\label{eq:OmegaFormula}
\end{align}
\end{lemma}
\begin{proof}
Lemma \ref{lemma:crit_points}, along with independence of $\bW_{n-1}, Z$ implies that, 
\begin{align}
\E[ | {\rm{det}} (\bW_{n-1} - t Z \bI_{n-1}) | \mathbf{1}_{ \{ Z \in \sqrt{2n} B\} } ] = \int_{ B} \E\Big[ | {\rm{det}} \Big(\bW_{n-1} - \sqrt{\frac{2k}{k-1}}x \bI_{n-1} \Big)| \Big] \sqrt{\frac{n}{\pi}} \exp{(- nx^2)} \de x. \label{eq:crit_int}
\end{align}
Further, we have, 
\begin{align}
\E[ | {\rm{det}} ( \bW_{n-1} - x \bI_{n-1} )| ]  = \E[ \exp{\Big[ (n-1) \int \log (\lambda - x) s_{n-1}(\de \lambda) \Big] }, \label{eq:det_exp}
\end{align}
where $s_n = \frac{1}{n} \sum_{i=1}^{n} \delta_{\lambda_i}$ is the empirical spectral measure of the matrix $\bW_n$. The Wigner semi-circle law governs that as $n\to \infty$, almost surely, 
\begin{align}
s_n \stackrel{w}{\Rightarrow} s_{\infty} (\de x) = \frac{\sqrt{4 - x^2}}{2 \pi} \mathbf{1}_{|x| \leq 2} \de x. \nonumber
\end{align}
We define 
\begin{align}
\Omega(x) \equiv  \int \log |\lambda - x| s_{\infty} (\de \lambda)\, .
\end{align}
An exercise in complex analysis reveals that $\Omega(x)$ is indeed given by 
Eq.~\eqref{eq:OmegaFormula}.

The following Lemma establishes the leading exponential behavior of the determinant. 
\begin{lemma}
\label{lemma:det_asymp}
As $n \to \infty$, 
\begin{align}
\E[ | {\rm{det}}( \bW_{n-1} - x \bI_{n-1}) | ] = \exp{(n \Omega (x) + o(n))}. \label{eq:det_asymp1}
\end{align}
\end{lemma}
We have, from \eqref{eq:crit_int} and Lemma \ref{lemma:det_asymp}, using Laplace method, 
\begin{align}
\E[ | {\rm{det}} ( \bW_{n-1} - t Z \bI_{n-1}) | \mathbf{1}_{ \{ Z \in \sqrt{2n} B\} } ] &\doteq \sup_{x \in B} \exp{\Big[n \Big(\Omega \Big(x \sqrt{\frac{2k}{k-1}}  \Big) - x^2  \Big)\Big]}  \nonumber 
\end{align}
 Therefore the desired claim follows from Lemma \ref{lemma:crit_points}.
\end{proof}

\begin{remark}
Equations \eqref{eq:FirstCplx} and \eqref{eq:OmegaFormula} imply that, for $|x|\ge \varepsilon_d
\equiv \sqrt{\frac{2(k-1)}{k}}$ the exponential growth rate takes the form
\begin{align}
S(x) =\frac{1}{2} \log(k-1) - \frac{k-2}{4(k-1)} x^2 - \frac{1}{2} \log \Big( \frac{k}{2} \Big) - \frac{x \sqrt{x^2 - \varepsilon_d^2}}{\varepsilon_d^2} - \log ( x + \sqrt{x^2 - \varepsilon_d^2})\, . 
\end{align}
\end{remark}

It remains to prove Lemma \ref{lemma:det_asymp}. 
We do not provide a complete proof and refer the reader to Auffinger, Ben-Arous, \u{C}ern\'{y} 
\cite{auffinger2013random} and Subag \cite{subag2017complexity} for formal arguments. 
The argument presented here only implies that the right-hand side of Eq.~\eqref{eq:det_asymp1}
is an upper bound on the desired expectation\footnote{Note that in many circumstances this is 
sufficient since $\E[C_n(B)]$ is anyway an upper bound on the typical number of critical points.}.

Consider Eq.~\eqref{eq:det_exp}. The Wigner semi-circle law suggests that as $n \to \infty$, 
\begin{align}
\E\Big[ \exp{\Big[ (n-1) \int \log (\lambda - z) s_{n-1}(\de \lambda) \Big] } \Big]\doteq \exp{( n \Omega(z) + o(n) )}, 
\end{align}
provided $s_n$ concentrates closely around $s_{\infty}$. This might be formalized using the 
following large deviation result of empirical spectral distribution of GOE random matrices,
due to Ben Arous and Guionnet \cite{arous1997large}.

\begin{theorem}[\cite{arous1997large}] 
For any two probability measures $\mu, \nu$, define the metric 
\begin{align}
d(\mu, \nu) = \sup \Big\{ \int f \, \de (\mu- \nu) (x) : f \,\,1- {\rm Lipschitz}, \| f \|_{\infty} \leq 1 \Big\}. \nonumber
\end{align} 
Then the sequence ${s_n}$ satisfies a large deviation principle (LDP) with respect to this topology. 
In particular,  for any $\varepsilon>0$, there exists a constant $c(\varepsilon) >0$ such that 
\begin{align}
\prob[ d(s_n, s_{\infty}) > \varepsilon] = \exp{(- c(\varepsilon) n^2 + o(n^2))}. \nonumber
\end{align}
\end{theorem}

This theorem implies that deviations of averages $\int f(\lambda) s_n(\de \lambda)$ from the asymptotic
value  $\int f(\lambda) s_\infty(\de \lambda)$ are exponentially small in $n^2$. It is then easy
to deduce the following lemma.
\begin{lemma}
For any $\psi : \mathbb{R} \to \mathbb{R}$ such that $0< \varepsilon \leq \psi(x)\leq M < \infty$ for all $x\in\reals$, 
\begin{align}
\lim_{n \to \infty} \frac{1}{n} \log \E \Big[ \prod \psi (\lambda_i) \Big] = \int \psi(\lambda) s_{\infty} (\de \lambda). \nonumber 
\end{align}
\end{lemma}
We note that, in order to prove Lemma \ref{lemma:det_asymp}, we would need to apply the last lemma to
$\psi (\lambda) = | \lambda - z|$ which is unbounded and not bounded away from $0$. 
Hence the result does not directly follow by an application of  Lemma \ref{lemma:det_asymp}.

Nevertheless, Lemma \ref{lemma:det_asymp} has a number of direct consequences that are nearly as 
useful, and we outline here:
\begin{itemize}
\item A better upper bound on the typical number of critical points is obtained by
computing a truncated first moment $\E[C_n(B)\bfone_{\cG_1}]$ 
where $\cG_n$ is a high-probability event (i.e. $\lim_{n\to\infty}\P(\cG_n)=1$). 

Fixing $\eta>0$, we can let $\cG_n\equiv\{\max_{i\le n}|\lambda_i(\bW_n)|\le 2+\eta\}$ 
which holds with high-probability by the Bai-Yin law. As a consequence, it is sufficient to have
$0< \varepsilon \leq \psi(x)\leq M < \infty$ for $x\in [-2,2]$, which holds as soon as $z\not\in[-2,2]$.
\item Using $\psi_{\eps} (\lambda) = | \lambda - z|\vee \eps$,
and taking $\eps\downarrow 0$ at after $n\to\infty$ yields an upper bound for all $z$.
\end{itemize}

Auffinger, Ben-Arous, \u{C}ern\'{y} \cite{auffinger2013random} carry out a more general 
analysis for critical points with fixed indices (recall that the index of a critical point
is the number of negative eigenvalues of the Hessian at that point). 
Their results can be roughly summarized as follows. 
\begin{theorem}[\cite{auffinger2013random}] 
If $C_{n,k}(B)$ denotes the number of critical points of $\hd(\sigma)$ in $B$ with index $n-k$
(i.e. $k$ positive eigendirections), then
\begin{align}
\E[C_{n,k}(B)] \doteq \sup_{x \in B} \exp{(nS_k(x) + o(n))}. \nonumber 
\end{align}
for some appropriate complexity function $S_k$. Further:
\begin{enumerate}
\item $S_0(x) = S(x)$ for $x\ge\eps_d$ and $S_0(x)=-\infty$ for $x<\eps_d$.
\item For any fixed $k\ge 1$, $S_{k}(x)<S_0(x)$ for $x>\eps_d$, $S_k(\eps_d) = S_0(\eps_d)$,
and  $S_k(x)=-\infty$ for $x<\eps_d$.
\end{enumerate}
\end{theorem}
In other words, the last result implies that local maxima dominate the total count of critical points
with energy $\hd(\bsigma)>\eps_d$. For $\hd(\bsigma)<\eps_d$ it turns out instead that the typical 
index of critical points is linear in $n$: the landscape is dominated by saddles with a
 large number of positive eigendirections.
 
 The above analysis has immediate implications on the ground state energy.
 Define
 \begin{align}
 \eps_*\equiv\sup\big\{x\in\reals: \;\; S(x) \ge 0\big\}\, .
 \end{align}
 \begin{corollary}
 The following holds almost surely
 \begin{align}
 \lim_{n\to\infty}\max_{\bsigma\in\sS^{n-1}}\hd(\bsigma) \le \eps_*\, .\label{eq:UB-BAC}
 \end{align}
 \end{corollary}
Of course we have $\eps_*>\eps_d$ strictly, and therefore there is a large interval of 
energy such that the energy landscape is dominated by local maxima.
This is believed to have a direct impact on algorithms: both sampling and search algorithms
are expected to fail at finding configurations with $\hd(\bsigma) \in(\eps_d,\eps_*]$,
and in fact we have rigorous evidence for this expectation.

Subag \cite{subag2017complexity} carries out a related second moment argument to 
establish that the typical behavior of local maxima is governed by that of
 its expectation in an interval of energies. In particular, this result implies
 that Eq.~\eqref{eq:UB-BAC} holds with equality. (The last conclusion also follows from rigorous
 proofs of Parisi formula.)


\section{Back to the replica method}
\label{sec:replica_critical} 

We now go back to the non-rigorous replica method, and explore some connections 
to the rigorous analysis of the last section. We use the replica method to zoom
 into the low temperature regime $T\to 0$. Recall the $1$
 RSB free energy functional 
\begin{align}
\Psi_{1\rm{RSB}} (q_1,m) = \frac{\beta^2}{4} [ 1- (1-m) q_1^k ] + \frac{1}{2m} \log (1 - (1-m) q_1) - \frac{1-m}{2m} \log (1-q_1). \nonumber 
\end{align}
The $1$-RSB free energy is given by $\Psi_{1\rm{RSB}}(q_{1, *}(T, m_*(T)), m_*(T))$, for any temperature $T$. For any $m$, the corresponding minimizer $q_{1,*}(T,m)$ satisfies the fixed point equation 
\begin{align}
q_{1,*}(T,m)^{k-2} (1 - q_{1,*}(T,m)) ( 1 - (1-m_{*}) q_{1,*}(T,m)) = \frac{2}{k} T^2 . \label{eq:rsb_fixed}
\end{align}
We assume $m= \mu T$ as $T \to 0$ for some constant $\mu$, and setting $q_{1,*} (T,m)= 1 - z_{*} T + o(T)$, Taylor expansion of  \eqref{eq:rsb_fixed} around $1$ yields that $z_*$ must satisfy the equation 
\begin{align}
z_* ( \mu + z_*) = \frac{2}{k}
\implies z_*:= z_*(\mu) = \sqrt{\frac{2}{k} + \frac{\mu^2}{4}} - \frac{\mu}{2}. \nonumber 
\end{align}
Plugging in these values, we have that 
\begin{align}
\lim_{\beta \to \infty} \frac{1}{\beta} \Psi_{1\rm{RSB}}(q_{1,*}(T,m), m) = e_{1\rm{RSB}}(\mu), \nonumber \\
e_{1\rm{RSB}}(\mu) = \frac{1}{4} [ \mu + k z_*(\mu)] + \frac{1}{2\mu} \log \Big(1 + \frac{\mu}{z_*(\mu)} \Big). \nonumber 
\end{align}
We set $e_* = \min_{\mu} e_{1\rm{RSB}} (\mu)$. $e_*$ is interpreted as the limiting ground state energy, that is 
\begin{align}
e_* = \lim_{n \to \infty} \E[ \max_{\bsigma} \hd(\bsigma)]. \nonumber 
\end{align}
Here we discover the first connection to the rigorous study of the last section. We obtain that 
$e_* = \inf \{ x \geq 0: S(x) \leq 0\}$. This suggests that for any $\delta >0$ there are exponentially many critical points with energy at least $\max_{\bsigma} \hd(\bsigma) - \delta$. Further, we define $\phi(\mu) = \mu e(\mu)$ and evaluate the Fenchel-Legendre transform 
\begin{align}
\Sigma(\varepsilon) = \sup_{\mu} \{ \phi(\mu) - \varepsilon \mu \}. \nonumber 
\end{align}
To find the dual $\Sigma(\varepsilon)$, we define $\mu^* = \argmin \,\, \phi(\mu)$ and set 
\begin{align}
\varepsilon = \frac{1}{2} \mu + \frac{1}{4} k z_*(\mu^*) + \frac{1}{2} \frac{1}{[\mu^* + z_*(\mu^*)]}. \nonumber 
\end{align}
Here we run into the second surprising fact--- the Fenchel- Legendre dual $\Sigma(\varepsilon)$ turns out to be in fact exactly equal to the complexity function $S(\varepsilon)$ derived in the last section ! 

This surprise was explained within the framework of the replica method by R\'{e}mi Monasson, and this argument is therefore usually referred to as Monasson's argument. We sketch this argument to explore the connections between the complexity and the free energy. We start with the partition function 
\begin{align}
Z_n = \int_{\sS^{n-1}} \exp{[n \beta H(\bsigma) ]} \de \nu_0(\bsigma). \nonumber 
\end{align}
Recall the intuitive picture of the random measure $\mu_{\beta}$. We believe that $\sS^{n-1} = \cup_{\alpha=1}^{N} \Gamma^{(\alpha)}$ may be decomposed into ``pure" states and the measure decomposes naturally into a convex combination of distributions over the pure states. Thus we have, $\mu_{\beta}(\de \sigma) = \sum_{\alpha=1}^{N} w^{(\alpha)} \mu_{\beta}^{(\alpha)} (\de \sigma)$, where $\mu_{\beta}^{(\alpha)} (\cdot)  \propto \mu_{\beta}(\cdot) \mathbf{1}_{\Gamma^{(\alpha)}}$ is the restriction of the measure onto the pure state.  We define 
\begin{align}
Z_n^{(\alpha)} = \int_{\Gamma^{(\alpha)}} \exp{[n \beta H(\bsigma)]} \de \nu_0 (\bsigma). \nonumber 
\end{align}
Thus the weight of a pure state $w^{(\alpha)} = \frac{Z_n^{(\alpha)}}{Z}$. We assume that $Z^{(\alpha)} \doteq \exp{[n f^{(\alpha)}]}$ and let $\widehat{N} = \sum_{\alpha} \delta_{f^{\alpha}}$ denote the counting measure corresponding to the limiting free energy of the pure states. Thus it suggests that 
\begin{align}
Z_n \doteq \exp[ n \sup_{f: \Sigma_{\beta}(f) \geq 0} ( f + \Sigma_{\beta}(f)) ], \nonumber 
\end{align}
where $\Sigma_{\beta}(f)$ is the cluster complexity functional, which captures the number of pure states with free energy $f$. This captures the energy-entropy tradeoff in determining the free energy. Now, consider the free energy of a coupled system with $m$ ``real" replicas. 
\begin{align}
Z_n(m) = \int_{(\sS^{n-1})^m} \exp{\Big[n \beta \sum_{a=1}^{m} H(\bsigma^a) + n \delta g(\bsigma^{1}, \cdots, \bsigma^{m}) \Big]} \prod_{i=1}^{m}\de \nu_0(\bsigma_i). \nonumber  
\end{align}
Here $g(\bsigma^{1}, \cdots, \bsigma^{m})$ is any function, for example, $g(\bsigma^1, \cdots, \bsigma^m) = \sum_{1\leq a <b \leq m} \langle \bsigma^a, \bsigma^b \rangle$, which favors configurations that align together, while $\delta>0$ is a small perturbation. We note that if we let $\delta \to 0$ for any fixed $n$, we obtain $Z_n(m) = (Z_n)^m$ and the system de-couples into $m$ independent systems. However, if $\delta>0$ is fixed while $n\to \infty$, we expect that the replicas will ``condense" on the same pure state. The perturbation with $g$ should be looked upon as a device to select one of the pure states. This is analogous to the choice of a pure state in the Curie-Weiss model by introducing an external magnetic field. Thus if we then send $\delta \to 0$, we expect, (since the system is $1$-RSB, once the pure state has been selected, everything is ``uniform")
\begin{align}
Z_n(m) \doteq \sum_{\alpha=1}^{N} (\exp{[n f^{(\alpha)}]})^m. \nonumber
\end{align}
Thus we expect that as $n \to \infty$, 
\begin{align}
Z_n(m) \doteq \exp{\Big[n \sup_{ f} (mf + \Sigma_{\beta}(f)) \Big]}. \nonumber 
\end{align}
We compute the asymptotics of $Z_n(m)$ using the replica method as earlier. For $r \geq 1$ a multiple of $m$, we obtain, 
\begin{align}
\E[(Z_n(m))^{r/m}] \doteq \exp{[n S(\bQ)]}, \nonumber 
\end{align}
where $S(\cdot)$ is the same action functional derived in \eqref{eq:RS_Action}. We evaluate the function on the $1$-RSB saddle point $\bQ^{1-\rm{RSB}}$. Finally, as before, we send $r \to 0$ to obtain 
\begin{align}
\lim_{r \to 0}m \frac{1}{n (r/m)} \log \E[(Z_n(m))^{r/m}] = m \Psi_{1\rm{RSB}}(m). \nonumber 
\end{align}
Now, this calculation suggests something remarkable! We derive, as a result, that for $m \in [0,1]$, 
\begin{align}
m\Psi_{1\rm{RSB}}(m) = \sup_{f} ( mf + \Sigma_{\beta}(f)). \label{eq:legendre}
\end{align}
Thus the complexity functional is related to the $1$-RSB free energy functional through Legendre-Fenchel duality. Formally, we can therefore invert the free energy functional $\Psi_{1\rm{RSB}}$ functional to obtain the complexity function $\Sigma_{\beta}(f)$. This explains the relation observed between the ground state energy functional $e_{1\rm{RSB}}$ and the complexity function $S(\cdot)$ at zero temperature. 

We continue to explore the connections between the free-energy, complexity and the RS to 1-RSB phase transition. We recall that intuitively, we expect 
\begin{align}
Z_n \doteq \exp{[n \max_{f : \Sigma_{\beta}(f) \geq 0} (f + \Sigma_{\beta}(f))]}. \nonumber 
\end{align}
When the maximum occurs at $f = f_*$, we intuitively believe that the free energy is dominated by $\exp[n \Sigma_{\beta}(f_*)]$ pure states, each having energy $f_*$. We invert \eqref{eq:legendre} to express $\Sigma_{\beta}$ as a function of $m$ and obtain the stationary conditions 
\begin{align}
\Sigma_{\beta} = m \Psi_{1\rm{RSB}} (m) - m f ,\nonumber \\
f = \frac{\de}{\de m} \Big[m \Psi_{1\rm{RSB}}(m) \Big], \nonumber 
\end{align}
where implicitly, we set $\Sigma_{\beta}$ and $f$ to be functions of $m$. Consider first the unconstrained problem $\max_{f}(f+ \Sigma_{\beta}(f))$. The stationary condition in this case is $ \Sigma_{\beta}'(f) = -1$. Further, it is easy to see that if $\Sigma_{\beta}$ and $f$ are expressed as functions of $m$, then $\frac{\de \Sigma_{\beta}}{\de f} = -m$. Therefore, the unconstrained maximum corresponds to $m=1$. In this case, 
\begin{align}
f =  \frac{\de}{\de m} \Big[m \Psi_{1\rm{RSB}}(m)\Big] \Big|_{m=1} = \Psi_{1\rm{RSB}}(1) + \Psi_{1\rm{RSB}}'(1), \quad
\Sigma_{\beta} = - \Psi_{1\rm{RSB}}'(1). \nonumber 
\end{align} 
We note that this implies that the unconstrained maximizer is a feasible solution to the problem only if $\Psi_{1\rm{RSB}}'(1) \leq 0$. Recall the static and dynamic critical temperatures $T_s$ and $T_d := T_d(1)$ defined earlier.  Further, recall from Fig \ref{fig:psivsm} the behavior of the function $\Psi_{1RSB}(m)$ as a function of $m$. 

For $T_s < T < T_d$, $\Psi_{1\rm{RSB}}'(1)\leq 0$, and thus the unrestricted maximum is feasible. In this case, we set $m=1$ and the free energy is the RS free energy. However, the complexity $\Sigma_{\beta} >0$. This validates the multiple pure states picture we had introduced earlier. Indeed, in this case, we expect to have exponentially many pure states, each corresponding to the energy $f_* = f(1)$. In case $T < T_s$, the complexity $\Psi_{1\rm{RSB}}'(1) >0$, and thus the unrestricted maximum is infeasible.  At this point, recall $m_*= m_*(T)$, the minimizer of $m \mapsto \Psi_{1\rm{RSB}}(m)$. In this case, using the stationary conditions obtained above, we have,
$\Sigma_{\beta} = - m^2 \Psi_{1\rm{RSB}}'(m )$ and therefore, $\Sigma_{\beta} \geq 0$ if and only if $\Psi_{1\rm{RSB}}'(m) \leq 0$, or equivalently, if $m \leq m_*$. Further, 
$\frac{\de}{\de f}(f+ \Sigma(f)) = (1 - m) \geq 0$ and thus the optimal solution in this case is to chose $m = m_*$. For $m = m_*$, $\Sigma_{\beta} = 0$ and $ f = \Psi_{1\rm{RSB}}(m_*)$. In this case, we expect that the free energy is governed by $O(1)$ pure states, each corresponding to the energy $f_s:= \Psi_{1\rm{RSB}}(m_*)$. 

 Further, physicists also predict the structure of the ``extremal process" of the pure states with $f = f_s + \frac{c}{n}$. For such states, using Taylor expansion around $f= f_s$, physicists predict that $\exp[n \Sigma_{\beta}(f)] \approx \exp[- n m_* (f - f_s)]$ and thus expect that the extremal process, suitably scaled, should converge to a Poisson point process with intensity $\exp[- m_* x]$. Such pure states are expected to have weight $w_{\alpha} = \exp{[ x_{\alpha}]}/ \sum_{\alpha} \exp{[  x_{\alpha}]}$, which is usually referred to as the Poisson-Dirichlet distribution with parameter $m_*$. This provides a very detailed description of the dominant states and their complexities at all temperatures. As a final sanity check, it is not hard to see that if states are distributed with weights given by Poisson-Dirichlet($m_*$), then marginal probability of two iid draws to be from the same cluster is $1-m_*$, which corroborates with our earlier interpretation of this parameter. 
 
 This completes our analysis of the $p$-spin model. We take up the Sherrington-Kirkpatrick model of spin glasses in the next chapter.

\section*{Bibliographic notes}
\addcontentsline{toc}{section}{\protect\numberline{}Bibliographic notes}

We refer to \cite{cover1991elements} for a general definition of mutual information, and its properties. 

The replica method was first  applied to the spherical $p$-spin model (corresponding to the $\lambda=0$ case of the model treated here)
by Crisanti and Sommers \cite{crisanti1992spherical}. The same authors also studied the energy landscape in \cite{crisanti1995thouless},
and the Langevin dynamics in \cite{crisanti1993sphericalp}.

The extension to  $\lambda>0$ is treated  in \cite{gillin2000p}. The spherical SK model (corresponding to $k=2$ and $\lambda=0$) was first studied by \cite{kosterlitz1976spherical} in physics.

The celebrated Baik-Ben Arous-Pech\'e phase transition
\cite{baik2005phase} was first discovered by Hoyle and Rattray using the replica method in
\cite{hoyle2003pca,hoyle2004principal}.

For a rigorous treatment of the spherical model, see \cite{talagrand2006free} which covers the case of even $k$, and
\cite{chen2013aizenman}.
which solves the general case. These papers do indeed consider the more general `mixed $p$-spin model'.

In the context of \eqref{eq:SpikedTensor}, the hypothesis testing problem
 $H_0: \lambda =0 $ vs. $H_1: \lambda >0$ has been examined extensively in the
  recent literature, under different assumptions on the spike $\bv_0$. In the 
  case of the uniform spherical prior on $\bv_0$, 
  \cite{montanari2014statistical,montanari2016limitation} established the existence of 
  $\lambda^- < \lambda^+$ such that for $\lambda > \lambda^+$, it is possible to detect the 
  signal with vanishing errors as $n \to \infty$. On the other hand, for $\lambda < \lambda^-$, 
  the Total Variation distance between these measures converges to zero. 
  \cite{perry2020statistical} analyzed this problem for spherical, iid Rademacher
   and sparse Rademacher prior on the spike $\bv_0$, and derived analogous constants 
   $\lambda^-, \lambda^+$ for each prior. Along the way, they improved the estimates in
    \cite{montanari2014statistical,montanari2016limitation} for the spherical prior. 
    \cite{lesieur2017statistical} characterized the limiting mutual information between 
    the tensor $\bY$ and the latent vector $\bv_0$ for any product prior on
     $\bv_0$, and derived a sharp constant $\lambda_p$ such that consistent detection 
     (i.e. with vanishing Type I and Type II errors) is possible whenever 
     $\lambda > \lambda_p$. Finally, \cite{chen2019phase} established that this is the sharp detection 
     threshold for this problem, in that the null and alternative distributions are asymptotically indistinguishable for any $\lambda < \lambda_p$.   

%
%
%

\section{Some omitted technical calculations}
\label{sec:omitted_calc} 

\subsection{Proof of Lemma~\ref{lemma:I-MMSE}}
\label{sec:appendix_proof_immse} 

\begin{proof}
Recall that 
\begin{align}
\Info(\bvz ;\bY) &=\E\left\{\log \Big(\frac{\de\mu_{\sBayes}}{\de \nu_0}(\bvz)\Big)\right\}\, = \E\Big[ \frac{n\lambda}{k!} \< \bY , \bvz^{\otimes k} \>  \Big] - n \Phi_{n, \sBayes}(\lambda). \nonumber 
\end{align}
The first identity follows upon dividing both sides by $n$ and noting that $\E[\< \bY, \bvz^{\otimes k} \>] = \lambda$. 
Therefore, 
\begin{align}
\frac{1}{n} \frac{\partial}{\partial \lambda} \Info(\bvz ; \bY) = \frac{2 \lambda}{k!} - \frac{\partial \Phi_{n, \sBayes}}{\partial \lambda} (\lambda). \nonumber 
\end{align}

\noindent 
Recall from \eqref{eq:phi_bayes} and \eqref{eq:SpikedTensor}  that 
\begin{align}
\Phi_{n, \sBayes}(\lambda) &= \frac{1}{n} \mathbb{E}\log \int_{\sS^{n-1}}\exp\Big\{\frac{n\lambda}{k!}\,
\<\bY,\bsigma^{\otimes k}\> \Big\} \, \nu_0(\de\bsigma) \nonumber \\
&= \frac{1}{n} \mathbb{E}\log \int_{\sS^{n-1}} \exp \Big\{ \frac{n \lambda^2}{k!} \< \bvz, \bsigma \>^k + \frac{n \lambda}{k!} \< \bW, \bsigma^{\otimes k} \> \Big\} \, \nu_0(\de\bsigma). \nonumber
\end{align} 
Differentiating in $\lambda$, we obtain, 
\begin{align}
\frac{\partial \Phi_{n, \sBayes}}{ \partial  \lambda}(\lambda) = \mathbb{E}\int_{\sS^{n-1}} \Big\{ \frac{2\lambda}{k!} \< \bvz, \bsigma\>^k + \frac{1}{k!} \< \bW, \bsigma^{\otimes k} \>\Big\} \mu_{\sBayes}(\de \bsigma). 
\end{align} 
Using Gaussian integration by parts and Bayes rule, we have, 
\begin{align}
\mathbb{E}\int_{\sS^{n-1}} \< \bW, \bsigma^{\otimes k} \> \mu_{\sBayes}(\de \bsigma) = \lambda ( 1 - \mathbb{E}\< \hbX_{\sBayes}(\bY),\bXz \>  ). \nonumber 
\end{align}  
Plugging this back, we obtain 
\begin{align}
\frac{1}{n} \frac{\partial}{\partial \lambda} \Info(\bvz ; \bY) = \frac{\lambda}{k!} (1- \mathbb{E}\<  \hbX_{\sBayes}(\bY),\bXz  \> )  = \frac{\lambda}{2 (k!)} \E\big\{\|\hbX_{\sBayes}(\bY)-\bXz\|_F^2\big\}\, . 
\end{align} 
This completes the proof.

\end{proof} 

\subsection{Proof of Eq.~(\ref{eq:Q-density})}
\label{sec:ProofQdensity}

We use the following change-of-measure trick.

\begin{lemma}
Let $(\Omega,\cF)$ be a measurable space and consider two $sigma$-finite measure $\mu$, $\nu$ on it. Let $\bX:\Omega\to\reals^d$ be measurable and 
assume it has densty $f_{\mu}(\bx)$, $f_{\nu}(\bx)$ with respect to Lebesgue measure, when $(\Omega,\cF)$ is endowed with $\mu$ or $\nu$. 
Assume that  $\nu$ is absolutely continuous with respect to $\mu$, with Radon-Nikodym derivative which depends on $\omega$ only through $\bX(\omega)$,
and write (with an abuse of notation) $\frac{\de \nu}{\de \mu}(\bX(\omega))$ for this derivative. Then
\begin{align}
f_{\nu}(\bx) = f_{\mu} (\bx)\frac{\de\nu}{\de \mu}(\bx)\, .
\end{align}
\end{lemma}

Applying this to $\nu = \gamma_{\bLambda}$, $\mu = \gamma_{\id}$, $\bX = \hbQ$, we obtain
\begin{align}
G_{\bLambda}(\bQ) = \det(\bLambda)^{n/2}\, e^{-n(\<\bLambda,\bQ\>-\<\id,\bQ\>)/2}\, G_{\id}(\bQ)\, .
\end{align}
Substituting in Eq.~(\ref{eq:Q-density}), and recalling $Q_{ii}=1$, we conclude that it is sufficient to prove the claim for $\bLambda = \id$, namely 
\begin{align}
f_{n,r}(\bQ) &=\frac{G_{\id}(\bQ)}{g_{n}(1)^{r+1}}\, .
\end{align}
This follows from the fact that the distribution of $\bsigma^a\sim\normal(0,\id_n/n)$, conditioned on $\|\bsigma^a\|_2=1$ (corresponding to 
the density on the right hand side) coincides with the uniform measure over $\bsigma^a\in \sS^{n-1}$.

\section*{Exercises}
\addcontentsline{toc}{section}{\protect\numberline{}Exercises}

\begin{exercise}\label{exer:GeneratingFunctions}
Prove Eqs.~(\ref{eq:PartialPhiH}), (\ref{eq:PartialPhiLambda}).
\end{exercise}

\begin{exercise}\label{exer:MixedSphericalModel}
Assume that for each $\ell\in\{1,\dots,k\}$, we observe a tensor $\bY_{\ell}\in (\reals^n)^{\otimes \ell}$, given by
\begin{align}
\bY_{\ell} = \lambda_{\ell}\, \bvz^{\otimes \ell} + \bW_{\ell}\, \label{eq:SpikedTensorMixed}
\end{align}
where $\bW_{\ell}\in (\reals^n)^{\otimes \ell}$ is a random Gaussian tensor distributed as in Section \ref{sec:Pspin}
(with $\ell$ replacing $k$), and $\bvz$ is a unit vector. We use $\bY = (\bY_1,\dots,\bY_{k})$ to denote collectively all the observations,
 $\bW = (\bW_1,\dots,\bW_{k})$ to denote all the noise tensors, and $\blambda = (\lambda_1,\dots,\lambda_k)$ for all the model parameters. Assume that the tensors $(\bW_1, \dots, \bW_{k})$ are independent. 
\begin{itemize}
\item[$(a)$] Write the maximum likelihood estimator for $\bvz$.
Assuming a uniform prior for $\bvz\in\sS^{n-1}$, derive the Bayesian posterior distribution of $\bvz$ given $\bY$.
\item[$(b)$] Introducing additional parameters $\beta,h$ define an analogue of the random probability measure (\ref{eq:PspinDef})
for this case. In particular, it should include the Bayes measure and the maximum likelihood problem as special cases.
\item[$(c)$] Derive the replica-symmetric expression for the free energy density in this model.
\end{itemize}
\end{exercise}

\begin{exercise}[Concentration of the free energy]
Consider the log-partition function as a function of the noise $\bG$ (where we use the construction (\ref{eq:Wconstruction})), namely,
for fixed $\beta,\lambda, h$, let
\begin{align}
Z_n(\bG) &=
\int_{\sS^{n-1}} 
\exp\Big\{\frac{n\beta\lambda}{\sqrt{2(k!)}}\, \<\bvz,\bsigma\>^k+\beta\sqrt{\frac{n}{2}} \<\bG,\bsigma^{\otimes k}\> +  n\,h \<\bvz,\bsigma\>
\Big\} \, \nu_0(\de\bsigma)\, ,\\
F_n(\bG) & \equiv \log Z_n(\bG)\, .
\end{align}
Use Gaussian concentration (Theorem \ref{thm:GaussianConcentration}) to prove that, for any $t\ge 0$ 
\begin{align}
\prob\Big(\big|\log Z_n-\E\log Z_n\big|\ge t\Big)\le 2\, e^{-t^2/(n\beta^2)}\, . \label{eq:GaussianConcentrationPspin}
\end{align}
\end{exercise}

\begin{exercise}[Second moment method]\label{exer:SecondMoment}
The second moment method is a simple rigorous technique that uses the first two moments of the partition function
to characterize the asymptotic free energy density.
\begin{itemize}
\item[$(a)$] Let $Z_n\ge 0$ be a sequence of random variables. Prove that, if
\begin{align}
\lim_{n\to\infty} \frac{1}{n}\log \frac{\E\{Z_n^2\}}{\{\E Z_n\}^2} &= 0\, ,\label{eq:ConditionSecondMoment}\\
\lim_{n\to\infty} \frac{1}{n}\log\E Z_n&= \phi\, ,\label{eq:AnnealedLimit}
\end{align}
then for any $\eps>0$, we have 
\begin{align}
\prob\Big(e^{-\eps n}\E Z_n\le Z_n\le e^{\eps n}\E  Z_n\Big)\ge  e^{-o(n)}\, .
\end{align}
[Hint: Use Markov and Paley-Ziegmund inequalities.]
\item[$(b)$] Let $Z_n$ be the partition function of the spherical model
  (\ref{eq:PspinDef}), for $h=0$. Use Gaussian concentration proved in the previous exercise,
cf. Eq.~(\ref{eq:GaussianConcentrationPspin}) to prove that, under assumptions  (\ref{eq:ConditionSecondMoment}), (\ref{eq:AnnealedLimit}),
we have $n^{-1}\log Z_n\to \phi$ in probability.
\item[$(c)$] Use Eq. (\ref{eq:Action-p-spin}) to compute the limit on the left-hand side of (\ref{eq:ConditionSecondMoment}), again
for $Z_n$ the partition function of the spherical model (\ref{eq:PspinDef}), for $h=0$.
\item[$(d)$] Following from previous point, assuming $k\ge 2$, prove that there exists $\beta_0$, $\lambda_0>0$ such that, if $\beta\le \beta_0$, $\lambda\le \lambda_0$,
Eq.~(\ref{eq:ConditionSecondMoment}) holds.
\item[$(e)$] Show that there exist values of $\beta,\lambda$ for which condition (\ref{eq:ConditionSecondMoment}) does not hold.

[We do not require a completely rigorous proof for this point.  It is acceptable to evaluate numerically the expression for the limit derived at point $(c)$.]
\end{itemize}
\end{exercise}

\begin{exercise}\label{exer:SpikedMatrix}
Consider the spherical model with $k=2$ and $\lambda=0$. In other words, we are interested in the partition function
\begin{align}
Z_n(\beta;\bW) = \int_{\sS^{n-1}} e^{n\beta\<\bsigma,\bW\bsigma\>/2} \nu_0(\de\bsigma)\, ,
\end{align}
where $\bW=\bW^{\sT}$ is a $\GOE(n)$ matrix, i.e. a matrix with entries $(W_{ij})_{i\le j}$ independent, with $W_{ij}\sim\normal(0,1/n)$ for
$i<j$ and $W_{ii}\sim\normal(0,2/n)$. The objective of this exercise is to obtain a rigorous proof of the free energy formula, for $\beta<1$.

It is useful to recall the following known facts about the eigenvalues of $\bW$, to be denoted by $\lambda_1(\bW)\ge \lambda_2(\bW)\ge \dots \ge \lambda_n(\bW)$.
\begin{theorem}
Let, for each $n\in\naturals$, $\bW_n\sim\GOE(n)$. Then we have, almost surely, $\lim_{n\to\infty}\lambda_1(\bW_n)=2$, $\lim_{n\to\infty}\lambda_n(\bW_n)=-2$.
Further, denoting by $s_n = n^{-1}\sum_{i=1}^n\delta_{\lambda_i(\bW)}$  the empirical spectral distribution , we have $s_n\stackrel{w}{\Rightarrow} s_{\infty}$ almost surely,
where $\stackrel{w}{\Rightarrow}$ denotes weak convergence and $s_\infty$ is the non-random probability distribution on $\reals$ defined by
\begin{align}
s_{\infty}(\de x) = \frac{1}{2\pi}\sqrt{4-x^2}\, \bfone_{|x|\le 2}\,\de x\, .
\end{align}
\end{theorem}
\begin{enumerate}
\item[$(a)$] For $\bM\in\reals^{n\times n}$ a symmetric positive definite matrix, let $f_{\bM}(\,\cdot\,)$ denote the density of 
$\|\bx\|^2$ when $\bx\sim\normal(0,(n\bM)^{-1})$. Prove the following identity, valid for all $z>\lambda_1(\bW)$,
\begin{align}
Z_n(\beta;\bW)  = \frac{(\beta^{-1} e^{\beta z -1})^{n/2}}{\det(z\id-\bW)^{1/2}} \left[\frac{f_{\beta(z\id-\bW)}(1)}{f_{\id} (1)}\right]\, .
\end{align}
\item[$(b)$] Prove that this implies, for $\beta<1$, almost surely
\begin{align}
\lim_{n\to\infty}\frac{1}{n}\log Z_n(\beta;\bW )& = \phi(\beta)\equiv -\frac{1}{2}\log \beta+\frac{1}{2}\beta z-\frac{1}{2}-\frac{1}{2}\int\log(z-\lambda)\, s_{\infty}(\de\lambda)\, ,
\end{align}
where $z>2$ solves the equation
\begin{align}
\beta = \int \frac{1}{z-\lambda}\, s_{\infty}(\de\lambda)\, .
\end{align} 
Use the above expression for $s_{\infty}$ to compute explicitly the integrals in the last two displays.
\item[$(c)$] Compare the result for $\phi(\beta)$ obtained at the previous point, with the replica symmetric formula in Section \ref{sec:SpikedMatrix}.
\end{enumerate}
\end{exercise}

\begin{exercise}
Consider the one-step replica symmetry breaking expression for the free energy density, given in Section \ref{sec:pSpin1RSB},
and assume $\lambda=0$, whence $b=0$. Further consider stationary points with $q_0=0$, and denote by 
$\Psi_{\rm{1RSB}}(q_1, m;\beta) =\Psi_{\rm{1RSB}}(b=0, q_0=0, q_1, m;\beta,\lambda=0)$
the corresponding free energy functional. 
\begin{enumerate}
\item[$(a)$] Write a program to find the correct stationary point $q_{1,*}$, $m_*$ for given $\beta, k$.
\item[$(b)$] Plot $q_{1,*}=q_{1,*}(\beta)$, $m_*=m_*(\beta)$ and the resulting free energy density 
$\Psi_{\rm{1RSB}}(q_{1,*}(\beta), m_*(\beta))$, as a function of $T=1/\beta$ for $k\in\{3,4,5\}$.
\item[$(c)$] Compute numerically the critical temperature $T_s$ for $k\in\{3,4,\dots,10\}$.
\end{enumerate}
\end{exercise}

\begin{exercise}
As in the previous exercise, assume $\lambda=b=q_0=0$, and denote by $\Psi_{\rm{1RSB}}(q_1, m;\beta) =\Psi_{\rm{1RSB}}(b=0, q_0=0, q_1, m;\beta,\lambda=0)$
the corresponding free energy functional. We will consider the  limit $\beta\to\infty$ (equivalently, $T=1/\beta\to 0$).
\begin{enumerate}
\item[$(a)$] Assuming $m = \mu\, T$ for some $\mu\in\reals_{>0}$ as $T\to 0$, and $q_{1,*}(\beta,m) = 1-c(\mu)T+o(T)$, compute $c(\mu)$.
\item[$(b)$] Substituting in the free energy functional, show that
\begin{align}
\Psi_{\rm{1RSB}}(q_{1,*}(\beta,m=\mu T), m=\mu T;\beta) =\beta\, e_{\rm{1RSB}}(\mu)+o(\beta)\, ,
\end{align}
and compute the function $e_{\rm{1RSB}}(\mu)$. 
\item[$(c)$] What is the meaning of $e^*_{\rm{1RSB}}=\min_{\mu>0}e_{\rm{1RSB}}(\mu)$? Compute numerically this quantity for $k\in\{3,4,\dots,10\}$.
\end{enumerate}
\end{exercise}

\begin{exercise}
Consider the spherical model  with $h=\lambda=0$, cf. Eq.~(\ref{eq:PspinDef}), with $k=3$.
\begin{enumerate}
\item[$(a)$] Choose two temperatures $\{T_1,T_2\}$ with $T_1\in (0,T_s)$ and $T_2\in(T_s, T_d)$.
For each of these temperatures, minimize the 1RSB expression of free energy $\Psi_{\sRSB{1}}(q_1,m)$, cf. Eq.~(\ref{eq:SimplePsi1RSB})
to compute the thermodynamic order parameter $q_{1,*}(m)$, and the corresponding free energy $\Psi(m) = \Psi_{\sRSB{1}}(q_{1,*}(m),m)$.
Plot $q_{1,*}(m)$ and $\Psi(m)$ as a function of $m$.
\item[$(b)$] Recall that the internal energy of pure states corresponding to parameter $m$ is given by $f(m) = \frac{\de\phantom{m}}{\de m}[m\Psi(m)]$.
Compute an analytic expression for this function by taking the derivative of  $m\Psi_{\sRSB{1}}(q_{1,*}(m),m)$.
Evaluate this expression numerically and plot $f(m)$ as a function of $m$.
\item[$(c)$] By taking the Legendre-Fenchel dual of the function $m\Psi(m)$, compute the complexity function $\Sigma(f)$
and plot it as a function of $f$, for each $T\in\{T_1,T_2\}$. 

[Concretely, you can obtain a plot of $\Sigma(f)$ by a parametric plot of $m\Phi(m)-mf(m)$ versus $f(m)$.] 
\item[$(d)$] Denote by $[f_d,f_s]$ the interval of free energies over which $\Sigma(f)\ge 0$ (with $f_d$ corresponding to
threshold states and $f_s$ to states with the largest free energy). Compute $f_s$ and $f_d$ for $T\in\{T_1,T_2\}$.
Check that $\Sigma'(f_s)>-1$ for $T=T_1$ and  $\Sigma'(f_s)<-1$ for $T=T_2$.
Compute also $f_*$, i.e. the value that maximizes $f+\Sigma(f)$ over $f\in [f_s,f_d]$.

How can you compute $f_s$, $f_d$, $f_*$ directly from $\Psi(m)$?
\item[$(e)$] Use the answer to the last question at the previous point to  compute $f_s(T)$, $f_d(T)$, $f_*(T)$ on
a grid of values of $T$ and plot them together as a function of $T$. Identify $T_s$, $T_d$ on this plot. 
\item[$(f)$] This question is \emph{optional}.
Recall that the picture behind 1RSB is that the measure $\mu$ decomposes in pure states
\begin{align}
\mu(\de\bsigma) = \sum_{\alpha=1}^N w^{(\alpha)}\mu^{(\alpha)}(\de\bsigma)\, .
\end{align}
The internal energy of state $\alpha$ can be defined by 
\begin{align}
u^{(\alpha})= \frac{1}{\sqrt{2n}} \int_{\sS^{n-1}} \<\bG,\bsigma^{\otimes k}\> \, \mu^{(\alpha)}(\de\bsigma)\, .
\end{align}
Can you use the replica method to compute the typical internal energy of states with 
internal free energy $f\in [f_s,f_d]$. Denote this by $u(f)$.

Use this expression for the temperatures $\{T_1,T_2\}$ considered above and plot $u(f)$ as a function
of $f$. Report the values of $u(f_s)$ and $u(f_d)$.
\end{enumerate}

\end{exercise}

\chapter{The Sherrington-Kirkpatrick model and $\mathbb{Z}_2$ synchronization}
\label{ch:SK}
\section{Introduction}

In this chapter, we study the Sherrington-Kirkpatrick (SK) model of spin glasses. 
This classical model was introduced by David Sherrington and Scott Kirkpatrick in 1975 
\cite{sherrington1975solvable} as a simplified model for spin glasses,
 which are a class of magnetic alloys. Sherrington and Kirkpatrick believed 
 that their model was amenable to
 more straightforward mathematical analysis than the  more
 realistic lattice model of Edwards and Anderson \cite{edwards1975theory}. 
 However, it took several years until Giorgio 
 Parisi came up with a heuristic derivation of the asymptotic free energy density, introducing 
 the idea of replica symmetry breaking \cite{parisi1979infinite}, and three decades before 
 Michel Talagrand proved Parisi's formulas
 \cite{talagrand2006parisi}.
 Efforts to analyze the SK model have enriched statistical physics and probability since
 its introduction, and still many interesting questions remain open. 
 
 As motivation for studying the SK model, we will use a statistical estimation problem, 
 group synchronization, which we introduce in Section \ref{sec:GroupSynch}. As further motivation,
 we discuss the connection with extremal cuts of random graphs in Section \ref{sec:Cuts}.


\subsection{The group synchronization problem}
\label{sec:GroupSynch}

An instance of the  group synchronization problem is defined by a finite graph 
$G = (V,E)$ and a group $\mathcal{G}$. Without loss of generality, we assume
 $V = [n]$ and assign a direction arbitrarily to each edge $\{i,j\} \in E$. 
 
  Let $\bx_0= (x_{0,1}, \cdots, x_{0,n}) \in \mathcal{G}^{V}$ 
 denote an unknown assignment of group elements to the vertices of the graph.   
We observe $\{Y_{ij}: (i,j) \in E \}$, where $Y_{ij}$ is a noisy version of 
 $x_{0,i}^{-1} x_{0,j}$, and would like to estimate the unknown object $\bx_0$.
 Of course this is possible only up to a global shift by a group element,
 since $\bx_0' = (gx_{0,1}, \cdots, gx_{0,n})$, $g\in\cG$ yields the same observations as $\bx_0$.
 
  A special case of practical interest is 
 $\cG = {\sf SO}(3)$, which arises in imaging applications,
  such as structure from motion \cite{ozyecsil2017survey}. In this case, the main question
   is to reconstruct a 3D image of an object from numerous 2D images taken by multiple 
   cameras placed in unknown positions. The first step is to determine the relative positions 
   of the cameras, which leads to the synchronization problem over ${\sf SO}(3)$. 

Arguably, the simplest version of this problem arises when
$\cG= \integers_2 = \{ \pm 1, \cdot \}$ and $G = K_n$, where $K_n$ is the complete graph 
on $n$ vertices. Given an unknown signal $\bx_0 = ( x_{01}, \cdots, x_{0n}) \in \{ \pm 1\}^n$ 
we observe, for $i < j$, 
\begin{align}
Y_{ij} = \frac{\lambda}{n} x_{0i}x_{0j} + W_{ij}, \label{eq:model}
\end{align}
where $W_{ij} \sim \normal(0,1/n)$ are independent random variables and $\lambda \geq 0$.
We can form a matrix $\bY\in \reals^{n\times n}$ by letting $Y_{ji}= Y_{ij}$ 
and assigning $Y_{ii}$ arbitrarily. A convenient choice is to take $Y_{ii} = (\lambda/n)+W_{ii}$
with $W_{ii}\sim \normal(0,2/n)$, whence the above model can be rewritten in matrix
notation as
\begin{align}
\label{eq:Z2model2}
\bY = \frac{\lambda}{n} \bx_{0}\bx_{0}^{\sT} + \bW\, ,\;\;\;\;\;\bW\sim\GOE(n)\, .
\end{align}

 Of course, this is in no way the most general form of the $\mathbb{Z}_2$ 
 synchronization problem.
 In particular, we could consider more general noise models defined by
 a transition probability kernel $Q$ (a `noisy channel'), which in this case is uniquely defined by 
 two probability measures  $Q(\cdot | +1)$ and $Q(\cdot | -1)$
 on a measurable space $\cX$. In this more general setup, for each edge $(i,j)$ we observe
 \begin{align}
 Y_{ij}\sim Q(\,\cdot\,| x_{0i} x_{0j})\, ,
 \end{align}
 and these observations are conditionally independent given $\bx_0$. 
 
 We notice in passing that ---if the transition kernel $Q$ is known--- there is no loss of 
 generality in assuming the observations to be real-valued, i.e. 
 $Y_{ij}\in \cX= \reals$. 
Indeed, without loss of generality  $Q(\, \cdot\, | +1), Q(\,\cdot\,|-1) \ll Q_0$ 
for some measure $Q_0$ on $\cX$ (take $Q_0(\,\cdot\, )=Q(\, \cdot\, | +1)+Q(\,\cdot\,|-1)$). 
The likelihood of the observations is
\begin{align}
L(\bY| \bx_0) = \prod_{i <j } \Big( \frac{\de Q(Y_{ij}|+1)}{ \de Q(Y_{ij} |-1)} \Big)^{\frac{x_{0i} x_{0j}}{2}} \sqrt{ \frac{\de Q(Y_{ij}|+1)}{\de Q_0} \cdot \frac{\de Q(Y_{ij}|-1)}{ \de Q_0} }. \nonumber 
\end{align} 
where we wrote  $\frac{\de Q(Y_{ij}|+1)}{ \de Q(Y_{ij} |-1)}$ for the ratio 
$\frac{\de Q(Y_{ij}|+1)}{\de Q_0} / \frac{\de Q(Y_{ij}|-1)}{ \de Q_0}$.
Thus, by the Fisher-Neyman factorization theorem \cite{lehmann2006theory},
the variables 
\begin{align}
\widehat{Y}_{ij} = \log \Big(\frac{\de Q(Y_{ij}| +1)}{\de Q(Y_{ij}| -1)} \Big) \, ,\;\;\;\;
i<j\le n
\end{align}
are sufficient statistics for $\bx_0$ given $\bY$. We can thus assume without loss of generality that our
 observations are real valued. 
 
We return to the study of the model \eqref{eq:model}, which observes the signal
 with additive Gaussian noise. For maximum likelihood estimation, we solve
\begin{align}\label{eq:maxlikelihood}
\begin{split}
\mbox{\rm maximize}&\;\;\;\;\;\<\bY,\bsigma \bsigma^{\sT}\>\, ,\\
\mbox{\rm subject to}&\;\;\;\;\;\bsigma\in \{ \pm 1\}^n . 
\end{split}
\end{align}
To carry out Bayesian inference, we assign a prior on $\bx_0$ which is uniform on 
$\{+1,-1\}^n$, and form the posterior distribution 
\begin{align}
\mu_{\sBayes} ( \bsigma) &= \frac{1}{\tilde{Z}_{\sBayes}(\lambda)} 
\exp\Big\{- \frac{n}{4}  \Big\| \boldsymbol{Y} - \frac{\lambda}{n}\bsigma \bsigma^{\sT} \Big\|_{F}^2 \Big\} \\
 &= \frac{1}{Z_{\sBayes}(\lambda)}  \exp\Big\{ \frac{\lambda}{2} \langle \bY, \bsigma \bsigma^{\sT} \rangle \Big\}. \label{eq:Bayes}
\end{align}
As in the previous chapter, this model  enjoys a large invariance group. 
Namely, it is equivariant under multiplications  $\bx_0\mapsto \bD \bx_0$
with $\bD$ a diagonal matrix with $\{+1,-1\}$ entries on the diagonal. The uniform
prior is the only invariant prior under this group and therefore, by the Hunt-Stein's 
theorem \cite{lehmann2006theory}, the Bayes estimator with respect to this prior is also 
minimax optimal.

As for the spiked tensor model, to study these estimation procedures in one unified framework,
 we introduce the following probability measure over $\bsigma\in\{+1,-1\}^n$:
\begin{align}
\mu_{\beta}(\bsigma) = \frac{1}{Z_n(\beta,h)} \exp{\Big\{\frac{\beta}{2} \langle \boldsymbol{Y}, \bsigma \bsigma^{\rm{T}} \rangle + h \langle\bsigma, \boldsymbol{x_0} \rangle \Big\}},  \label{eq:gibbs}
\end{align}
where $\beta >0$ is the inverse temperature and $Z_n(\beta, h)$ is the partition function.
 The case $h=0$, $\lambda = \beta$ corresponds to Bayes estimation, while we recover 
 maximum likelihood for $h=0$, $\beta \to \infty$. 

\subsection{Extremal cuts of random graphs}
\label{sec:Cuts}

The Sherrington-Kirkpatrick model is intimately connected to 
some canonical optimization problems on sparse random graphs. 
The most direct connection is with the max-cut problem.
Given a graph $G = (V, E)$ max-cut  seeks to
  partition the vertices $V$ into two groups (not necessarily of equal size) 
  such that the number of edges connecting the two parts is maximized.
This problem is known  to be NP-hard in the worst case, and indeed hard to approximate
\cite{papadimitriou1991optimization,arora2005non}.

The analysis of this problem on random instances 
provides a concrete benchmark for practical algorithms, and has directly motivated the 
 development of novel algorithms. 
 Two standard random graph distributions are the \ER 
 ensemble  $\boldsymbol{G}^{\sf{ER}}(n, d/n)$  and the 
 random regular graph $\boldsymbol{G}^{\sf{Reg}}(n,d)$. 

An \ER random graph $G\sim \boldsymbol{G}^{\sf{ER}}(n,d/n)$ has $n$ vertices, 
and is generated by connecting each pair of vertices $\{i,j\}$ by an edge 
independently with probability $d/n$. We will consider the sparse case in which $d>0$ is a 
constant independent of $n$. In this model, the degree of a randomly chosen vertex converges 
in distribution to $\Poisson(d)$ as $n \to \infty$.
In particular, $d$ is the average vertex degree.

 On the other hand, a random regular graph $G\sim \boldsymbol{G}^{\sf{Reg}}(n,d)$ is 
 constructed by a uniform draw among all $d$-regular (simple) graphs on $n$ vertices.
 (The set of such graphs is non-empty if $nd$ is even.)

Let ${\sf{MaxCut}}(G)$ denote the number of edges in the max-cut of the graph $G$. The 
next result connects the typical value of the max-cut on \ER and random regular graphs to the 
``ground state energy'' of the Sherrington-Kirkpatrick model. 

We begin by defining a constant ${\sf{P}_*}$  as follows: 
\begin{align}
{\sf{P}_*} = \lim_{n \to \infty} \frac{1}{2n} \mathbb{E}\Big[ \max_{\bsigma \in \{ \pm 1\}^n} 
\<\bsigma,\bW\bsigma\> \Big], \;\;\;\;\; \bW\sim\GOE(n)\, . \nonumber 
\end{align} 
Note that the optimization problem on the right-hand side corresponds to finding the mode 
of the distribution \eqref{eq:gibbs} in the special case $\lambda=h=0$.

The existence of the constant ${\sf{P}_*}$ follows directly from the analyses presented in the 
subsequent sections, and a variational characterization of this constant is
given by the Parisi formula (see Section \ref{sec:SKrigorous} and references 
provided there).
\begin{theorem}[\cite{dembo2017extremal}]
\label{thm:maxcut} 
Let $G_n$ be distributed as $\boldsymbol{G}^{\sf{ER}}(n, d/n)$ or $\boldsymbol{G}^{\sf{Reg}}(n,d)$. Then we have, 
\begin{align}
\frac{{\sf{MaxCut}}(G_n)}{n} = \frac{d}{4} + {\sf{P}}_* \sqrt{\frac{d}{4}} + o_d(\sqrt{d}). \label{eq:maxcut_result} 
\end{align} 
\end{theorem} 
(We say that a sequence of random variables satisfies $X_n = o_d(\sqrt{d})$ if there exists a
 deterministic function $f(d) = o(\sqrt{d})$ such that $\mathbb{P}[|X_n| \leq f(d)] \to 1$ as
  $n \to \infty$.)

Before providing a heuristic argument to justify Theorem \ref{thm:maxcut}, 
let us interpret this result. Given any graph $G$, a naive randomized 
strategy is to assign the vertices into two groups uniformly at random. 
This strategy cuts about half of the edges in expectation. As 
$\E[|E|] = nd/2+o_n(n)$, the expected cut size for this strategy is $nd/4+ o_n(n)$,
and the actual cut size is tightly concentrated around the expectation\footnote{This can be proven by
computing the variance, or using a bounded difference martingale argument.}.
We thus obtain the lower bound 
\begin{align}
\frac{{\sf{MaxCut}}(G_n)}{n} \ge \frac{d}{4} + o_n(1)\, .
\end{align}
which explains the $d/4$ term appearing in \eqref{eq:maxcut_result}.
 The $\sqrt{d}$-term captures the effect of optimization. The theorem
 states that---to leading order---optimizing the cut size in a random graph
 is equivalent (from the point of view of the value achieved) to optimizing 
 the random quadratic form $\<\bsigma,\bW\bsigma\>$ over the hypercube 
 $\bsigma \in \{ \pm 1\}^n$ (this coincides with optimizing the SK Hamiltonian).
 
 Next, we provide a heuristic justification for this
 correspondence.  We assume, to be definite, 
 $G_n \sim \boldsymbol{G}^{\sf{ER}}(n, d/n)$. Let $\bA = (A_{ij})_{i,j\le n}$ denote the 
 adjacency matrix of $G_n$; formally, $A_{ij} = \boldsymbol{1}(\{i,j\} \in E)$. We have
\begin{align}
\mathbb{P}(A_{ij} = 1 ) = 1 - \mathbb{P}(A_{ij} =0) = \frac{d}{n}. \nonumber  
\end{align} 

Observe that we can encode a vertex partition via a vector $\bsigma \in \{\pm 1\}^n$:
 the vertices with $\sigma_i=+1$ form one group, while vertices with $\sigma_i=-1$ form
  the complementary group. Let ${\sf{cut}}(\bsigma)$ denote the number of 
  cut-edges in this partition, i.e. edges $\{i,j\} \in E$ with $\sigma_i \sigma_j =-1$. 
  We observe that 
\begin{align}
{\sf{cut}}(\bsigma) = \frac{1}{2} \sum_{i<j} A_{ij}(1- \sigma_i \sigma_j)\, .  \label{eq:cut_defn} 
\end{align}
This implies
\begin{align}
\frac{{\sf{MaxCut}}(\boldsymbol{G}_n)}{n}  &=\frac{1}{2n}\max_{\bsigma \in \{ \pm 1\}} \sum_{i<j} A_{ij}(1- \sigma_i \sigma_j) \nonumber \\
&= \max_{\bsigma \in \{ \pm 1\}} \Big[ \frac{d}{2n^2} \Big(\frac{n^2}{2} - \Big( \sum_i \sigma_i \Big)^2 \Big) +  \frac{1}{2n} \sum_{i<j} \Big(A_{ij} - \frac{d}{n} \Big) (1- \sigma_i \sigma_j)  \Big] \nonumber \\
&=  \sqrt{d} \cdot \max_{\bsigma \in \{ \pm 1\}} \Big[ \frac{\sqrt{d}}{2n^2} \Big(\frac{n^2}{2} - \Big( \sum_i \sigma_i \Big)^2 \Big) -  \frac{1}{2n \sqrt{d} } \sum_{i<j} \Big(A_{ij} - \frac{d}{n} \Big)  \sigma_i \sigma_j  \Big] + o_n(1), \label{eq:opt_int1}  
\end{align} 
where the last display follows from Chebychev inequality. Indeed, for any $\varepsilon>0$, 
\begin{align}
\mathbb{P} \Big[ \Big| \sum_{i<j} \Big(A_{ij} - \frac{d}{n} \Big)  \Big| > \varepsilon n   \Big] \leq \frac{\mathbb{E}\Big[  \Big(  \sum_{i<j} \Big(A_{ij} - \frac{d}{n} \Big)  \Big)^2 \Big] }{\varepsilon^2 n^2} \leq \frac{d}{2 \varepsilon^2 n} \to 0. \nonumber 
\end{align} 
Returning to Eq \eqref{eq:opt_int1}, we make two observations: $(i)$~The 
first term in the optimization problem is deterministic, while the second term is random;
 $(ii)$~The random variables $\{ A_{ij} - \frac{d}{n} : i <j \}$ are independent, with
 mean zero and $\mathbb{E}[(A_{ij} - \frac{d}{n})^2] = \frac{d}{n} + o_n(1)$. 
 Theorem \ref{thm:maxcut} is proven by establishing the following: 
\begin{itemize}
\item[$(i)$] At the cost of a $o_d(\sqrt{d})$ term, we can replace the $A_{ij}-\frac{d}{n}$ 
variables by independent Gaussians with mean zero and matching variances. Formally, 
one can show that (recall that $\ed$ denotes equality in distribution):
\begin{align}
\frac{{\sf{MaxCut}}(\boldsymbol{G}_n)}{n} \ed
\max_{\bsigma \in \{ \pm 1\}} \Big[ \frac{d}{2n^2} \Big(\frac{n^2}{2} - \Big( \sum_i \sigma_i \Big)^2 \Big) -  \frac{\sqrt{d} }{2n} \sum_{i<j} W_{ij} \sigma_i \sigma_j  \Big] + o_d(\sqrt{d} ). \label{eq:opt_gaussian} 
\end{align} 
\item[($ii$)] In \eqref{eq:opt_gaussian}, the Gaussian optimization term will yield a term of 
order $\sqrt{d}$, whereas the first deterministic term is of order $d$. Thus for large 
$d>0$, the optimum value is approximately obtained by first optimizing the deterministic 
part, and subsequently optimizing the random part subject to this constraint. 
In this specific setting, note that the deterministic part is optimized by setting $\sum_{i\le n} \sigma_i =0$; incorporating this constraint, we wish to solve
\begin{align}
\mbox{\rm maximize}&\;\;\;\;\; \sum_{i<j} W_{ij} \sigma_i \sigma_j ,\\
\mbox{\rm subject to}&\;\;\;\;\;\bsigma\in \{\pm 1\}^n, \sum_{i=1}^n \sigma_i =0. \, ,
\end{align}

\item[($iii$)] A direct application of Gaussian concentration (see Appendix \ref{app:Probability})
 establishes that the optimal value above concentrates around its expectation. Finally, it 
 is relatively straightforward to see that the constraint $\sum_{i=1}^n \sigma_i=0$ can 
 be dropped without changing significantly the value of this optimization problem. 
\end{itemize} 
This completes a sketch of the proof of Theorem \ref{thm:maxcut}.

The techniques outlined above can be used for analyzing several related graph cut
 quantities on random graphs.  Recall the definition of ${\sf{cut}}(\bsigma)$  from
  \eqref{eq:cut_defn}. For a graph $G$ with adjacency matrix 
  $\boldsymbol{A} = (A_{ij})_{i,j\le n}$, we introduce, 
\begin{align}
{\sf{MaxBis}}(G) = \max_{\bsigma \in \{ \pm 1\}^n: \sum_i \sigma_i =0} {\sf{cut}}(\bsigma), \nonumber \\
{\sf{MinBis}}(G) = \min_{\bsigma \in \{ \pm 1\}^n: \sum_i \sigma_i =0} {\sf{cut}}(\bsigma). \nonumber 
\end{align}  
We refer to  these as the max-bisection and min-bisection of $G$:
they are the extremal sizes of cuts that partition the graph into equal components.
 The next result collects the analogue of Theorem~\ref{thm:maxcut} for these quantities. 
\begin{theorem}[\cite{dembo2017extremal}]
\label{thm:bisection} 
Let $G_n$ be distributed as $\boldsymbol{G}^{\sf{ER}}(n, d/n)$ or $\boldsymbol{G}^{\sf{Reg}}(n,d)$. Then we have, 
\begin{align}
\frac{{\sf{MaxBis}}(\boldsymbol{G}_n)}{n} = \frac{d}{4} + {\sf{P}}_* \sqrt{\frac{d}{4}} + o_d(\sqrt{d}),  \label{eq:maxbis_result} \\
\frac{{\sf{MinBis}}(\boldsymbol{G}_n)}{n} = \frac{d}{4} - {\sf{P}}_* \sqrt{\frac{d}{4}} + o_d(\sqrt{d}). \label{eq:minbis_result}
\end{align} 
\end{theorem} 

\begin{remark}
Based on the cavity method, \cite{zdeborova2010conjecture} conjectured that 
\begin{align}
{\sf{MinBis}}(\boldsymbol{G}^{\sf{Reg}}(n,d)) + {\sf{MaxCut}}(\boldsymbol{G}^{\sf{Reg}}(n,d)) = o_n(n). \nonumber 
\end{align}  
Combining \eqref{eq:maxcut_result} and \eqref{eq:minbis_result}, we have, 
\begin{align}
{\sf{MinBis}}(\boldsymbol{G}^{\sf{Reg}}(n,d)) + {\sf{MaxCut}}(\boldsymbol{G}^{\sf{Reg}}(n,d)) = n o_d(\sqrt{d}). \nonumber 
\end{align}  
This is significantly weaker than the claim, but provides some supporting evidence to the veracity of this conjecture. 
\end{remark}

We finally note that the connection between optimization with random instances and spin 
glass models is significantly deeper than what is sketched here.
In particular, generalizations of the techniques developed in the next sections exist
for random sparse graphs. These allow, to cite just one example, to derive exact
(albeit non-rigorous) predictions for $\lim_{n\to\infty}{\sf{MaxCut}}(G_n)/n$
in the above graph models, \emph{at fixed $d$}. We refer to \cite{MezardMontanari}
for an entry-point to that literature.

\noindent
\textbf{Organization:} The rest of the chapter is organized as follows. We describe the 
replica symmetric approximation for this problem in Section \ref{sec:RS-Z2-a}. In Section
 \ref{sec:replica_symmetric_cavity}, we introduce the replica symmetric cavity method. The cavity 
 method provides an alternative to the replica method to arrive at the same predictions. 
 The replica method is based on a sequence of formal manipulations that,
 although not not well-defined mathematically and somewhat mysterious, are very
 rigid. 
 In contrast, the cavity method is closer to a probabilistic argument (an induction over $n$) 
 but highlights and makes use of certain deep probabilistic assumptions that 
 underly both approaches. The replica symmetric 
 cavity method naturally suggests some algorithmic approaches to computing the low 
 dimensional marginals of the probability distributions $\mu_{\beta}(\cdot)$. We develop
  and rigorously explore these algorithms in Section \ref{sec:algorithmic}. As in the 
  previous chapter, the replica symmetric approximation is inaccurate for certain parameters
   $(\beta,h)$---we overcome this limitation using the replica symmetry breaking (RSB) scheme in
    Section \ref{sec:kRSB-Replica-All}. Although the full replica symmetry breaking (RSB)
     approach yields the right approximation for all parameters $(\beta,h)$, the probabilistic
      assumptions inherent in this approach are more opaque. To gain more insights into 
      the structure of the distribution $\mu_{\beta}(\cdot)$, we introduce the replica 
      symmetry breaking version of the cavity method in Section \ref{sec:CavityRSB}; 
      this sheds valuable light into the independence structures inherent in the 
      replica symmetry breaking approach. Finally, we prove the rigorous replica 
      symmetry breaking upper bounds on the free energy in Section \ref{sec:SKrigorous}.

\section{The replica symmetric asymptotics}
\label{sec:RS-Z2-a}

\subsection{The replica calculation}
\label{sec:RS-Z2}
In this section, we initiate the study of the partition function using the replica method 
and study the replica symmetric asymptotics. For the sake of simplicity we will focus on the
case $h=0$, and drop the arguments of the partition function  $Z_n=Z_n(\beta,h=0)$.
Writing explicitly the dependence of $\bY$ on $\bx_0$ and the random matrix $\bW$, 
cf. Eq.~\eqref{eq:Z2model2}, we have
\begin{align}
Z_n = \sum_{\bsigma \in \{\pm 1\}^n } \exp{\Big\{\frac{\beta \lambda}{2n} \langle \bx_0 , \bsigma \rangle^2 + \frac{\beta}{2} \langle \bW, \bsigma \bsigma^{\rm{T}} \rangle \Big\}}. \label{eq:partition}
\end{align}
Recall that $\boldsymbol{x}_0$ is uniformly distributed on $\{ \pm 1 \}^n$ and 
$\bW = (\bG + \bG^{\sT})/\sqrt{2n}$, where $\bG = (G_{ij})_{n \times n}$ is a square matrix with 
i.i.d. $\normal(0, 1)$ entries. We will compute $\E[Z_n(\beta)^r]$ for integer $r$, then send
 $n\to \infty$, followed by $r \to 0$ to evaluate the free energy in the replica method. 
 
Taking expectations with respect to the randomness of $\boldsymbol{x}_0$ and $\bW$, we have, 
\begin{align}
\E[Z_n^r] &= \frac{1}{2^n} \sum_{\bsigma^0, \bsigma^1, \cdots, \bsigma^r} \exp{\Big\{ \frac{\beta \lambda}{2n} \sum_{a=1}^{r} \langle \bsigma^0, \bsigma^{a} \rangle^2 \Big\}} \E\exp{\Big\{ \frac{\beta}{\sqrt{2n}} \langle \bG, \sum_{a=1}^{r}\bsigma^{a}(\bsigma^{a})^{\sT} \rangle\Big\}} \nonumber  \\
&= \frac{1}{2^n} \sum_{\bsigma^{0}, \cdots, \bsigma^{r}} \exp{\Big\{ \frac{\beta\lambda}{2n} \sum_{a=1}^{r} \langle \bsigma^0, \bsigma^a \rangle^2 + \frac{\beta^2}{4n} \Big\| \sum_{a=1}^{r} \bsigma^a (\bsigma^a)^{\sT} \Big\|_{F}^2 \Big\} }  \nonumber \\
&= \frac{1}{2^n} \sum_{\bsigma^0, \cdots, \bsigma^{r}} \exp{\Big\{ \frac{\beta^2 r n}{4}  + \frac{\beta \lambda}{2n} \sum_{a=1}^{r} \langle \bsigma^0, \bsigma^a \rangle^2  + \frac{\beta^2 }{ 2n} \sum_{1 \leq a < b \leq r} \langle \bsigma^a, \bsigma^b \rangle^2  \Big\}} \, .\nonumber 
\end{align}
We note that the summands depend on the replicas $\bsigma^0, \cdots, \bsigma^{r}$ only via
their overlaps $\widehat{Q}_{ab} = \frac{1}{n} \langle \bsigma^a , \bsigma^b \rangle $. We are therefore in a situation similar to the one encountered when applying the replica
method to the tensor PCA model. We need to estimate an integral (or sum) over a high-dimensional space
(in this case $\{+1,-1\}^{nr}$) and the integrand only depends on a low-dimensional function 
on that space (in the present case $r(r-1)/2$-dimensional).
This sum can be estimated by large deviations tools.

 In the present case, the calculation is simplified by the following elementary 
 observation. For $x \in \reals$, 
\begin{align}
e^{x^2/2}= \int_{\reals} \!\exp{\Big( - \frac{q^2}{2} + x q  \Big)}\, \frac{\de q}{\sqrt{2\pi}}. \nonumber 
\end{align} 
This is often referred to as the `Gaussian disintegration trick' or as the 
`Hubbard-Stratonovich transform' in the physics literature. An application of this 
trick for each overlap linearizes the expectation and we obtain, 
\begin{align}
\E[Z_n^r] &= \int \frac{1}{2^n} \sum_{\bsigma^0, \cdots, \bsigma^r} \exp{\Big\{ n C(\bQ) +  \beta^2 \sum_{1 \leq a < b \leq r} Q_{a,b} \langle \bsigma^a, \bsigma^b \rangle  + \beta \lambda \sum_{a=1}^{r} Q_{0,a} \langle \bsigma^0, \bsigma^a \rangle \Big\}} D\bQ, \nonumber \\
C(\bQ) &=  \frac{\beta^2 r }{4} - \frac{\beta^2 }{2} \sum_{1 \leq a < b \leq r} Q_{a,b}^2 - \frac{\beta \lambda }{2} \sum_{a=1}^{r} Q_{0,a}^2 ,\nonumber  
\end{align}
where we use $D\bQ$ to denote the measure
\begin{align}
D\bQ = \prod_{a=1}^{r} \sqrt{\frac{\beta \lambda n}{2\pi}} \de Q_{0,a} \prod_{1\leq a < b \leq r} \sqrt{\frac{\beta^2 n}{2\pi}} \de Q_{a,b}. \nonumber 
\end{align}
We define 
\begin{align}
z_r(\bQ) = \frac{1}{2} \sum_{\sigma^0, \cdots, \sigma^r \in \{\pm 1 \}} \exp{ \Big[\beta \lambda \sum_{a=1}^{r} Q_{0,a} \sigma^{a} \sigma^{0} + \beta^2 \sum_{1\leq a < b \leq r} Q_{a,b} \sigma^{a} \sigma^{b} \Big]}. \nonumber 
\end{align}
We define the ``action" functional $S(\bQ) = C(\bQ) + \log z_r(\bQ)$ and note that,
by Laplace method:
\begin{align}
\E[Z_n^r] = \int \exp\big\{n S(\bQ)\big\} D\bQ = \exp\big\{n \max_{\bQ} S(\bQ)+o(n)\big\} \, .\label{eq:LaplaceSK} 
\end{align}

Thus, using the replica method, we expect the asymptotic free energy density to be given by
\begin{align}
 \phi &= \lim_{r \to 0} \frac{1}{r} S(\bQ^*)\, ,\;\;\;\; 
 \left.\frac{\partial S}{\partial \bQ}\right|_{\bQ^*}=0\, ,\\
 S(\bQ) & = \frac{\beta^2 r }{4} - \frac{\beta^2 }{2} \sum_{1 \leq a < b \leq r} Q_{a,b}^2 - \frac{\beta \lambda }{2} \sum_{a=1}^{r} Q_{0,a}^2 + \log z_r(\bQ)\, .
 \label{eq:SK_Action}
 \end{align}
 Note that we replaced  $\max_{\bQ} S(\bQ)$ which appears in Eq.~\eqref{eq:LaplaceSK}
 with $S(\bQ^*)$ where $\bQ^*$ is only required to be a stationary point of $S(\bQ)$. 
 Indeed, as we saw in the previous chapter, in the limit $r\to 0$, the relevant 
 stationary point is often not a maximum.

We will begin by evaluating the action functional on the replica symmetric saddle point. 
This saddle point is again motivated by symmetry considerations, in that the action 
$S(\, \cdot\, )$ is invariant to a permutation of the indices $\{1, \cdots , r \}$. 
Thus in this case, we consider 
\begin{align}\label{eq:RSansatz-Z2}
\bQ^{\sRS}_{a, b} = \begin{cases}
1 \quad {\rm{if }}\,\, a=b  \\
b \quad {\rm{if }}\,\, a=0, b \in \{1, \cdots, r\}  \\
q \quad {\rm{o.w.}} 
\end{cases}\nonumber 
\end{align}
Notice that the diagonal elements are once again set by the condition $\<\bsigma^a,\bsigma^a\>/n=1$.
We wish to evaluate $S(\bQ^{\sRS})$. To this end, we have, 
\begin{align}
z_r(\bQ^{\sRS}) &= \frac{1}{2} \sum_{\sigma^0, \cdots, \sigma^{r} \in \{ \pm 1 \}} \exp{\Big[ \beta \lambda b \sum_{a=1}^{r} \sigma^0 \sigma^{a} + \beta^2 q \sum_{1 \leq a <b \leq r} \sigma^{a}\sigma^{b} \Big]} \nonumber \\
&= \sum_{\sigma^{1}\cdots, \sigma^{r} \in \{ \pm 1 \}} \exp{\Big[ \beta \lambda b \sum_{a=1}^{r}  \sigma^{a} + \frac{\beta^2 q }{2} \Big(\sum_{a=1}^{r} \sigma^{a} \Big)^2 - \frac{\beta^2 q r }{2} \Big]} \nonumber \\
&= \E\Big[ \sum_{\sigma^1, \cdots, \sigma^r}  \exp{ \Big[ \beta \lambda b \sum_{a=1}^{r} \sigma^a + \beta \sqrt{q} g \sum_{a=1}^{r} \sigma^a - \frac{\beta^2 q r}{2} \Big]} \Big] \nonumber \\
&= \exp{\Big[ - \frac{\beta^2 q r }{2} \Big]} \E\Big[ ( 2 {\rm{cosh}} \beta(\lambda b + \sqrt{q} g) )^r \Big], \nonumber 
\end{align}
where $g \sim \normal(0,1)$. Finally, as $r \to 0$, we have, 
\begin{align}
\frac{1}{r} \log z_r(\bQ_{\rm{RS}} ) = - \frac{1}{2} \beta^2 q + \E\log [ 2 {\rm{cosh}} \beta (\lambda b + \sqrt{q} g) ] + o(r). \nonumber 
\end{align}
Further, we have, at the replica symmetric saddle point, $\sum_{a=1}^{r} Q_{0,a}^2 = r b^2$ and $\sum_{1\leq a < b \leq r } Q_{a,b}^2 = - \frac{r}{2} q^2+ O(r^2)$. Plugging these back into $S(\bQ)$, we obtain the replica-symmetric free energy functional
\begin{align}
&\lim_{r \to 0} \frac{1}{r} S(\bQ^{\sRS}) = \Psi_{\sRS}(b,q) \, ,\\
&\Psi_{\sRS}(b,q) = \frac{1}{4} \beta^2 (1- q)^2 - \frac{1}{2} \beta \lambda b^2 + \E[\log 2 {\rm{cosh}} (\beta (\lambda b + \sqrt{q} g) ) ].  \label{eq:rs-functional}
\end{align}
For sanity check, we note that in case we plug in $b = q=0$, we obtain
 $\Psi_{\sRS}(0,0) = \beta^2/4 +\log 2$. The stationary point $b=q=0$  corresponds usually
  to the `annealed' free energy density
  $\lim_{n\to\infty}n^{-1} \log \E[Z_n(\beta)]$. Indeed it is easy to check that this quantity coincides with
  $\beta^2/4 +\log 2$.
  
As in the previous chapter, we look for stationary points of  
$(b,q)\mapsto\Psi_{\sRS}(b,q)$  to derive the replica symmetric free energy 
density. To this end, we study the stationarity conditions
\begin{align}
\frac{\partial \Psi_{\sRS}(b,q)}{\partial q} &= \frac{\beta^2 q}{2} - \frac{\beta^2}{2} \E[\tanh^2 (\beta(\lambda b + \sqrt{q} G))] =0 , \nonumber \\
\frac{\partial \Psi_{\sRS}(b,q)}{\partial b} &= - \beta \lambda b + \beta \lambda \E[\tanh (\beta(\lambda b + \sqrt{q} G ))] =0 , \nonumber 
\end{align} 
where $G\sim\normal(0,1)$ and 
the first equation is derived by an application of Gaussian integration by parts. This 
implies the following system of fixed point equations for $b,q$: 
\begin{align}
b &= \E[ \tanh (\beta ( \lambda b + \sqrt{q} G))] , \nonumber \\
q &= \E[ \tanh^2 ( \beta (\lambda b + \sqrt{q} G))]. \label{eq:rs-fixed}
\end{align}

Before we start analyzing the phase diagram resulting from the replica symmetric ansatz,
 we discuss the interpretation of the parameters $b,q$, ant their connection with the marginal 
 distribution of each spin $\sigma_i$ under the measure 
 $\mu_{\beta}$. 
Set $\hbx_{\beta}(\bY) = \sum_{\bsigma} \bsigma \mu_{\beta}(\bsigma)$, 
the expectation under the measure $\mu_{\beta}$. Note that,
for any  $\beta$, $\hbx_{\beta}(\bY)$ is a valid estimator of $\bx_0$, and it coincides with 
the Bayes estimator for $\beta=\lambda$, and with the maximum likelihood estimator
for $\beta\to\infty$.  We will write $\hbx_{\beta} = \hbx_{\beta}(\bY)$ for simplicity.

 Recall from the previous chapter that $b$ is interpreted as the asymptotic overlap of 
 $\hbx$ and $\bx_0$ (equivalently, this is the expected overlap
 between $\bsigma\sim \mu_{\beta}$ and $\bx_0$), while $q$ should be the overlap of $\hbx$ with itself
 (equivalently, this is the expected overlap
 between $\bsigma^1$ and $\bsigma^2$, where $\bsigma^1,\bsigma^2\sim^{iid} \mu_{\beta}$).
 In formulas 
 \begin{align}
 b&= \lim_{n \to \infty} \frac{1}{n} \E[ \langle \hbx_{\beta}, \bx_0 \rangle ] \, ,\\
 q &= \lim_{n \to \infty} \frac{1}{n} \E[ \|\hbx_{\beta}\|^2_2 ] \, .
 \end{align}
 Without loss of generality, we can set $\bx_0 = \bfone$. In this case, 
 using the exchangeability of the coordinates of $\sigma$, we can rewrite
$b = \lim_{n \to \infty} \E[\hat{x}_{\beta,1}]$
and  $q= \lim_{n \to \infty} \E[ \hat{x}_{\beta,1}^2]$.

Note that $\mu_{\beta}$ is a random distribution and consequently $\mu_{\beta,i}(\,\cdot\,)$,
 the marginal distribution of $\sigma_i$ under $\mu_{\beta}$ is a random distribution supported on 
 $\{ +1,-1 \}$. We parametrize the  distribution as 
\begin{align}
\mu_{\beta,i}(\sigma_i ) = \frac{e^{\beta h_i \sigma_i }}{2 \cosh (\beta h_i )}.  \label{eq:FirstMargSK}
\end{align}
The quantity $h_i$ is referred to as the `effective field' on the spin $\sigma_i$. 
Under this parametrization, we have $\hat{x}_{\beta,1} =\tanh (\beta h_1)$ and therefore
\begin{align}
b = \lim_{n\to\infty }\E[\tanh (\beta h_1)] , \quad \quad 
q = \lim_{n\to\infty }\E[ \tanh^2 (\beta h_1)]. \nonumber 
\end{align}
Comparing with the fixed point equation system \eqref{eq:rs-fixed},
these equations suggest that the law of  $h_1$ converges to $\normal(\lambda b, q)$
as $n\to\infty$. With the replica method, this can be confirmed by computing
other moments $\E[ \tanh^{\ell}(\beta h_1)]$.
This heuristic will be extremely useful in the discussion of the 
replica symmetric cavity method in Section \ref{sec:replica_symmetric_cavity}.

Summarizing, the replica symmetric ansatz corresponds to the following picture
of the measure $\mu_{\beta}$:
\begin{itemize}
\item Under $\bsigma\sim \mu_{\beta}$, the coordinates
of $\bsigma = (\sigma_1,\dots,\sigma_n)$ are roughly  independent
(in the sense of finite-dimensional 
distributions).
\item The marginals of this measure can be parametrized as in Eq.~\eqref{eq:FirstMargSK}.
\item The effective fields  $h_i$ are roughly i.i.d. (in the sense of finite-dimensional 
distributions) with common law $\normal(\lambda b, q)$. 
\item The parameters $b,q$ 
solve
the fixed point equations \eqref{eq:rs-fixed}.
\end{itemize}
We will see that the replica symmetric picture is correct at high temperature (and in particular
on the line $\beta=\lambda$), but in general it is \emph{incorrect}. As in the previous
chapter, replica symmetry breaking will be needed to get the correct behavior.

\subsection{The replica symmetric phase diagram}

We next study the solutions of the fixed point equations \eqref{eq:rs-fixed} 
and their dependence on $(\beta,\lambda)$. When
multiple solutions exist, the asymptotic free energy density, and the qualitative behavior
of $\mu_{\beta}$ are determined by the solutions that maximize the replica symmetric free
energy functional $\Psi_{\sRS}$ with respect to $b$ and minimize it with respect to $q$. 

Qualitatively different solutions correspond to qualitatively different behaviors of $\mu_{\beta}$
(phases):
\begin{enumerate}
\item \emph{Uninformative/paramagnetic phase.}
  $b_{\sP} = q_{\sP} = 0$ is always a solution to the fixed point equations \eqref{eq:rs-fixed}.
For $(\lambda, \beta)$ sufficiently small, this is the only solution,
 and the system is said to be in the paramagnetic phase. According to the interpretation
 of $b,q$ in the previous section,  we expect
 $\lim_{n \to \infty} n^{-1} \E[\<\hbx_\beta, \bx_0\> ]=0$, 
  $\lim_{n \to \infty} n^{-1} \E[\|\hbx_\beta\|^2 ]=0$, and the effective fields are $h_i\approx 0$. 
  In other words, the Gibbs measure 
  $\mu_{\beta}$ is roughly uniform (in the sense of the finite-dimensional marginals), 
  and uncorrelated with the signal $\bx_0$.
  In this phase, the free energy density coincides with the annealed one 
  $\Psi_{\sRS}(b_P,q_P) = \frac{\beta^2}{4} + \log 2$. 
\item \emph{Spin glass phase.}  For $\lambda$ sufficiently small, 
and $\beta$ sufficiently large, another fixed point $b_{\sSG} =0, q_{\sSG} >0$ 
emerges and the system is referred to be in the spin-glass phase. 
In this phase, $\lim_{n \to \infty} \frac{1}{n} \E[\< \hbx_{\beta}, \bx_0 \> ] = 0$ while 
$\lim_{n \to \infty} \frac{1}{n} \E[ \| \hbx_\beta \|_2^2 ] = q_{\sSG} > 0$. 
In words, the Gibbs measure is far from uniform, since its barycenter  $\hbx_\beta$
is far from zero. On the other hand, it is uncorrelated with the true signal $\bx_0$.

The cavity fields are asymptotically i.i.d. $h_i\sim \normal(0,q_*)$.
Hence the finite dimensional marginals of the Gibbs measure have approximate product 
form but are not uniform, and carry no information about $\bx_0$.

 The free energy prediction is
\begin{align}
\Psi_{\sRS}(0, q_{\sSG}) &= \frac{\beta^2}{4} ( 1- q_{\sSG})^{2} + \E[ \log2 \cosh (\beta \sqrt{q_\sSG} G)] , \nonumber \\
q_{\sSG} &= \E[ \tanh^2 (\beta\sqrt{q_\sSG} G)]. \nonumber 
\end{align} 
By studying the last equation, it is possible to see that such a solution exists for 
$\beta>1$, and indeed it is the dominant solution for all $\beta>1$ and $\lambda<\lambda_c(\beta)$,
for some $\lambda_c(\beta)$. We leave this calculation as an exercise for the reader.
\item \emph{Recovery/ Ferromagnetic phase.} For $\lambda$ sufficiently large, we have a 
fixed point $b_{\sR} >0$, $q_{\sR} >0$.
 In this case, $\lim_{n \to \infty} n^{-1}
  \E[ \langle \hbx_\beta , \bx_0 \rangle ] = b_{\sR}>0$ and thus the estimator 
  is positively correlated with the signal. 
  The cavity fields have asymptotic distribution $h_i\sim\normal(\lambda b_{\sR}x_{0,i},q_{\sR})$.
\end{enumerate}
\begin{remark}\label{rmk:SymmetryZ2}
The reader has probably noticed that, strictly speaking, $\hbx_{\beta}(\bY) = \bfzero$
identically because the Gibbs measure $\mu_{\beta}$ of Eq.~\eqref{eq:gibbs} is symmetric under
$\bsigma\to -\bsigma$ for $h=0$. We already encountered this problem in tensor
PCA, see Section \ref{sec:TensorEstimator}, and the same considerations apply 
to the present case.

Summarizing, we can redefine $\hbx_{\beta}(\bY)$ by using the principal eigenvector 
$\bv_1(\hbM)$ and principal eigenvalue $\lambda_1(\hbM)$ of the matrix
of $\hbM_{\beta}\equiv \sum_{\bsigma}\bsigma\bsigma^{\sT}\mu_{\beta}(\bsigma)$,
as $\hbx^{+}_{\beta}(\bY) = \sqrt{\lambda_1(\hbM)}\, \bv_1(\hbM)$. The estimation error
of $\hbx^{+}_{\beta}(\bY)$ (for losses that are invariant under $\hbx^{+}_{\beta}\mapsto -\hbx^{+}_{\beta}$)
is expected to be asymptotically the same as for $\hbx_{\beta,h}(\bY)$
(where the measure $\mu_{\beta}$ is perturbed with $h>0$), in the limit $h\to 0$
after $n\to\infty$.
\end{remark}

\subsection{Bayes optimal and maximum likelihood estimation}

\begin{figure}
\begin{center}
\includegraphics[scale=0.6]{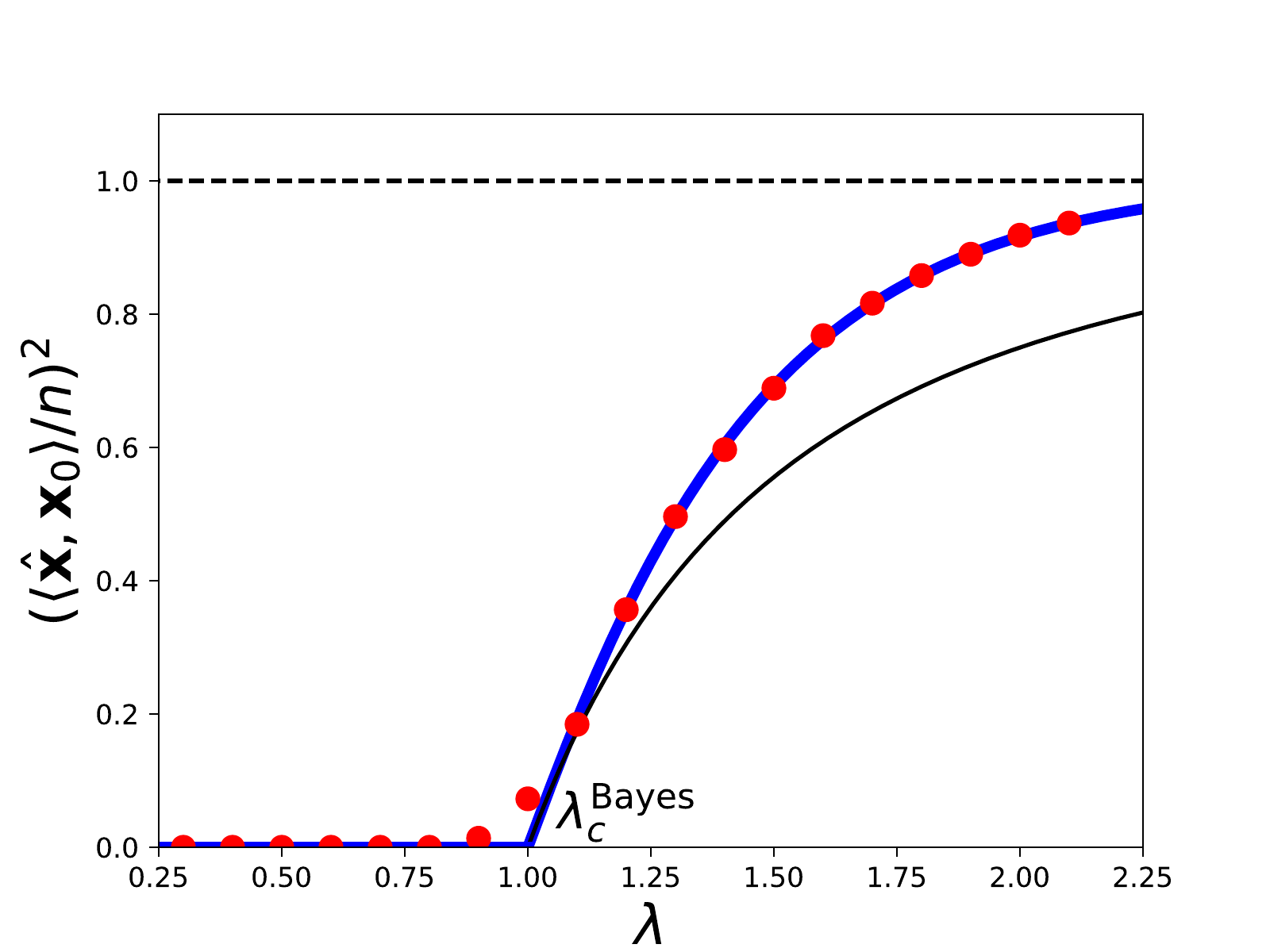}
\end{center}
\caption{Reconstruction accuracy in the $\integers_2$ synchronization model,
as a function of the signal-to-noise ratio $\lambda$. Blue line: asymptotic accuracy of
the Bayes-optimal estimator as predicted by the replica symmetric formula.}
\label{fig:rsb_saddle1}
\end{figure}

For $\beta = \lambda$, the Gibbs measure coincides with 
 the Bayes posterior. Using properties of
 posterior expectation, it is easy to see that $b_* = q_*$, where
  $b_*$ is determined as the solution of the fixed point equation 
\begin{align}
b_* = \E[ \tanh (\lambda^2 b_* + \sqrt{ \lambda^2 b_*} G)]. \label{eq:FP_Bayes-Z2}
\end{align}
The free energy density on this line is $\Psi_{\sBayes} (b_*)$ where 
\begin{align}
\Psi_{\sBayes} (b) \equiv \frac{\lambda^2}{4} (1- 2b - b^2) + \E[\log 2 \cosh ( \lambda^2 b + \sqrt{\lambda^2b} g)]. \nonumber 
\end{align}
The asymptotic free energy in this case has been derived rigorously in
 \cite{korada2009exact} and the corresponding estimation error was
 characterized in  \cite{deshpande2016asymptotic}. 
 
 By studying Eq.~\eqref{eq:FP_Bayes-Z2} the reader can show that 
 a solution $b_*>0$ exists (and dominates the free energy) for $\lambda>1$, while for $\lambda\le 
 1$, the only solution to this equation is $b_*=0$. 
 Since we are studying the Bayes-optimal estimator, the phase transition
 at $\lambda_{c}^{\sBayes} =1$ has the interpretation of a fundamental statistical barrier 
 for any estimator:
 \begin{itemize}
 \item For $\lambda\le \lambda_{c}^{\sBayes}$ no estimator can achieve a positive correlation with the
 signal, i.e. $\lim_{n \to \infty} n^{-1}
  \E[ \langle \hbx , \bx_0 \rangle ] = 0$ for any estimator $\hbx$.
 \item For $\lambda>\lambda_{c}^{\sBayes}$ the optimal estimator achieves a positive correlation with the
 signal, i.e. $\lim_{n \to \infty} n^{-1}
  \E[ \langle \hbx_{\sBayes} , \bx_0 \rangle ] > 0$
  (with $\hbx_{\sBayes}=\hbx_{\beta=\lambda}$ the Bayes estimator).
  \end{itemize}
  A phenomenon of this type is sometimes referred to as a `weak-recovery threshold.'
  We note in passing that ---in this specific model--- the same threshold is achieved by 
  a simpler principal component analysis (PCA) estimator that computes the leading eigenvector of 
  $\bY$. However, the PCA estimator does not achieve the optimal correlation above the threshold 
  $\lambda_{c}^{\sBayes}$.

 On the Bayes line, the vector of cavity fields is asymptotically
 distributed as  $\bh = \lambda b_* \bx_0 + \sqrt{b_*}\, \bg$
 for $\bg\sim \normal(0,\id_n)$. Letting $\by \equiv \bh/(\lambda b_*)$
 and $\tau \equiv (\lambda^2 b_*)^{-1/2}$,   the marginals of the posterior read
 \begin{align}
 \mu_{\sBayes,i}(\sigma_i) \approx \frac{e^{\lambda h_i\sigma_i}}{2\cosh(\lambda h_i)}
 = \frac{1}{z(y_i)}\, \exp\Big\{-\frac{1}{2\tau^2}(y_i-\sigma_i)^2\Big\}\, .
 \end{align}
 In other words, the Bayes posterior is asymptotically equivalent (in the sense of
 finite-dimensional marginals) to that of a much simpler model in 
  which we want to estimate a vector $\boldsymbol{x}_0 \in \{\pm 1\}^n$, 
 observed under additive Gaussian noise, $\by = \bx_0 + \tau \bg$,  
 $\bg \sim \normal(0,\id_n)$.

How does the maximum likelihood estimator (MLE) $\hbx_{\sML}=\hbx_{\beta=\infty}$ 
(which solves Eq.~\eqref{eq:maxlikelihood})
compare to Bayes estimation $\hbx_{\sBayes}=\hbx_{\beta=\lambda}$? It is easy to show 
that, for specific metrics, maximum likelihood is strictly sub-optimal
to the Bayes estimator. The simplest case is the one of (normalized)
matrix mean square error. Letting $\hbM_{\sBayes}(\bY) \equiv\E[\bx_0\bx_0^{\sT}|\bY]=
\hbM_{\beta=\lambda}(\bY)$, and $\hbM_{\sML}(\bY) \equiv \hbx_{\sML}\bx_{\sML}^{\sT}=
\hbM_{\beta=\infty}(\bY)$ (cf. Remark \ref{rmk:SymmetryZ2} for the definition of
$\hbM_{\beta}(\bY)$), we have:
\begin{align*}
\MSE_n(\hbM_{\sBayes})&\equiv\frac{1}{n^2}\E\big\{
\big\|\bx_0\bx_0^{\sT}-\E[\bx_0\bx_0^{\sT}|\bY]\big\|_F^2\big\}\\
& =\frac{1}{n^2}\E\big\{\big\|\bx_0\bx_0^{\sT}-\hbM_{\sML}(\bY)\big\|_F^2\big\}-\frac{1}{n^2}
\E\big\{\big\|\hbM_{\sML}(\bY)-\hbM_{\sBayes}(\bY)\big\|_F^2\big\}\\
& \le \MSE_n(\hbM_{\sML})-\frac{1}{n^2}\Big(
\E\big\{\big\|\hbM_{\sML}(\bY)\big\|_F^2\big\}^{1/2}-
\E\big\{\big\|\hbM_{\sBayes}(\bY)\big\|_F^2\big\}^{1/2}\Big)^2\\
& \le \MSE_n(\hbM_{\sML})-(1-q_{*,M})^2+o_n(1)\, .
\end{align*}
In the last step we used the fact that 
$n^{-2}\E\big\{\big\|\hbM_{\sBayes}(\bY)\big\|_F^2\big\} =q_{*,M}^2+o_n(1)$.
The replica method predicts $q_{*,M}=q_*<1$, but for obtaining a gap between maximum likelihood
and Bayes, it is sufficient to show $q_{*,M}<1$, which relatively easy to argue.

A much more subtle question is whether the MLE achieves the optimal 
weak recovery threshold.
We denote the MLE threshold as
\begin{align}
\lambda_{c}^{\sML} \equiv \inf\{ \lambda >0 : \lim_{n \to \infty} 
\frac{1}{n}\E[ |\<\hbx_{\sML}(\bY),\bx_0\>| ] >0 \}.
\end{align}
Interestingly, the replica symmetric calculation predicts $\lambda_c^{\sML} >1$.
On the other hand, taking into account replica symmetry breaking yields 
$\lambda_c^{\sML} =1$. A complete proof that $\lambda_c^{\sML} =1$ is at present an
open problem. (See \cite{javanmard2016phase} for further discussion.)

\section{The replica symmetric asymptotics: cavity approach}
\label{sec:replica_symmetric_cavity} 

In this section, we re-derive the results of the previous section
using a different approach: the cavity method. The cavity method offers an alternative
route to derive the results of the replica method. While still non-rigorous---
 it does not make use of formal tricks (number of replicas going to zero) 
 and its assumptions are more explicit. 
As for the replica method, it provides a hierarchy of approximations, that are
referred to as replica symmetric (RS), one-step replica-symmetry breaking (1RSB), and so on, 
matching the analogous hierarchy in the replica method. This terminology 
is sometimes misleading for a newcomer, since replicas do not really appear in 
the cavity method, and assumptions are in terms of probabilistic dependency structures,
rather than RSB patterns.

We will emphasize the key assumptions as items {\sf C1}, {\sf C2},
and so on.

 We want to compute the free energy density, 
 $ \phi \equiv \lim_{n \to \infty} \frac{1}{n}\E \log Z_n(\beta, \lambda)$,
 where we recall  the partition function  $Z_n(\beta, \lambda)$ from Eq.~\eqref{eq:partition}. 
The cavity method starts with the following simple observation 
\begin{align}
\frac{1}{n} \E\log Z_{n}(\beta, \lambda) = \frac{1}{n} \sum_{i=0}^{n-1} A_i , \quad \quad
A_i = \E \Big[\log \frac{Z_{i+1}(\beta, \lambda)}{Z_i(\beta, \lambda)} \Big]. \nonumber 
\end{align}
Now, elementary considerations about C\'{e}saro averages 
imply that $\phi = \lim_{n \to \infty} A_n$, provided the limit exists. 
The cavity method seeks to  compute this limit. Physically, $A_n$
is the expected change in free energy of the system when a single spin is added
to a model with $n$ spins. 
The basic idea is to write this change in terms of the Gibbs measure of the $n$-spins model.
It is therefore crucial to characterize the marginals of this measure.

\subsection{The distribution of cavity fields}
\label{sec:RSCavityFields}

Recall the Gibbs measure of the $n$-spins system, cf. Eq.~\eqref{eq:gibbs},
 which we rewrite here in the case $h=0$, for the reader's convenience
\begin{align}
\mu(\bsigma) = \frac{1}{Z_n} \, \exp{\Big\{ \beta \sum_{i<j\le n}Y_{ij}\sigma_i \sigma_j \Big\}}\, .
\end{align}
 For any $A \subset [n]$, we define $\bsigma_A = \{ \sigma_i : i \in A\}$. For any two subsets
  $A, B$ of $[n]$ with $A \cap B = \emptyset$, we define 
  $\nu_{A \to B}(\bsigma_A)$ to be the marginal of the $\bsigma_A$ spins in the modified 
  model where the $\bsigma_B$ variables have been ``removed". Formally:
\begin{align}
\nu_{A\to B}(\bsigma_{A})
 \propto \sum_{\{\sigma_i: i\in B^c\setminus A\}}
 \exp{\Big\{\beta \sum_{i<j: i,j \in B^c} Y_{ij} \sigma_i \sigma_j \Big\}}. \nonumber 
\end{align} 
In physics language, we are creating a `cavity' by removing the spins 
$\bsigma_B = \{ \sigma_i : i \in B\}$ from the system: this image is at the origin of 
the name of the method. In the special case in which $B=\{j\}$ is a single vertex, we will write
$\nu_{A\to j} = \nu_{A\to\{j\}}$, and analogously if $A= \{i\}$.
We also denote by $\mu_B$ the marginal distribution of $\bsigma_B$
when $\bsigma\sim \mu$.
We will refer to $\nu_{A\to B}$ as a `cavity marginal' and to $\mu_B$ as a `marginal'.

We can express ordinary marginals in terms of cavity marginals. 
For instance, considering the case $B=\{i\}$, 
 we have 
\begin{align}
\mu_i(\sigma_i) &=\frac{1}{z'_i} \sum_{\bsigma_{[n]\setminus i} }
\exp{\Big\{\beta \sum_{k<l : k,l \in[n]\setminus i} Y_{kl} \sigma_k \sigma_l + \beta \sigma_i \sum_{j \in [n]\setminus i} Y_{ij} \sigma_j \Big\}} \nonumber \\
&= \frac{1}{z''_i}\sum_{\bsigma_{[n]\backslash i} } 
\nu_{[n]\setminus i \to i}(\bsigma_{[n]\setminus i}) 
\exp{\Big\{\beta \sigma_i \sum_{j \in [n]\setminus i }Y_{ij} \sigma_j \Big\}}. \label{eq:FirstCavity}
\end{align}
(Here and in the following lines, we will introduce normalization constants denoted by $z$,
$z'$, etc, which are implicitly defined by the normalization condition of quantities on the left-hand side.)

We note that the pairwise interactions among the spin variables $\sigma_k$, $\sigma_{l}$
under  $\nu_{[n]\backslash i \to i}$ are of order $Y_{kl}=O(1/\sqrt{n})$,
 and thus it is reasonable to assume that the entries are approximately independent. 
 \begin{itemize}
 \item[{\sf C1}] We assume that a negligible error is made in
 replacing $\nu_{[n]\setminus i \to i}(\bsigma_{[n]\setminus i})$ by the product of marginals
  $\prod_{l \in [n]\backslash i } \nu_{l \to i}(\sigma_l)$ in Eq.~\eqref{eq:FirstCavity}.
 \end{itemize}
 The cavity method assumes this explicit independence and thus posits that 
\begin{align}
\mu_i(\sigma_i) &= \frac{1}{z_i}\sum_{\bsigma_{[n]\setminus i} } \prod_{l \in [n]\backslash i } \nu_{l \to i}(\sigma_l)  \exp{\Big[\beta \sigma_i \sum_{l \in [n]\backslash i }Y_{il} \sigma_l \Big]}
+o_n(1). \label{eq:MuNu}
\end{align}

Now, we parametrize 
\begin{align}
\mu_{i}(\sigma_i) = \frac{e^{\beta h_{i} \sigma_i }}{2 \cosh (\beta h_{i})} \,\;\;\;
\nu_{i \to j}(\sigma_i) = \frac{e^{\beta h_{i\to j} \sigma_i }}{2 \cosh (\beta h_{i \to j})},
 \label{eq:cavityfields}
\end{align}
where $h_{i \to j}$ is usually referred to as the cavity field. 
Using Eq.~\eqref{eq:MuNu}, we get
\begin{align}
h_{i} &= \frac{1}{2\beta}  \log \Big[ 
\frac{ \sum_{\sigma_l : l \in [n]\setminus i} \exp[\beta Y_{il} \sigma_l + \beta h_{l \to i } \sigma_l] }{ \sum_{\sigma_l : l \in[n]\setminus i } \exp[- \beta Y_{il} \sigma_l + \beta h_{l \to i} \sigma_l ]} \Big]+o_n(1)
\nonumber\\
 & = \frac{1}{2\beta} \sum_{l \in [n]\setminus i} \log \Big[ \frac{\cosh [\beta Y_{il} + \beta h_{l \to i} ]}{ \cosh [ - \beta Y_{il} + \beta h_{l\to i} ]} \Big]+o_n(1)
 \label{eq:CavityRec1}\\
 & =\sum_{l \in [n]\setminus i} u (Y_{il}, h_{l\to i}) +o_n(1), \nonumber 
\end{align}
where we defined
\begin{align}
u(Y,h) = \frac{1}{\beta} {\rm atanh}\big[\tanh(\beta Y) \tanh(\beta h)\big]\, .
\end{align}
(We used the trigonometric identities $\cosh(x+y) = \cosh(x) \cosh(y) + \sinh(x) \sinh(y)$ and
$\rm{atanh}(x) = \frac{1}{2} \log \frac{1+x}{1-x}$. Further, we set $Y_{ii} = 0$ without loss of generality. )

We can further simplify the above formula by using the fact that
 $Y_{ij} = O(1/\sqrt{n})$. By Taylor expansion 
 $u(Y,h) = Y\tanh(\beta h) +O(Y^3)$ and therefore
\begin{align}
h_{i} = \sum_{l \in [n]\setminus i} Y_{il} \tanh(\beta h_{l \to i}) + o_n(1). \label{eq:selfconsistent}
\end{align}

Note that the cavity fields $(h_{l \to i})_{l\in [n]\setminus i}$ are random variables 
independent of  $(Y_{il})_{l\in [n]\setminus i}$ (since they are functions of
$(Y_{kl})_{k,l\in [n]\setminus i}$). Further recalling the definition of $\bY$,
we get
\begin{align}
h_{i} = \frac{\lambda}{n}x_{0,i}\sum_{l \in [n]\setminus i} x_{0,l} \tanh(\beta h_{l \to i})
+\sum_{l \in [n]\setminus i} W_{il} \tanh(\beta h_{l \to i})
 + o_n(1). 
\end{align}
We will consider, without loss of generality,  $\bx_0$ to be fixed (the reader can consider $\bx_0=\bfone$).
Let $\rP^{(n)}_{h_i|\bY_{-i}}$ be the conditional distribution of $h_i$
given $(Y_{kl})_{k,l\in [n]\setminus i}$. Since the $W_{il}$ are independent Gaussian random 
variables\footnote{This conclusion does not actually rely on Gaussianity, and follows more generally
using the central limit theorem.},
we get
\begin{align}
&{\sf dist}\big(\rP^{(n)}_{h_i|\bY_{-i}} ,\normal(\lambda b_i x_{0,i},q_i)\big) =o_n(1)\, ,\\
& b_i\equiv \frac{1}{n}\sum_{l \in [n]\setminus i} x_{0,l} \tanh(\beta h_{l \to i})\, ,\\
& q_i \equiv \frac{1}{n}\sum_{l \in [n]\setminus i} \tanh^2(\beta h_{l \to i})
\end{align}
In other words $\rP^{(n)}_{h_i|\bY_{-i}}$ is approximately Gaussian with mean $\lambda b_i x_{0,i}$ 
and variance  $q_i$. (The mathematically minded
reader can think of ${\rm dist}$ as a distance metrizing weak topology.)

 Note that both $b_i$ and $q_i$ are averages over 
$n$ terms. It is reasonable to think that these terms concentrate, which is our next assumption.
\begin{itemize}
\item[{\sf C2}] We assume that $b_i$ and $q_i$ concentrate around their expectations
$\ob^{(n)}$ and $\oq^{(n)}$:
\begin{align}
\ob^{(n)} = \E[ x_0 \tanh (\beta h_{\to })]\, ,\;\;\;\;\;\;
\oq^{(n)} = \E[ \tanh^2 ( \beta h_{ \to })] . \label{eq:RSCavityQB}
\end{align}
Here $(x_0,h_{\to })$ is a random variable distributed as any one of the $(x_{0,l},h_{l\to i})$.
\end{itemize}
Note that the latter assumption is implied by independence of $(h_{l\to i})_{l\in [n]\setminus i}$,
but much weaker than it. Under this assumption, the conditional 
distribution $\rP^{(n)}_{h_i|\bY_{-i}}$ is essentially independent of
$\bY_{-i}$, and hence we conclude that the unconditional distribution 
$\rP^{(n)}_{h_i}$ is also approximately normal:
\begin{align}
&{\sf dist}\big(\rP^{(n)}_{h_i},\normal(\lambda \ob^{(n)} x_{0,i},\oq^{(n)})\big) =o_n(1)\, .
\end{align}

At this point we need to `close' the equations  \eqref{eq:RSCavityQB} for $\ob^{(n)}$, $\oq^{(n)}$ which 
parametrize the distribution of the effective fields. To this hand, we begin by noting that
the distribution of the cavity fields $h_{ \to }$ in the system with $n$ spins is essentially the same 
as the distribution of the effective fields in a system with $n-1$ spins.
Both quantities are parametrization of single spin marginals as per Eq.~\eqref{eq:cavityfields},
in a system with $n-1$ spins. There is however one difference: for the system with $n-1$ 
spins, the mean and variance of the $Y_{ij}$ are scaled by a factor $1/(n-1)$
instead of $1/n$. This amounts to a difference of order $1/n$ in $\lambda$ and $\beta$.
It is reasonable to believe 
\begin{itemize}
\item[{\sf C3}] The distribution of the effective fields $\rP^{(n)}_{h}$ has
a limit as $n\to\infty$.
Further, the limit distribution depends continuously on $\beta,\lambda$.
\end{itemize}

Together, these assumptions imply that the distribution of $h_{\to}$ 
on the right-hand side of Eq.~\eqref{eq:RSCavityQB} is approximately 
$\normal(\lambda \ob^{(n)} x_{0},\oq^{(n)})$. Therefore $\ob^{(n)}$, $\oq^{(n)}$
satisfy 
\begin{align*}
\ob^{(n)}= \E[ x_0\tanh (\beta ( \lambda x_0 \ob^{(n)} + \sqrt{\oq^{(n)}} G ))] +o_n(1)\, ,\\
\oq^{(n)} = \E[\tanh^2( \beta ( \lambda x_0 \ob^{(n)} + \sqrt{\oq^{(n)}} G )) ]+o_n(1)\, .
\end{align*}
where $G \sim \normal(0,1)$. 
These equations imply that $\ob^{(n)}\to b_*$, $\oq^{(n)}\to q_*$, where 
$b_*,q_*$ solve the limit equations 
\begin{align}
b= \E[ \tanh (\beta ( \lambda b + \sqrt{q} G ))]\, , \quad \quad 
q= \E[\tanh^2( \beta ( \lambda b + \sqrt{q} G )) ]\, . \label{eq:fixedpt}
\end{align}
Note that this is exactly the same set of equations derived from the stationary
 point conditions of the replica symmetric free energy for this problem \eqref{eq:rs-fixed}!  
 The present derivation was longer, but the assumptions are transparent, 
 and the Gaussian distribution appears more naturally. 

\subsection{The free energy density}

Armed with the description of cavity fields derived in the previous section, we return to the computation of the 
free energy as outlined in the beginning of the section. Recall that we wish to compute
 $\phi = \lim_{n \to \infty} A_n$, where 
$A_{n} = \E \log [ Z_{n+1}(\beta,\lambda)/Z_n(\beta,\lambda)]$.
 We define $\tilde{\bW} = \sqrt{n} \bW$ and without loss of generality, set $\bx_{0i} =\bfone$. 
 Thus  we have, 
\begin{align}
Z_{n+1}(\beta ,\lambda) = \sum_{\bsigma\in \{ \pm 1 \}^{n+1}} 
\exp{\Big\{\frac{\beta \lambda}{n+1} \sum_{1\leq i < j \leq n+1} \sigma_i \sigma_j + \frac{\beta}{\sqrt{n+1}} \sum_{1\leq i < j \leq n+1} \tilde{W}_{ij} \sigma_i \sigma_j \Big\} }. \nonumber 
\end{align}
Setting $\beta_{n+1} = \beta \sqrt{\frac{n+1}{n}}$, $\lambda_{n+1} = \lambda \sqrt{\frac{n+1}{n}}$, direct computation yields 
\begin{align}
Z_{n+1}(\beta_{n+1}, \lambda_{n+1}) &= \sum_{\bsigma} \exp{\Big\{ \frac{\beta \lambda}{n} \sum_{1\leq i < j \leq n} \sigma_i \sigma_j + \frac{\beta}{\sqrt{n}} \sum_{1\leq i < j \leq n} \tilde{W}_{ij} \sigma_i \sigma_j + \beta \sigma_{n+1} \sum_{i=1}^{n}Y_{i (n+1)} \sigma_i  \Big\}}. \nonumber \\
&= Z_n(\beta, \lambda) \sum_{\bsigma} \nu_{[n] \to (n+1)}(\bsigma_{[n]}) \exp{\Big\{\beta \sigma_{n+1} \sum_{i=1}^{n} Y_{i (n+1)} \sigma_i \Big\}}. \nonumber 
\end{align}
We assume that the distribution $\nu_{[n]\to (n+1)}(\, \cdot\, )$ factorizes as in 
point {\sf C1} above to get 
\begin{align}
\frac{Z_{n+1}(\beta_{n+1}, \lambda_{n+1})}{Z_n(\beta, \lambda)} &= \sum_{\bsigma} \prod_{i=1}^{n} \nu_{i \to (n+1)} (\sigma_i) \, e^{\beta Y_{i (n+1)} \sigma_i \sigma_{n+1}} \cdot\big(1+o_n(1)\big) \nonumber \\
&= \left\{\sum_{\sigma _{n+1}} \prod_{i=1}^{n}  \frac{\cosh \beta(Y_{i(n+1)} \sigma_{n+1} + h_{i \to (n+1)}) }{ \cosh (\beta h_{i \to (n+1)} )} \right\} \cdot  \big(1+o_n(1)\big) ,  \nonumber 
\end{align}
where we introduce the cavity fields as in \eqref{eq:cavityfields}. Next, using the trigonometric identities $\cosh(x+y) = \cosh (x) \cosh(y) + \sinh(x) \sinh(y)$ and $\sinh(\sigma x) = \sigma \sinh(x)$ for $\sigma \in \{\pm 1 \}$, we have, 
\begin{align}
&\frac{Z_{n+1}(\beta_{n+1}, \lambda_{n+1})}{Z_n(\beta, \lambda)} = \Big[\prod_{i=1}^{n} \cosh(\beta Y_{i (n+1)}) \Big] \sum_{\sigma_{n+1}} \prod_{i=1}^{n} ( 1 + \sigma_{n+1} \tanh(\beta Y_{i (n+1)} ) \tanh(\beta h_{i \to (n+1)})) \nonumber \\
&= \Big[\prod_{i=1}^{n} \cosh(\beta Y_{i (n+1)}) \Big]  \sum_{\sigma_{n+1}} \exp{ \Big[\sum_{i=1}^{n} \log ( 1 + \sigma_{n+1} \tanh(\beta Y_{i(n+1)}) \tanh(\beta h_{i \to (n+1)})  \Big]}. \label{eq:intermediate}
\end{align}
First, we note that each $Y_{i(n+1)}$ is $O(\frac{1}{\sqrt{n}})$ and therefore, using Taylor expansion, we expect 
\begin{align}
\prod_{i=1}^{n} \cosh(\beta Y_{i(n+1)}) = \prod_{i=1}^{n} \exp{\Big[ \frac{\beta^2 Y_{i(n+1)}^2}{2} + O\Big(\frac{1}{n^{3/2}} \Big) \Big]} = \exp{\Big[ \frac{\beta^2}{2} + o(1) \Big]}. \nonumber 
\end{align}
Further, using Taylor expansion, we expect, 
\begin{align}
&\sum_{i=1}^{n}\log (1 + \sigma_{n+1} \tanh(\beta Y_{i(n+1)}) \tanh(\beta h_{i \to (n+1)}) )= \nonumber \\
&\sigma_{n+1} \sum_{i=1}^{n} \tanh(\beta Y_{i(n+1)}) \tanh(\beta h_{i\to (n+1)}) - \frac{1}{2} \sum_{i=1}^{n} \tanh^2(\beta Y_{i(n+1)}) \tanh^2 (\beta h_{i \to (n+1)}) + o_n(1). \nonumber \\
&= \beta \sigma_{n+1} \sum_{i=1}^{n} Y_{i(n+1)} \tanh(\beta h_{i \to (n+1)}) - \frac{\beta^2}{2} \E[\tanh^2(\beta h_{i \to (n+1)})]+o_n(1), \nonumber 
\end{align} 
where we use the approximation that $\tanh(x) = x+O(|x|^3)$ as $x \to 0$ and $\E[Y_{i(n+1)}^2] = 1/n+o(1/n)$. 
 Recalling the equation for the effective field \eqref{eq:selfconsistent} and the 
 fixed point equations for $q,b$ \eqref{eq:fixedpt}, we obtain 
\begin{align}
\sum_{i=1}^{n} \log (1 + \sigma_{n+1} \tanh(\beta Y_{i(n+1)}) = \beta h_{n+1}\sigma_{n+1}  - 
\frac{\beta^2}{2} q+o_n(1)\, . \nonumber 
\end{align}
Plugging these back into \eqref{eq:intermediate}, we obtain, 
\begin{align}
\frac{Z_{n+1}(\beta_{n+1}, \lambda_{n+1})}{Z_n(\beta, \lambda)} &= \exp{\Big[\frac{\beta^2}{2} (1-q) \Big]} \sum_{\sigma_{n+1}} \exp{[\beta  h_{n+1}\sigma_{n+1} ]}
 \cdot\big(1+o_n(1)\big)\\
&= \exp{\Big[\frac{\beta^2}{2} (1-q) \Big]} 2\cosh(\beta h_{n+1 \to i})\cdot\big(1+o_n(1)\big)
\, . \nonumber 
\end{align}
Taking expectations and using the fact that the law of $h_{n+1}$ 
converges to $\normal(\lambda b_*, q_*)$, we obtain, 
\begin{align}
\E \log \frac{Z_{n+1}(\beta_{n+1}, \lambda_{n+1})}{Z_n(\beta, \lambda)} = 
\frac{\beta^2}{2}(1-q_*) + \E[\log 2 \cosh(\beta (\lambda b_* + \sqrt{q_*} G))] + o_n(1), 
\label{eq:FirstPieceCavity}
\end{align}
where $G\sim \normal(0,1)$. 

Finally, we need to evaluate $\log \big[Z_n(\beta_n, \lambda_n)/Z_n(\beta, \lambda) \big]$. 
To this end, we note that (for $\mu = \mu_{n,\beta,\lambda}$)
\begin{align}
\frac{Z_n(\beta_n, \lambda_n)}{Z_n(\beta, \lambda)} &= \sum_{\bsigma} \mu(\bsigma) 
\exp{\left\{ \frac{\beta_n \lambda_n - \beta \lambda}{n} \sum_{i <j} \sigma_i \sigma_j  + \frac{\beta_n - \beta}{\sqrt{n}} \sum_{i<j} \tilde{W}_{ij} \sigma_i \sigma_j \right\}}\nonumber\\
 &= \sum_{\bsigma} \mu(\bsigma) \Big[ 1 + \frac{\beta \lambda}{n^2} \sum_{i <j} \sigma_i \sigma_j + \frac{\beta}{2n^{3/2}} \sum_{i < j } \tilde{W}_{ij} \sigma_i \sigma_j + o_n(1) \Big]. \nonumber 
\end{align}
Thus we have, 
\begin{align}
& \E\log \frac{Z_n(\beta_n,\lambda_n)}{Z_n( \beta, \lambda)}  = T_1 + T_2 + o_n(1), \nonumber \\
&T_1 = \frac{\beta\lambda}{n^2} \sum_{i<j} \E \Big[\sum_{\bsigma} \mu(\bsigma) \sigma_i \sigma_j \Big], \quad\quad
T_2 = \frac{\beta}{2n^{3/2}} \sum_{i <j} \E[\tilde{W}_{ij} \sum_{\bsigma} \mu(\bsigma) \sigma_i \sigma_j]. \nonumber 
\end{align}
We analyze each term in turn. In order to estimate $T_1$, we assume, as 
in point {\sf C1} above, that $\sigma_i$, $\sigma_j$ are approximately independent under $\mu$, 
and hence
\begin{align*}
\sum_{\bsigma} \mu(\bsigma) \sigma_i \sigma_j &=\Big(\sum_{\sigma_i}\mu_i(\sigma_i)\,\sigma_i\Big)
\Big(\sum_{\sigma_j}\mu_j(\sigma_j)\,\sigma_j\Big)+o_n(1)\\
& = \tanh(\beta h_i)\tanh(\beta h_j)+o_n(1)\, .
\end{align*}
 Therefore  
\begin{align*}
T_1 &= \frac{\beta\lambda}{2} \cdot 
\E\left\{\Big(\frac{1}{n}\sum_{i=1}^n \tanh(\beta h_i)\Big)\right\} +o_n(1)\\
& \stackrel{(a)}{=} \frac{\beta\lambda}{2} \cdot 
\E\big\{\tanh(\beta h_1)\big\}^2 +o_n(1)\\
& \stackrel{(b)}{=} \frac{\beta\lambda}{2} \cdot 
\E\big\{\tanh(\beta (\lambda b_*+\sqrt{q_*}G))\big\}^2+o_n(1) \\
& \stackrel{(c)}{=}\frac{\beta\lambda}{2} \cdot b_*^2+o_n(1)\, .
\end{align*}
In step $(a)$ we used the assumption that $n^{-1}\sum_{i=1}^n \tanh(\beta h_i)$
concentrates around its expectation, as per point {\sf C2}; in $(b)$ we used the asymptotic
characterization of the distribution of effective fields derived above;
in $(c)$ we used the fixed point equations \eqref{eq:fixedpt}.

Next, using Gaussian integration by parts, we obtain
\begin{align}
T_2 = \frac{\beta}{2n^{3/2}} \sum_{ i <j } 
\E\Big[ \frac{\partial \phantom{\tilde{W}_{ij}}}{\partial \tilde{W}_{ij}} \sum_{\bsigma} \mu(\bsigma) \sigma_i \sigma_j \Big]
= \frac{\beta^2}{2n^2} \sum_{i<j} \E \Big[ 1 - \Big( \sum_{\bsigma} \mu_{\beta}(\bsigma) \sigma_i \sigma_j \Big)^2 \Big]. \label{eq:T2}
\end{align}
Again, assuming as per point {\sf C1} (and as in the calculation of $T_1$) that 
$\sigma_i$, $\sigma_j$ are approximately independent under $\mu$,
and recalling the definition of effective fields, we get
\begin{align*}
T_2 &= \frac{\beta^2}{2n^2}  \sum_{i<j} \E \big[ 1 - \big( \tanh\beta h_i\tanh\beta h_j \big)^2 \big]+o_n(1)\\
&= \frac{\beta^2}{2} -   \frac{\beta^2}{2}\E\left\{\left(\frac{1}{n}
\sum_{i=1}^n (\tanh\beta h_i)^2\right)\right\}+o_n(1)\\
&\stackrel{(a)}{=} \frac{\beta^2}{2} -   \frac{\beta^2}{2} \E\{(\tanh\beta h_1)^2\}^2+o_n(1)\\
&\stackrel{(b)}{=}\frac{\beta^2}{2} \left\{1-  \E\{\tanh(\beta (\lambda b_*+\sqrt{q_*}G))^2\}^2+\right\}+o_n(1)\\
&\stackrel{(c)}{=} \frac{\beta^2}{2}\big(1-q_*^2\big)+o_n(1)\, .
\end{align*}
Here we proceeded analogously to the computation of $T_1$:
in step $(a)$ we used the assumption that $n^{-1}\sum_{i=1}^n \tanh^2(\beta h_i)$
concentrates around its expectation, as per point {\sf C2}; in $(b)$ we used the asymptotic
characterization of the distribution of effective fields derived above;
in $(c)$ we used the fixed point equations \eqref{eq:fixedpt}.

Finally, combining Eq.~\eqref{eq:FirstPieceCavity} with these results for $T_1$, $T_2$, we obtain
\begin{align*}
A_n &= \E\log \frac{Z_{n+1}(\beta,\lambda)}{Z_n(\beta,\lambda)}\\
& = \E\log \frac{Z_{n+1}(\beta_{n+1},\lambda_{n+1})}{Z_n(\beta,\lambda)}  - \E \log \frac{Z_{n+1}(\beta_{n+1},\lambda_{n+1})}{Z_{n+1}(\beta,\lambda)} \\
&= \frac{\beta^2}{4} (1-q_*)^2 - \frac{\beta\lambda}{2}b_*^2 + \E[ \log 2 \cosh (\beta (\lambda b_* + \sqrt{q_*}\, G) )]  + o_n(1)\\
& = \Psi_{\sRS}(b_*,q_*) + o_n(1)\,.
\end{align*}
where $\Psi_{\sRS}(b,q)$ is the replica symmetric  free energy functional 
which we already derived in Section \ref{sec:RS-Z2}.

Thus the cavity method successfully 
recovers the RS free energy density prediction $\phi=\Psi_{\sRS}(b_*,q_*)$,
without relying on replicas, and instead making specific assumptions 
{\sf C1}, {\sf C2},  {\sf C3} about the structure of the Gibbs measure. As we will see in Section
\ref{sec:CavityRSB},  these assumptions can be modified to obtain RSB predictions. 

\section{Algorithmic questions}
\label{sec:algorithmic} 

As we discussed above, one way to perform estimation is to compute the mean
of the measure $\mu$, i.e.
\begin{align}\label{eq:MeanEstimator}
\hbx_{\beta}(\bY) \equiv \sum_{\bsigma\in\{+1,1\}^n}
\mu_{n,\beta}(\bsigma) \,\bsigma\, .
\end{align}
We sweep under the carpet the technical nuisance that, in the present model, the mean on the right-hand side is zero
for symmetry reasons. We assume throughout that this symmetry is broken, e.g. by
additional side information or by a spectral initialization.

Writing Eq.~\eqref{eq:MeanEstimator} does not solve the problem because 
evaluating the right-hand side requires summing over exponentially many 
vectors $\bsigma$.
In this section we briefly discuss the question of computing
this average via efficient algorithms.

\subsection{From the cavity method to message passing algorithms}

Computing the effective fields $\{h_i\}_{i\le n}$ is of course equivalent to 
computing the marginals of the Gibbs measure $\mu$, and is therefore 
sufficient to compute the mean vector $\hbx_{\beta}(\bY)$. We saw in turn 
that the effective fields are closely related to the cavity fields $\{h_{i\to j}\}$,
and it is therefore natural to try to compute the latter.
Indeed $h_i=h'_{i\to 0}$ where $h'_{i\to j}$ are cavity fields in a fictitious model in 
which we added a variable $\sigma_0$.

By repeating the argument to derive Eq.~\eqref{eq:CavityRec1}, we obtain a set of
equations that ---under similar assumptions--- are approximately satisfied by the cavity 
fields, namely 
\begin{align}
h_{i\to j}& =\sum_{l \in [n]\setminus \{i,j\}} u (Y_{il}, h_{l\to i}) +o_n(1)\, ,\\
u(Y,h) &= \frac{1}{\beta} {\rm atanh}\big[\tanh(\beta Y) \tanh(\beta h)\big]\, .
\end{align}
Indeed note that Eq.~\eqref{eq:CavityRec1}, applied to the system from which
$\sigma_j$ has been removed, yields 
$h_{i\to j} =\sum_{l \in [n]\setminus \{i,j\}} u (Y_{il}, h_{l\to \{i,j\}}) +o_n(1)$,
and the last display follows from the approximation $h_{l\to \{i,j\}}\approx h_{i\to j}$.

It is useful to rewrite these equations for an Ising model on a general graph
$G=(V=[n],E)$ (i.e., for the Gibbs measure \eqref{eq:gibbs} with $Y_{ij}$ set to $0$ whenever 
$(i,j)\not\in E$). We have 
\begin{align}
x_{i\to j}& =\sum_{l: (i,l)\in E, l\neq j } u (Y_{il}, x_{l\to i})\, ,
\;\;\; \forall (i,j)\in E\, .\label{eq:BP-FP}
\end{align}
This provides a set of $2|E|$ equations for the $2|E|$ unknowns $(x_{i\to j}:\, (i,j)\in E)$.
Here we changed our notation from $h_{i\to j}$ to $x_{i\to j}$ to emphasize the
 fact that we regard these as equations in the unknowns $(x_{i\to j}:\, (i,j)\in E)$,
 whose solutions might (or might not) be related  to the actual cavity fields.
 
Two questions arise naturally: $(i)$~What algorithms can be used to 
solve the system of equations \eqref{eq:BP-FP} in polynomial time? $(ii)$~What is
the relation between the solutions (or approximate solutions) found by such algorithms and the actual
cavity fields $(h_{i\to j}:\, (i,j)\in E)$?

 Given the form in which the equations \eqref{eq:BP-FP} are written there is 
 natural iterative algorithm that we might hope to use:
\begin{align}
x^{(t+1)}_{i\to j}& =\sum_{l: (i,l)\in E, l\neq j } u (Y_{il}, x^{(t)}_{l\to i})\, ,
\;\;\; \forall (i,j)\in E\, .\label{eq:BP-Ising}
\end{align}
This is known as belief propagation (BP) or the sum-product algorithm
\cite{gallager1962low,wainwright2008graphical,MezardMontanari}. 
Following \cite{RiU08}, we refer to it as a `message passing algorithm'. This name refers to 
the fact that we can think of each variable $x^{(t)}_{i\to j}$ as a message passed from vertex $i$ to vertex
$j$.

 Abstractly, we
can think of it as defining a map $\sF_{\bY}(\,\cdot\,):\reals^{2|E|}\to \reals^{2|E|}$,
for which the above iteration can be written as
\begin{align}
\bx^{(t+1)} = \sF_{\bY}(\bx^{(t)})\, ,\;\;\;\; \bx^{(t)}\equiv (x^{(t)}_{i\to j}:\, (i,j)\in E)\, ,
\end{align}
Existence of fixed points of this recursion (and hence existence of solutions of Eqs.~\eqref{eq:BP-FP})
follows from Brouwer's fixed point theorem (applied to the 
variables $u_{i\to j} = u (Y_{ij}, x^{(t)}_{i\to j})$). Convergence and
relation with the actual cavity fields are far more subtle questions.

There is one case in which these questions have a simple answer, as stated 
below.
\begin{proposition}
If $G = (V,E)$ is a finite tree, then the belief propagation iteration 
\eqref{eq:BP-Ising} converges to a unique fixed point $\bx^* = 
(x^*_{i\to j}: (i,j)\in E)$ in at most ${\rm diam}(G)$ iterations.

Further, at this fixed point, the variables $x^*_{i\to j}$ coincide with
the cavity field $x^*_{i\to j}=h_{i\to j}$, and the effective fields are given by
$h_i = \sum_{l: (i,l)\in E } u (Y_{il}, x^*_{l\to i})$.
\end{proposition}
We leave the proof as an exercise for the reader \cite{MezardMontanari}. 
If the graph $G$ is not a tree, the last proposition can still be of use.
One case of interest is the one of locally tree-like graphs. We refer 
to the bibliography section for pointers to this literature.

\subsection{Approximate Message Passing}

We are interested in a setting that is seemingly opposite
to the tree case described above. Namely, in our setting $G=K_n$ is the complete graph over $n$
vertices. We will see that nevertheless, the iteration \eqref{eq:BP-Ising}
(or its close relatives that we will develop here) produce correct marginals.

We begin by noting that $Y_{ij} = O(1/\sqrt{n})$ and therefore we can linearize
the function $u(Y,x)$ in $Y$, as we did in the previous sections.
We thus replace the BP iteration by the following simpler
form:
\begin{align*}
x^{(t+1)}_{i\to j}& =\sum_{l: (i,l)\in E, l\neq j } Y_{il}\tanh(\beta x^{(t)}_{l\to i})\, .
\end{align*}
Notice that the connection with the Gibbs measure \eqref{eq:gibbs} is encoded in the
nonlinear function $\tanh(\beta x)$ appearing on the right-hand side. 
The following arguments do not depend on this specific function and 
they are more transparent if we replace $\tanh(\beta x)$ by a generic function $f_t:\reals\to \reals$
(possibly dependent on the iteration number):
\begin{align*}
x^{(t+1)}_{i\to j}& =\sum_{l: (i,l)\in E, l\neq j } Y_{il}f_t( x^{(t)}_{l\to i})\, .
\end{align*}
If $G$ is the complete graph, we can rewrite this iteration as
(setting for simplicity $Y_{ii}=0$ for all $i\in [n]$)
\begin{align*}
x^{(t+1)}_{i\to j}& =\sum_{l=1}^n Y_{il}f_t( x^{(t)}_{l\to i}) - Y_{ij}f_t( x^{(t)}_{j\to i}) \, .
\end{align*}
In other words, we can write $x_{i\to j}^{(t)} = x_i^{(t)} + \delta_{i \to j}^{(t)} $, 
where $\delta^{(t)}_{i\to j} = O(1/\sqrt{n})$ and
\begin{align*}
x^{(t+1)}_{i}& =\sum_{l=1}^n Y_{il}f_t( x^{(t)}_{l\to i})\, ,\\
\delta_{i \to j}^{(t+1)} & = - Y_{ij}f_t( x^{(t)}_{j\to i}) \, .
\end{align*}
Substituting $x_{i\to j}^{(t)} = x_i^{(t)} + \delta_{i \to j}^{(t)} $ on the right-hand
side and Taylor expanding, we get
\begin{align*}
x^{(t+1)}_{i}& =\sum_{l=1}^n Y_{il}f_t( x^{(t)}_{l})+\sum_{l=1}^n Y_{il}f'_t( x^{(t)}_{l})
\delta^{(t)}_{l\to i}+O(1/\sqrt{n})\, ,\\
\delta_{i \to j}^{(t+1)} & = - Y_{ij}f_t( x^{(t)}_{j})+O(1/\sqrt{n}) \, .
\end{align*}
Finally, using the second equation in the first one, we get
\begin{align*}
x^{(t+1)}_{i}& =\sum_{l=1}^n Y_{il}f_t( x^{(t)}_{l})-
\left(\sum_{l=1}^n Y^2_{il}f'_t( x^{(t)}_{l})\right) f_{t-1}( x^{(t-1)}_{i})+
O(1/\sqrt{n})\\
& =\sum_{l=1}^n Y_{il} f_t ( x^{(t)}_{l})- \sd_t f_{t-1}( x^{(t-1)}_{i})+
O(1/\sqrt{n})\, ,\\
\sd_t &\equiv \frac{1}{n}\sum_{l=1}^n f'_t( x^{(t)}_{l})\, .
\end{align*}
In the last step we used a central limit theorem heuristic to replace the
sum $\sum_{l=1}^n Y^2_{il}f'_t( x^{(t)}_{l})$ by $\sd_t$.

If we drop the $O(1/\sqrt{n})$ terms in the above iteration, 
and rewrite it in matrix notation, we have derived the following iteration
\begin{align}
\bx^{(t+1)}& =\bY f_t( \bx^{(t)}_{l})- \sd_t f_{t-1}( \bx^{(t-1)})\, , \label{eq:AMPGeneral}\\
 &\sd_t = \frac{1}{n}\sum_{i=1}^n f'_t( x^{(t)}_{i})\, .
\end{align}
Here it is understood that $f_t:\reals\to \reals$ acts entrywise on vectors,
namely $f_t(\bx) \equiv (f_t(x_1),\dots,f_t(x_n))$. We also note that the diagonal entries of $\bY$
have a negligible impact of this iteration (indeed $Y_{ii}\sim\normal(\lambda/n,2/n)$)
and therefore we include them. Also, by convention we set $\sd_0=0$.

Regardless of the sequence of arguments that brought us to this iteration, it
defines a perfectly reasonable class of algorithms. A specific algorithm in this class
is defined by two sequence of functions $\{f_t\}_{t\ge 0}$ and $\{g_t\}_{t\ge 0}$
(see below). Spelling out the details, the algorithm proceeds as follows:
\begin{enumerate}
\item \emph{Initialize $\bx^{(0)}$.} The initialization can be taken to be a function of the data
matrix $\bY$ or of additional side information. An example of the first type is
the spectral initialization $\bx^{(0)}= \sqrt{n} \bv_1(\bY)$ (where $\bv_1(\bY)$ is the 
principal eigenvector of $\bY$). An example of the second type arises if we 
observe $\by = \eps \bx_0+\bg$ for $\bg\sim\normal(0,\id_n)$ and set $\bx^{(0)} = \by$.
\item \emph{Iterate} the map \eqref{eq:AMPGeneral} for a number $t_{*}$ of steps.
\item \emph{Output}
\begin{align}
\hbx^{(t_*)} = g_t(\bx^{(t_*)})\, .\label{eq:hbX}
\end{align}
Here again, it is understood that $g_t$ is applied to $\bx^{(t_*)}$ entrywise.
\end{enumerate}
This family of algorithms is an example of an even broader class of
algorithms known as Approximate Message Passing (AMP)  algorithms.

Remarkably, AMP admits an exact characterization in the limit $n\to\infty$,
for any constant number of iterations. 
\begin{theorem}[\cite{BM-MPCS-2011}]\label{thm:StateEvolution}
Let $F:\reals^2\to \reals$ be any function that is locally Lipschitz and
with at most polynomial growth $|F(\bx)|\le C(1+\|\bx\|_2)^k$ for constants $C,k$.
Further assume the initial condition $\bx^{(0)}$ and true signal $\bx_0$ to be 
deterministic and such that $n^{-1}\sum_{i=1}^n \delta_{x_{0,i},x^{(0)}_i}$ converges
weakly and in $k$ moments to the law of $(X_0,X^{(0)})$.
Finally assume the functions $\{f_t\}$ to be Lipschitz.

Define the sequences $a_t, q_t$, $t\ge 0$, by letting, for $t\ge 0$
\begin{align}
a_{t+1} &= \lambda \E[X_0f_t(a_tX_0 + \sqrt{q_t} G)],  \quad
 q_{t+1} = \E[ f_t(a_tX_0 + \sqrt{q_t} G) ^2], \nonumber\\
 &(X_0,G)\sim {\sf Unif}(\{+1,-1\})
 \otimes \normal(0,1),\label{eq:StateEvolution}
\end{align}
with initialization $a_0 \equiv \E\{X_0 f_0(X^{(0)})\}$, $q_0 \equiv \E\{ f_0(X^{(0)})^2\}$.

Then we have, for each $t\ge 1$, almost surely,
\begin{align}
\lim_{ n \to \infty} \frac{1}{n} \sum_{i=1}^{n} F(x_{0,i},\bx_i^{(t)}) 
=\E F(X_0, a_tX_0 + \sqrt{q_t} G)\, ,
\end{align}
where expectation is with respect to $G\sim \normal(0,1)$ independent of $X_0$.
\end{theorem}
\begin{proof}[Proof idea]
We will only sketch the main idea of the proof that uses a technique introduced by 
Bolthausen \cite{bolthausen2014iterative}, and subsequently generalized in \cite{BM-MPCS-2011}
and follow up work.

For the sake of simplicity, we consider the case $\lambda=0$, $\bx^{(0)}=0$, which implies $a_t=0$
for all $t$.  For $t=1$, we have, $\bx^{(1)} = \bW f_0(0)$, which is a 
Gaussian vector with approximately i.i.d. $\normal(0,f_0(0)^2)$ entries. 
Indeed, for a fixed unit vector $\bu\in\reals^n$, we have
$\bW\bu\ed \sqrt{2/n} g_0\bu+\bP_{\bu}\bg/\sqrt{n}$, with $g_0\sim\normal(0,1)$ independent 
of $\bg\sim\normal(0,\id_n)$.

If we try to use the same argument to characterize $\bW f_1(\bx^{(1)})$
 we run into some difficulty as $\bx^{(1)}$ is 
no-longer independent of $\bW$. Indeed it is a useful exercise to check that, for a non-constant
function $f_1$, the vector $\bW f_1(\bx^{(1)})$ is not well approximated by normal vector with i.i.d. 
entries. 
In general, at the $t$-th step, $\bx^{(t)}$ is not independent of $\bW$ and thus this problem persists.

The difficulty is that $\bx^{(t)}$ is a complicated non-linear function of $\bW$
and hence there does not seem to be hope to obtain its distribution in closed form.
However, things simplify if we try to compute the conditional distribution of $\bW$ given 
$\bx^{(1)}, \cdots, \bx^{(t)}$. 
Note  that conditional on 
$\bx^{(1)}, \cdots, \bx^{(t)}$, $f_0(\bx^{(0)}), \cdots , f_t(\bx^{(t-1)})$ are also given, and therefore, 
conditioning on $\bx^{(1)}, \cdots , \bx^{(t)}$ is equivalent to conditioning on 
\begin{align}
\mathcal{E}_t = \Big\{ \bW:\; \bx^{(1)} = \bW f ( \bx^{(0)} ), \cdots, \bx^{(t)} + \sd_{t-1} f_{t-2}(\bh^{(t-2)}) = 
\bW f_{t-1}(\bh^{(t-1)}) \Big\}. \nonumber 
\end{align}
This is equivalent to conditioning on a set of linear measurements of a Gaussian vector
$\bW$. Equivalently, we are interested in the conditional distribution of the Gaussian vector $\bW$
subject to $\bW\in\cE_t$ an affine space. This is well known to be the conditional
expectation plus the projection onto $\cE^0_t$ (the linear space corresponding to $\cE_t$) 
of a fresh Gaussian vector. In formulas
\begin{align}
\bW |_{\cE_t} \ed \E[ \bW | \bx^{(0)}, \cdots, \bx^{(t)}] +
 \bP_{t}^{\perp} \bW^{\snew} \bP_{t}^{\perp}\, ,
\end{align}
where $\bW^{\snew}$ is an independent copy of $\bY$ and $\bP_{t}$ is 
the orthogonal projector onto the linear span of
$f_0(\bx^{(0)})$, $\cdots$ , $f_t(\bx^{(t-1)})$, and $\bP_{t}^{\perp}=\id-\bP_{t}$
is its orthogonal complement.
 Thus conditioning on $\bx^{(0)}, \cdots, \bx^{(t-1)}$, we have
 \begin{align*}
\bW f_t(\bx^{(t)})|_{\cE_t} 
\ed \E[ \bW | \bx^{(0)}, \cdots, \bx^{(t)}]\, f_t(\bx^{(t)}) +
\bW^{\snew}\bP_{t}^{\perp} f_t(\bx^{(t)}) - \bP_{t}\bW^{\snew}
 \bP_t^{\perp}f(\bx^{(t)}) \, .
 \end{align*}
The term $\bW^{\snew}\bP_{t}^{\perp} f_t(\bx^{(t)})$ is approximately Gaussian 
with independent entries of variance $\|\bP_{t}^{\perp} f_t(\bx^{(t)})\|_2^2/n$, since
$\bW^{\snew}$ is independent of $\bx^{(0)}, \cdots, \bx^{(t)}$.
 The term $\bP_{t}\bW^{\snew}
 \bP_t^{\perp}f(\bx^{(t)})$ is small since it is a low-dimensional projection of a 
 high-dimensional Gaussian.
 
It remains to track the $\E[ \bW | \bx^{(0)}, \cdots, \bx^{(t)}]\, f_t(\bx^{(t)})$.
This can be done by induction showing that it can be approximately 
decomposed into a term that is canceled by the correction
$- \sd_t f_{t-1}( \bx^{(t-1)})$ of Eq.~\eqref{eq:AMPGeneral},
and a linear combination of $\bx^{(1)},\dots ,\bx^{(t)}$, which is Gaussian by induction hypothesis.
 We refer the reader to \cite{BM-MPCS-2011} for a formalization of these ideas and a rigorous proof. 
\end{proof}

\begin{remark}
The last theorem is an application of the more general theorem proved in 
\cite{BM-MPCS-2011}, and substantial generalizations of the latter have been proved in the literature. 
These include non-symmetric matrices, non-Gaussian matrices, including matrices
with independent or non-identically distributed entries, algorithms with memory and
so on. We refer to the bibliography section for pointers to this literature.

The case of random initializations $\bx^{(0)}$ is covered by Theorem \ref{thm:StateEvolution}
as long as $\bx^{(0)}$ is independent of $\bW$, and the empirical distribution
$n^{-1}\sum_{i=1}^n \delta_{x_{0,i},x^{(0)}_i}$ converges almost surely.

On the other hand, the case of a spectral initialization $\bx^{(0)} = \sqrt{n}\, \bv_1(\bY)$
is not covered by this theorem, since this initialization 
is correlated with the noise matrix $\bW$. On the other hand, as shown in \cite{montanari2021estimation},
spectral initializations can be treated within  the same framework discussed here.
\end{remark}

The iteration \eqref{eq:StateEvolution} for parameters $a_t,q_t$ is known as `state evolution,'
a term introduced in \cite{DMM09}.
Theorem  \ref{thm:StateEvolution} allows one to determine the accuracy of the estimator
$\hbx^t$ produced by the algorithm, see Eq.~\eqref{eq:hbX}. Choosing
$F(x_0,x)= (x_0-g_t(x))^2$ in the theorem, we obtain, 
\begin{align}
\lim_{ n \to \infty} \frac{1}{n}\big\|\hbx^t-\bx_0\big\|_2^2 
=\E \big\{\big(X_0- g_t(a_tX_0 + \sqrt{q_t} Z)\big)^2\big\}\, .
\end{align}
The function $g_t$ that minimizes the right-hand side is  
$g_t(y) = \E\{X_0|a_tX_0 + \sqrt{q_t} Z=y\}$. Note that the resulting 
error only depends on the `signal-to-noise' ratio $b_t =  a_t^2/(\lambda^2 q_t)$,
and therefore  we would like to choose functions $f_1,\dots,f_t$ as to maximizes this
quantity. (The factor $\lambda^2$ is introduced for future convenience.)

 A recursive argument shows that the optimal functions, and the corresponding 
 quantities $(b_s)_{s\ge 1}$ are given by
 \begin{align*}
b_{t+1} &= \E\big\{X_0\E[X_0|\lambda b_tX_0+\sqrt{b_t} G]\big\}\, ,\\
f_t(y) & = \E[X_0|\lambda b_tX_0+\sqrt{b_t} G=y]\, .
\end{align*}
Since $X_0\sim {\sf Unif}(\{+1,-1\})$, it is easy to compute these conditional expectations to get
 \begin{align}
b_{t+1} &= \E\big\{\tanh(\lambda^2 b_t+\sqrt{\lambda^2 b_t} G)\big\}\, ,\\
f_t(y) & = \tanh(\lambda y)\, .
\end{align}
We therefore recover the nonlinearity $f_t(y) = \tanh(\beta y)$ that was derived heuristically 
in the last section, with the Bayes choice $\beta = \lambda$.

We proved that AMP with nonlinearity $f_t(y)= \tanh(\lambda y)$ is optimal among all the algorithms of
the form  Eq.~\eqref{eq:AMPGeneral}. We refer to this specific algorithm as  `Bayes AMP,'
because it appears to make, at each step, the Bayes-optimal choice.
Note that the fixed point condition corresponding to Eq.~\eqref{eq:Bayes}
coincides with the one that characterizes the (global) Bayes optimal estimator,
Eq.~\eqref{eq:FP_Bayes-Z2}. This points to a connection between the two.

This finding motivates two questions:
\begin{itemize}
\item Does Bayes AMP converge (in the large $n$ and large $t$ limits) to the Bayes optimal estimator?
\item Fixing a number of iterations $t$, is Bayes AMP optimal (in the sense of
statistical accuracy) among all algorithm with comparable computational complexity?
\end{itemize}
The answer is positive to both questions, provided a small amount of side information 
or a spectral initialization are used to break symmetry at the first iteration
\cite{deshpande2017asymptotic,celentano2020estimation}. Figure~\ref{fig:rsb_saddle1} illustrates these points by reporting the accuracy achieved by Bayes AMP and comparing it 
with the optimal Bayes accuracy in the large $n$ limit.

\section{Replica symmetry breaking} 
 \label{sec:kRSB-Replica-All}
 
We have derived the replica symmetric asymptotic free energy density of  the model
\eqref{eq:gibbs} using two alternative techniques ---the replica method and the cavity method.
In both cases, we made specific assumptions that allowed us to complete
the calculation. Within the replica method, we assumed the overlap matrix to 
take the replica symmetric form of Eq.~\eqref{eq:RSansatz-Z2}. Within the cavity method,
we assumed specific properties  {\sf C1},  {\sf C2},  {\sf C3} of the Gibbs measure. 

On the other hand, we know that at low enough temperature (large $\beta$),
the replica symmetric prediction for the free energy density is necessarily incorrect. 
Indeed, as explained in Section \ref{sec:NeedRSB}, if the prediction was
correct, i.e. if we had $\phi(\beta,\lambda) = \Psi_{\sRS}(b_*,q_*;\beta,\lambda)$
at all temperatures (here we are restoring the dependence of these quantities upon $\beta,\lambda$),
this would imply $\Entro(\mu_{\beta,\lambda})<0$ which is impossible (Recall that $\Entro(\mu)$
denotes the entropy of probability distribution $\mu$).

In this section, we correct this inconsistency by introducing the full 
Replica Symmetry Breaking (fRSB) asymptotics for the free energy, originally due to Parisi
\cite{parisi1979infinite}. 
Parisi formula was eventually established rigorously thanks to the work of 
 Talagrand \cite{talagrand2006parisi}, Panchenko 
  \cite{panchenko2014parisi} and many other authors. 
  In this section we will derive this formula using the replica method. 
  In Section \ref{sec:CavityRSB}, we will show how to modify the cavity method to 
  incorporate replica-symmetry breaking. Finally we will describe some rigorous 
  results using interpolation techniques in Section \ref{sec:SKrigorous}.
  
  To simplify our treatment, we will focus on the case $\lambda=0$, $h=0$.
  This case is not very interesting from the point of view of statistical estimation in the 
  $\integers_2$-synchronization model. Indeed, it corresponds to the case in which the data 
  matrix $\bY=\bW$ is pure noise. However, the generalization to $\lambda>0$ 
  is relatively straightforward, and left as an exercise for the reader.
  
  \subsection{$k$-step replica symmetry breaking}
  \label{sec:kRSB-Replica}
  
Recall the action functional derived by the replica method in Section \ref{sec:RS-Z2}.
The replica method predicts the  asymptotic free energy density to be given by
  $\phi  = \lim_{r \to 0} \frac{1}{r} S(\bQ_*)$, where 
\begin{align}
S(\bQ) &= \frac{\beta^2r }{4} - \frac{\beta^2}{2} \sum_{1 \leq a < b \leq r} Q_{ab}^2 + \log z_r(\bQ)\, ,\label{eq:sk_action_again} \\
z_r(\bQ) &= \sum_{\sigma^{1}, \cdots, \sigma^{r} \in \{ \pm 1 \}} \exp{\Big\{\beta^2 \sum_{1\leq a< b \leq r} Q_{ab} \sigma^{a} \sigma^{b} \Big\}},  \nonumber
\end{align}
and $\bQ_*$ is a suitable stationary point of $S$.
Note that we slightly simplified this expression with respect to Eq.~\eqref{eq:SK_Action}.
Indeed, since we are assuming $\lambda =0$, we can take $Q_{0a}=b=0$ for all 
$a\in\{1,\dots, r\}$ and redefine $\bQ$ to be an $r \times r$ matrix. 
For general $\lambda$, we would have to keep track of $b$ and optimize the action $S$ with respect to
$b$ as well. 

The RS stationary points $\bQ^{\sRS}$ are the only ones that are symmetric 
with respect to permutations of the replicas $\{1,\dots,r\}$. 
In the context of the $p$-spin model, we improved on the RS free energy by introducing the 
$1$RSB ansatz. The $1$RSB subspace is defined by partitioning 
$\{ 1, \cdots, r \} = \cup_{\ell=1}^{r/m} \mathcal{I}_{\ell}$ with $|\mathcal{I}_{\ell}|=m$
and setting
\begin{align}
Q^{\sRSB{1}}_{ab} &= 
\begin{cases}
1 & \mbox{if } a=b \in \{0, 1, \cdots, r\},\\
 q_1 & \mbox{if } a\neq b \in \mathcal{I}_{\ell}, \ell \in \{1, \cdots, r/m\}, \\ 
q_0 & \mbox{otherwise.}
\end{cases}
\end{align}
Substituting in $S(\bQ)$, taking the limit $r\to 0$
and optimizing over $q_0, q_1$ and $m$ obtains the $1$RSB free energy.
This calculation was originally carried out in \cite{parisi1979toward}.
It yields a better behavior at low temperature, but predicted  
$\lim_{\beta\to\infty}\lim_{n\to\infty}n^{-1}\Entro(\mu_{\beta})$ to be approximately $-0.01$,
which is impossible.
The same paper, however, suggested that iterating the same construction infinitely many times would yield
the correct stationary point\footnote{Apparently, this informal remark in 
\cite{parisi1979toward}  was 
criticized by one of the referees of the manuscript.} of $S$. 

Indeed the correct asymptotic free-energy is obtained by 
constructing a $k$-step RSB ansatz: a subspace of overlap matrices with reduced symmetry,
that can be obtained by iterating the $1$RSB ansatz.
 The correct limiting free energy density can be obtained in 
  the limit $k\to\infty$.
This more general calculation was carried out in \cite{parisi1979infinite}.

We construct a  hierarchical partition of $\{1, \cdots , r\}$ into blocks as follows. 
At the first step, set $\{1, \cdots ,r \} = \sqcup_{l_1=1}^{r/m_0} B_{l_1}$ for some integer $m_0 \geq 1$, with $|B_{l_1}| = m_0$. Next, partition each $B_{l_1} = \sqcup_{l_2 =1}^{m_0/m_1} B_{l_1 l_2}$, with $|B_{l_2}|= m_1$. We proceed iteratively, and finally partition 
$B_{l_1 l_2 \cdots l_{k-1}} = \cup_{l_k = 1}^{m_{k-2}/ m_{k-1}} B_{l_1, \cdots, l_k}$ with
 $|B_{l_1, \cdots, l_k}| = m_{k-1}$ for each $l_1, \cdots, l_k$.  Let $\bJ$ denote the 
 $r \times r$ matrix of all ones, and let $\bJ_{S}$ denote the $r \times r$ 
 matrix with ones on the $|S| \times |S|$ block indexed by 
  $S \subset \{1, \cdots , r \}$ and zero otherwise. Given the hierarchical partition 
  $\{B_{l_1, \cdots, l_m} : m\leq k\}$ and $0\leq q_0 \leq q_1 \leq \cdots \leq q_k \leq 1$, the 
  $k$RSB ansatz consists of matrices of the form
\begin{align}
\bQ^{\sRSB{k}} =& q_0 \bJ + (q_1 - q_0) \sum_{l_1} \bJ_{B_{l_1}} 
+ (q_2 - q_1) \sum_{l_1 l_2} \bJ_{B_{l_1 l_2}} + \nonumber \\
& \cdots + (q_k - q_{k-1}) \sum_{l_1, \cdots, l_k } \bJ_{B_{l_1,\cdots, l_k}} + (1-q_k) \bI
\, . \nonumber 
\end{align}
A schematic representation of the $k$RSB saddle point is given in Figure \ref{fig:rsb_saddle}. 
\begin{figure}
\begin{center}
\includegraphics[scale=0.6]{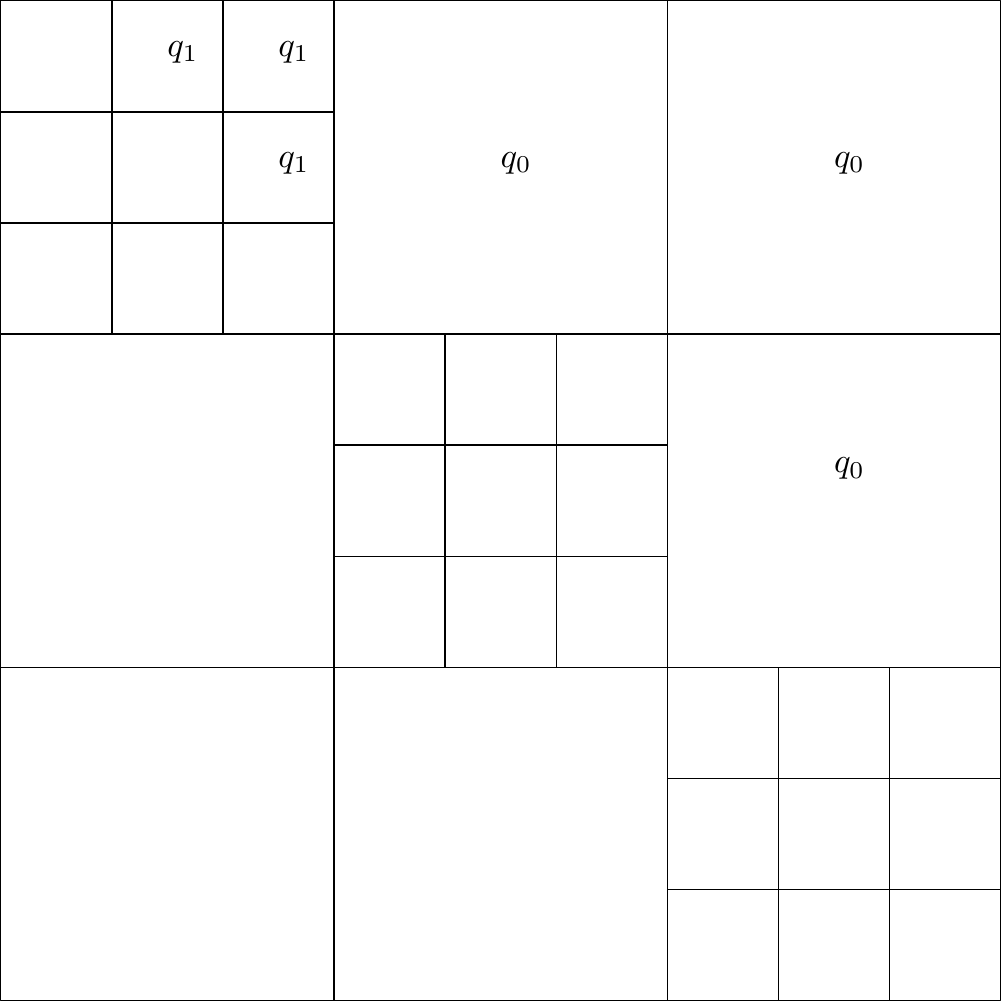}
\end{center}
\caption{A schematic representation of the $k$RSB saddle point}
\label{fig:rsb_saddle}
\end{figure}

Before we attempt to study the $k$RSB free energy functional, 
it is useful to interpret the $k$RSB ansatz in terms of the overlap distribution. 
This will also clarify the meaning of parameters $m_0, \cdots, m_{k-1}$,
$q_0, \cdots, q_{k}$, and the range over which they should be optimized.
 Recall the replica analysis leading to \eqref{eq:overlap3}, which predicts that if 
 $\rho_n$ is the distribution of
  $n^{-1}|\langle \bsigma^{1}, \bsigma^{2} \rangle|$ for $\bsigma^{1}, \bsigma^{2}$ i.i.d.
   samples from $\mu_{n}(\cdot)$, then $\rho_n\Rightarrow \rho$, where, for any continuous 
   bounded function $f$,
\begin{align}
\int_{[0,1]}\! f(q) \, \rho(\de q) = \lim_{r \to 0} \frac{1}{r(r-1)} \sum_{a \neq b} f(\bQ^*_{ab}), \nonumber 
\end{align}
where $\bQ^*$ is the stationary point of $S(\bQ)$ that achieves the correct free energy density.
 
We next evaluate the last expectation on the $k$RSB stationary point. We have
\begin{align}
\label{eq:kRSBfq}
\sum_{a \neq b} f(Q_{ab}) = \Big( r^2 - r m_0 \Big) f(q_0) + \frac{r}{m_0} \Big(m_0^2 - m_0m_1 \Big) f(q_1) + \cdots + f(q_k) \frac{r}{m_{k-1}} (m_{k-1}^2 - m_{k-1}). 
\end{align}
Taking the limit $r\to0$, we obtain, 
\begin{align}
\int_{[0,1]} \!f(q)\, \rho(\de q) = m_0 f(q_0) + (m_1 - m_0) f(q_1) + \cdots + (1- m_{k-1}) f(q_k), \nonumber 
\end{align}
which implies
\begin{align}\label{eq:RhoKrsb}
 \rho = \sum_{l=0}^{k} (m_l - m_{l-1}) \delta_{q_l}\, , \;\;\;\; m_k= 1, \;\; m_{-1}= 0\, .
 \end{align}
In words,  the $k$RSB ansatz assumes that the asymptotic overlap distribution
is formed by   $k+1$ point masses at  $q_l$, $l\in\{0,\dots,k\}$, 
with probabilities $(m_l - m_{l-1})$ respectively. 
 This also implies that we should take 
 $0\leq m_0 \leq m_1 \leq \cdots \leq m_{k-1} \leq 1$ 
 in the limit $r\to 0$.
 
 Viceversa,  any $k$RSB saddle point $\bQ^{\sRSB{k}}$ can be specified 
 by a probability measure $\rho$ on the interval $[0,1]$ with $k+1$ point masses.
 Equivalently, it can be specified by the distribution function $q\mapsto\rho([0,q])$.
 In physics, this is referred to as the ``functional order parameter.'' 
 We can therefore think of the RSB action as a function of the 
 probability distribution  $\rho$.

Next, we turn to the computation of the $k$RSB free energy functional. We will encode the parameters 
in vectors 
$\bq = (q_0, \cdots, q_k)$ and $\bm = (m_0, \cdots, m_{k-1})$. Note that using $f(q) = q^2$ 
in Eq.~\eqref{eq:kRSBfq}   above, we have, 
\begin{align}
\lim_{r \to 0} \frac{1}{r} \sum_{a<b} (Q^{\sRSB{k}}_{ab})^2 = - \frac{1}{2} \Big\{ m_0 q_0^2 + \cdots + (1- m_{k-1}) q_k^2 \Big\}. \label{eq:krsbfirst}
\end{align}
Therefore, it remains to compute $\lim_{r \to 0} r^{-1}\log z_r(\bQ^{\sRSB{k}})$. 
To this end, we note that 
\begin{align}
\sum_{a \neq b} Q_{ab}^{\sRSB{k}} \sigma^a \sigma^b = & \sum_{a,b} Q_{ab}^{k\rm{RSB}} \sigma^a \sigma^b - r \nonumber \\
= &q_0 \Big(\sum_a \sigma^a \Big)^2 + (q_1 - q_0) \sum_{l_1} \Big(\sum_{a \in B_{l_1}} \sigma^a \Big)^2 + \cdots \nonumber \\
&+ 
(q_k- q_{k-1}) \sum_{l_1, \cdots, l_k} \Big(\sum_{a \in B_{l_1, \cdots, l_k}} \sigma^a \Big)^2 + (1- q_k) r - r. \label{eq:QuadraticExponent}
\end{align}
Therefore, \eqref{eq:sk_action_again} implies, 
\begin{align}
&z_r(\bQ^{k\rm{RSB}} ) = \sum_{\sigma^1, \cdots , \sigma^r \in \{\pm 1 \}} \exp{\Big\{ \frac{\beta^2}{2} \sum_{a\neq b} Q_{ab}^{\sRSB{k}} \sigma^{a}\sigma^{b} \Big\}} = 
e^{-\beta^2 q_k r/2} \cdot T,\label{eq:krsbsecond}\\
T &= \E\Big\{ \sum_{\sigma^1, \cdots, \sigma^r} \exp{\Big[ \beta \sqrt{q_0} g_0 \Big( \sum_a \sigma^a \Big) + \beta \sum_{m=1}^{k} \sqrt{q_m - q_{m-1}} \sum_{l_1, \cdots, l_m} g_{l_1, \cdots, l_m} \Big( \sum_{a \in B_{l_1, \cdots, l_m} } \sigma^a \Big) \Big]} \Big\}\, ,\nonumber 
\end{align}
where expectation is with respect to $\{ g_0, g_{l_1, \cdots, l_m} : l_1,\dots,l_m\in\naturals,
m\leq k \}$, a collection of independent   $\normal(0,1)$ random variables.
The above expression is obtained by applying 
the  ``Gaussian disintegration" trick to each term in the exponent
\eqref{eq:QuadraticExponent}, as to linearize the exponent with respect to the $\sigma^a$'s.

We next carry out the sum over $\sigma^1,\dots,\sigma^r$.
 First, we fix $a \in \{1, \cdots, r\}$. Then there exist unique $l_1, l_2, \cdots, l_k$ 
 such that $a \in B_{l_1}$, $a \in B_{l_1l_2}$, $\cdots$, $a \in B_{l_1l_2\cdots l_k}$. 
 The terms involving $\sigma^a$ are
\begin{align}
&\sum_{\sigma^a} \exp{\Big[\beta \sigma^a \Big( \sqrt{q_0} g_0 + \sqrt{q_1 - q_0} g_{l_1} + \cdots+ \sqrt{q_k - q_{k-1}} g_{l_1 l_2 \cdots l_k}  \Big) \Big]} \nonumber \\
&= 2\cosh \Big[ \beta \Big( \sqrt{q_0} g_0 + \sqrt{q_1 - q_0} g_{l_1} + \cdots+ \sqrt{q_k - q_{k-1}} g_{l_1 l_2 \cdots l_k}  \Big) \Big]. \label{eq:exp_simplification}
\end{align} 
Further, for any fixed $(l_1, \cdots, l_k)$, there are exactly $m_{k-1}$ many 
$a \in \{1, \cdots, r\}$ such that $\sigma^a$ belongs to the same partitions 
$B_{l_1}, \cdots, B_{l_1l_2 \cdots l_k}$. These sums will give
 the same contribution as \eqref{eq:exp_simplification} and thus we obtain, 
\begin{align}
z_r(\bQ^{\sRSB{k}}) &= \E\Big[ \prod_{l_1, \cdots , l_k} \Big( 2\cosh \Big[ \beta \Big( \sqrt{q_0} g_0 + \sqrt{q_1 - q_0} g_{l_1} + \cdots+ \sqrt{q_k - q_{k-1}} g_{l_1 l_2 \cdots l_k}  \Big) \Big] \Big)^{m_{k-1}} \Big] \nonumber \\
&=: \E\Big[\prod_{l_1, \cdots, l_k} \Zrep_{l_1,\cdots,l_k}^{m_{k-1}} \Big]. \nonumber 
\end{align}
Here $\E$ denotes expectation with respect to the random variables
$\{ g_0, g_{l_1, \cdots, l_m} : l_1,\dots,l_m\in\naturals,\,
m\leq k \}$. This expectation can be  computed recursively, by evaluating only $k+1$
one-dimensional integrals as we next explain.

 For any fixed $l_1, \cdots, l_{k-1}$, $(g_{l_1 l_2 \cdots l_k})_{l_k\le m_{k-2}/m_{k-1}}$
  are i.i.d. random variables. Thus for fixed $l_1, \cdots, l_{k-1}$, we can first compute the 
  expectation 
\begin{align}
\tZrep_{l_1, \cdots, l_{k-1}}^{m_{k-1}} := \E_{g_{l_1 \cdots l_{k-1},l_k}} \Big[ \Zrep_{l_1 l_2 \cdots l_{k-1},l_k}^{m_{k-1}}\Big]=
\E_{g_{l_1 \cdots l_{k-1},1}} \Big[ \Zrep_{l_1 l_2 \cdots l_{k-1},1}^{m_{k-1}}\Big], \nonumber 
\end{align}
where we note, that for each fixed $l_1, \cdots, l_{k-1}$, this expectation is the same for all $l_k$. 
The exponent $m_{k-1}$ was introduced on the left-hand side for future convenience. 
Indeed, there are $m_{k-2}/ m_{k-1}$ such terms, and thus we have
\begin{align}
z_r(\bQ^{\sRSB{k}}) = \E\Big\{ \prod_{l_1, \cdots, l_{k-1}} \Big[ \tZrep_{l_1, \cdots, l_{k-1}}\Big]^{m_{k-2}} \Big\}\, . \nonumber 
\end{align}
We repeat this operation iteratively to define 
$\tZrep_{l_1, \cdots, l_{k-2}}^{m_{k-1}} = 
 \E_{g_{l_1 l_2 \cdots l_{k-1}} } \Big[ \Big[ \tZrep_{l_1, \cdots, l_{k-1}}\Big]^{m_{k-2}} \Big] \Big]$,
  $\tZrep_{l_1, \cdots, l_{k-3}}$ and so on. Finally, we will obtain
$z_r(\bQ^{k\rm{RSB}}) = \E_{g_0} [ (\tZrep_0)^{r}]$, where $\tZrep_0$ is only a function of $g_0$. 

Summarizing, let $g_0,\dots,g_k\sim \normal(0,1)$, and define the sequence of
random variables $(\Zrep_i)_{0\le i\le k}$ via the recursion 
\begin{align}\label{eq:RecursiveFirst}
\begin{split}
\Zrep_k& =
 2\cosh \Big[ \beta \Big( \sqrt{q_0} g_0 + \sqrt{q_1 - q_0} g_{1} + \cdots+ \sqrt{q_k - q_{k-1}} g_{k}  \Big) \Big] \,,\\
 \Zrep_{i-1} & = \E\big[\Zrep^{m_{i-1}}_{i}\big|g_0,\dots, g_{i-1}\big]\,,\\
 F(\bq, \bm)& =\E\log\Zrep_0\, .
 \end{split}
 \end{align}
 We note that each $\Zrep_{i}$ is a function of $g_0,\dots,g_{i}$
 and the correspondence with the variables defined above is that $\tZrep_{l_1\dots l_i}$
 is distributed as $\Zrep_i$ for each $i\le r$. 
Thus we obtain 
\begin{align}
\lim_{r\to 0} \frac{1}{r}\log z_r(\bQ^{\sRSB{k}}) = F(\bq, \bm)\, . \nonumber 
\end{align}
 Finally, plugging everything back into \eqref{eq:sk_action_again}, using \eqref{eq:krsbfirst} 
 and \eqref{eq:krsbsecond}, we obtain $\lim_{r\to 0}r^{-1}S(\bQ^{\sRSB{k}})
  =\Psi_{\sRSB{k}}(\bq, \bm)$, where
 \begin{align}
\Psi_{\sRSB{k}}(\bq, \bm) =  
\frac{\beta^2}{4}(1- 2 q_k) + \frac{\beta^2}{4} \sum_{l=0}^{k} (m_l - m_{l-1}) q_l^2 + F(\bq, \bm). \label{eq:k-RSBfreeenergyfunctional}
\end{align}
This completes the derivation of the $k$RSB free energy functional. As a sanity check,
 we check that in case $k=0$, we recover the RS free energy. 
 In this case, we set $m_0 =1$ and $q_k = q$ to get 
\begin{align}
\Psi_{\sRSB{k}}(\bq, \bm)\Big|_{k=0} = \frac{\beta^2}{4} (1- 2q) + 
\frac{\beta^2}{4} q^2 + \E[\log 2 \cosh(\beta (\sqrt{q}g))] = \Psi_{\sRS}(q)\, .
\end{align}
Another case of interest is $1$RSB. We set $\bm = (m_0, m_1)$ with $0 \leq m_0 \leq m_1 =1$ and 
$\mathbf{q} = (q_0, q_1)$ with $0 \leq q_0 \leq q_1 \leq 1$. In this case, 
the function $F$ reads:
\begin{align}
F(q_0, q_1, m_0) = \frac{1}{m_0} \E_{g_0} \Big[ \log \Big( \E_{g_1} \Big[ \Big( 2 \cosh \beta ( \sqrt{q_0} g_0 + \sqrt{q_1 - q_0} g_1 ) \Big)^{m_0} \Big] \Big) \Big]. \nonumber
\end{align}

In the next sections, we will try to elucidate the meaning and origin of the $k$RSB free energy
 functional by describing other approaches towards its derivation (cavity and interpolation methods).
The fact that it yields the correct free energy density \emph{at all temperatures}
was eventually proven by Talagrand \cite{talagrand2006parisi} and Panchenko \cite{panchenko2014parisi}. 
\begin{theorem}[\cite{talagrand2006parisi}]\label{thm:MainTalagrand}
Consider the SK model \eqref{eq:model}. Setting $Z_n(\beta)$ 
to be the partition function of Eq.~\eqref{eq:partition} at $\lambda=0$, we have, 
\begin{align}
\phi:=\lim_{n \to \infty} \frac{1}{n} \log Z_n(\beta) = \inf_{k , \bq, \bm} \Psi_{k\rm{RSB}}(\bq, \bm), \nonumber 
\end{align}
where the limit holds almost surely and in $L^1$.
\end{theorem}

\subsection{Taking the limit of full RSB}

Given a $k$RSB order parameter $\bq = (q_0, \cdots, q_k)$ and $\bm = (m_0, \cdots, m_{k-1})$,
it is easy to `embed it' in the $(k+1)$RSB space, e.g.,
\begin{align}
\bq' = (q_0, \cdots, q_k,q_{k+1}=1),\;\;\;\;\; \bm' = (m_0, \cdots, m_{k-1},m_k=1)\, .
\end{align}
We leave it to the reader to check that $\Psi_{\sRSB{(k+1)}}(\bq', \bm') = \Psi_{\sRSB{k}}(\bq, \bm)$,
whence
\begin{align}
 \inf_{\bq, \bm} \Psi_{\sRSB{(k+1)}}(\bq, \bm)\le  \inf_{\bq, \bm} \Psi_{\sRSB{k}}(\bq, \bm) \, .
 \end{align}
In other words  $\Psi^*_{\sRSB{k}}:= \inf_{\bq, \bm} \Psi_{\sRSB{k}}(\bq, \bm)$
form a sequence of non-increasing upper bounds on the free energy density, and
\begin{align}
\phi= \inf_{k\ge 0}\Psi^*_{\sRSB{k}} = \lim_{k\to\infty}\Psi^*_{\sRSB{k}}\, .
\end{align}
Two cases are possible:
\begin{itemize}
\item The infimum is achieved. Letting $k$ denote  the smallest integer such that  
$\Psi^*_{\sRSB{k}} =\inf_{k'\ge 0}\Psi^*_{\sRSB{k'}}$, the model (the Gibbs measure) 
is said be in a $k$RSB phase.

When this happens, most often $k=0$ (RS) or $k=1$, although models with any $k$
can be constructed.
\item The infimum is not achieved. In this case the correct 
free energy density is only obtained by taking the limit $k\to\infty$.
This scenario is known as \emph{full replica symmetry breaking (fRSB)}.

Full RSB is believed to hold for the SK model at any $\beta>1$, although a 
rigorous proof is missing.
\end{itemize}

Motivated by these remarks we will next derive a description of the $k\to\infty$ 
limit. 
We begin by noting that  we can assume without loss of generality that $q_0 =0$---indeed any $k$RSB order parameter can be expressed as an equivalent $(k+1)$RSB
order parameter with $q_0=0$, simply by taking $m_0=0$.
Equivalently, setting $q_0=0$ amounts to adding an atom at $0$
of weight $m_0$ to the functional order parameter $\rho$. 

Next, we reformulate the recursive construction of the Parisi functional
in Eq.~\eqref{eq:RecursiveFirst}. We construct a sequence of functions 
$\Phi_{i}:\reals\to\reals$,  proceeding backwards for $i\in\{0,\dots,k\}$.
Letting  $\E[\;\cdot\;]$ denote expectation with respect to $g\sim \normal(0,1)$ 
\begin{align}
\Phi_k(x) &= \log 2 \cosh(\beta x), \nonumber \\
\Phi_i (x) &= \frac{1}{m_i} \log \E\Big[ \exp \Big( m_i \Phi_{i+1} ( x + \sqrt{ q_{i+1} - q_i} g) \Big) \Big], \forall i\in\{0,\dots, k-1\}\, .
\end{align}
Using Eq.~\eqref{eq:RecursiveFirst}, it is not hard to see that 
$F(\bq, \bm) = \Phi_0(0)$.  More generally, the relation between the functions $\Phi_i$,
and the random variables $\Zrep_{i}$ defined there is
\begin{align}
\log\Zrep_i = \Phi_i\Big(\sum_{\ell=1}^{i}\sqrt{q_i - q_{i-1}} g_{i}\Big)\, .
\end{align}

We next claim that for each $i\in\{0,\dots k\}$, 
\begin{align}
\Phi_i(x)=\Phi(x,q_i)-\frac{\beta^2}{2}(1-q_k)\, , \label{eq:ClaimPhi}
\end{align}
 where
$\Phi:[0,1]\times \reals\to \reals$, $(q,x)\mapsto \Phi(q,x)$ is the unique solution of the following 
anti-parabolic PDE (we denote by $\partial_q$, $\partial_x$ the partial derivatives with 
respect to $q$ and $x$ and by $\partial_{xx}$ the second derivative with respect to $x$): 
\begin{align}\label{eq:Parisipde}
\begin{split}
&\partial_q \Phi+ \frac{1}{2} \big[ \partial_{xx} \Phi + \rho(q) \big( \partial_x \Phi\big)^2  \big]=0 ,  \nonumber \\
&\Phi(x,1) = \log 2 \cosh(\beta x)\, .
\end{split}
\end{align}
Here $\rho(q) = \rho([0,q])$ is the cumulative distribution function corresponding to 
the asymptotic overlap distribution $\rho(\de q)$ and encodes the $k$RSB order parameter
of Eq.~\eqref{eq:RhoKrsb}.
In particular (recall that $q_0=0$, $q_{k+1}=1$, $m_k=1$):
\begin{align}
\rho(q) = m_i \;\;\; \mbox{ for } q\in [q_i,q_{i+1}),\;\; i\in\{0,\dots, k\}\, .
\end{align}
Finally, Eq.~\eqref{eq:Parisipde} is understood to be solved backwards in `time' $q$
starting from the boundary condition at $q=1$.

The claim \eqref{eq:ClaimPhi} can be proved by induction, proceeding backward from $i=k$.
Indeed notice that, for $q\in [q_i,q_{i+1})$, the solution of the PDE
\eqref{eq:Parisipde} must satisfy 
\begin{align}\label{eq:Parisipde_one_interval}
&\partial_q \Phi+ \frac{1}{2} \big[ \partial_{xx} \Phi + m_i \big( \partial_x \Phi \big)^2  \big]=0\, .
\end{align}
In this interval, define $W(x,q)\equiv \exp(m_i\Phi(x,q))$, a change of variables known as
`Cole-Hopf transformation'. Note that
\begin{align*}
\partial_q W &= m_i e^{m_i\Phi }\partial_q \Phi , \;\;\; \;
\partial_x W = m_i e^{m_i\Phi }\partial_x \Phi\, , \\
\partial_{xx} W &= m_i e^{m_i \Phi } \big\{  \partial_{xx} \Phi+ m_i
(\partial_x \Phi )^2 \big\}\, , 
\end{align*}
and therefore, $W$ satisfies the (time reversed) heat equation 
\begin{align*}
&\partial_q W+ \frac{1}{2}  \partial_{xx} W=0\, .
\end{align*}
The solution of this equation is well known to admit a probabilistic representation:
for any $q' \geq q$, we have  $W(x,q) = \E[ W(x + \sqrt{q' -q} \,g, q')]$,
where expectation is with respect to  $g\sim\normal(0,1)$.
Hence the function $\Phi$ has the following representation for $q\in [q_i,q_{i+1})$:
\begin{align}
\Phi(x,q) = \frac{1}{m_i} \log \E\Big[ \exp\big( m_i \Phi( x + \sqrt{q_{i+1}-q} g , q_{i+1} ) \big) \Big]. \label{eq:ColeHopf}
\end{align}
Using this representation, we can easily prove claim \eqref{eq:ClaimPhi}.
For $i=k$, notice that $m_k=1$ and therefore
\begin{align*}
\Phi(x,q_k)&= \log \E\Big[ 2\cosh\beta\big( x + \sqrt{1-q_{k}} \,g  \big) \Big]\\
&= \frac{1}{2}\beta^2(1-q_k)+\log 2\cosh\beta x\, ,
\end{align*}
as claimed.  Next, we assume that the claim holds for $\Phi_{i+1}(\, \cdot\, )$ and
$\Phi(\, \cdot\, ,q_{i+1})$, and prove it for $\Phi_{i}(\, \cdot\, )$ and
$\Phi(\, \cdot\, ,q_{i})$. By Eq.~\eqref{eq:ColeHopf} and the induction hypothesis,
we have
\begin{align*}
\Phi(x,q_i)&= 
 \frac{1}{m_i} \log \E\Big[ \exp\big( m_i \Phi_{i+1}( x + \sqrt{q_{i+1}-q_i} g )+
 (\beta^2/2)(1-q_k) \big) \Big]\\
 & = \Phi_i( x )+\frac{1}{2}\beta^2(1-q_k)\, .
\end{align*}

We can now represent the $k$RSB free energy functional  \eqref{eq:k-RSBfreeenergyfunctional} 
in terms of the solution the Parisi PDE, which we denote by $\Phi(x,q;\rho)$
in order to emphasize its dependence on $\rho$:
\begin{align}
\Psi_{k\rm{RSB}} (\bq,\bm) = -\frac{1}{4} \beta^2 + \frac{\beta^2}{4} \int_{[0,1]} q^2 \rho(\de q) + 
\Phi(0,0; \rho). \nonumber  
\end{align}
Note that the right-hand side makes sense for any probability measure $\rho$,
not only a probability measure formed by a finite number of 
atoms. 

One mathematical difficulty here is that defining the functional for arbitrary $\rho$
requires showing existence and uniqueness of solutions of the PDE
 \eqref{eq:Parisipde} for $\rho(q)$ a general distribution function (not just piecewise constant).
 While this is possible (see below), an alternative approach was originally developed by
Guerra \cite{guerra2003broken}. This paper established that for two probability measures 
$\mu, \nu$ on $[0,1]$ with finitely many atoms, $|\Phi(0,0;\mu) - \Phi(0,0;\nu) | \leq \int_{0}^{1}
 | \mu([0,x]) - \nu([0,x]) | \de x$. Thus, if the space of all probability measures on $[0,1]$
  is metrized by ${\sf d}(\mu , \nu) = \int_{0}^{1} | \mu([0,x]) - \nu([0,x]) | \de x$, then 
  $\rho\mapsto \Phi(0,0;\rho)$ is 
  Lipschitz and can thus be extended uniquely to the space of all probability measures on $[0,1]$. 
  
  Thus for any probability measure $\rho$ on $[0,1]$, we define the free energy functional
\begin{align}
\label{eq:ParisiFRSB}
\cuP(\rho) = - \frac{\beta^2}{4} + \frac{\beta^2}{4} \int q^2 \rho(\de q) + \Phi(0,0; \rho). 
\end{align}
In terms of the free energy functional $\mathscr{P}$,
the asymptotics of the free energy density is 
\begin{align}
\lim_{n \to \infty} \frac{1}{n}\log Z_n(\beta) = \inf_{\rho} \cuP(\rho). \label{eq:FRSB-Formula}
\end{align}

As a final sanity check, we note that for $\lambda= h=0$ and $\beta$ sufficiently small, 
we expect the free energy functional to be minimized at $\rho = \delta_0$
(indeed this can be proved to be the case).
In this case, $\Phi(0,0;\delta_0) = \log \E[ 2 \cosh ( \beta g)]$ for $g \sim \normal(0,1)$.
 Thus $\Phi(0,0;\delta_0) = \beta^2/2 + \log 2$ and $\mathscr{P}(\delta_0) = \beta^2/4 + \log 2 
 = \lim_{n \to \infty} \frac{1}{n}\log \E[Z_n(\beta)]$. This agrees with the intuition
 that for sufficiently high temperature, and in zero external field,
 the asymptotic free energy density coincides with  the annealed free energy density.

At first sight, the formula \eqref{eq:FRSB-Formula} for the free energy 
density might appear difficult to use because requires to solve an infinite-dimensional 
optimization problem. On the contrary, it turns out to be very useful.

The minimization over $\rho$ can be approximated numerically by restricting to 
distributions with $k$ atoms either at fixed of variable locations.
A particular consequence of \eqref{eq:FRSB-Formula} is the value of the ground state energy:
\begin{align}
\lim_{n\to\infty}\frac{1}{n}\max_{\bsigma\in\{+1,-1\}^n}H(\bsigma) = \sP_*\equiv
\lim_{\beta\to\infty}\frac{1}{\beta}\inf_{\rho}\cuP(\rho) \, .
\end{align}
An evaluation of the right-hand side yields $\sP_* \approx 0.763 166 726 567$
\cite{crisanti2002analysis,schmidt2008replica}.

Several rigorous results are available about the functional 
$\rho \mapsto \cuP(\rho)$. Auffinger-Chen \cite{auffinger2015parisi} and
 Jagannath-Tobasco \cite{jagannath2016dynamic} establish that $\cuP(\rho)$ is strictly convex in
  $\rho$ and therefore has a unique minimizer (recall that,
  as mentioned above,  $\cuP(\rho)$ is continuous in $\rho$). Recall that $\rho^*$ 
  is expected to be give the asymptotic overlap distribution 
  as $n\to\infty$.
   Figure \ref{fig:parisi_meas} illustrates the expected behavior of $\rho^*$ for
    different temperatures $\beta$. We refer the reader to 
    \cite{auffinger2015properties} for rigorous results about the minimizer. 

\begin{figure}
\begin{center}
\includegraphics[scale=0.6]{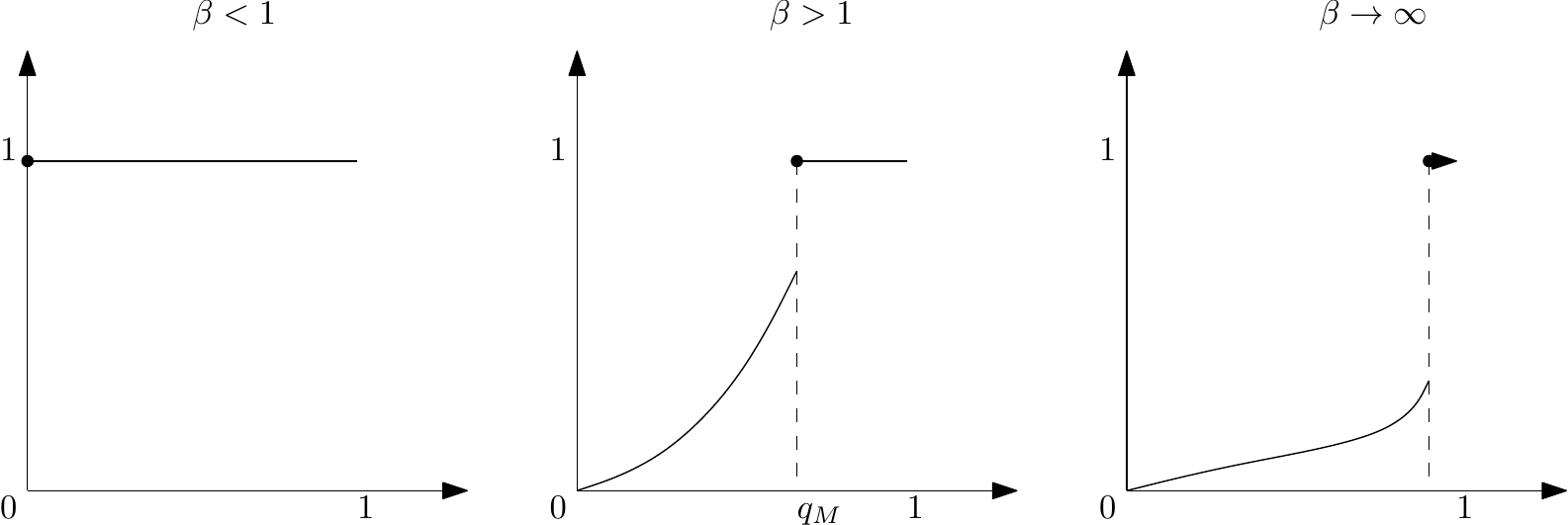}
\end{center}
\caption{A schematic representation of the distribution function of $\rho^*$ for different values 
of the inverse temperature $\beta$}.
\label{fig:parisi_meas}
\end{figure}

\subsection{On the physical interpretation of Parisi's PDE}

Let $\Phi(x,q) = \Phi(x,q;\rho)$ be the solution of the PDE \eqref{eq:Parisipde}
for a certain probability distribution $\rho$.
Define $s(x, q) = \beta^{-1} \partial_x \Phi(x, q)$. Differentiating the 
PDE \eqref{eq:Parisipde} with respect to $x$, we note that 
$s(\cdot, \cdot)$ solves
\begin{align}
&\partial_q s +\frac{1}{2} \partial_{xx} s + \beta \rho(q) s(x, q) \partial_x s=0\, 
, \label{eq:pde1} \\
 &s(x, 1) = \tanh(x)\, . \nonumber
\end{align}
For a general function $s : \reals \times [0,1] \to \reals$ (under suitable regularity conditions), 
we can also define the 
stochastic differential equation (SDE):
\begin{align*}
\de X_q = \beta \rho(q) s(X_q, q) \de q + \de B_q\, . 
\end{align*}
where $\{ B_q : q \geq 0 \}$ is a standard Brownian motion.
 We first observe that if $s$ solves \eqref{eq:pde1}, then $s(X_q, q)$ is a martingale
 (i.e. $\E[s(X_{q'},q')|X_q] = s(X_q,q)$ for all $q\le q'$). 
 This follows upon a direct application of Ito's lemma. However, to keep the discussion elementary 
 and self-contained, we sketch a proof below. We have, 
\begin{align}
\E&[s(X_{q+\eps},q+\eps)] | X_q = x] - s(x, q) \nonumber \\
&= \partial_x s(x,q) \E[ X_{q+ \varepsilon} - X_q | X_q = x] + \varepsilon \partial_q s(x, q) + \frac{1}{2} \partial_{xx}s(x,q) \E[ (X_{q + \varepsilon} - X_ q)^2 | X_q =x] + o(\varepsilon). \nonumber \\
&= \varepsilon \Big( \partial_q s + \frac{1}{2} \partial_{xx} s + \beta \rho(q) s \partial_x s  \Big)  + o(\varepsilon).  \nonumber 
\end{align}
Upon differentiating $\E[s(X_{q'},q')|X_q]$ with respect to $q'$, we obtain the desired conclusion.
Notice in particular that the martingale property implies
\begin{align}
s(x,q) = \E[\tanh(\beta X_1)|X_q=x]\, .\label{eq:SMartingale}
\end{align}

Given any probability distribution $\rho$, we can construct $\Phi$, $s$ and the stochastic
process $\{ X_q : q \geq 0 \}$, so that $s(X_q,q)$ is a martingale. 
However these quantities assume a special physical significance when 
$\rho=\rho^*$ is the unique minimizer of the Parisi functional $\cuP(\rho)$. 
  We will refer to the corresponding diffusion as 
 $\{ X^*_q: q \geq 0 \}$. 
 
 Among other properties, using the replica method,  de Almeida-Lage \cite{de1983internal}
 predicted that the law $X^*_1$ is the asymptotic distribution of the cavity fields. 
 Recall the Gibbs measure $\mu=\mu_{\beta}$ of Eq.~\eqref{eq:gibbs}, where we are here assuming $\lambda=h=0$.
 We already introduced the parametrization of its one-dimensional marginals 
 in terms of effective fields $h_i$: 
\begin{align}
\mu_i(\sigma_i) = \frac{e^{\beta h_i \sigma_i }}{ 2 \cosh (\beta h_i)}. \nonumber 
\end{align}
Consider the probability distribution of the effective field at vertex $i$ (with respect to the
random matrix $\bW$) $\rP_{h_i}^{(n)}$. By exchangeability of the vertices, this is independent of
$i\in\{1,\dots,n\}$, so we can focus on $\rP_{h_1}^{(n)}$. 
In the context of the
 replica symmetric cavity method, we derived the behavior of 
 $\rP_{h_1}^{(n)}$ and argued that (as long as the replica symmetric assumption holds)
 it converges weakly to a Gaussian distribution, see Section \eqref{sec:RSCavityFields}.
 
 
As we discussed above, the replica symmetric ansatz only holds at high temperature (small $\beta$)
or large enough external field. In general, 
$\rP_{h_1}^{(n)}$ converges weakly as $n \to \infty$ to the law of $X_{q_M}^*$,
where $q_M\equiv\sup\{q:\; q\in\supp(\rho^*)\}$ is the maximum overlap in the support of $\rho^*$.
It is a useful exercise for the reader to check that the Gaussian limit of
Section \eqref{sec:RSCavityFields} is recovered in the replica symmetric phase (i.e. 
if $\rho^* = \delta_{q_*}$ is a point mass).

Finally, we note that explicit stationary conditions can be derived
for the Parisi measure $\rho^*$, which involve the diffusion $(X_{t})_{t\in[0,q]}$.
To be definite, we  $\rho(q) = \rho([0,q])$ as a distribution function,
namely a non-negative, non-decreasing left continuous function on the interval $[0,1]$ , with $\rho(1) = 1$.
Omitting technical details, the variation of $\cuP(\rho)$, cf. Eq.~\eqref{eq:ParisiFRSB}, 
with respect to 
a change in $\rho$ can be shown to take the form:
\begin{align}
\cuP(\rho+\eps\Delta)-\cuP(\rho) =
\frac{1}{2}\beta^2\eps\int_{0}^1\Big\{\E\big[s(X_q,q)^2\big]-q\Big\}\, \Delta(q)\, \de q+o(\eps)\, .
\end{align}
Hence we obtain the following stationarity conditions:
\begin{align}
F_{\rho}(q)&\equiv \int_{q}^{q_M}\Big\{\E\big[s(X_t,t)^2\big]-t\Big\}\de t\, ,\\
q\in{\rm supp}(\rho^*) &\;\;\; \Rightarrow\;\;\; F_{\rho}(q) = 0,\, \;\; F'_{\rho}(q) = 0\, ,\\
 q \in [0,1] \, \;  &\;\;\; \Rightarrow\;\;\;  F_{\rho}(q) \ge 0\, ,
 \end{align}
The martingale condition \eqref{eq:SMartingale} can also be used to express 
the above conditions in a different form. 
Given $X_q$,
   let $\{ X^{(1)}_q : q \geq q_0 \}$ and $\{ X^{(2)}_q : q \geq q_0 \}$ be two independent copies 
   the diffusion $X_{t}$ started at $X^{(1)}_q = X^{(2)}_q = X_{q}$. 
 Then we have
\begin{align}
\E\big[s(X_t,t)^2\big] = \E[ \tanh (\beta X^{(1)}_{1} ) \tanh( \beta X^{(2)}_{1})]. \nonumber 
\end{align}

\section{Cavity method with replica symmetry breaking}
\label{sec:CavityRSB}

In Section \ref{sec:replica_symmetric_cavity}, we studied the SK model using the cavity method.
We made a certain number of assumptions and recovered the replica symmetric free energy. 
However, the replica symmetric free energy is incorrect at sufficiently low temperature 
(i.e. $\beta$ sufficiently large). 
This prompts the natural questions:
\textit{Can we sharpen the replica symmetric cavity method to study the 
free energy of the system at low temperature? In particular, can we modify/relax
our assumptions as to derive the RSB free energy?}

In this section, we will accomplish this goal, in the sense that we will
recover the $1$RSB free energy. This approach can be extended to $k$-steps
replica symmetry breaking in a natural way, but we will refrain from 
doing this. 

The replica symmetric cavity method hinges on the absence of global correlations 
in the model. In particular, two fixed coordinates $\sigma_i$, $\sigma_j$ are approximately 
independent when $\bsigma\sim \mu_{\beta}$. We stated a stronger form (or consequence)
of this assumptions as {\sf C1} in Section \ref{sec:replica_symmetric_cavity}.

At low temperature, approximate independence does not hold any longer, and the replica symmetric 
results are inaccurate. Note that the distribution of the overlap, and the fact that it does 
not concentrate is intimately related to this emergence of global correlations.
It is perhaps useful to revisit this point in a simple way. We use the 
shorthand $\<f(\bsigma^1,\dots,\bsigma^k)\>_{\mu}\equiv  \int f(\bsigma^1,\dots,\bsigma^k)\, \mu^{\otimes k}(\de \bsigma)$
for the expectation with respect to the Gibbs measure (even if the expectation is 
actually a finite sum). Considering the mean square 
correlation of two variables $\sigma_i$, $\sigma_j$, we have
\begin{align*}
{\sf Corr}^2_n & = 
\frac{1}{n^2} \sum_{i,j=1}^n\E\Big\{\big(\<\sigma_i\sigma_j \>_{\mu}-
\<\sigma_i\>_{\mu}\<\sigma_j\>_{\mu}\big)^2\Big\}\\
& = \frac{1}{n^2} \sum_{i,j=1}^n\E\Big\{\<\sigma_i\sigma_j\>_{\mu}^2
\Big\}-\frac{2}{n^2} \sum_{i,j=1}^n\E\Big\{\<\sigma_i\sigma_j\>_{\mu}
\<\sigma_i\>_{\mu}\<\sigma_j\>_{\mu}\Big\}+
\frac{1}{n^2} \sum_{i,j=1}^n\E\Big\{\<\sigma_i\>^2_{\mu}\<\sigma_j\>^2_{\mu}
\Big\}\\
& = \frac{1}{n^2} \sum_{i,j=1}^n\E\Big\{\<\sigma^1_i\sigma^1_j\sigma^2_i\sigma^2_j\>_{\mu}\Big\}
-\frac{2}{n^2} \sum_{i,j=1}^n
\E\Big\{\<\sigma^1_i\sigma^1_j\sigma^2_i\sigma^3_j \>_{\mu}
\Big\}+
\frac{1}{n^2} \sum_{i,j=1}^n
\E\Big\{\<\sigma^1_i\sigma^2_j\sigma^3_i\sigma^4_j \>_{\mu} \Big\}\\
& = \E\{Q_{12}^2\}-2\E\{Q_{12}Q_{13}\}+\E\{Q_{12}Q_{34}\}\, .
\end{align*}
Here in the last equation, we introduced the overlap $Q_{ab}\equiv n^{-1}\sum_{i=1}^n
\sigma_i^a\sigma_i^b$, and used $\E\{f(\bsigma^1,\dots,\bsigma^k)\}$
for $\E\int f(\bsigma^1,\dots,\bsigma^k)\, \mu^{\otimes k}(\de \bsigma)$.

If the overlap concentrates, i.e. if the overlap distribution converges weakly to a point mass
$\rho_n\to \rho$, $\rho=\delta_{q_*}$, then the above calculation
shows that the mean square correlation vanishes asymptotically 
 $\lim_{n\to\infty}{\sf Corr}^2_n=0$. 
 We need therefore to give up the assumption that variables $\bsigma=(\sigma_1,\dots,\sigma_n)$
 are approximately pairwise independent under the Gibbs measure. 
 This seems hopeless since non-product form measures are very complex to describe.
 
Remarkably, M\'{e}zard, Parisi and Virasoro \cite{mezard1986sk,SpinGlass} 
developed a set of assumptions that allow one to carry out the cavity calculation,
and recover $k$RSB asymptotics for any $k$. 
This assumption can be described as a specific form of conditional independence of
the coordinates of $\bsigma$. This will be our topic of interest for the rest of the
section (for $k=1$), although we will use a somewhat more mathematical language than
the original physics papers. We will pause in the next section to develop 
some of this language.

\subsection{A digression into Poisson processes} 
\label{sec:Poisson}

Our starting point is a Poisson Point Process (PPP) $\{ \phi^{\alpha}: \alpha \in \mathbb{N} \}$ 
with intensity measure $\Lambda(x) \de x$ on $\mathbb{R}$. This is the unique point process on 
$\mathbb{R}$ such that, letting
$N(B) = \sum_{\alpha} \mathbf{1}(\phi^{\alpha}\in B)$ denote the number of points in the Borel
set $B\subset \mathbb{R}$, the following two properties hold:
\begin{itemize}
\item  For any Borel set $B$,  $N(B)\sim \Poisson(\int_{B} \Lambda(x) \de x)$.
\item  For any collection of disjoint Borel sets $B_1, \cdots, B_k$, the 
random variables $N(B_1), \cdots, N(B_{k})$ are mutually independent.
\end{itemize}
This definition generalizes immediately to PPP's on $\mathbb{R}^d$ for any $d \geq 1$. 
We recall that an immediate consequence of the definition is that, for any (measurable)
bounded function $g:\reals\to\reals$:
\begin{align}\label{eq:ExpectationPPP}
\E\Big\{\sum_{\alpha}g(\phi^{\alpha})\Big\} = \int g(x)\,\Lambda(x)\, \de x\, .
\end{align}
(This can be proved by first considering the case in which $g$
is a \emph{simple function}, i.e. $g(x) = \sum_{k=1}^mc_k\bfone_{B_k}$ for some Borel sets $B_k$.)

In fact, the following characterization of PPP's (see e.g. \cite[Chapter 7.4]{daley1998introduction})
will be useful. 
\begin{lemma}
\label{lemma:marks2}
$\{x_{\alpha}: \alpha \in \mathbb{N} \}$ follows ${\rm{PPP}}$ with intensity $\rho$ if and only if for all bounded $g : \mathbb{R} \to \mathbb{R}$, 
\begin{align}
\E\Big[\exp{\sum_{\alpha} g(x_{\alpha})} \Big] = \exp\Big[ \int \Big( {\rm{e}}^{g(x)} - 1 \Big) \rho(\de x) \Big]. \nonumber 
\end{align}
Further, the above are equivalent to this identity holding for all $g$ bounded continuous.
\end{lemma}
The reader can easily check that this identity holds when $g$ 
is a simple function.

We will be specifically interested in the PPP $\{ \phi^{\alpha}: \alpha \in \mathbb{N} \}$
 with intensity function $\Lambda(x) = \exp[- m x]$ for $x \in \mathbb{R}$, $m>0$. 
 The following properties of the process will be specifically relevant for us. 
\begin{lemma}\label{lemma:PPP-Basic}
The PPP $\{\phi^{\alpha}: \alpha \in \mathbb{N} \}$ with intensity $\Lambda(x)\de x = e^{- mx}\de x$ 
has the following properties:
\begin{enumerate}
\item The process has a finite maximum almost surely, i.e. there exists a point
$\phi^{\alpha_0}$ such that $ \phi^{\alpha_0}>\phi^{\alpha}$ for every $\alpha\neq \alpha_0$. 
\item  Almost surely, there exists a (measurable) reordering of the points 
$\{\phi^{\alpha} : \alpha \in \mathbb{N} \}=\{\phi^{(k)} : k \in \mathbb{N} \}$ 
such that $\phi^{(1)} \geq \phi^{(2)} \geq \cdots$. 
\item For $0<m<1$, $\sum_{\alpha} e^{\phi^{\alpha}} < \infty$ almost surely. 
\end{enumerate}
\end{lemma}
\begin{proof}
To prove point 1, note that the expected number of points $\phi^{\alpha}\ge 0$ is
\begin{align}
\E N(\reals_{\ge 0}) = \int_0^{\infty}e^{-mx}\, \de x = \frac{1}{m}<\infty\, .
\end{align}
Hence, almost surely, the process contains a finite number of non-negative
(distinct) points. Any time this happens, there is a maximum point as well, since a finite
set always has a maximum.

Let us now consider point 3. We separate the contributions of positive and negative points:
\begin{align*}
\sum_{\alpha} e^{\phi^{\alpha}} &= 
\sum_{\alpha} e^{\phi^{\alpha}}\bfone_{\phi^{\alpha}\ge 0} 
+\sum_{\alpha} e^{\phi^{\alpha}} \bfone_{\phi^{\alpha}< 0}\\
& \equiv S_{\ge 0}+S_{<0}\, .
\end{align*}
By point 1, the first sum $S_{\ge 0}$ only contains finitely many terms and therefore is finite almost surely.
As for the second sum, by Eq.~\eqref{eq:ExpectationPPP} we have
\begin{align*}
\E\{S_{<0}\} = \int_{-\infty}^0 e^x e^{-mx}\, \de x = \frac{1}{1-m}<\infty\, .
\end{align*}
Therefore $S_{<0}<\infty$ almost surely, which proves the claim. Note that
here we used in a crucial way the fact that $m\in (0,1)$.

Finally, point 2 is both very easy to accept and somewhat subtle from
a measure-theoretic perspective. We refer the reader to \cite[Chapter 2]{panchenko2013sherrington} for 
a rigorous proof.
\end{proof}

 Our next lemma introduces the crucial property that makes this PPP particularly important:
 its behavior under random i.i.d. perturbations. 
\begin{lemma}
\label{lemma:marks1}
Let $\{\Delta^{\alpha} : \alpha \in \mathbb{N} \}$ be a sequence of i.i.d.
random variables (`marks') independent of the PPP $\{ \phi^{\alpha} : \alpha \in \mathbb{N} \}$.
Assume that the common distribution $\nu$ of the $\Delta^{\alpha}$ is 
such that $C_{\nu} = \int \exp(m \Delta)\, \nu(\de \Delta) < \infty$. Then we have, 
\begin{align}
\{ \phi^{\alpha} + \Delta^{\alpha} : \alpha \in \mathbb{N} \} \sim {\rm{PPP}} (C_{\nu} \Lambda). \nonumber 
\end{align}
\end{lemma}
\begin{proof}
The proof follows from a direct application of Lemma \ref{lemma:marks2}. 
We first note that $\{ (\phi^{\alpha}, \Delta^{\alpha}): \alpha \in \mathbb{N} \}$ is a PPP on
$\mathbb{R}^2$ with intensity $ \exp(- mx) \de x \otimes  \nu(\de\Delta)$. Thus, 
 for any $g: \mathbb{R} \to \mathbb{R}$ bounded 
\begin{align}
\E\Big\{ \exp \Big[ \sum_{\alpha} g(\phi^{\alpha} + \Delta^{\alpha}) \Big] \Big\} &= 
\exp\Big\{ \int ( e^{g(x+y)} -1 )  e^{-mx}\,\de x \,\nu(\de y) \Big\} \nonumber \\
&= \exp\Big\{ \int  (e^{g(x)}-1) e^{-m(x - y)} \, \de x \, \nu(\de y)   \Big\} \\
&= 
\exp\Big\{ C_{\nu} \int  ( e^{g(x)} -1 ) \, e^{- m x}\,  \de x\Big\}\, . \nonumber 
\end{align}
By Lemma \ref{lemma:marks2}, this completes the proof.
\end{proof}
 Lemma \ref{lemma:marks1} establishes a stability property of this 
 PPP. First note that the intensity $C_{\nu}\Lambda(x)\de s=C_{\nu}e^{-mx}\de x$ can be rewritten as
 $C_{\nu}\Lambda(x)\de x = \Lambda(x-x_0)\de x$ for $x_0 = -m^{-1}\log C_{\nu}$,
 i.e. the change in intensity is just a translation. Hence the lemma
 can be restated as follows: if we perturb the PPP by i.i.d.
 random shifts, the resulting point process is distributed according to the original one,
 shifted by a deterministic amount $x_0$.
 
 An instructive image is proposed in  \cite{ruzmaikina2005characterization}.
 Imagine that the $\{\phi^{\alpha}\}$ are positions of cars in a car race. In a given interval of
 time (say, one minute), each car moves by an independent and identically distributed 
 amount $\Delta^{\alpha}$. The new car positions are therefore $\{\phi^{\alpha}+\Delta^{\alpha}\}$.
 The overall race moved forward by a distance $x_0$ but---once this shift is accounted for---the distribution of car position is unchanged apart from a reordering.
 This image motivates the name Indy500 for this process.
 
Liggett \cite{liggett1978random} studied this PPP in connection with interacting 
particle systems. His analysis implies that---under certain technical conditions---it is the only point processes that enjoys the above stability property.
Of course, uniqueness holds up to a global shift or, equivalently, multiplying the intensity
$e^{-mx}\de x$ by a constant.
Notice that, as a consequence of stationarity, the 
point process of gaps $\{\phi^{(1)}-\phi^{(\alpha)}:\; \alpha\in\naturals\}$ remains unchanged after 
random shifts.
Ruzmaikina-Aizenman \cite{ruzmaikina2005characterization} showed that 
all the point processes satisfying this weaker condition are convex combinations of PPP's
with intensity $e^{-m x}\de x$ (with varying $m$).

 A strengthening of the above invariance property is established in  
 the following result.
\begin{lemma}
\label{lemma:marks3}
Let $(\cX,\cF,\mu)$ be a probability space and $\Delta: \mathcal{X} \to \mathbb{R}$,
 $h: \mathcal{X} \to \mathbb{R}$ two measurable functions. Let 
 $\{ (\phi^{\alpha} , x_{\alpha}): \alpha \in \mathbb{N} \}$ be a PPP on $\mathbb{R} \times \mathcal{X}$ 
 with intensity measure $\exp(- m \phi ) \de \phi \otimes \nu(\de x)$. We assume
  $C_{\nu} = \int e^{m \Delta(x)} \nu(\de x) < \infty$ for some $m >0$ and define
   the probability measure $\nu_{m}(\de x) \equiv C_{\nu}^{-1} e^{m \Delta(x)} \nu(\de x)$. 
   Then we have:
\begin{align}
\Big\{ \big( \phi^{\alpha} + \Delta(x_{\alpha}), h(x_{\alpha}) \big) :\;\alpha\in\naturals\Big\} 
\sim {\rm{PPP}} \Big(C_{\nu} e^{-m\phi} \de \phi \otimes   \nu_m\circ h^{-1} \Big) . \nonumber  
\end{align}
(Here $\nu_{m}\circ h^{-1}$ denotes the law of $h(X)$ where $X\sim \nu_{m}$.)
\end{lemma}
\begin{proof}
The proof proceeds again by an application of Lemma \ref{lemma:marks2}. We have, for any bounded continuous $g: \mathbb{R}^2 \to \mathbb{R}$, 
\begin{align}
\E\Big\{ \exp\Big[ \sum_{\alpha} g(\phi^{\alpha} + \Delta(x_{\alpha}), h(x_{\alpha})) \Big] \Big\}&
 = \exp\Big\{ \int \Big[ e^{g(u+\Delta(x) , h(x)) } -1 \Big] e^{-m u} \de u\, \nu(\de x) \Big\} \nonumber \\
&= \exp\Big\{ \int \exp\Big[ e^{g(u, h(x))} -1  \Big] e^{m \Delta(x) } e^{-m u} \de u \,\nu(\de x) \Big\} \nonumber \\
&= \exp\Big\{ C_{\nu} \int \Big[ e^{g(u, h(x))} -1 \Big] e^{-m u} \de u \, \nu_{m}(\de x) \Big\}. \nonumber 
\end{align}
This completes the proof.
\end{proof}

\subsection{The 1RSB cavity recursion} 

Armed with this machinery, we return to the study of the cavity method.
Recall that the cavity method adds the $(n+1)^{th}$ spin to a 
 system comprising of $n$ spins (in our earlier analysis, we deleted a node, 
 but this is basically the same analysis), computes the effective field on the 
 $(n+1)^{th}$ spin in terms of the effective fields in the system comprising $n$ spins, 
 and finally imposes the constraint that these fields should have the same law in 
 the large system limit. 
 The definition of cavity fields and effective fields was given in 
 Eq.~\ref{eq:cavityfields}, and in Eq.~\eqref{eq:CavityRec1} we derived the recursion 
\begin{align}
&h_{n+1} = \sum_{i=1}^{n} u(W_{n+1,i}, h_{i \to n+1})+o_n(1),\\
&u(W,h) = \frac{1}{\beta} {\rm{atanh}} \Big[ \tanh(\beta W) \tanh(\beta h) \Big]\, . 
\label{eq:rs_approx}
\end{align}
%

As discused above, at low temperature, the approximate independence that
was used to derive the last equation no longer holds. As already discussed in the previous
chapter, within a $1$RSB phase, we expect the Gibbs measure $\mu$ to decompose (approximately) 
into the convex combination of $O(1)$ `pure states' $\mu^{\alpha}$.
Namely, letting $\{ +1,-1 \}^n = \uplus_{\alpha=1}^M\Omega_{\alpha}\uplus\Omega_0$
a partition of the set of spin configurations, we have 
 \begin{align*}
 \mu(\bsigma) &= \sum_{\alpha=1}^Mw_{\alpha}\, \mu^{\alpha}(\bsigma)+\eps(\bsigma)\, ,\;\;\;\;
  \sum_{\bsigma}|\eps(\bsigma)|=o_n(1)\, ,\\
 \mu^{\alpha}(\bsigma) & \equiv \frac{1}{Z_{\alpha}}\, e^{\beta H(\bsigma)}\bfone_{\bsigma\in \Omega_{\alpha}}
\,,\;\;\;\;\; Z^{\alpha} \equiv\sum_{\bsigma\in \Omega_\alpha}e^{\beta H(\bsigma)}\, ,
\;\;\;Z\equiv\sum_{\bsigma\in\{+1,-1\}^n}e^{\beta H(\bsigma)}\, .
\end{align*}
(Recall that $H(\bsigma) = \langle \bW, \bsigma \bsigma^{\sT} \rangle$ is the Hamiltonian.) 
Each component $\mu^{\alpha}$ is referred to as a `pure state' under the $1$RSB heuristic. 
   Under this decomposition, we have, $Z = \sum_{\alpha=1}^{M} Z^{\alpha}+o_n(1)$. Each pure state is expected to have the same energy on the macroscopic scale, so that $\log Z^{\alpha} - \log Z = O(1)$ for all $\alpha \in \{1, \cdots, M\}$. We expect that as $n \to \infty$, the number of pure states $M$ also diverges; further, we expect that a suitably re-scaled version of the ``pure-state" free energies $\{ \phi^{\alpha}=  \log Z^{\alpha} - a_n \}$ converges to a PPP with intensity $\exp(-m x) \de x$, for some $ 0\leq m  < 1$. 
   The PPP interpretation of the pure state free energies implies that $w_{\alpha} \propto \exp[ \phi^{\alpha}]$ is distributed as Poisson-Dirichlet distribution with parameter $m$. 

Within a $1$RSB heuristic we expect each $\mu^{\alpha}$ to be approximately product form.
More precisey, we expect that average correlations between $\sigma_i$ 
and $\sigma_j$ for fixed $i\neq j$ to be small when  $\bsigma\sim \mu^{\alpha}$.
Notice that we can define effective fields for $\mu^{\alpha}$ exactly as we did for $\mu$.
Namely, letting $\mu^{\alpha}_{n+1}$ be the $(n+1)$-th marginal of $\mu^{\alpha}$,
we introduce the parametrization 
 \begin{align*}
 \mu^{\alpha}_{n+1}(\sigma_{n+1}) = \frac{e^{\beta h^{\alpha}_{n+1}\sigma_{n+1}}}{2\cosh(\beta h^{\alpha}_{n+1}
 \sigma_{n+1})}\, .
 \end{align*}
Analogously, we denote the cavity fields within state $\alpha$ 
by  $\{ h^{\alpha}_{i \to n+1}\}_{i\le n}$. 
 We expect that the number of pure states $M\to \infty$ as $n \to \infty$, and 
we would like to  to study the distribution of $h_{i\to n+1}^{\alpha}$ for 
 a ``random" pure state $\alpha$. 
 This is an important difference with respect to the replica
  symmetric cavity method at this point: even for fixed disorder variables 
  $\{W_{ij}: 1\leq i,j \leq n \}$, $h^{\alpha}_{i \to n+1}$ has a different
   value across  pure states $\alpha$ and 
   we will describe this variability by a probability distribution  on the real
    line which we denote by $\onu_i$. 
Note that the distribution $\onu_i$ will in general depend on $i$ and the realization of 
       $\{W_{ij} : 1 \leq i , j \leq n \}$ variables:
       therefore $\onu_1,\dots,\onu_n$ are random probability distributions.
       Our objective will be to derive a distributional recursion for 
       their law.


When the $(n+1)^{th}$ spin is added to a system with $n$
spins, the free energy $\phi^{\alpha}\equiv \log Z^{\alpha}$ of each state $\alpha$ changes, and thus changes their 
relative weights $w_{\alpha}$.
Eventually we would like to focus on the effective fields in states with large 
free energy, and therefore we should keep track of these changes as well. 

We already computed the effective field on $\sigma_{n+1}$ in Section
\ref{sec:replica_symmetric_cavity}:  
\begin{align}\label{eq:CavityFieldUpdate}
h^{\alpha}_{n+1} &= \sum_{i=1}^{n} u(W_{n+1,i}, h^{\alpha}_{i \to n+1})+o_n(1) \equiv {\sf h}(\{h^{\alpha}_{i \to n+1} \})+o_n(1)\, .
\end{align}
A similar computation (under the same assumptions) allows one to compute the change in free energy:
\begin{align}
\Delta^{\alpha}&=\log \Big[ \frac{Z_{n+1}^{\alpha}(\beta_{n+1})}{Z_{n}^{\alpha}(\beta)} \Big]
\nonumber \\
 &=  \Delta_0  + \log 2 \cosh\Big[ \beta \sum_{i=1}^{n} u(W_{n+1,i}, h_{i \to n+1}^{\alpha}) \Big] 
- \sum_{i=1}^{n} \log \cosh \Big[ \beta u(W_{n+1,i}, h_{i \to n+1}^{\alpha}) \Big] \nonumber \\
&= \Delta_0+\Delta(\{ h^{\alpha}_{i \to n+1} \} ) +o_n(1)\, , \label{eq:delta}
\end{align}
 where the constant $\Delta_0\equiv \sum_{i\le n}\log\cosh\beta W_{i,n+1}$ is independent of
 the state $\alpha$ and hence will be irrelevant in the cavity recursion.
 
In words, the addition of the new spin transforms the weights of the pure states 
by shifting their free energies $\{\phi^{\alpha} \}$ to $\{ \Delta^{\alpha}+ \phi^{\alpha} \}$ .
The shifts $\Delta^{\alpha}$ depend on the cavity fields.
If we assume that these shifts are i.i.d. across $\alpha$, and that the distribution of
$\{\phi^{\alpha}\}$ is independent of $n$ up to a global shift, then
it follows that $\{\phi^{\alpha}\}$ must be a PPP with intensity $C e^{-mx}\de x$
(as we saw in the last section). Since a global shift in these free energies is immaterial, we can set 
$C=1$.

We denote by $x_{\alpha} = ( h^{\alpha}_{i \to n+1}: 1\leq i \leq n)$ the collection of cavity
fields and think of this quantity as
 a ``mark" taking values in $\reals^{n}$. 
 Following the above discussion, we obtain that 
 $\{(\phi^{\alpha},x_{\alpha}): \alpha\in\naturals\}$ (approximately for large $n$) is a PPP:
 \begin{align}
\{(\phi^{\alpha},x_{\alpha}): \alpha\in\naturals\}\sim  
\textrm{PPP}\big(e^{-mx}\de x  \otimes \prod_{i=1}^n \onu_i(\de h_i)\big)\, .
 \end{align}
Lemma \ref{lemma:marks3} 
 suggests that  
\begin{align}
\{\phi^{\alpha}+ \Delta^{\alpha} , h^{\alpha}_{n+1} \} &\sim 
\textrm{PPP}\big( C_{\nu} e^{-mx } \de x \otimes \nu_{m} \circ {\sf h}^{-1}\big), \\
\nu_m(\de x) &\equiv\frac{1}{C_{\nu}} e^{m \Delta(x)} \prod_{i=1}^n  \onu_i(\de x_i)\, ,
\;\;\;C_{\nu} \equiv\int e^{m \Delta(x)} \prod_{i=1}^n  \onu_i(\de x_i)\, .
\end{align}
Hence, in the system with $n+1$ spins $\{(\phi_{n+1}^{\alpha},h^{\alpha}_{n+1})\}$
is approximately distributed as a PPP with intensity $C_{\nu}e^{-mx}\de x \otimes \onu_{n+1}$,
consistently with our assumption at size $n$. Further $\onu_{n+1}$ satisfies,
for any $x\in\reals$,
\begin{align} 
&\onu_{n+1}(h_{n+1}\le x) \approx_n \frac{1}{Z^{\nu}_{n+1}} \int \mathbf{1}\Big( \sum_{i=1}^{n} u(W_{n+1,i}, h_{i \to n+1}) \le x \Big) \, e^{m\Delta(\{ h_{i \to n+1} \})}\,
  \prod_{i=1}^{n} \onu_i(\de h_{i \to n+1 })\, \label{eq:dist_recursion}
\end{align}
(Here $Z^{\nu}_{n+1}$ is a normalization constant and
we neglected errors due to finite $n$).

Note that this procedure specifies $\onu_{n+1}$ as a (complicated) function of the 
probability measures $\{\onu_i : i \leq n\}$. In the next section we will see that
this recursion can be further simplified for large $n$, but it is instructive
to step back and have a look at the mathematical structure that emerges from the 
1RSB cavity method. In settings that are more complicated than the SK model
(e.g. models on sparse graphs), this structure survives, and further
simplifications are not possible.

Denoting by ${\sf Pr}(\Omega)$
the space of probability measures over a Polish space $\Omega$, the above recursion 
\eqref{eq:dist_recursion} defines a 
 map
$U_n : ({{\sf Pr }}(\mathbb{R}))^n \to {{\sf Pr }}(\mathbb{R})$.  Now, assume that when 
viewed as a function of $\{W_{ij}: 1\leq i <j \leq n+1\}$ variables, $\onu_i$ are i.i.d. random 
measures with distribution $\cQ$.
We expect therefore that the law of $U_n(\onu_1, \cdots, \onu_n)$ is approximately 
$\cQ$, when $\{ \onu_i : i \leq n \}$ are i.i.d. samples from $\cQ$. In other words:
\begin{align}
\onu_{n+1}\ed U_n(\onu_1, \cdots, \onu_n)+o_n(1)\, .
\end{align}
Here $o_n(1)$ means that a suitable distance\footnote{As our derivation is heuristic, we will not
make the distance notion precise, but the mathematically minded reader can think, for instance of bounded Lipschitz distance.}
  between $\onu_{n+1}$ and 
$U_n(\onu_1, \cdots, \onu_n)$ vanishes as $n\to\infty$. 

Hence the 1RSB cavity method leads, in general, to a distributional recursion in the space
of probability measures ${\sf Pr}(\reals)$. Solving these recursion amounts to finding 
a distributional fixed point $\cQ$.

\subsection{Simplifying the cavity recursion} 

We emphasize that our discussion of the 1RSB cavity method above is quite general,
and does not depend too heavily on the features of the model. For instance,
the same argument can be carried out for spin glasses on sparse random graphs, 
with very minor modifications. 
 For the SK model, it is possible to simplify the distributional recursion
 \eqref{eq:dist_recursion}, by using a central limit theorem (CLT) heuristic.
  Throughout this calculation, we will
  write $A_n \approx_n B_n$ to mean that ${\sf dist}(A_n,B_n)\to 0$ for a suitable notion of distance,
  which we will not specify.

We begin by simplifying the free energy shift  
$\Delta(\{ h_{i} \}_{i\le n} )$ of Eq.~\eqref{eq:delta}.
Using $W_{n+1,j} = O(1/\sqrt{n})$, and the fact that, for  $x$ small, 
$\log\cosh (x) = x^2/2+O(x^4)$, and that $u(W,h)= W\tanh(\beta h) +O(n^{-3/2})$, we get 
\begin{align}
\Delta(\{ h_{i} \}_{i\le n} )&=
\log 2 \cosh\Big[ \beta \sum_{i=1}^{n} u(W_{n+1,i}, h_{i }) \Big] -
\sum_{i=1}^{n} \log \cosh \Big[ \beta u(W_{n+1,i}, h_{i}) \Big] \nonumber\\
&= \log 2 \cosh\Big[ \beta \sum_{i=1}^{n} W_{n+1,i}\tanh(\beta h_{i }) \Big]-
\frac{\beta^2}{2} \sum_{i=1}^{n} W_{n+1,i}^2 \big( \tanh(\beta h_{i}) \big)^2+O(n^{-1})\nonumber\\
 & =  \log 2 \cosh\Big[ \beta \sum_{i=1}^{n} u(W_{n+1,i}, h_{i }) \Big] + \tilde{\Delta}_n\,.
 \label{eq:DeltaApprox}
\end{align}
The term $\tilde{\Delta}_n$ is approximated by the sum of $n$ independent sub-Gaussian quantities,
each of order $1/n$, and hence will tightly concentrate around a constant. We can therefore
neglect it in the cavity recursion (as it will be cancelled by the normalization of
the measure $\onu_{n+1}$).

We next approximate the cavity field update \eqref{eq:CavityFieldUpdate} to get
\begin{align}
{\sf h}(\{h_{i} \}) =\sum_{i=1}^{n} W_{n+1,i}\tanh(\beta h_{i})+O(n^{-1})\,.
\end{align}
Using the last equation together with Eq.~\eqref{eq:DeltaApprox}, 
 we can 
replace Eq.~\eqref{eq:dist_recursion} by the following  distributional 
recursion
\begin{align}
\onu_{n+1} (h_{n+1}\le x) &\approx_n
\frac{1}{\tilde{Z}^{\nu}_{n+1}}\int \mathbf{1}\Big( \sum_{i=1}^{n} \tu_i(h_i) \le x \Big)
\,  \Big( 2 \cosh\Big[\beta \sum_{i=1}^n \tu_i(h_i) \Big] \Big)^m \prod_{i=1}^{n} \onu_i(\de h_{i} ), \label{eq:dist_fixed}\\
& \tu_i(h_i) \equiv W_{n+1,i} \tanh(\beta h_i) \, .\nonumber
\end{align}

Next, we analyze the limit of this distributional recursion
(it is important to remind ourselves that the distributions $\onu_{n+1}$ are themselves random).
Notice that  Eq.~\eqref{eq:dist_fixed} takes the form of the composition of three steps:
$(i)$ a change of variables from $h_i$ to $\tu_i(h_i)$; $(ii)$~a convolution;
$(iii)$ a reweighting by the factor $(2\cosh(\cdots ))^m$. To emphasize this structure,
we write
\begin{align}
\onu_{n+1} (\de h) &\approx_n \frac{1}{\tilde{Z}^{\nu}_{n+1}} \big(2\cosh\beta h\big)^m
P_{n+1}(\de h)\, ,\\
P_{n+1} &= \mbox{Law}(Y_{n+1})\, ,\;\;\;  Y_{n+1}\equiv \sum_{i=1}^{n} W_{n+1,i} \tanh(\beta h_i)\, .
\end{align}
Conditional on $\bW$ (and therefore on the measures $\onu_1,\dots,\onu_n$), 
$Y_{n+1}$ is the sum of $n$ independent terms, each of order $1/\sqrt{n}$. We can therefore
hope to apply the CLT and conclude that
\begin{align}
P_{n+1}& \approx_n \normal(\oh_{n+1}, \sigma_{n+1}^2)\, ,\\
\oh_{n+1}  & = \sum_{i=1}^{n} W_{n+1,i} \E_{\onu_i}\big\{\tanh (\beta h_i)\big\}\, ,
\label{eq:RecOH}\\
\sigma_{n+1}^2 & = \frac{1}{n} \sum_{i=1}^{n} {{\rm Var }}_{\onu_i}\big\{\tanh(\beta h_i)\big\}\, .
\label{eq:RecSigma}
\end{align}
We therefore conclude that, for large $n$, the measures $\onu_i$, $i\le n+1$,
 are well approximated by a two parameter family
\begin{align}
\onu_i(\de h_i) \approx_n \frac{1}{\hat\cZ^{\nu}_i}\,
(2\cosh \beta h_i)\, \phi_G(h_i;\oh_i,\sigma^2_i)\, \de h_i\,  ,\label{eq:1RSBOnu}
\end{align}
where $\phi_G(x;a,v^2)$ denotes the Gaussian density with mean $a$ and variance $v^2$.
The parameters $\oh_i$ and $\sigma^2_i$ satisfy the recursion
of Eqs.~\eqref{eq:RecOH}, \eqref{eq:RecSigma}. 

The above calculation derives a complete description of the 
variability of cavity fields with respect to
the pure state, \emph{for a given vertex $i$ and a given realization of $\bW$}.
We next turn our attention to the distribution of the pair $(\oh_i,\sigma^2_i)$ with respect
to randomness $\bW$. 

We begin by considering Eq.~\eqref{eq:RecSigma}.
The measures $\{\onu_i: i \leq n\}$ are i.i.d. with common law  
$\cQ$. Equivalently, because of Eq.~\eqref{eq:1RSBOnu}, we  can think of
$\cQ$ as the common distribution of  $(\oh_i,\sigma_i^2)$.
 Thus using a law of large numbers, we expect as $n \to \infty$,
\begin{align}
\sigma_n^2 \to \sigma^2 : = \E_{\nu \sim  \cQ} \Big[  {{\rm Var }}_{h \sim \nu}[\tanh(\beta h)] \Big]. \nonumber 
\end{align}

Next consider Eq.~\eqref{eq:RecOH}.
Again, since the measures $\{\onu_i: i \leq n\}$ are i.i.d. with common law  
$\cQ$, and each term on right-hand side is of order $1/\sqrt{n}$ (because of the factor $W_{n+1,i}$),
we can apply the central limit theorem, and conclude that 
$\oh_{n+1}$ is approximately Gaussian with mean $0$ and
variance $\tau_n^2 = \frac{1}{n} \sum_{i=1}^{n} \Big( \E_{\onu_i} [ \tanh(\beta h_i)] \Big)^2$. 
We thus conclude that $\cQ$ (which we now think as a distribution over $(\oh_i,\sigma^2)$)
has a particularly simple form:
\begin{align}
\cQ(\de\oh_i,\de\sigma_i^2) \approx_n \phi_G(\oh_i;0,\tau^2)\de\oh_i\otimes \delta_{\sigma^2}\,,
\end{align}
where $\tau^2$, $\sigma^2$ are deterministic parameters that satisfy
the following equations:
\begin{align}
\sigma^2 & =  \E_{\oh\sim\normal(0,\tau^2)} 
\Big[  {{\rm Var }}_{h \sim \onu_{\oh}}[\tanh(\beta h)] \Big]\, ,\label{eq:FixedPt1RSB-A}\\
\tau^2 & = \E_{\oh\sim\normal(0,\tau^2)} \Big[ \Big( \E_{h \sim \onu}[\tanh(\beta h) ]\Big)^2 \Big]\, ,\label{eq:FixedPt1RSB-B}\\
\onu_{\oh}(\de h)&\equiv \frac{1}{\cZ(\oh)}\,
(2\cosh \beta h)\, \phi_G(h;\oh,\sigma^2)\, \de h\,  .
\end{align}
In other words, the distribution of the cavity field $h$ in 
these equations takes the form $h = \tau\,  G_0 + \sigma Z_1$, 
where  $G_0\sim \normal(0,1)$, and conditional on $G_0$,  $Z_1\sim \onu_{\tau G_0}$. 
We can write expectation with respect to $Z_1$ in terms of expectation
with respect to $G_1\sim\normal(0,1)$ independent of $G_0$, via a change of measure
\begin{align}
\E_{Z_1}[f(G_0,Z_1)] = \frac{\E_{G_1} [ f(G_0, G_1 ) ( 2 \cosh ( \beta [\tau G_0 + \sigma G_1] ) )^m ]  }{ \E_{G_1} [ (2 \cosh(\beta [ \tau G_0 + \sigma G_1 ]))^m]}. \nonumber 
\end{align}
Using this representation and  defining $q_0\equiv \tau^2$, $q_1\equiv \tau^2+\sigma^2$,
we can rewrite Eqs.~\eqref{eq:FixedPt1RSB-A}, \eqref{eq:FixedPt1RSB-B} as
\begin{align}
q_0 &= \E_{G_0} \Big[ \Big( \frac{\E_{G_1}[\tanh(\beta h) (2 \cosh (\beta h) )^m]}{ \E_{G_1}[ (2 \cosh (\beta h ))^m ] } \Big)^2 \Big], \\ 
q_1 &= \E_{G_0} \Big[ \frac{\E_{G_1}[ (2 \cosh( \beta h))^m (\tanh(\beta h))^2 ]}{ \E_{G_1} [ (2 \cosh(\beta h) )^m ]} \Big], \\
& h = \sqrt{q_0} G_0 + \sqrt{q_1 - q_0} G_1\, .
\end{align}
These coincide exactly with  the stationary point conditions for the $1$RSB free energy functional
derived in Section \ref{sec:kRSB-Replica}.

We note that a similar cavity analysis can be carried out for $k$-steps replica 
 symmetry breaking. The key idea is that now the decomposition in states has a 
 hierarchical structure with states clustered in ancestral states, and so on.
We will describe this hierarchical structure in section \ref{sec:SKrigorous}.

\subsection{Computing the free energy density}
\label{sec:1RSB_Cavity_Free}
The free energy density may also be computed using the cavity method,
under a $1$RSB scenario.
As in the  replica symmetric case, the starting point is to write the 
asymptotic free energy density as
\begin{align}
\phi(\beta) & = \lim_{n\to\infty}A_n\, ,\;\;\; A_n\equiv
\E \Big[ \log\frac{Z_{n+1}(\beta)}{Z_n(\beta)}  \Big]\, ,\\
A_n &= \E\Big[\log  \frac{Z_{n+1}(\beta_{n+1})}{Z_{n}(\beta)} \Big] -  
\E\Big[\log \frac{Z_{n+1}(\beta_{n+1})}{Z_{n+1}(\beta)} \Big] \label{eq:partition_intermediate} \\
& \equiv \E\{B_{n}^{(1)}\}- \E\{B_{n}^{(2)}\}\, . 
\end{align}
for $\beta_{n+1} = \beta\sqrt{n+1/n}$.  The first term captures the effect 
of adding a spin going from $n$ to $n+1$ spins, and the second is the effect of rescaling 
the temperature on the original $n$ spins. (Note that the change  is negligible for what concerns the 
couplings between $\sigma_{n+1}$ and $\sigma_i$, $i\le n$.)

Let us start from the first term, that is the effect of adding a spin.
We already saw in the previous section that this has the effect of shifting the free energy 
of pure states from $\phi^{\alpha}$ to $\phi^{\alpha}+\Delta^{\alpha}$, where, by 
Eq.~\eqref{eq:delta},
\begin{align}
\Delta^{\alpha}&= \Delta_0+\Delta(\{ h^{\alpha}_{i \to n+1} \} ) +o_n(1)
\end{align}
Further, by the Taylor expansion $\log\cosh(x) = x^2/2+O(x^4)$ and the law of large
numbers:
\begin{align}
\Delta_0\equiv \sum_{i\le n}\log\cosh\beta W_{i,n+1} =\frac{\beta^2}{2}+O(n^{-1})\,.
\end{align}
Since the total partition function is the sum of partition functions for each 
pure state, we have
\begin{align}
B^{(1)}_n = \log\Big[\sum_{\alpha}e^{\phi^{\alpha}+\Delta^{\alpha}}\Big]
- \log\Big[\sum_{\alpha}e^{\phi^{\alpha}}\Big] \, .
\end{align}
Under the 1RSB assumption, the shifts $\Delta^{\alpha}$ are independent of the free energies 
$\phi^{\alpha}$.
We saw in Lemma \ref{lemma:marks1}
that (asymptotically)  $\{\phi^{\alpha}+\Delta^{\alpha}\}_{\alpha\in\naturals}$ are distributed as 
$\{\tilde{\phi}^{\alpha}+x_0\}_{\alpha\in\naturals}$ where 
$\{\tilde{\phi}^{\alpha}\}$ is a copy of the process $\{\phi^{\alpha}\}$,
and $x_0\equiv m^{-1}\log \E_{s}[e^{m\Delta}]$, where $\E_{s}$ denotes expectation with 
respect to the pure states (at given disorder.)
 Therefore
\begin{align*}
B^{(1)}_n &\approx_n 
\log\Big[\sum_{\alpha}e^{\tilde\phi^{\alpha}}\Big] +\frac{1}{m}\log \E_s[e^{m\Delta}]- \log\Big[\sum_{\alpha}e^{\phi^{\alpha}}\Big] \\
& \approx_n \frac{\beta^2}{2}+\frac{1}{m}\log \E_{s}\exp\big\{m\Delta(\{ h_{i \to n+1} \})\big\}+
\tilde{F}-F\, .
\end{align*}
Here $\tilde F$ and $F$ are identically distributed random variables. Therefore
\begin{align*}
\E\{B^{(1)}_n\} &\approx_n
\frac{\beta^2}{2}+\frac{1}{m}\E \log \E_{s}\exp\big\{m\Delta(\{ h_{i \to n+1} \})\big\}, .
\end{align*}
Recall that by the approximations in Eq.~\eqref{eq:DeltaApprox}, we have
\begin{align*}
\Delta(\{ h_{i} \}_{i\le n} )&=
\log 2 \cosh\Big[ \beta \sum_{i=1}^{n} W_{n+1,i}\tanh(\beta h_{i }) \Big]-
\frac{\beta^2}{2} \sum_{i=1}^{n} W_{n+1,i}^2 \big( \tanh(\beta h_{i}) \big)^2+O(n^{-1})\\
&=\log 2 \cosh\Big[ \beta \sum_{i=1}^{n} W_{n+1,i}\tanh(\beta h_{i }) \Big]-
\frac{\beta^2}{2n} \sum_{i=1}^{n} \E_{h_i\sim \onu_i}\big\{\big( \tanh(\beta h_{i}) \big)^2
\big\}+o_n(1)
\end{align*}
Further, by the CLT  $\sum_{i=1}^{n} W_{n+1,i}\tanh(\beta h_{i })$ 
is approximately Gaussian, with variance $(q_1-q_0)$ and mean 
$\oh_{n+1}: =  \sum_{i=1}^{n} W_{n+1,i}\E_{\onu_i}\tanh(\beta h_{i })$.
We therefore conclude that
\begin{align*}
\E\{B^{(1)}_n\} \approx_n&
\frac{\beta^2}{2}- \frac{\beta^2}{2}\E_{\onu\sim \cQ}\E_{h\sim \onu}\{\big( \tanh(\beta h) \big)^2\}\\
&+\frac{1}{m}\E_{\oh_{n+1}}\log \E_{G_1\sim\normal(0,1)}
\Big\{\big(2\cosh\beta(\oh_{n+1}+\sqrt{q_1-q_0}G_1)\big)^m
\Big\}\, .
\end{align*}
Finally,using the 1RSB fixed point equations for the first term, and recalling that
$\oh_{n+1}$ is approximately $\normal(0,q_0)$,  we get
\begin{align}
\E\{B^{(1)}_n\} &\approx_n
\frac{\beta^2}{2}(1-q_1)+
\frac{1}{m}\E_{G_0}\log \E_{G_1\sim\normal(0,1)}
\Big\{\big(2\cosh\beta(\sqrt{q_0}+\sqrt{q_1-q_0}G_1)\big)^m
\Big\}\, .\label{eq:B1}
\end{align}
Here it is understood that $G_0,G_1\sim\normal(0,1)$ independent.

We next consider term $B^{(2)}_n$. By Eq.~\eqref{eq:T2} and calculations thereafter 
we have%
\begin{align}
B^{(2)}_n= \frac{\beta}{2n^{3/2}} \sum_{ i <j } \E\Big[ \frac{\partial}{\partial \tilde{W}_{ij}} \sum_{\sigma} \mu(\bsigma) \sigma_i \sigma_j \Big]
= \frac{\beta^2}{2n^2} \sum_{i<j} \E \Big[ 1 - \Big( \sum_{\bsigma} \mu(\bsigma) \sigma_i \sigma_j \Big)^2 \Big]. \nonumber 
\end{align}
Introducing replicas and recalling the notation $\<\,\cdot\,\>_{\mu}$
for expectation with respect to the Gibbs measure $\mu^{\otimes 2}$, we have
\begin{align*}
\frac{1}{n^2} \sum_{i,j} \Big( \<\sigma_i \sigma_j\>_{\mu} \big)^2& =
 \frac{1}{n^2} \sum_{i,j}  \<\sigma_i^1 \sigma_j^1 \sigma_i^2 \sigma_j^2 \>_{\mu}
 =\<Q_{12}^2\>_{\mu}, \nonumber 
\end{align*}
where, as usual $Q_{12}\equiv \<\bsigma^2,\bsigma^2\>/n$.
Therefore
\begin{align}
\E\{B^{(2)}_n\}&\approx_n \frac{\beta^2}{4} \Big( 1 - \E\{\<Q_{12}^2\>_{\mu}\} \Big)\nonumber\\
& \approx_n\frac{\beta^2}{4} \Big( 1 - \int q^2\,\rho(\de q)\Big)\label{eq:B2}\\
& \approx_n \frac{\beta^2}{4} \big( 1 - mq_0^2-(1-m)q_1^2\big)\, .\nonumber
\end{align}
Putting together Eqs.~\eqref{eq:B1} and \eqref{eq:B2}, we obtain
the 1RSB prediction $\phi=\Psi_{\sRSB{1}}(q_0,q_1,m)$, where
\begin{align*}
\Psi_{\sRSB{1}}(q_0,q_1,m)=&
\frac{\beta^2}{2}(1-q_1)+\frac{\beta^2}{4} \big( 1 - mq_0^2-(1-m)q_1^2\big)\\
&+
\frac{1}{m}\E_{G_0}\log \E_{G_1\sim\normal(0,1)}
\Big\{\big(2\cosh\beta(\sqrt{q_0}+\sqrt{q_1-q_0}G_1)\big)^m
\Big\}\, .
\end{align*}
This is easily seen to coincide with the formula derived in Section \ref{sec:kRSB-Replica}
using the replica method.

\section{Rigorous bounds via interpolation}
\label{sec:SKrigorous}

A complete proof of  the Parisi formula for the Sherrington-Kirkpatrick model
goes beyond the scope of these notes, and we refer to the textbook
by Panchenko \cite{panchenko2013sherrington} for an in-depth treatment. 
We will limit ourselves to stating and proving two key results on the path towards a complete proof:
the existence of the limit of the free energy density, and the fact that 
the $k$RSB free energy functional provides an upper bound on the actual 
free energy density. 
Both results are proved via a surprising interpolation argument,
introduced in \cite{guerra2002thermodynamic,guerra2003broken}.
Our presentation will differ from the original one, in that
it will be based on the construction of Ruelle Probability Cascades,
as first introduced in \cite{aizenman2003extended}. This allows for a more transparent argument.

We next state these results, and will devote the rest of this section to their proof.
\begin{theorem}[Thermodynamic limit \cite{guerra2002thermodynamic}]\label{thm:thermodynamic}
Let $Z_n(\beta)$ denote the partition function \eqref{eq:partition}, with $\lambda=0$. 
As $n \to \infty$, $n^{-1}\log Z_n(\beta)$ converges almost surely to a deterministic constant. 
\end{theorem}
\begin{theorem}[Replica symmetry breaking upper bound \cite{guerra2003broken}]
\label{thm:guerra_upper}
With the same notations as in the previous theorem
\begin{align}
\phi := \lim_{n \to \infty} \frac{1}{n} \log Z_n(\beta) \leq \inf_{k,\mathbf{q}, \mathbf{m}} \Psi_{k\rm{RSB}}(\mathbf{q}, \mathbf{m}), \nonumber
\end{align}
where $\Psi_{k\rm{RSB}}(\cdot, \cdot)$ is the $k$RSB free energy functional \eqref{eq:k-RSBfreeenergyfunctional}.
\end{theorem}

In the next subsection we provide an outline of the proof of Theorem
\ref{thm:thermodynamic}. The rest of this section is devoted to the proof of Theorem 
\ref{thm:guerra_upper}.

\subsection{Thermodynamic limit: proof outline}

First recalling $\bW= (\bG+ \bG^{\sT})/\sqrt{2n}$, for $(G_{ij})_{i,j\le n}\sim^{iid}\normal(0,1)$,
we have (omitting for simplicity the dependence on $\beta$)
\begin{align}
Z_n =\sum_{\bsigma\in\{+1,-1\}^n}\exp\Big\{\frac{\beta}{\sqrt{2n}}\<\bsigma,\bG\bsigma\>\Big\}\,.
\end{align}
It is then easy to show that $\bG\mapsto n^{-1}\log Z_n(\beta)$ is a Lipschitz function,
with Lipschitz constant $c_0\beta/\sqrt{n}$, and therefore by Gaussian
concentration (Theorem \ref{thm:GaussianConcentration}), we have
\begin{align}
\prob\Big\{\Big|\frac{1}{n}\log Z_n-\frac{1}{n}\E\log Z_n\Big|\ge t\Big\}
\le 2\, e^{-Cnt^2/\beta}\, .
\end{align}
It is therefore sufficient to prove that $n^{-1} \E\log Z_n$ has a limit
(by Borel-Cantelli).

The existence of a limit follows from sub-additivity. Recall the following
 basic analysis lemma \cite{fekete1923verteilung}. 
\begin{lemma}
Let $(a_n)_{n\in\naturals}$ be a sequence of real numbers, satisfying, for all $n_1$, 
$n_2$:
\begin{align}
a_{n_1+n_2}\ge a_{n_1}+a_{n_2}\, .
\end{align}
Then the limit $\lim_{n\to\infty}a_n/n$ exists (eventually equal to $\sup_{n} a_n/n$, which could be $\infty$).
\end{lemma}
The key step is then to prove the following inequality
%
%
\begin{align}
\E\log Z_{n_1+n_2}\ge \E\log Z_{n_1}+\E\log Z_{n_2}\, .
\end{align}
Physically, this means that a system of $n_1+n_2$ spins has (expected) free energy that is 
smaller or equal than the sums of free energies of a system of $n_1$ 
spins and a system of $n_2$ spins.

The proof of this inequality is obtained by defining a Hamiltonian 
$H_t(\bsigma)$, $\bsigma\in\{+1,-1\}^n$ that continuously interpolates between
the two cases (two separate systems and a single one), and showing that the free energy 
is monotone along this path.

\subsection{Proof of the replica symmetric upper bound}
\label{sec:RS-Upper-Bound}

In this subsection, we prove the Guerra upper bound, Theorem \ref{thm:guerra_upper},
in the $k=0$ case (replica symmetric free energy formula).

 Recall that in the replica symmetric phase, we expect the coordinates 
 of $\bsigma\sim \mu$ to be  approximately 
 independent, with effective fields which are Gaussian with mean $0$ and variance $q$. 
 Replica symmetric interpolation explicitly builds this intuition into the proof.
 
 Let $\bW= (\bG+ \bG^{\sT})/\sqrt{2n}$, for $(G_{ij})_{i,j\le n}\sim^{iid}\normal(0,1)$,
 and $\bg \sim \normal(0, \id_n)$ an independent standard Gaussian vector.
  We introduce the interpolating partition function (we suppress as before the dependence on $\beta$)
\begin{align}
Z_n(t) = \sum_{\bsigma\in\{+1,-1\}^n} \exp\Big\{ \frac{\beta}{\sqrt{2n}} \sqrt{t} \langle \bsigma, \bG \bsigma \rangle + \beta \sqrt{q(1-t)} \langle \bg , \bsigma \rangle \Big\}, \nonumber 
\end{align}
and define the interpolating free energy density $\phi_n(t) \equiv n^{-1} \E\log Z_n(t)$, where 
$\E$ denotes the expectation with respect to both $\bG$ and $\bg$ variables. 

We note that $\phi_n(1) = \phi_n$ is the free energy density of the SK model. 
For $t=0$, the system de-couples, and the free energy can be computed directly. We have, 
\begin{align}
\phi_n(0) &= \frac{1}{n} \E\Big\{ \log \sum_\bsigma \exp[\beta \sqrt{q} \sum_i g_i \sigma_i] \Big\} 
\nonumber\\
&= \E\log 2 \cosh ( \beta \sqrt{q} g) , \nonumber
\end{align} 
for $g\sim \normal(0,1)$.

 Now, we have, 
$\phi_n = \phi_n(1) = \phi_n(0) + \int_{0}^{1} \phi_n'(t) \de t  $. 
Straightforward differentiation yields
\begin{align}
\frac{\de \phi_n(t)}{\de t} =  \E\Big[ \mu_t \Big( \frac{\beta \langle \bsigma , \bG \bsigma \rangle}{\sqrt{8tn^3}}  \Big)  - \mu_t\Big( \frac{\beta \sqrt{q} }{ 2 n\sqrt{1-t} } \langle \bg, \bsigma \rangle \Big) \Big] := T_1 + T_2, \nonumber
\end{align}
where $\mu_t(\cdot)$ denotes the average with respect to the Gibbs measure corresponding to 
the interpolated system. Namely
\begin{align}
\mu_t(\bsigma) \equiv \frac{1}{Z_n(t)}\,\exp\Big\{ \frac{\beta}{\sqrt{2n}} \sqrt{t} \langle \bsigma, \bG \bsigma \rangle + \beta \sqrt{q(1-t)} \langle \bg , \bsigma \rangle \Big\}\, .
\end{align}
Using Gaussian integration by parts, we have, 
\begin{align}
T_1 = \frac{\beta^2}{4n^2} \sum_{i,j =1}^{n} \E\Big[ 1 - \Big( \mu_t(\sigma_i \sigma_j ) \Big)^2 \Big].  \nonumber 
\end{align}
Introducing the replicated Gibbs measure, we have, 
\begin{align}
T_1 = \frac{\beta^2}{4n^2} \sum_{i,j =1}^{n} \E\Big[1 - \mu_t^{\otimes 2} (\sigma^1_i \sigma^1_j \sigma^2_i \sigma^2_j) \Big] = \frac{\beta^2}{4} \Big( 1 - \E[ \mu_t^{\otimes 2} (Q_{12}^2)] \Big), \nonumber 
\end{align}
where $\bsigma^1, \bsigma^2$ are two i.i.d. replicas drawn from $\mu$ and $Q_{12} = \frac{1}{n} \langle \bsigma^1, \bsigma^2 \rangle$ is the ``overlap" between the replicas. A similar computation yields
\begin{align}
T_2 = \frac{\beta^2 q}{2} \Big( 1 - \E[ \mu_t^{\otimes 2}(Q_{12})] \Big). \nonumber 
\end{align}
Combining, we obtain, 
\begin{align}
\frac{\de \phi_n(t)}{\de t} = T_1 + T_2 = \frac{\beta^2}{4} (1- q)^2 - \frac{\beta^2}{4} \E\Big[ \mu_t^{\otimes 2} \Big( (Q_{12} -q)^2 \Big) \Big] \leq \frac{\beta^2}{4} (1- q)^2. \nonumber 
\end{align}
Thus we obtain
\begin{align*}
\phi_n &= \phi_n(0) + \int_{0}^{1} \frac{\de \phi_n(t)}{\de t} \, \de t\\
&  = \E \log 2 \cosh ( \beta \sqrt{q} g)  + \frac{\beta^2}{4} (1 -q )^2
- \frac{\beta^2}{4} \E\Big[ \mu_t^{\otimes 2} \Big( (Q_{12} -q)^2 \Big) \Big]\\
 &= \Psi_{\sRS}(q)- \frac{\beta^2}{4} \E\Big[ \mu_t^{\otimes 2} \Big( (Q_{12} -q)^2 \Big) \Big]\,.
\end{align*}
where the  expression $\Psi_{\sRS}(q)$ coincides with the replica symmetric formula already derived in Section
\ref{sec:RS-Z2-a}. 
Note that the above sequence of equalities have two interesting consequences:
\begin{enumerate}
\item $\phi_n\le \inf_{q\ge 0} \Psi_{\sRS}(q)$ for any $n$. The replica symmetric formula
provides an entirely non-asymptotic upper bound.
\item The RS asymptotics is correct (i.e. $\lim_{n\to\infty}\phi_n=\inf_q\Psi_{\sRS}(q)$)
if and only if $\E\Big[ \mu_t^{\otimes 2} \Big( (Q_{12} -q)^2 \Big) \Big]\to 0$
(in probability over $t$), i.e. if and only if the overlap concentrates around a $t$-independent value.
\end{enumerate}

\subsection{Ruelle Probability Cascades}
To establish Guerra's replica symmetry breaking upper bound, a family of random measures 
called the Ruelle Probability Cascades (henceforth denoted as RPCs) will prove to be extremely useful. 
We introduce these random measures in this section, and study some of their properties to familiarize
 ourselves with them.
 
  Our starting point is the PPP $\{ x_{\alpha} : \alpha \in \mathbb{N} \}$ 
 with intensity function $\Lambda(x) = m \exp(-mx)$, which we
 already studied in Section \ref{sec:Poisson} (with a slightly different normalization). 
  We first examine some integrability properties of these point processes. 
\begin{lemma}
\label{lemma:integrality1}
For $0 < a < m$, $\E\Big[ \Big( \sum_{\alpha} e^{x_{\alpha} } \Big)^a  \Big] < \infty$. 
\end{lemma}
\begin{proof}
We note that for $x,y >0$ and $ a \leq 1$, $(x+ y)^a \leq x^a  + y^a$ and therefore, 
\begin{align}
\E\Big[ \Big( \sum_{\alpha} e^{x_{\alpha} }  \Big)^a \Big] \leq \E\Big[ \Big( \sum_{\alpha}
 e^{x_{\alpha} } \mathbf{1}_{\{ x_{\alpha} >0   \}} \Big)^a \Big] + \E\Big[ \Big( \sum_{\alpha} 
 e^{x_{\alpha} } \mathbf{1}_{\{ x_{\alpha}  < 0   \}} \Big)^a \Big].  \nonumber
\end{align}
To control the first term, we use the same observation as above to obtain 
\begin{align}
\E\Big[ \Big( \sum_{\alpha} e^{x_{\alpha}} ] \mathbf{1}_{\{ x_{\alpha} >0   \}} \Big)^a \Big]  &\leq \E \Big[\sum_{\alpha} e^{a x_{\alpha}} \mathbf{1}( x_{\alpha} >0 ) \Big] \nonumber\\
&= \int_{0}^{\infty} m  e^{- (m-a)x}\, \de x = \frac{m}{m-a} < \infty. \nonumber
\end{align}
We note that for $z>0$ and $a<1$, $z^a \leq 1 + z$ and thus, for the second term, we have, 
\begin{align}
\E\Big[ \Big( \sum_{\alpha} e^{x_{\alpha} } \mathbf{1}_{\{ x_{\alpha}  < 0   \}} \Big)^a \Big]& \leq 1 + \E\Big[ \sum_{\alpha} e^{x_{\alpha}} \mathbf{1}_{x_{\alpha} <0 } \Big] \nonumber \\
&= 1 + \int_{- \infty}^{0} m e^{- (1- m) x }\, \de x = 1 + \frac{m}{1 - m}. \nonumber 
\end{align}
This completes the proof. 
\end{proof}
Lemma \ref{lemma:integrality1} allows us to establish the following property of these Poisson processes. 

\begin{lemma} 
\label{lemma:integrability2}
For $\{x_{\alpha} : \alpha \in \mathbb{N} \}$ a {\rm PPP} with intensity $\Lambda(x) = m \exp(- mx )\de x$.
Then $\log \Big( \sum_{\alpha} e^{x_{\alpha}} \Big)$ is integrable. 
\end{lemma}
\begin{proof}
It suffices to establish that $ \Big[ \log \Big( \sum_{\alpha} e^{x_{\alpha}} \Big) \Big]_{+}$ and
 $\Big[ \log \Big( \sum_{\alpha} e^{x_{\alpha} }\Big) \Big]_{-}$ are integrable. 
 To this end, we note that for $ z \geq 1$, $ \log z \leq z^a / a$ for some $a$ small 
 enough, and thus 
\begin{align}
 \Big[ \log \Big( \sum_{\alpha} e^{x_{\alpha}} \Big) \Big]_{+} \leq \frac{1}{a}  \Big( \sum_{\alpha} e^{x_{\alpha} } \mathbf{1}_{\{ x_{\alpha} >0   \}} \Big)^a. \nonumber 
\end{align}
Lemma \ref{lemma:integrality1} immediately establishes the integrability of the positive part. 
The control of the negative part is more direct. We have $\sum_{\alpha} \exp[x_{\alpha}] \geq \exp[x^{(1)}]$, 
where $x^{(1)}$ denotes the largest point in the process. Thus the proof is complete if we can prove that 
$\E[(x^{(1)})_{-}] < \infty$. This follows from the classical observation that $x^{(1)}$  has the Gumbel distribution 
and that a Gumbel distribution is integrable. Indeed, 
\begin{align}
\P[ x^{(1)} \leq t ] = \P \Big[ { \rm {Poisson}} \Big( \int_{t}^{\infty} m \, e^{-mx}\,
 \de x \Big) = 0  \Big] = \exp\big( - e^{-t}\big). \nonumber 
\end{align}
This completes the proof. 
\end{proof}

Let $\{ \Delta_{\alpha}: \alpha \in \mathbb{N} \}$ be a set of i.i.d. marks such that 
 $\E[ \exp ( m \Delta_1 ) ] < \infty$. Lemma \ref{lemma:marks1} establishes that 
 $\{ x_{\alpha} + \Delta_{\alpha} : \alpha \in \mathbb{N} \}$ is another PPP with intensity 
 $C_m \Lambda(x)\de x$, with $C_m = \E[ \exp(m \Delta)]$.
 
This invariance property, together with Lemma \ref{lemma:integrability2}, allows us to establish the following 
fact, which we already used informally in the context of the 1RSB cavity method
in Section \ref{sec:1RSB_Cavity_Free}. 
\begin{lemma}
\label{lemma:invariance1}
Let $\{x_{\alpha}: \alpha \in \mathbb{N} \}$ be a PPP with intensity $\Lambda (x) = m \exp[- mx] \de x$ and let 
$\{ \Delta_{\alpha}: \alpha \in \mathbb{N} \}$ be i.i.d. marks with $C_m = \E[ \exp[ m \Delta_{1}] ] < \infty$. Then we have, 
\begin{align}
\E\Big[ \log \Big( \sum_{\alpha}  \exp[ x_{\alpha} + \Delta_{\alpha} ] \Big) \Big] = \E\Big[ \log \Big( \sum_{\alpha} \exp[\alpha] \Big) \Big] + \frac{1}{m} \log \E[ \exp[ m \Delta_1]]. \nonumber 
\end{align}
\end{lemma}
\begin{proof}
Using Lemma \eqref{lemma:marks1}, we immediately have, 
\begin{align}
\log \Big( \exp\Big[ \sum_{\alpha} e^{x_{\alpha} + \Delta_{\alpha}} \Big]  \Big) \ed 
\log \Big( \sum_{\alpha} e^{x_{\alpha}} \Big) + \frac{1}{m} \log C_m. \nonumber 
\end{align}
Both sides are integrable using Lemma \ref{lemma:integrability2}, and taking an expectation completes the proof. 
\end{proof}
We will see that similar formulae will play a crucial role in the context of
 RPCs and the proof of the Guerra replica symmetry breaking upper bound. 

Note that the last statement can be rewritten  
in terms of the probability measure $\nu_{\alpha} \propto \exp(x_{\alpha} )$:
\begin{align}
\E\Big[ \log \Big( \sum_{\alpha} \nu_{\alpha} e^{\Delta_{\alpha}}\Big) \Big] = 
\frac{1}{m} \log \E[ e^{m \Delta_1} ]\, .
\end{align}
We already saw in Section  \ref{sec:1RSB_Cavity_Free} that the right hand side 
gives rise to a term in the $1$RSB free energy functional.

To define the RPC, we fix $k \geq 1$ and a set of values $0< m_1 < \cdots < m_{k} < 1$. 
The measure constructed will be indexed by $\mathbb{N}^k$. It is very helpful to think of 
$\mathbb{N}^{k}$ as the leaves of a rooted infinite tree $T = (V,E)$, with vertex set 
$V = \mathbb{N}^{0} \cup \mathbb{N}^1 \cup \cdots \mathbb{N}^k$.
The root of the tree will be denoted by $\root$. The vertices $\{1, 2, \cdots \} $ will be 
 the children of the root, and joined to the root via edges in the tree 
$T$. The vertices $\{(1,n) : n \geq 1 \}$ will similarly be children of the vertex $1$.
 In general, for any $1 \leq l \leq k-1$, consider the vertex
  $ \alpha = (n_1, \cdots , n_l)$. Then all vertices 
  $\alpha n = ( n_1, \cdots , n_l ,n)$ will be children of $\alpha$ and joined to
   $\alpha$ via edges. We denote by $|\alpha|$ the level of vertex $\alpha \in V$. 
   Further, for $\alpha = (n_1, \cdots, n_l)$, we denote by
    $\mathcal{P}(\alpha) = \{ \root, n_1, n_1 n_2, \cdots, n_1 n_2 \cdots n_l \}$ 
    the unique path in the tree joining the root $\root$ to the vertex $\alpha$. 
    
For each $\alpha \in V$, $| \alpha | = l \leq k$, let $\Pi_{\alpha} $ be a PPP with 
intensity measure $\Lambda_{l+1}(x)\de x = m_{l+1} \exp(- m_{l+1} x)\de x$. The point processes 
$\{ \Pi_{\alpha}: \alpha \in V \}$ are mutually independent. We note that each $\Pi_{\alpha}$ 
is almost surely countable, and therefore, for $\alpha \in V$ with $| \alpha | \leq k-1$, 
the elements of $\Pi_{\alpha}$ may be naturally associated to the children of $\alpha$. 
Formally, for $\Pi_{\alpha} = \{ x_{\alpha1}, x_{\alpha 2}, \cdots \}$, we associate 
$x_{\alpha n}$ with the child $\alpha n$ of the vertex $\alpha$.  For $|\alpha| \leq k$, 
we define $z_{\alpha} = \sum_{\beta \in \mathcal{P}(\alpha) } x_{\beta}$. 

Finally, we define a (random) probability measure  indexed by $\alpha\in \mathbb{N}^k$:
\begin{align}
\nu_{\alpha} = \frac{1}{Z^{\nu}}\, e^{z_{\alpha}}\, ,\;\;
\;\;\;\; \alpha \in \mathbb{N}^{k} \, ,\;\;\;\;\; Z^{\nu}\equiv \sum_{\beta\in\naturals^k}e^{z_{\beta}}. 
\end{align}
This is the RPC corresponding to the parameters $\bm = ( m_1, \cdots, m_{k})$. 
The next lemma establishes that the RPC is well defined, in that the weights 
can be normalized almost surely to yield a valid probability measure. 
\begin{lemma}
For $k \geq 1$ and $\bm = (m_1, \cdots , m_{k})$,
 $\sum_{\alpha \in \mathbb{N}^{k}} \exp(z_{\alpha}) < \infty$ almost surely. 
\end{lemma}
\begin{proof}
The proof proceeds by induction on $k$. For $k =1$, $\{ z_{\alpha}: \alpha \in \mathbb{N} \}$ is 
the PPP with intensity $\Lambda_1(x)\de x = m_1 \exp(- m_1 x )\de x$. 
This case was already established in Lemma \ref{lemma:PPP-Basic}, point 3.

Now we assume that the result has been verified for RPC's with $k-1$ levels. 
Then we have, using the definition of RPC
\begin{align}
\sum_{\alpha: | \alpha| = k } e^{z_{\alpha}} = \sum_{\alpha : |\alpha| = k-1} e^{z_{\alpha}} \Big( \sum_{n} e^{x_{\alpha n}} \Big)= \sum_{\alpha: |\alpha| = k-1} \exp[z_{\alpha}]  U_{\alpha}, \nonumber 
\end{align}
where we define $U_{\alpha} = \sum_{n} \exp(x_{\alpha n})$ for $\alpha \in V$ with $|\alpha| \leq k-1$.
 Note that for any $\alpha$ with $|\alpha | = k-2$, $\{x_{\alpha n} : n \in \mathbb{N} \}$ is a PPP with intensity $ m_{k-1} \exp(- m_{k-1}x ) \de x $. 
 Further $\{ U_{\alpha n} : n \in \mathbb{N} \}$ are i.i.d. marks, independent of
  $\{ x_{\alpha n} : n \in \mathbb{N} \}$. Finally, using the definition of
   $U_{\alpha n}$ and Lemma \ref{lemma:integrality1}, we have, for $|\alpha|=k-2$,
\begin{align}
\E[ U_{\alpha 1}^{m_{k-1}} ] = \E\Big[ \Big( \sum_{n} \exp[x_{\alpha 1 n} ] \Big)^{m_{k-1}} \Big] < \infty, \nonumber 
\end{align}
since $\{ x_{ \alpha 1 n} : n \in \mathbb{N} \}$ is a PPP with intensity 
$m_{k} \exp( - m_{k} x )\de x$ and $m_{k-1} < m_{k}$. The proof of 
Lemma \ref{lemma:invariance1} establishes that, always for $|\alpha|=k-2$,
\begin{align}
\sum_{n} e^{x _{ \alpha n}}  U_{\alpha n} \ed C_{k-2} \sum_{n} e^{x_{\alpha n}}, \nonumber 
\end{align}
for $C_{k-2} = \Big( \E[ U_{\alpha 1}^{m_{k-1}} ]  \Big)^{1/m_{k-1}}$. Thus we have, 
\begin{align}
\sum_{\alpha: | \alpha| = k } e^{z_{\alpha}} \ed C_{k-2} \sum_{\alpha: | \alpha| = k-2} e^{z_{\alpha}} \sum_{n} e^{x_{\alpha n}} = C_{k-2} \sum_{\alpha: |\alpha| = k-1} e^{z_{\alpha}} < \infty \,\, {\rm{a.s.}}, \nonumber
\end{align}
where the last assertion follows from the induction hypothesis. This completes the proof. 
\end{proof}

\subsection{The RSB upper bound: Proof of Theorem \ref{thm:guerra_upper}}
In this section, we establish Guerra's replica symmetry breaking upper bound.

We start by reminding the form of the $k$RSB free energy functional
of Section \ref{sec:kRSB-Replica-All}, albeit in a slightly different form.
 For $k \geq 1$, fix $0=m_0 < m_1 \leq \cdots \leq m_{k-1} \leq m_k =1$ and a function 
 $\psi_k: \mathbb{R}^{k} \to \mathbb{R}$. For $g_1, \cdots, g_k \sim^{iid} \normal(0,1)$, 
 we define the following sequence of random variables
\begin{align}
\psi_k &= \psi_k(g_1, \cdots, g_k), \nonumber\\
&\vdots \nonumber \\
\psi_l &= \psi_l(g_1, \cdots, g_l) = \frac{1}{m_{l+1}} \log 
\E\Big[ \exp(m_{l+1} \psi_{l+1} (g_1, \cdots, g_{l+1}) )\Big|g_1,\dots, g_l \Big], \nonumber \\
&\vdots \nonumber  \\
\psi_0 &= \frac{1}{m_1} \log \E\Big[ \exp[m_1 \psi_1(g_1) ] \Big]. \label{eq:recursive_parisi}
\end{align}
Note that, as implied by the notation $\psi_{l}$ is a measurable function of $(g_1,\dots,g_l)$.
 Jensen's inequality implies that the construction above is well
 defined as soon as $\E\exp(m_{k} \psi_k(g_1, \cdots, g_k)) < \infty$. 
 We recall that the $k$RSB free energy is recovered as follows:
 \begin{align}
 \psi_k(x_1, \cdots, x_k) &= \log 2\cosh \Big(\beta \sum_{i=1}^{k} \sqrt{q_{i+1} - q_{i}} \, x_i \Big)\, ,
 \nonumber\\
 \Psi_{\sRSB{k}}(\bq,\bm) &= -\frac{1}{4}\beta^2+
 \frac{1}{4}\beta^2\sum_{i=0}^k(m_{i+1}-m_i)q_i^2+\psi_0\, , \label{eq:kRSB-Last}
 \end{align}
 where $0=q_0\le q_1\le \dots\le q_k\le q_{k+1}=1$ and $m_{k+1}=1$.
 
The next lemma establishes a connection between the recursive construction 
of the random variable $\psi_k$, and the RPC's introduced in the last section. 
\begin{lemma}
\label{lemma:representation_rpc1}
Consider the RPC $\{\nu_{\alpha} : |\alpha| = k\}$ corresponding to the parameters 
$\bm = (m_0 , \cdots, m_{k-1},m_k)$. Let $\{ g_{\alpha} : \alpha \in \mathbb{N} \cup \cdots \cup \mathbb{N}^k\}$ be a
family of independent standard Gaussian variables independent of the RPC. 
For any $\alpha = (n_1, n_1 n_2 , \cdots, n_1 n_2 \cdots n_k)$, we define 
$G_{\alpha} = (g_{n_1}, g_{n_1n_2},  \cdots, g_{n_1 n_2 \cdots n_k} )$. 

Then, for any function  $\psi_k:\reals^k\to\reals$ such that
$\E\exp( \psi_k(g_1, \cdots, g_k)) < \infty$, we have 
\begin{align}
\psi_0 = \E\Big[ \log \Big( \sum_{\alpha} \nu_{\alpha} \exp[ \psi_k (G_{\alpha}) ] \Big) \Big]. \nonumber 
\end{align}
\end{lemma}

\begin{proof}
Recall that $\nu_{\alpha} \propto \exp(z_{\alpha})$, $z_{\alpha} = \sum_{\beta \in \mathcal{P}(\alpha)} x_{\beta}$, 
where $\{x_{\beta} : \beta \in V \backslash \root \}$ is constructed using i.i.d. PPP's at each level 
in the $k$-level infinite tree. Thus it suffices to prove
\begin{align}
\E\Big[ \log \Big( \sum_{\alpha} \exp(z_{\alpha } + \psi_k(G_{\alpha})) \Big) \Big] = \E\Big[ \log \Big( \sum_{\alpha} \exp(z_{\alpha})\Big) \Big] + \psi_0. \nonumber 
\end{align}
We use induction on $k$. In fact, we prove the stronger statement: for any sequence $\{ h_{\alpha}: \alpha \in \mathbb{N}^{k} \}$ 
satisfying $\E[ \exp(  m_kh_1 )] < \infty$ and independent of $\{\nu_{\alpha} : \alpha \in \mathbb{N}^{k} \}$ and 
$\{ g_{\alpha} : \alpha \in V \backslash \phi \}$, we have, 
\begin{align}
\E\Big[ \log \Big( \sum_{\alpha} \exp\big( z_{\alpha } + \psi_k(G_{\alpha}) + h_{\alpha}\big) \Big) \Big] = \E\Big[ \log \Big( \sum_{\alpha} \exp\big(z_{\alpha} + h_{\alpha} \big)\Big) \Big] + \psi_0. \label{eq:induction_hyp1}
\end{align}
We note that the Lemma follows upon setting $h_{\alpha } = 0$. 

The case $k=1$ corresponds to a PPP with intensity $m_1\exp(-m_1 x ) \de x $ and follows 
immediately using Lemma \ref{lemma:invariance1}. Indeed,
\begin{align}
\E\Big[ \log \Big( \sum_{\alpha} \exp\big( z_{\alpha } + \psi_k(G_{\alpha}) + h_{\alpha}\big)  \Big)
 \Big] &= \E\Big[ \log \Big( \sum_{\alpha} \exp\big( z_{\alpha} \big) \Big) \Big] + 
 \frac{1}{m_1} \log \E[ e^{m_1 h_1} ] + \psi_0 \nonumber \\
&= \E\Big[ \log \Big( \sum_{\alpha} \exp\big( z_{\alpha} + h_{\alpha} \big) \Big) \Big] + \psi_0. \nonumber 
\end{align}

Next, we assume that \eqref{eq:induction_hyp1} holds up to $k-1$ levels. Then we have, 
\begin{align}
\sum_{\alpha \in \mathbb{N}^{k}} \exp[ z_{\alpha} + \psi_k(G_{\alpha}) + h_{\alpha} ] = 
\sum_{\alpha \in \mathbb{N}^{k-1}} e^{z_{\alpha}} \sum_{n} \exp[ x_{\alpha n} + \psi_k (G_{\alpha}, g_{\alpha n}) + h_{\alpha n}]. \nonumber 
\end{align}
We set $Q_{\alpha} = \sum_{n } \exp( x_{\alpha n}  + \Delta_{\alpha n})$, with 
$\Delta_{\alpha n} =  \psi_k (G_{\alpha}, g_{\alpha n}) + h_{\alpha n}$. Define the sigma-field
$\mathscr{F}_{k-1} = \sigma( x_{\alpha}, g_{\alpha}: | \alpha | \leq k-1)$. Note that 
$\{ x_{\alpha n} : n \in \mathbb{N} \}$ is a PPP with intensity $m_{k}\exp(- m_{k} x ) \de x$ 
and conditional on $\mathscr{F}_{k-1}$, $\Delta_{\alpha n}$ are i.i.d. marks, independent of
 $\{x_{\alpha n} : n \in \mathbb{N} \}$ for each $\alpha \in \mathbb{N}^{k-1}$. Further, we have, 
\begin{align}
\Big( \E\Big[ \exp[ m_{k} \Delta_{\alpha n}  ] | \mathscr{F}_{k-1}  \Big] \Big)^{\frac{1}{m_{k} }} = \Big( \E\Big[ \exp[m_{k} h] \Big] \Big)^{\frac{1}{m_{k}}}\Big( \E_{g_{\alpha n}} \Big[ \exp[m_{k} \psi_k (G_{\alpha n} , g_{\alpha n} ) ] \Big] \Big)^{\frac{1}{m_{k}}}. \nonumber 
\end{align}
We set  $c^{m_{k}} = \E[ \exp(m_{k} h) ]$ and define 
$\sum_n \exp(x_{\alpha n}) = \exp(u_{\alpha})$. Using the recursive construction of the 
functions $\psi_l$, we note that 
\begin{align}
Q_{\alpha} \big|_{\mathscr{F}_{k-1}}\ed c e^{\psi_{k-1}(G_{\alpha}) } \sum_n e^{x_{\alpha n}}\,.
\end{align}
(The two random variables have the same distribution conditional on $\mathscr{F}_{k-1}$.)

 Note that $\{ x_{\alpha n} : n \in \mathbb{N} \}$ is independent of $\{ G_{\alpha} : | \alpha| \leq k-1 \}$ and thus the equality in distribution is true unconditionally. Thus we have established that 
\begin{align}
 \sum_{ {\alpha} \in \mathbb{N}^{k} }\exp[ z_{\alpha } + \psi_k(G_{\alpha}) + h_{\alpha}  ] 
 &\ed  c \sum_{\alpha \in \mathbb{N}^{k-1}} \exp[z_{\alpha} + \psi_{k-1}(G_{\alpha}) + u_{\alpha}], \nonumber \\
 \sum_{ {\alpha} \in \mathbb{N}^{k} }e^{z_{\alpha } + h_{\alpha}  }& \ed
  c \sum_{\alpha \in \mathbb{N}^{k-1}} e^{z_{\alpha} + u_{\alpha}}, \nonumber
\end{align}
where the last relation follows upon setting $\psi_k =0 $. Finally, Lemma \ref{lemma:integrality1} 
implies that $\E[\exp[m_{k-1} u_{\alpha}] ] < \infty$ and thus 
\begin{align}
\E\Big[ \log \Big( \sum_{\alpha \in \mathbb{N}^k } \exp\big( z_{\alpha } + \psi_k(G_{\alpha}) + h_{\alpha} \big) \Big) \Big] &= \log c + \E\Big[ \log \Big( \sum_{\alpha \in \mathbb{N}^{k-1}} \exp\big(z_{\alpha} + \psi_{k-1}(G_{\alpha}) + u_{\alpha}\big)\Big) \Big], \nonumber \\
\E\Big[ \log \Big( \sum_{\alpha \in \mathbb{N}^k} \exp\big(z_{\alpha} + h_{\alpha} \big)\Big) \Big] &= \log c + \E\Big[ \log \Big(\sum_{\alpha \in \mathbb{N}^{k-1}} e^{z_{\alpha} + u_{\alpha}} \Big) \Big]. \nonumber 
\end{align}
Therefore Eq.~\eqref{eq:induction_hyp1} follows directly for $k$-level RPCs
using the induction hypothesis. This completes the proof. 
\end{proof}

We now collect these results to express the $k$RSB free energy functional 
of Eq.~\eqref{eq:kRSB-Last} in terms of  RPCs. 
 %
 \begin{lemma}
\label{lemma:representation_rpc2}
Let $\{\nu_{\alpha} : \alpha \in \mathbb{N}^{k}\}$ be the weights of a RPC with parameters 
$\bm = (m_0 , \cdots , m_{k-1},m_k )$ with $m_0=0\le m_1\le \dots\le m_{k-1}<m_k=m_{k+1}=1$. Let 
$\{ g_{\alpha} : \alpha \in \mathbb{N} \cup \cdots \cup \mathbb{N}^k \}$ be i.i.d. standard Gaussian 
random variables, independent of the RPC. Define the following processes 
indexed by $\alpha\in \naturals^k$:
\begin{align}
h_{\alpha} &\equiv  \beta \sum_{\gamma \in \mathcal{P} (\alpha) } \sqrt{ q_{|\gamma|} - q_{|\gamma|- 1} } \, g_{\gamma}\, ,\label{eq:Halpha}\\
s_{\alpha} &= \beta \sum_{\gamma \in \mathcal{P}(\alpha)}g_{\gamma} \sqrt{\frac{q_{|\gamma|}^2}{2} - \frac{q_{|\gamma|-1}^2}{2} } \,.\label{eq:Salpha}
\end{align}
where $q_0=0\le q_1\le \dots\le q_{k-1}\le q_k=1$.
 Then we have, 
\begin{align}
\Psi_{\sRSB{k}}(\bq,\bm)=
 \E\Big[ \log \Big( 2\sum_{\alpha } \nu_{\alpha}  \cosh (h_{\alpha}) \Big)\Big] - 
 \E\Big[ \log \Big(\sum_{\alpha} \nu_{\alpha} \, e^{s_{\alpha}} \Big) \Big]. 
 \label{eq:Psi-RPC}
\end{align}
\end{lemma}

\begin{proof}
Equation~\eqref{eq:kRSB-Last} already implies that
\begin{align*}
\Psi_{\sRSB{k}}(\bq,\bm)=
 \E\Big[ \log \Big( 2\sum_{\alpha } \nu_{\alpha}  \cosh (h_{\alpha}) \Big)\Big] 
  -\frac{1}{4}\beta^2+
 \frac{1}{4}\beta^2\sum_{i=0}^k(m_{i+1}-m_i)q_i^2\, .
 \end{align*}
 It is therefore sufficient to compute the second term of the formula
 \eqref{eq:Psi-RPC}.

Define the function $\tilde{\psi}_k : \mathbb{R}^k \to \mathbb{R}$ by
\begin{align*}
 \tilde{\psi}_k (x_1, \cdots , x_k) = \beta \sum_{i=1}^{k} \sqrt{\frac{q_{i}^2}{2} - \frac{q_{i-1}^2}{2} } x_i\, .
 \end{align*}
  We define the random variable $\tilde{\psi}_0$ using the recursive construction  \eqref{eq:recursive_parisi}, 
  starting with $\tilde{\psi}_k$. Note that 
\begin{align*}
\tilde{\psi}_{k-1} (g_1,\dots,g_{k-1})&= 
\frac{1}{m_{k}} \log \E_{g_k} \Big[ 
\exp\Big(\beta m_{k} \sum_{i=1}^{k} \sqrt{\frac{q_i^2}{2} - \frac{q_{i-1}^2}{2} } g_{i}\Big) \Big]\\
&= \beta \sum_{i=1}^{k-1}  \sqrt{\frac{q_i^2}{2} - \frac{q_{i-1}^2}{2} } g_{i} + \frac{\beta^2}{4} m_{k} (q_k^2 - q_{k-1}^2). \nonumber 
\end{align*}
Continuing this recursion, we obtain
\begin{align*}
 \tilde{\psi}_0 &= \frac{\beta^2}{4} \sum_{i=1}^{k} m_{i} (q_i^2 - q_{i-1}^2)\\
 & = \frac{\beta^2}{4}-\frac{\beta^2}{4} \sum_{i=0}^{k} (m_{i+1}-m_i) q_i^2 \,.
 \end{align*} 
 An application of Lemma \ref{lemma:representation_rpc1} completes the proof. 
\end{proof}
Before proving the replica symmetry breaking upper bound, 
we need one final ingredient--- a Gaussian integration by parts lemma.
 This result is a version of Stein's lemma whose proof is immediate with the only technical difficulty
 that we consider Gaussian processes indexed by countable sets
 (see, e.g., \cite{panchenko2013sherrington}). 
\begin{lemma}
\label{lemma:integrationbyparts}
Let $\Sigma$ be a countable set and let $\{x(\sigma) , y(\sigma) : \sigma \in \Sigma\}$ be 
two centered Gaussian processes indexed by $\Sigma$. Define $C(\sigma^1, \sigma^2) = \E[x(\sigma^1) y(\sigma^2)]$. For any measure $\mu_0$ on $\Sigma$, consider the probability measure $\mu(\sigma) \propto \exp[y(\sigma)] \mu_0(\sigma)$. Denoting the expectation of a function $f$ under $\mu$ as $\mu(f)$, we have, 
\begin{align}
\E[\mu(x(\sigma))] = \E[\mu^{\otimes 2} ( C(\sigma^1, \sigma^1) - C(\sigma^1, \sigma^2) ) ], \nonumber 
\end{align}
where $\sigma^1, \sigma^2$ are i.i.d. samples drawn from $\mu$. 
\end{lemma}

We are finally in position to prove Theorem \ref{thm:guerra_upper}. 
The argument is similar to the one for replica symmetric upper bound 
given in Section \ref{sec:RS-Upper-Bound}. The main difference is that instead of
interpolating between the Sherrington-Kirkpatrick measure and a product measure with i.i.d. gaussian effective fields, we interpolate to a system where the effective field on each spin is given by a RPC. Indeed, this form for the effective field is crucial for the emergence of the Parisi formula.

\begin{proof}[Proof of Theorem \ref{thm:guerra_upper}]
Recall the SK Hamiltonian $H(\bsigma) = (\beta/2) \sum_{i,j\le n} W_{ij} \sigma_i \sigma_j
= (\beta/2)\<\bsigma,\bW\bsigma\>$, where $\bW= (\bG+ \bG^{\sT})/ \sqrt{2n}$, and
$\bG = (G_{ij})_{i,j\le n}$ are i.i.d. standard Gaussian random variables.
It is useful to recall that $\{H(\bsigma)\}_{\bsigma\in\{+1,-1\}}$ is a centered Gaussian process
with covariance 
\begin{align}
\E[H(\bsigma^1)H(\bsigma^2)]= \frac{\beta^2}{2n}
\E\big[ \langle \bG, \sigma^1 \otimes \sigma^1 \rangle \langle \bG, \sigma^2 \otimes \sigma^2\rangle \big]
=  \frac{\beta^2}{2n}\<\bsigma^1,\bsigma^2\>^2\, .\label{eq:HamiltonianCov}
\end{align}

Fix parameters $\bm = (m_0 , \cdots , m_{k-1},m_k )$, $\bq = (q_0 , \cdots , q_{k-1},q_k )$
as in the statement of Lemma \ref{lemma:representation_rpc2},
and let $\{ \nu_{\alpha} : \alpha \in \mathbb{N}^k \}$ be an RPC corresponding to the parameters 
$\bm$, independent of the disorder variables
 $\{W_{ij}: i,j\le n\}$. 
 Recall also the Gaussian fields 
 $\{ h_{\alpha}: \alpha \in \mathbb{N} \cup \cdots \cup \mathbb{N}^k \}$ and 
 $\{ s_{\alpha} : \alpha \in \mathbb{N} \cup \cdots \cup \mathbb{N}^k \}$
 introduced in Eqs.~\eqref{eq:Halpha}, \eqref{eq:Salpha}.
 Let $\{ h_{\alpha ,i } : \alpha \in \mathbb{N}^k , 1 \leq i \leq n \}$ and $\{s_{\alpha,i } : \alpha \in \mathbb{N}^k , 1 \leq i \leq n\}$ be i.i.d. copies of 
  $\{ h_{\alpha}\}$ $\{ s_{\alpha}\}$ respectively, independent of both $\{W_{ij} :i,j\le n \}$ and $\{ \nu_{\alpha} : \alpha \in \mathbb{N}^k \}$. 
  
  For $t\in [0,1]$ define the interpolating Hamiltonian 
\begin{align}
H_t(\bsigma, \alpha) = \sqrt{t} H(\bsigma) + \sqrt{t} \sum_{i=1}^{n} s_{\alpha , i}  + \sqrt{1- t} \sum_{i=1}^{n} h_{\alpha,i} \sigma_i. \label{eq:interpolatinghamiltonian}
\end{align}
and define the interpolating partition function and free energy density as
\begin{align}
Z_n(t) &\equiv \sum_{\bsigma, \alpha } \nu_{\alpha} \exp[H_t(\bsigma, \alpha)]\, ,\\
\phi_n(t) &\equiv \frac{1}{n} \E[\log Z_n(t)]\, ,
\end{align}
 where $\E$ denotes  expectation over the joint distribution of all random variables. 
 
 First, we evaluate the interpolating free energy at $t=0$ and $t=1$. At $t=0$, we have,
\begin{align}
\phi_n(0) = \frac{1}{n} \E\Big[ \log \Big[\sum_{\alpha} \nu_{\alpha} \Big[ \prod_{i=1}^{n} 2 \cosh ( h_{\alpha,i}) \Big] \Big] \Big]. \nonumber 
\end{align}
To evaluate the right-hand side, let $\bg_1, \cdots, \bg_n \sim \normal(0,\id_k)$ be i.i.d. 
random vectors, with 
$g_i = ( g_{l,i}: 1 \leq l \leq k)$. Consider now 
$\psi_k^{(n)} = \log \Big[ \prod_{i=1}^{n} 2 \cosh ( \beta \sum_{l=1}^{k} \sqrt{q_l - q_{l-1} } g_{l,i}) \Big]$. 
We define $\psi_l^{(n)}$, $0\le l\le k$, by the same recursion as \eqref{eq:recursive_parisi}, 
with the expectation over $g_l$ replaced by  the expectation with respect to the variables
$\{g_{l,i }: 1 \leq i \leq n\}$. An obvious modification of Lemma 
\ref{lemma:representation_rpc1} now implies that 
$\phi_n(0) = \psi_0^{(n)}/n$. Finally, we note that 
\begin{align}
\psi_k^{(n)} = \sum_{i=1}^{n} \log 2 \cosh ( \beta \sum_{l=1}^{k} \sqrt{q_l - q_{l-1}} g_{l,i} ). \nonumber 
\end{align}
Thus using the independence of the Gaussian variables, we have, 
\begin{align*}
\phi_n(0) &= \frac{1}{n}\psi_0^{(n)} = \psi_0  = \E\Big[ \log \Big( 2\sum_{\alpha } \nu_{\alpha}  \cosh (h_{\alpha}) \Big)\Big] \, .
\end{align*}

Next, consider the case $t=1$:
\begin{align}
\phi_n(1) &= \frac{1}{n} \E\Big[ \log \Big(\sum_{\bsigma} e^{\beta H(\bsigma)} \Big) \Big] + \frac{1}{n} \E\log \Big( \sum_{\alpha} \nu_{\alpha} \exp\Big[\sum_{i=1}^n s_{\alpha,i }  \Big]\Big) \\
&= \phi_n +  \E\Big[\log \Big( \sum_{\alpha} \nu_{\alpha} e^{s_{\alpha}} \Big) \Big], \nonumber 
\end{align}
where $\phi_n \equiv n^{-1} \E\Big[\log \Big( \sum_{\bsigma} \exp[ H(\bsigma)] \Big) \Big]$
is the SK free energy density, and the last equality
 follows from the same considerations as for $\phi_n(0)$. 

The claim of the theorem follows if we establish that $\phi_n(1) \leq \phi_n(0)$. 
 To show this, we will proceed as in the replica symmetric upper bound, namely we will prove that
 $\phi_n'(t) \leq 0$ for all $t\in [0,1]$. It is useful to define the random Gibbs measure 
 $\mu_{t}$ on the augmented space $\{+1,-1 \}^n \times \mathbb{N}^k$ associated 
 with the Hamiltonian $H_t(\bsigma, \alpha)$. Namely, 
  for $(\bsigma, \alpha ) \in \{\pm 1 \}^n \times \mathbb{N}^k$, we have, 
\begin{align}
\mu_{t}(\bsigma, \alpha) \equiv \frac{1}{Z_{n}(t)}  \nu_{\alpha} \, e^{H_{t}(\bsigma,\alpha)}\, . \nonumber 
\end{align}
We express the derivative of the interpolating free energy in terms of this Gibbs measure.
Denoting by $\mu_{t}(f)\equiv \sum_{\bsigma,\alpha}f(\bsigma,\alpha)\mu_{t}(\bsigma, \alpha) $ 
 the expectation of any function under the measure $\mu_{t}(\cdot)$ and 
 defining $C((\bsigma^1, \alpha^1), (\bsigma^2, \alpha^2)) \equiv \E[\partial t
 H_{t}(\bsigma^1, \alpha^1) H_t(\bsigma^2, \alpha^2)]$, we have, 
\begin{align}
\phi_n'(t) &= \frac{1}{n} \E\Big[ \mu_{t} \Big( \frac{\partial}{\partial t} H_t(\bsigma, \alpha) \Big) \Big]\nonumber\\
& = \frac{1}{n} \E\Big[ \mu_{t}\Big( C( (\bsigma^1, \alpha^1), (\bsigma^1, \alpha^1)) - C( (\bsigma^1, \alpha^1 ), (\bsigma^2, \alpha^2 ) ) \Big) \Big]. \label{eq:PhiPrimeRSB}
\end{align}
The last equality in the display above follows using Gaussian integration 
by parts, as presented in  Lemma \ref{lemma:integrationbyparts}, with $(\bsigma^1, \alpha^1), (\bsigma^2, \alpha^2)$ two i.i.d. 
samples drawn from the measure $\mu_{t}(\cdot)$. Now, \eqref{eq:interpolatinghamiltonian} implies 
\begin{align}
\frac{\partial}{\partial t} H_{t}(\bsigma, \alpha) = \frac{1}{2\sqrt{t}} H(\bsigma) + \frac{1}{2\sqrt{t}} \sum_{i=1}^{n} s_{\alpha,i} - \frac{1}{2\sqrt{1- t}} \sum_{i=1}^{n} h_{\alpha,i} \sigma_i. \nonumber 
\end{align}
The covariance $C((\bsigma^1, \alpha^1), (\bsigma^2, \alpha^2))$ may now be obtained directly. 
\begin{align}
\frac{1}{n}C((\bsigma^1, \alpha^1), ( \bsigma^2, \alpha^2)) &= \frac{1}{2n} \E[H(\bsigma^1) H(\bsigma^2)] + \frac{1}{2} \E[s_{\alpha^1} s_{\alpha^2}] - \frac{1}{2n } \sum_{i=1}^{n} \E[h_{\alpha^1}h_{\alpha^2} ] \sigma^1_i \sigma^2_i. \nonumber\\
&:= T_1 + T_2 + T_3. \nonumber
\end{align}
 Term $T_1$, was already computed in Eq.~\eqref{eq:HamiltonianCov}, yielding
\begin{align}
T_1 = \frac{\beta^2}{4} Q_{12}^2, \nonumber
\end{align}
where we set, as usual, $Q_{12} \equiv \langle \sigma^1, \sigma^2 \rangle/n$. 

To evaluate $T_2$ and $T_3$ we recall that $\alpha^1, \alpha^2 \in \mathbb{N}^{k}$ may be thought of as leaf vertices in the infinite tree with $k$-levels. We denote by $\alpha^1 \wedge \alpha^2$ the most recent ancestor of $\alpha^1, \alpha^2$ and use $|\alpha^1 \wedge \alpha^2|$ to denote the level of this ancestor. Therefore, we have, 
\begin{align}
T_2 = \frac{\beta^2}{2} \E\Big[s_{\alpha^1} s_{\alpha^2} \Big] = \frac{\beta^2}{4} \sum_{\gamma \in \mathcal{P}(\alpha^1 \wedge \alpha^2)} (q_{|\gamma |}^2 - q_{|\gamma| -1}^2) = \frac{\beta^2}{4} q_{\alpha^1 \wedge \alpha^2}^2. \nonumber
\end{align}
An analogous computation yields $T_3 = \frac{\beta^2}{2} q_{|\alpha^1 \wedge \alpha^2|} Q_{12}$. Combining, we obtain, 
\begin{align}
\frac{1}{n}C((\sigma^1, \alpha^1), (\sigma^2, \alpha^2)) = T_1 + T_2 + T_3 = \frac{\beta^2}{4} (Q_{12} - q_{| \alpha^1 \wedge \alpha^2 |})^2. \nonumber
\end{align}
We note that $C((\sigma^1, \alpha^1) , (\sigma^1, \alpha^1))=0$ and thus,
Eq.~\eqref{eq:PhiPrimeRSB} yields
\begin{align}
\phi_n'(t) &= -\frac{\beta^2}{4} \E\Big\{\mu_t^{\otimes 2}
\big((Q_{12} - q_{| \alpha^1 \wedge \alpha^2 |})^2\big)\Big\}\le 0\, .
\end{align}
This  completes the proof. 
\end{proof}

\section*{Bibliographic notes}
\addcontentsline{toc}{section}{\protect\numberline{}Bibliographic notes}

We refer the interested reader to \cite{panchenko2013sherrington,talagrand2010mean,talagrand2011mean} for rigorous textbook introductions to mean-field spin glasses and related models arising in statistical physics. The book \cite{panchenko2013sherrington} also contains an excellent account of the historical development of this area in mathematical probability. 

The RS formula for the limiting mutual information of the $\mathbb{Z}_2$ synchronization problem 
(corresponding to the case $\beta=\lambda$) was established 
rigorously using several different approaches.  Among others, \cite{deshpande2016asymptotic} 
derive the limit of the free energy by utilizing an algorithmic approach based on the Bayes optimal 
Approximate Message Passing scheme, and a related approach was developed in \cite{barbier2018rank}. 
A result for general distributions of the spike $\bx_0$ (always with i.i.d. coordinates)
was established in \cite{lelarge2019fundamental}. The upper bound is derived using
 Guerra interpolation, while the lower bound is established the Aizenman-Sims-Starr scheme
  along with an ``overlap concentration" formula. An alternative proof is obtained using the elegant
  ``adaptive interpolation method" \cite{barbier2019adaptive}. In this approach, one establishes
   first that the overlap concentrates at a fixed value, and then \emph{adapts} the interpolation path according 
   to the value of this overlap. A third independent proof was obtained in \cite{el2018estimation}---they directly analyzed the partition function restricted to specific values of the overlap between the sample and the planted signal, and established that the global optimum is attained at the replica symmetric fixed point. Moving beyond the limiting log-partition function, \cite{el2020fundamental} characterized $O(1)$ fluctuations for the log-partition function below the IT threshold. This establishes that below the IT threshold, the null model (i.e. $\lambda=0$) and the planted model are mutually contiguous, and that detection and estimation have the same information theoretic threshold in these spiked matrix models. 

The connections between extremal cuts of random graphs and ground states of mean-field spin glasses 
has been explored in-depth. \cite{sen2018optimization} extended the results in \cite{dembo2017extremal}
to a general class of discrete optimization problems on sparse random graphs and hypergraphs. 
The special case of the random KSAT problem was treated in \cite{panchenko2018k}, and unbalanced cuts in 
\cite{jagannath2020unbalanced}. 
These results are intimately connected to the question of \emph{universality} of the ground state of
 the Sherrington-Kirkpatrick model to the law of the disorder
 \cite{carmona2006universality,chatterjee2005simple}.
 
Approximate Message Passing algorithms have been used crucially in recent years to analyze diverse estimators in high-dimensional inference problems e.g. the LASSO \cite{BayatiMontanariLASSO},  M-estimators \cite{donoho2016high}, Maximum Likelihood Estimators \cite{sur2019modern}. We refer the interested reader to the recent survey \cite{feng2021unifying} for an in-depth discussion of these results. 

The Approximate Message Passing Algorithm has been extended in several ways, and stronger state evolution guarantees are now known; see e.g. \cite{javanmard2013state,RanganGAMP,berthier2020state,rangan2019vector,fan2020approximate}. 
\cite{chen2021universality} extended the universality guarantees of the AMP algorithm to Lipschitz nonlinearities, and AMP algorithms which depend on all the past iterates. 

Analogues of the Parisi formula have been established for vector spin glasses \cite{panchenko2018freevector}, Potts spin glass models \cite{panchenko2018free}, and multispecies models with positive definite interactions \cite{panchenko2015free,barra2015multi,bates2021free}. This line of work crucially exploits the notion of ``synchronization" due to Panchenko, which facilitates the lower bound computation in these models.

\section*{Exercises}
\addcontentsline{toc}{section}{\protect\numberline{}Exercises}

\begin{exercise}\label{exer:IsingPspin}
The Ising p-spin model is the following probability distribution over $\bsigma\in\{+1,-1\}^n$:
\begin{align}
\mu_{\beta,n}(\bsigma)  = \frac{1}{Z_n(\beta)}\, \exp\left\{\frac{\beta}{\sqrt{2n^{k-2}(k!)}}\<\bW,\bsigma^{\otimes k}\>\right\}\, ,
\end{align}
where $\bW\in (\reals^n)^{\otimes k}$ is a zero-mean Gaussian tensor, with the same distribution as in Section \ref{sec:Pspin}.
In other words, this is the same as the spherical p-spin model studied in Chapter \ref{ch:Pspin}, except for the normalization 
of $\beta$, and the fact that $\bsigma\in\{+1,-1\}^n$. For $k=2$, it corresponds to the Sherrington-Kirkpatrick model studied in this chapter.

Use the replica method to derive the replica symmetric free energy and the corresponding stationarity conditions. It might be useful to recall 
the following consequence of Varadhan's lemma. Let $\bx_n$ be a sequence of random vectors taking values in a fixed bounded convex set
$K\subseteq \reals^d$, with law $\mu_n$.  Assume that $\mu_n$ satisfies a large deviations principle with convex rate function $I(\bx)$. Let
$f:K\to \reals$ be continuous. Then
\begin{align}
\lim_{n\to \infty}\frac{1}{n}\log\E\{e^{nf(\bx_n)}\} &= \sup_{\bx_0\in K}\inf_{\blambda} A(\blambda,\bx_0)\, ,\\
A(\blambda,\bx_0) & = f(\bx_0)-\<\blambda,\bx_0\>+\lim_{n\to \infty}\frac{1}{n}\log\E\{e^{n\<\blambda,\bx_n\>}\}\, .
\end{align}
\end{exercise}

\begin{exercise}
Consider the Sherrington-Kirkpatrick model, namely the probability measure (\ref{eq:gibbs})  with $\lambda=0$. In this case the replica symmetric prediction 
is given by Eq.~(\ref{eq:rs-functional}), with $b=0$. We denote the free energy functional by $\Psi_{\sRS}(q;\beta)$.
\begin{itemize}
\item[$(a)$] Plot the  function $\Psi_{\sRS}(q;\beta)$ versus $q$ for two values of the inverse temperature: $\beta_1<1$, and $\beta_2>1$.
\item[$(b)$] Write a program to solve the stationary condition for $q$, and plot the resulting value $q_*(\beta)$ as a function of $1/\beta$.
\item[$(c)$] Compute the replica symmetric prediction for the free energy density $\Psi_{\sRS}(q_*(\beta);\beta)$ and plot it as a function of $1/\beta$.
\item[$(d)$] By taking suitable derivatives of the free energy, compute the replica symmetric prediction for 
the entropy density $\lim_{n\to\infty}H(\mu_n)/n$. Check  that this becomes negative at low temperature. 
\item[$(e)$] By considering the limit $\beta\to \infty$, compute the replica symmetric prediction for the maximum energy
$\lim_{n\to\infty}\max_{\bsigma\in\{+1,-1\}^n}\<\bsigma,\bW\bsigma\>/n$. 
\end{itemize} 
\end{exercise}

\begin{exercise}
We consider again the Sherrington-Kirkpatrick model, with $\lambda=0$.
\begin{itemize}
\item[$(a)$]  Write the 1RSB expression for the free energy $\Psi_{\sRSB{1}}(q_0,q_1,m;\beta)$. This expression depends
on the overlap parameters $0\le q_0\le q_1\le 1$ and the RSB parameter $0\le m\le 1$.  
\item[$(b)$] Write a program to evaluate $\Psi_{\sRSB{1}}(q_0,q_1,m;\beta)$. Fix a value of the inverse temperature $\beta>1$ and plot 
$\Psi_{\sRSB{1}}(q_0=0,q_1,m;\beta)$ as a function of $(q_1,m)\in [0,1]\times [0,1]$. Minimize this function numerically over $q_1$ and $m$ to 
determine the minimizer $(q_{1,*},m_*)$ as well as the corresponding free energy estimate  $\Psi_{\sRSB{1}}(q_0=0,q_{1,*},m_*;\beta)$.
\item[$(c)$] Recall that we defined the effective field on spin $\sigma_i$ as 
$h_i = (1/\beta)\atanh\big\{\sum_{\sigma}\mu_n(\sigma)\sigma_i\big\}$, and the asymptotic distribution of effective fields
$Q_h$ as follows (where the limit is understood in the weak sense)
\begin{align}
Q_h = \lim_{n\to\infty}\E\left\{\frac{1}{n}\sum_{i=1}^n\delta_{h_i}\right\}\, .
\end{align}
Assuming the 1RSB approximation, and hence overlap distribution $\rho = (1-m)\delta_{q_{1}}+m\delta_{q_0}$,
write an expresssion for $Q_h$.

Evaluate the expression for $q_0=0$, $q_1=q_{1,*}$, $m=m_*$, and inverse temperature $\beta$ chosen for the previous part.
Plot the resulting density $Q_h$. Compare this density with the Gaussian prediction obtained within the replica symmetric theory.
\end{itemize}
\end{exercise}

\begin{exercise}
Consider the model (\ref{eq:gibbs}) with $\bY$ given by Eq.~(\ref{eq:model}) with $\lambda>0$ and $h=0$. From the point
of view of statistics, this is the ${\mathbb Z}_2$ synchronization problem with a non-vanishing signal-to-noise ratio. From the physics perspective, 
you can look at this as the Sherrington-Kirkpatrick model with a ferromagnetic interaction.
\begin{itemize}
\item[$(a)$] Write the 1RSB expression for the free energy $\Psi_{\sRSB{1}}(q_0,q_1,b;m)$. Notice that this will depend on overlap parameters
$q_0, q_1$, on the RSB parameter $m$ and on the parameter $b$ that corresponds to the bias in the direction of the signal $\bx_0$.
\item[$(b)$] Write a program to find the stationary point of the 1RSB free energy, $(q_{0,*},q_{1,*},m_*,b_*)$.

How can you compute the asymptotic accuracy $M(\beta,\lambda)\equiv \lim_{n\to\infty}\E\{|\<\bx_0,\hat{\bsigma}(\beta)\>|\}/n$ in terms of this stationary point?
Here $\hat{\bsigma}(\beta)\equiv\sum_{\bsigma}\tilde{\mu}_{\beta}(\bsigma)\bsigma$ (where $\tilde{\mu}_{\beta}$ is obtained from $\mu_{\beta}$ tilting 
by an infinitesimal amount in the direction of $\bx_0$).
\item[$(c)$] Plot the 1RSB result for $M(\beta,\lambda)$ as a function of $\lambda\in [0,3]$ for $\beta\in \{1,2,4,8\}$ (thus obtaining four curves).
\end{itemize}
\end{exercise}

\appendix

\chapter{Probability theory inequalities}
\label{app:Probability}
\def\bddsubg{C_{\rm sg}}
\def\bddsube{C_{\rm se}}

Good references on this subject are \cite{boucheron2013concentration,Vershynin-CS}.

\section{Basic facts} 

\begin{theorem}[Gaussian Integration by Parts]
Let $X \sim \normal(0,1)$ and $f : \mathbb{R} \to \mathbb{R}$ weakly differentiable. Assume that  $\mathbb{E}[Xf(x)]$ and $\mathbb{E}[f'(X)]$ are both well-defined. Then we have, 
\begin{align*}
\mathbb{E}[Xf(X)] = \mathbb{E}[f'(X)]. \nonumber 
\end{align*} 
\end{theorem} 

\section{Basic inequalities}

\begin{lemma}[Markov inequality]
 Let $X$ be a non-negative random variable with $\mathbb{E}[X] < \infty$. Then for any $a>0$, $\mathbb{P}[X>a] \leq \frac{\mathbb{E}[X]}{a}$. 
\end{lemma}

\begin{lemma}[Chebychev inequality]
Let $X$ be a real-valued random variable with $\mathbb{E}[X^2] < \infty$. Then for any $a>0$, 
\begin{align}
\mathbb{P}[|X - \mathbb{E}[X]| > a ] \leq \frac{\mathrm{Var}(X)}{a^2}. \nonumber
\end{align}
\end{lemma}

\begin{theorem}[Paley-Ziegmund inequality]
Let $X$ be a random variable with finite second moment and $\E(X)>0$. Then, for any $t>0$, we have
\begin{align}
\prob\big(X\ge t\E X\big) \ge (1-t)^3 \frac{(\E X)^2}{\E(X^2)} \, .
\end{align} 
\end{theorem}

\section{Concentration inequalities}

\begin{theorem}[McDiarmid inequality]
Let $F:\cX^n\to \reals$ be any function satisfying 
\begin{align}
\big|F(x_1,\dots,x_{i-1},x_i,x_{i+1},\dots, x_n)-F(x_1,\dots,x_{i-1},x_i',x_{i+1},\dots, x_n)\big|\le L_i
\end{align}
for all $\bx = (\bx_1,\dots,\bx_n)\in\cX^n$, and $x'_i\in\cX$. If $\bX=(X_1,\dots,X_n)$ is a random vector taking values in 
$\cX^n$, with independent coordinates, then, for any $t\ge 0$
\begin{align}
\prob\Big(F(\bX)\ge \E\, F(\bX)+t\Big)\le \exp\Big\{-\frac{t^2}{2\sum_{i=1}^nL_i^2}\Big\}\, .
\end{align}
\end{theorem}

\begin{theorem}[Gaussian concentration]\label{thm:GaussianConcentration}
Let $F:\reals^{d}\to \reals$ be an $L$-Lipschitz function, i.e. a function such that, for all $\bx,\by\in\reals^d$,
\begin{align}
|F(\bx)-F(\by)|\le L\, \|\bx-\by\|_2\, .
\end{align}
If $\bX\sim\normal(\bfzero,\id_d)$ is a standard normal vector, then, for any $t\ge 0$
\begin{align}
\prob\big\{F(\bX)\ge \E F(\bX) +t\big\}\le e^{-\frac{t^2}{2L^2}}\, .
\end{align}
\end{theorem}

\chapter{Summary of notations}

\section{Matrix Norms} 
For any matrix $\bA \in \mathbb{R}^{d \times d}$, define the operator norm $\|\bA \|_2 = \sup_{\| \bx\|_2 =1} \| \bA \bx \|_2$. On the other hand, we define the Frobenius norm $\|\bA\|_F = \Big(\sum_{i,j} A_{ij}^2 \Big)^{1/2}$.

\section{Asymptotics}

Given two (strictly positive) functions $f(n)$, $g(n)$, we write $f(n)\doteq g(n)$ if they are equivalent to
leading exponential order, i.e. if $\lim_{n\to\infty} n^{-1} \log[f(n)/g(n)] = 0$. For any two sequences $a_n$, $b_n$, we say $a_n = O_n(b_n)$ if there exists a constant $C>0$ 
such that $a_n/b_n \leq C$ for all $n \geq 1$. Similarly, we say $a_n = o_n(b_n)$ if $a_n/b_n \to 0$ as $n \to \infty$.

\section{Probability} 

\begin{itemize}
\item[$(i)$] Convergence in Probability---For a sequence of random variables $\{X_n : n \geq 1\}$ and $X$ defined on the same probability space $(\Omega, \mathcal{F}, \mathbb{P})$, we say that $X_n \stackrel{P}{\to} X$ if for every $\varepsilon>0$, $\mathbb{P}[|X_n - X|> \varepsilon] \to 0$ as $n \to \infty$. Given any positive sequence of real numbers $\{a_n : n \geq 1\}$, we say that $X_n = o_n(a_n)$ if $X_n/a_n \stackrel{P}{\to} 0$. 

\item[$(ii)$] Convergence almost surely---For a sequence of random variables $\{X_n : n \geq 1\}$ and $X$ defined on the same probability space $(\Omega, \mathcal{F}, \mathbb{P})$, we say that $X_n \stackrel{a.s.}{\to}X$ if $\mathbb{P}[\{\omega: X_n(\omega) \to X(\omega)\}] =1$. 

\item[$(iii)$] Convergence in distribution---Let $\mathcal{X}$ be a complete, separable metric space. Let $\{X_n : n \geq 1\}$, $X$ be random variables taking values in $\mathcal{X}$. We say that $X_n \stackrel{w}{\Rightarrow} X$ if for all $f: \mathcal{X} \to \mathbb{R}$ bounded continuous, $\mathbb{E}[f(X_n)] \to \mathbb{E}[f(X)]$. Equivalently, for any sequence of probability measures $\{\mu_n : n \geq 1\}$ and $\mu$ on $\mathcal{X}$, we say $\mu_n \stackrel{w}{\Rightarrow} \mu$ if $\int f \d \mu_n \to \int f \d \mu$ for all $f : \mathcal{X} \to \mathbb{R}$ bounded continuous. 

\item[$(iv)$] For any two real-valued random variables $X,Y$, we say $X \stackrel{d}{=} Y$ if for all $x \in \mathbb{R}$, $\mathbb{P}[X\leq x] = \mathbb{P}[Y\leq x]$. 
\end{itemize}

\section{Probability distributions}

\begin{itemize}
\item[$(i)$] A real-valued random variable $X \sim \normal(0,1)$ if $X$ has a probability density function $f_X(x) = \frac{1}{\sqrt{2\pi}} \exp(- \frac{x^2}{2})$ on $\mathbb{R}$. For any $\mu \in \mathbb{R}$ and $\sigma >0$, we say $Y \sim \normal(\mu,\sigma^2)$ if $Y \stackrel{d}{=} \mu + \sigma X$. For $\sigma>0$, $Y$ has probability density $f_Y(x) = \frac{1}{\sqrt{2\pi} \sigma} \exp(- \frac{(x-\mu)^2}{2\sigma^2})$. 

For any $d \geq 1$, an $\mathbb{R}^d$ valued random vector $\mathbf{g} \sim \normal(0,\bI_d)$ if $\boldsymbol{g} = (g_1, \cdots, g_d)$, $g_1, \cdots, g_d \sim \normal(0,1)$ are i.i.d random variables.  For any $\boldsymbol{\mu} \in \mathbb{R}^d$ and $\boldsymbol{\Sigma} \in \mathbb{R}^{d \times d}$ non-negative definite, we say $\bh \sim \normal(\boldsymbol{\mu}, \boldsymbol{\Sigma})$ if $\bh \stackrel{d}{=} \boldsymbol{\mu} + \boldsymbol{\Sigma}^{1/2} \bg$, where $\mathbf{\Sigma}^{1/2}$ is the matrix square-root of $\mathbf{\Sigma}$. If $\mathbf{\Sigma}$ is positive definite, $\bh$ has density 
\begin{align}
f_{\bh}(\bx) = \frac{1}{ ({2\pi})^{d/2} \sqrt{\mathrm{det}(\mathbf{\Sigma})}} \exp\Big(- \frac{1}{2} (\bx - \boldsymbol{\mu})^{\top} \mathbf{\Sigma}^{-\frac{1}{2}} (\bx - \boldsymbol{\mu}) \Big), \,\, \mathbf{x} \in \mathbb{R}^d.  \nonumber 
\end{align} 

\item[$(ii)$] We say $X \sim \Poisson(\lambda)$ if $\mathbb{P}[X=k] = \exp(-\lambda) \frac{\lambda^{k}}{k!}$, $k \in \mathbb{N}:=\{0,1,2,\cdots\}$. 

\item[$(iii)$] Let $\mathcal{X}$ be a closed subset of $\mathbb{R}^d$. We define a Poisson Point Process (PPP) $\mathcal{N}$ with intensity $m:\mathcal{X} \to \mathbb{R}_{+}$ as follows:
\begin{itemize}
\item For any $k \geq 1$ and $\mathcal{B}_1, \cdots, \mathcal{B}_k$ measurable disjoint subsets of $\mathcal{X}$, $\mathcal{N}(\mathcal{B}_1), \cdots \mathcal{N}(\mathcal{B}_k)$ are independent random variables taking values in $\mathbb{N}$. 
\item For $1\leq i \leq k$, $\mathcal{N}(\mathcal{B}_i) \sim \Poisson(\int_{\mathcal{B}_i} m(x) \d x)$. 
\end{itemize}

\item[$(iv)$] Let $\bG = (G_{ij}) \in \mathbb{R}^{n \times n}$ be a random matrix with i.i.d. $\normal(0,1)$ entries. Set $\bW = (\bG + \bG^{\top} )/\sqrt{2n}$. $\bW$ is a symmetric random matrix, with $\{W_{ij}: i \leq j\}$ independent, mean-zero normal random variables. Further, $\mathrm{Var}(W_{ij}) = (1 + \mathbf{1}_{i=j})/n$. This distribution is referred to as the ``Gaussian Orthogonal Ensemble" (GOE) in the literature. For notational compactness, we use $\bW_n\sim\GOE(n)$ in these notes. 
\end{itemize}

\backmatter

\addcontentsline{toc}{section}{References}


\bibliographystyle{amsalpha}

\bibliography{all-bibliography}

\newcommand{\etalchar}[1]{$^{#1}$}
\providecommand{\bysame}{\leavevmode\hbox to3em{\hrulefill}\thinspace}
\providecommand{\MR}{\relax\ifhmode\unskip\space\fi MR }
\providecommand{\MRhref}[2]{%
  \href{http://www.ams.org/mathscinet-getitem?mr=#1}{#2}
}
\providecommand{\href}[2]{#2}
\begin{thebibliography}{SVdPDMS11}

\bibitem[AA{\v{C}}13]{auffinger2013random}
Antonio Auffinger, G{\'e}rard~Ben Arous, and Ji{\v{r}}{\'\i} {\v{C}}ern{\`y},
  \emph{Random matrices and complexity of spin glasses}, Communications on Pure
  and Applied Mathematics \textbf{66} (2013), no.~2, 165--201.

\bibitem[ABE{\etalchar{+}}05]{arora2005non}
Sanjeev Arora, Eli Berger, Hazan Elad, Guy Kindler, and Muli Safra, \emph{On
  non-approximability for quadratic programs}, Foundations of Computer Science,
  2005. FOCS 2005. 46th Annual IEEE Symposium on, IEEE, 2005, pp.~206--215.

\bibitem[AC15a]{auffinger2015properties}
Antonio Auffinger and Wei-Kuo Chen, \emph{On properties of parisi measures},
  Probability Theory and Related Fields \textbf{161} (2015), no.~3-4, 817--850.

\bibitem[AC15b]{auffinger2015parisi}
\bysame, \emph{{The Parisi formula has a unique minimizer}}, Communications in
  Mathematical Physics \textbf{335} (2015), no.~3, 1429--1444.

\bibitem[AG97]{arous1997large}
G~Ben Arous and Alice Guionnet, \emph{Large deviations for wigner's law and
  voiculescu's non-commutative entropy}, Probability theory and related fields
  \textbf{108} (1997), no.~4, 517--542.

\bibitem[And88]{anderson1988spin}
Philip~W Anderson, \emph{{Spin glass I: A scaling law rescued}}, Physics Today
  \textbf{41} (1988), no.~1, 9--11.

\bibitem[And90]{anderson1990reference}
\bysame, \emph{{Spin glass VII: Spin glass as a paradigm}}, Physics Today
  (1990).

\bibitem[ASS03]{aizenman2003extended}
Michael Aizenman, Robert Sims, and Shannon~L Starr, \emph{Extended variational
  principle for the sherrington-kirkpatrick spin-glass model}, Physical Review
  B \textbf{68} (2003), no.~21, 214403.

\bibitem[AT07]{adler2007random}
Robert~J Adler and Jonathan~E Taylor, \emph{Random fields and geometry},
  vol.~80, Springer, 2007.

\bibitem[BBAP05]{baik2005phase}
Jinho Baik, G{\'e}rard Ben~Arous, and Sandrine P{\'e}ch{\'e}, \emph{Phase
  transition of the largest eigenvalue for nonnull complex sample covariance
  matrices}, Annals of Probability (2005), 1643--1697.

\bibitem[BCMT15]{barra2015multi}
Adriano Barra, Pierluigi Contucci, Emanuele Mingione, and Daniele Tantari,
  \emph{Multi-species mean field spin glasses. rigorous results}, Annales Henri
  Poincar{\'e}, vol.~16, Springer, 2015, pp.~691--708.

\bibitem[BDM{\etalchar{+}}18]{barbier2018rank}
Jean Barbier, Mohamad Dia, Nicolas Macris, Florent Krzakala, and Lenka
  Zdeborov{\'a}, \emph{Rank-one matrix estimation: analysis of algorithmic and
  information theoretic limits by the spatial coupling method},
  arXiv:1812.02537 (2018).

\bibitem[BLM13]{boucheron2013concentration}
St{\'e}phane Boucheron, G{\'a}bor Lugosi, and Pascal Massart,
  \emph{Concentration inequalities: A nonasymptotic theory of independence},
  Oxford university press, 2013.

\bibitem[BM11]{BM-MPCS-2011}
Mohsen Bayati and Andrea Montanari, \emph{{The dynamics of message passing on
  dense graphs, with applications to compressed sensing}}, IEEE Trans. on
  Inform. Theory \textbf{57} (2011), 764--785.

\bibitem[BM12]{BayatiMontanariLASSO}
\bysame, \emph{{The LASSO risk for gaussian matrices}}, IEEE Trans. on Inform.
  Theory \textbf{58} (2012), 1997--2017.

\bibitem[BM16]{barak2016noisy}
Boaz Barak and Ankur Moitra, \emph{Noisy tensor completion via the
  sum-of-squares hierarchy}, Conference on Learning Theory, PMLR, 2016,
  pp.~417--445.

\bibitem[BM19]{barbier2019adaptive}
Jean Barbier and Nicolas Macris, \emph{The adaptive interpolation method: a
  simple scheme to prove replica formulas in bayesian inference}, Probability
  Theory and Related Fields \textbf{174} (2019), no.~3-4, 1133--1185.

\bibitem[BMN20]{berthier2020state}
Raphael Berthier, Andrea Montanari, and Phan-Minh Nguyen, \emph{State evolution
  for approximate message passing with non-separable functions}, Information
  and Inference: A Journal of the IMA \textbf{9} (2020), no.~1, 33--79.

\bibitem[Bol14]{bolthausen2014iterative}
Erwin Bolthausen, \emph{{An iterative construction of solutions of the TAP
  equations for the Sherrington--Kirkpatrick model}}, Communications in
  Mathematical Physics \textbf{325} (2014), no.~1, 333--366.

\bibitem[BS21]{bates2021free}
Erik Bates and Youngtak Sohn, \emph{Free energy in multi-species mixed $ p
  $-spin spherical models}, arXiv preprint arXiv:2109.14790 (2021).

\bibitem[CH06]{carmona2006universality}
Philippe Carmona and Yueyun Hu, \emph{Universality in
  sherrington--kirkpatrick's spin glass model}, Annales de l'Institut Henri
  Poincare (B) Probability and Statistics, vol.~42, Elsevier, 2006,
  pp.~215--222.

\bibitem[Cha05]{chatterjee2005simple}
Sourav Chatterjee, \emph{A simple invariance theorem}, arXiv preprint
  math/0508213 (2005).

\bibitem[Cha21]{VirasoroInterview}
Patrick Charbonneau, \emph{{History of RSB Interview: Miguel Virasoro}},
  transcript of an oral history conducted 2021 by Patrick Charbonneau and
  Francesco Zamponi, History of RSB Project, CAPHES \'Ecole normale
  supérieure, Paris.

\bibitem[Che13]{chen2013aizenman}
Wei-Kuo Chen, \emph{The aizenman-sims-starr scheme and parisi formula for mixed
  $ p $-spin spherical models}, Electronic Journal of Probability \textbf{18}
  (2013), 1--14.

\bibitem[Che19]{chen2019phase}
\bysame, \emph{{Phase transition in the spiked random tensor with Rademacher
  prior}}, The Annals of Statistics \textbf{47} (2019), no.~5, 2734--2756.

\bibitem[CHS93]{crisanti1993sphericalp}
Andrea Crisanti, Heinz Horner, and H-J Sommers, \emph{The sphericalp-spin
  interaction spin-glass model}, Zeitschrift f{\"u}r Physik B Condensed Matter
  \textbf{92} (1993), no.~2, 257--271.

\bibitem[CL21]{chen2021universality}
Wei-Kuo Chen and Wai-Kit Lam, \emph{Universality of approximate message passing
  algorithms}, Electronic Journal of Probability \textbf{26} (2021), 1--44.

\bibitem[CMW20]{celentano2020estimation}
Michael Celentano, Andrea Montanari, and Yuchen Wu, \emph{The estimation error
  of general first order methods}, Conference on Learning Theory, PMLR, 2020,
  pp.~1078--1141.

\bibitem[CR02]{crisanti2002analysis}
Andrea Crisanti and Tommaso Rizzo, \emph{{Analysis of the $\infty$-replica
  symmetry breaking solution of the Sherrington-Kirkpatrick model}}, Physical
  Review E \textbf{65} (2002), no.~4, 046137.

\bibitem[CS92]{crisanti1992spherical}
Andrea Crisanti and H-J Sommers, \emph{The spherical p-spin interaction spin
  glass model: the statics}, Zeitschrift f{\"u}r Physik B Condensed Matter
  \textbf{87} (1992), no.~3, 341--354.

\bibitem[CS95]{crisanti1995thouless}
\bysame, \emph{Thouless-anderson-palmer approach to the spherical p-spin spin
  glass model}, Journal de Physique I \textbf{5} (1995), no.~7, 805--813.

\bibitem[CT91]{cover1991elements}
Thomas~M. {Cover} and Joy~A. {Thomas}, \emph{Elements of information theory},
  1991.

\bibitem[DAL83]{de1983internal}
JRL De~Almeida and EJS Lage, \emph{Internal field distribution in the
  infinite-range ising spin glass}, Journal of Physics C: Solid State Physics
  \textbf{16} (1983), no.~5, 939.

\bibitem[DAM16]{deshpande2016asymptotic}
Yash Deshpande, Emmanuel Abbe, and Andrea Montanari, \emph{Asymptotic mutual
  information for the balanced binary stochastic block model}, Information and
  Inference: A Journal of the IMA \textbf{6} (2016), no.~2, 125--170.

\bibitem[DAM17]{deshpande2017asymptotic}
\bysame, \emph{Asymptotic mutual information for the balanced binary stochastic
  block model}, Information and Inference: A Journal of the IMA \textbf{6}
  (2017), no.~2, 125--170.

\bibitem[DM16]{donoho2016high}
David Donoho and Andrea Montanari, \emph{High dimensional robust m-estimation:
  Asymptotic variance via approximate message passing}, Probability Theory and
  Related Fields \textbf{166} (2016), no.~3-4, 935--969.

\bibitem[DMM09]{DMM09}
David~L. Donoho, Arian Maleki, and Andrea Montanari, \emph{{Message Passing
  Algorithms for Compressed Sensing}}, Proceedings of the National Academy of
  Sciences \textbf{106} (2009), 18914--18919.

\bibitem[DMS17]{dembo2017extremal}
Amir Dembo, Andrea Montanari, and Subhabrata Sen, \emph{Extremal cuts of sparse
  random graphs}, The Annals of Probability \textbf{45} (2017), no.~2,
  1190--1217.

\bibitem[DVJ98]{daley1998introduction}
DJ~Daley and D~Vere-Jones, \emph{Introduction to the general theory of point
  processes}, An Introduction to the Theory of Point Processes, Springer, 1998,
  pp.~197--233.

\bibitem[EA75]{edwards1975theory}
Samuel~Frederick Edwards and Phil~W Anderson, \emph{Theory of spin glasses},
  Journal of Physics F: Metal Physics \textbf{5} (1975), no.~5, 965.

\bibitem[EAK18]{el2018estimation}
Ahmed El~Alaoui and Florent Krzakala, \emph{Estimation in the spiked wigner
  model: A short proof of the replica formula}, 2018 IEEE International
  Symposium on Information Theory (ISIT), IEEE, 2018, pp.~1874--1878.

\bibitem[EAKJ20]{el2020fundamental}
Ahmed El~Alaoui, Florent Krzakala, and Michael Jordan, \emph{Fundamental limits
  of detection in the spiked wigner model}, The Annals of Statistics
  \textbf{48} (2020), no.~2, 863--885.

\bibitem[Fan20]{fan2020approximate}
Zhou Fan, \emph{Approximate message passing algorithms for rotationally
  invariant matrices}, arXiv preprint arXiv:2008.11892 (2020).

\bibitem[Fek23]{fekete1923verteilung}
Michael Fekete, \emph{{\"U}ber die verteilung der wurzeln bei gewissen
  algebraischen gleichungen mit ganzzahligen koeffizienten}, Mathematische
  Zeitschrift \textbf{17} (1923), no.~1, 228--249.

\bibitem[FVRS21]{feng2021unifying}
Oliver~Y Feng, Ramji Venkataramanan, Cynthia Rush, and Richard~J Samworth,
  \emph{A unifying tutorial on approximate message passing}, arXiv preprint
  arXiv:2105.02180 (2021).

\bibitem[Gal62]{gallager1962low}
Robert Gallager, \emph{Low-density parity-check codes}, IRE Transactions on
  information theory \textbf{8} (1962), no.~1, 21--28.

\bibitem[GS00]{gillin2000p}
Peter Gillin and David Sherrington, \emph{$p> 2$ spin glasses with first-order
  ferromagnetic transitions}, Journal of Physics A: Mathematical and General
  \textbf{33} (2000), no.~16, 3081.

\bibitem[GT02]{guerra2002thermodynamic}
Francesco Guerra and Fabio~Lucio Toninelli, \emph{The thermodynamic limit in
  mean field spin glass models}, Communications in Mathematical Physics
  \textbf{230} (2002), no.~1, 71--79.

\bibitem[Gue03]{guerra2003broken}
Francesco Guerra, \emph{Broken replica symmetry bounds in the mean field spin
  glass model}, Communications in mathematical physics \textbf{233} (2003),
  no.~1, 1--12.

\bibitem[HKZ12]{hsu2012spectral}
Daniel Hsu, Sham~M Kakade, and Tong Zhang, \emph{A spectral algorithm for
  learning hidden markov models}, Journal of Computer and System Sciences
  \textbf{78} (2012), no.~5, 1460--1480.

\bibitem[HR03]{hoyle2003pca}
DC~Hoyle and M~Rattray, \emph{Pca learning for sparse high-dimensional data},
  EPL (Europhysics Letters) \textbf{62} (2003), no.~1, 117.

\bibitem[HR04]{hoyle2004principal}
David~C Hoyle and Magnus Rattray, \emph{Principal-component-analysis eigenvalue
  spectra from data with symmetry-breaking structure}, Physical Review E
  \textbf{69} (2004), no.~2, 026124.

\bibitem[JM13]{javanmard2013state}
Adel Javanmard and Andrea Montanari, \emph{State evolution for general
  approximate message passing algorithms, with applications to spatial
  coupling}, Information and Inference: A Journal of the IMA \textbf{2} (2013),
  no.~2, 115--144.

\bibitem[JMRT16]{javanmard2016phase}
Adel Javanmard, Andrea Montanari, and Federico Ricci-Tersenghi, \emph{Phase
  transitions in semidefinite relaxations}, Proceedings of the National Academy
  of Sciences \textbf{113} (2016), no.~16, E2218--E2223.

\bibitem[JS20]{jagannath2020unbalanced}
Aukosh Jagannath and Subhabrata Sen, \emph{On the unbalanced cut problem and
  the generalized sherrington--kirkpatrick model}, Annales de l’Institut
  Henri Poincar{\'e} D \textbf{8} (2020), no.~1, 35--88.

\bibitem[JT16]{jagannath2016dynamic}
Aukosh Jagannath and Ian Tobasco, \emph{{A dynamic programming approach to the
  Parisi functional}}, Proceedings of the American Mathematical Society
  \textbf{144} (2016), no.~7, 3135--3150.

\bibitem[KM09]{korada2009exact}
Satish~Babu Korada and Nicolas Macris, \emph{Exact solution of the gauge
  symmetric p-spin glass model on a complete graph}, Journal of Statistical
  Physics \textbf{136} (2009), no.~2, 205--230.

\bibitem[KSS13]{kreimer2013tensor}
Nadia Kreimer, Aaron Stanton, and Mauricio~D Sacchi, \emph{Tensor completion
  based on nuclear norm minimization for 5d seismic data reconstruction},
  Geophysics \textbf{78} (2013), no.~6, V273--V284.

\bibitem[KTJ76]{kosterlitz1976spherical}
John~M Kosterlitz, David~J Thouless, and Raymund~C Jones, \emph{Spherical model
  of a spin-glass}, Physical Review Letters \textbf{36} (1976), no.~20, 1217.

\bibitem[LC06]{lehmann2006theory}
Erich~L Lehmann and George Casella, \emph{Theory of point estimation}, Springer
  Science \& Business Media, 2006.

\bibitem[Lig78]{liggett1978random}
Thomas~M Liggett, \emph{Random invariant measures for markov chains, and
  independent particle systems}, Zeitschrift f{\"u}r Wahrscheinlichkeitstheorie
  und Verwandte Gebiete \textbf{45} (1978), no.~4, 297--313.

\bibitem[LL10]{li2010tensor}
Nan Li and Baoxin Li, \emph{Tensor completion for on-board compression of
  hyperspectral images}, 2010 IEEE International Conference on Image
  Processing, IEEE, 2010, pp.~517--520.

\bibitem[LM19]{lelarge2019fundamental}
Marc Lelarge and L{\'e}o Miolane, \emph{Fundamental limits of symmetric
  low-rank matrix estimation}, Probability Theory and Related Fields
  \textbf{173} (2019), no.~3-4, 859--929.

\bibitem[LML{\etalchar{+}}17]{lesieur2017statistical}
Thibault Lesieur, L{\'e}o Miolane, Marc Lelarge, Florent Krzakala, and Lenka
  Zdeborov{\'a}, \emph{Statistical and computational phase transitions in
  spiked tensor estimation}, 2017 IEEE International Symposium on Information
  Theory (ISIT), IEEE, 2017, pp.~511--515.

\bibitem[LMWY13]{liu2013tensor}
Ji~Liu, Przemyslaw Musialski, Peter Wonka, and Jieping Ye, \emph{Tensor
  completion for estimating missing values in visual data}, IEEE Transactions
  on Pattern Analysis and Machine Intelligence \textbf{35} (2013), no.~1,
  208--220.

\bibitem[MM09]{MezardMontanari}
Marc M{\'e}zard and Andrea Montanari, \emph{{Information, Physics and
  Computation}}, Oxford, 2009.

\bibitem[M{\o}r11]{morup2011applications}
Morten M{\o}rup, \emph{Applications of tensor (multiway array) factorizations
  and decompositions in data mining}, Wiley Interdisciplinary Reviews: Data
  Mining and Knowledge Discovery \textbf{1} (2011), no.~1, 24--40.

\bibitem[MPV86]{mezard1986sk}
M~M{\'e}zard, G~Parisi, and MA~Virasoro, \emph{Sk model: The replica solution
  without replicas}, EPL (Europhysics Letters) \textbf{1} (1986), no.~2, 77.

\bibitem[MPV87]{SpinGlass}
Marc M\'ezard, Giorgio Parisi, and Miguel~A. Virasoro, \emph{Spin glass theory
  and beyond}, World Scientific, 1987.

\bibitem[MR14]{montanari2014statistical}
Andrea Montanari and Emile Richard, \emph{{A statistical model for tensor
  PCA}}, Advances in Neural Information Processing Systems \textbf{27} (2014).

\bibitem[MRZ16]{montanari2016limitation}
Andrea Montanari, Daniel Reichman, and Ofer Zeitouni, \emph{On the limitation
  of spectral methods: From the gaussian hidden clique problem to rank one
  perturbations of gaussian tensors}, IEEE Transactions on Information Theory
  \textbf{63} (2016), no.~3, 1572--1579.

\bibitem[MV21]{montanari2021estimation}
Andrea Montanari and Ramji Venkataramanan, \emph{Estimation of low-rank
  matrices via approximate message passing}, The Annals of Statistics
  \textbf{49} (2021), no.~1, 321--345.

\bibitem[{\"O}VBS17]{ozyecsil2017survey}
Onur {\"O}zye{\c{s}}il, Vladislav Voroninski, Ronen Basri, and Amit Singer,
  \emph{A survey of structure from motion}, Acta Numerica \textbf{26} (2017),
  305--364.

\bibitem[Pan13]{panchenko2013sherrington}
Dmitry Panchenko, \emph{{The Sherrington-Kirkpatrick model}}, Springer Science
  \& Business Media, 2013.

\bibitem[Pan14]{panchenko2014parisi}
\bysame, \emph{{The Parisi formula for mixed $ p $-spin models}}, The Annals of
  Probability \textbf{42} (2014), no.~3, 946--958.

\bibitem[Pan15]{panchenko2015free}
\bysame, \emph{{The free energy in a multi-species Sherrington--Kirkpatrick
  model}}, The Annals of Probability \textbf{43} (2015), no.~6, 3494--3513.

\bibitem[Pan18a]{panchenko2018freevector}
\bysame, \emph{Free energy in the mixed $ p $-spin models with vector spins},
  The Annals of Probability \textbf{46} (2018), no.~2, 865--896.

\bibitem[Pan18b]{panchenko2018free}
\bysame, \emph{{Free energy in the Potts spin glass}}, The Annals of
  Probability \textbf{46} (2018), no.~2, 829--864.

\bibitem[Pan18c]{panchenko2018k}
\bysame, \emph{{On the K-sat model with large number of clauses}}, Random
  Structures \& Algorithms \textbf{52} (2018), no.~3, 536--542.

\bibitem[Par79a]{parisi1979infinite}
Giorgio Parisi, \emph{Infinite number of order parameters for spin-glasses},
  Physical Review Letters \textbf{43} (1979), no.~23, 1754.

\bibitem[Par79b]{parisi1979toward}
\bysame, \emph{Toward a mean field theory for spin glasses}, Physics Letters A
  \textbf{73} (1979), no.~3, 203--205.

\bibitem[PWB20]{perry2020statistical}
Amelia Perry, Alexander~S Wein, and Afonso~S Bandeira, \emph{Statistical limits
  of spiked tensor models}, Annales de l'Institut Henri Poincar{\'e},
  Probabilit{\'e}s et Statistiques, vol.~56, Institut Henri Poincar{\'e}, 2020,
  pp.~230--264.

\bibitem[PY91]{papadimitriou1991optimization}
Christos~H Papadimitriou and Mihalis Yannakakis, \emph{Optimization,
  approximation, and complexity classes}, Journal of computer and system
  sciences \textbf{43} (1991), no.~3, 425--440.

\bibitem[RA05]{ruzmaikina2005characterization}
Anastasia Ruzmaikina and Michael Aizenman, \emph{Characterization of invariant
  measures at the leading edge for competing particle systems}, The Annals of
  Probability \textbf{33} (2005), no.~1, 82--113.

\bibitem[Ran11]{RanganGAMP}
S.~Rangan, \emph{{Generalized Approximate Message Passing for Estimation with
  Random Linear Mixing}}, IEEE Intl. Symp. on Inform. Theory (St. Perersbourg),
  August 2011.

\bibitem[RSF19]{rangan2019vector}
Sundeep Rangan, Philip Schniter, and Alyson~K Fletcher, \emph{Vector
  approximate message passing}, IEEE Transactions on Information Theory
  \textbf{65} (2019), no.~10, 6664--6684.

\bibitem[RU08]{RiU08}
Thomas~J. Richardson and R\"udiger Urbanke, \emph{{Modern Coding Theory}},
  Cambridge University Press, Cambridge, 2008.

\bibitem[SC19]{sur2019modern}
Pragya Sur and Emmanuel~J Cand{\`e}s, \emph{A modern maximum-likelihood theory
  for high-dimensional logistic regression}, Proceedings of the National
  Academy of Sciences \textbf{116} (2019), no.~29, 14516--14525.

\bibitem[Sch08]{schmidt2008replica}
Manuel~J Schmidt, \emph{Replica symmetry breaking at low temperatures}, Ph.F.
  Thesis, 2008.

\bibitem[Sen18]{sen2018optimization}
Subhabrata Sen, \emph{Optimization on sparse random hypergraphs and spin
  glasses}, Random Structures \& Algorithms \textbf{53} (2018), no.~3,
  504--536.

\bibitem[SK75]{sherrington1975solvable}
David Sherrington and Scott Kirkpatrick, \emph{Solvable model of a spin-glass},
  Physical review letters \textbf{35} (1975), no.~26, 1792.

\bibitem[Sub17]{subag2017complexity}
Eliran Subag, \emph{The complexity of spherical $ p $-spin models?a second
  moment approach}, The Annals of Probability \textbf{45} (2017), no.~5,
  3385--3450.

\bibitem[SVdPDMS11]{signoretto2011tensor}
Marco Signoretto, Raf Van~de Plas, Bart De~Moor, and Johan~AK Suykens,
  \emph{Tensor versus matrix completion: a comparison with application to
  spectral data}, IEEE Signal Processing Letters \textbf{18} (2011), no.~7,
  403--406.

\bibitem[Tal06a]{talagrand2006free}
Michel Talagrand, \emph{Free energy of the spherical mean field model},
  Probability theory and related fields \textbf{134} (2006), no.~3, 339--382.

\bibitem[Tal06b]{talagrand2006parisi}
\bysame, \emph{{The Parisi formula}}, Annals of Mathematics (2006), 221--263.

\bibitem[Tal10]{talagrand2010mean}
\bysame, \emph{Mean field models for spin glasses: Volume i: Basic examples},
  vol.~54, Springer Science \& Business Media, 2010.

\bibitem[Tal11]{talagrand2011mean}
\bysame, \emph{Mean field models for spin glasses: Volume ii: Advanced
  replica-symmetry and low temperature}, vol.~55, Springer Science \& Business
  Media, 2011.

\bibitem[Ver12]{Vershynin-CS}
R.~Vershynin, \emph{Introduction to the non-asymptotic analysis of random
  matrices}, Compressed Sensing: Theory and Applications (Y.C. Eldar and
  G.~Kutyniok, eds.), Cambridge University Press, 2012, pp.~210--268.

\bibitem[WJ08]{wainwright2008graphical}
Martin~J Wainwright and Michael~I Jordan, \emph{Graphical models, exponential
  families, and variational inference}, Foundations and Trends in Machine
  Learning \textbf{1} (2008), no.~1-2, 1--305.

\bibitem[ZB10]{zdeborova2010conjecture}
Lenka Zdeborov{\'a} and Stefan Boettcher, \emph{A conjecture on the maximum cut
  and bisection width in random regular graphs}, Journal of Statistical
  Mechanics: Theory and Experiment \textbf{2010} (2010), no.~02, P02020.

\end{thebibliography}
\end{document}